%% file: wstlspg_main_reorg.tex
\pdfoutput=1
\documentclass[computermodern]{elsarticle}
\input{packages}
\input{commands}
\graphicspath{{figures/}}
\setstackEOL{\cr}

\definecolor{matcha}{cmyk}{0.21, 0, 0.68, 0}

\journal{Elsevier}
\DeclareFontFamily{OT1}{pzc}{}
\DeclareFontShape{OT1}{pzc}{m}{it}{<-> s * [1.200] pzcmi7t}{}
\DeclareMathAlphabet{\mathpzc}{OT1}{pzc}{m}{it}

\DeclareSymbolFont{slant}{OT1}{\familydefault}{m}{sl}
\SetSymbolFont{slant}{bold}{OT1}{\familydefault}{bx}{sl}
\DeclareSymbolFontAlphabet{\mathsl}{slant}
\biboptions{sort&compress}
\begin{document}

\begin{frontmatter}

\title{Windowed space--time least-squares Petrov--Galerkin method for nonlinear model order reduction}

\author[mymainaddress]{Yukiko S. Shimizu\corref{mycorrespondingauthor}}\cortext[mycorrespondingauthor]{Corresponding author}
\ead{yshimiz@sandia.gov}
\author[mymainaddress]{Eric J. Parish}
\ead{ejparis@sandia.gov}
\address[mymainaddress]{Extreme-scale Data Science and Analytics, Sandia National Laboratories, Livermore, CA 94550, USA}

\begin{abstract}
This work presents the windowed space--time least-squares Petrov--Galerkin method (WST-LSPG) for model reduction of nonlinear parameterized dynamical systems. WST-LSPG is a generalization of the space--time least-squares Petrov--Galerkin method (ST-LSPG). The main drawback of ST-LSPG is that it requires solving a dense space--time system with a space--time basis that is calculated over the entire \emph{global} time domain, which can be unfeasible for large-scale applications. Instead of using a temporally-global space--time trial subspace and minimizing the discrete-in-time full-order model (FOM) residual over an entire time domain, the proposed WST-LSPG approach addresses this weakness by (1) dividing the time simulation into time windows, (2) devising a unique low-dimensional space--time trial subspace for each window, and (3) minimizing the discrete-in-time space--time residual of the dynamical system over each window. This formulation yields a problem with coupling confined within each window, but sequential across the windows. To enable high-fidelity trial subspaces characterized by a relatively minimal number of basis vectors, this work proposes constructing space--time bases using tensor decompositions for each window. WST-LSPG is equipped with hyper-reduction techniques to further reduce the computational cost. Numerical experiments for the one-dimensional Burgers' equation and the two-dimensional compressible Navier--Stokes equations for flow over a NACA 0012 airfoil demonstrate that WST-LSPG is superior to ST-LSPG in terms of accuracy and computational gain.
 
\end{abstract}

\end{frontmatter}

%

\section{Introduction}
%
%
Simulating large-scale nonlinear dynamical systems plays an essential role in numerous fields of science and engineering. However, executing such simulations often comes at the price of prohibitively high computational costs; this is the case when the dynamical system has a large state-space dimension. As a result, analysts often rely on low-cost reduced-order models (ROMs) that generate approximate solutions to the high-fidelity FOM. These ROMs can then be used more effectively, e.g., in many-query problems and time-critical applications, examples of which include uncertainty quantification, optimization, error estimation, and mesh adaptation.

Projection-based ROMs in particular can efficiently generate accurate approximate solutions at a low computational cost. These techniques are performed on a FOM, which can be either a nonlinear continuous-in-time set of ordinary differential equations (ODE) or a discrete-in-time set of ordinary difference equations (O$\Delta$E).  Generally speaking, projection-based ROMs operate by (1) restricting the state to live in a low-dimensional trial subspace by performing proper orthogonal decomposition (POD) and (2) executing a projection or residual minimization process. This process yields a ROM whose dimension is much less than that of the FOM. Projection-based ROMs have proven to be particularly effective for linear time-invariant (LTI) systems~\cite{wilcox_benner_rev,moore,roberts,GugercinIRKA,rovas_thesis}, as methods that account for, e.g., observability and controllability, $\mathcal{H}^2$ optimality, non-affine parametric dependence, and \textit{a posteriori} error bounds have been developed. While the development of projection-based ROMs for nonlinear dynamical systems has shown to be more challenging than that of LTI systems, significant advancements have been made in recent years. In particular, the maturation of ROMs based on variable transformations and lifting approaches~\cite{QIAN2020132401,doi:10.2514/1.J057791,KRATH2020109959}, residual minimization principles, space--time settings, and so-called ``hyper-reduction" techniques~\cite{everson_sirovich_gappy,eim,qdeim_drmac}, has enabled the construction of efficient and accurate ROMs of complex nonlinear dynamical systems.
This work seeks to expand upon these developments by introducing a novel space--time residual minimization approach that leverages the concept of windowing. To this end, this paper provides a brief review of the state-of-the-art residual minimization method, a study of the space--time model reduction approaches, and a discussion of the outstanding challenges.

Residual minimization approaches in reduced-order modeling operate by computing an approximate solution, which lies within a low-dimensional trial subspace that minimizes the residual of the FOM ODE or FOM O$\Delta$E~\cite{bui2007model, bui2008model, bui2008parametric, carlberg2011efficient, carlberg2011model, carlberg2013gnat, carlberg2017galerkin, legresley2000airfoil, abgrall2018model, parish2019windowed, collins2019output, collins2020petrov}. For model reduction of dynamical systems,  Refs.~\cite{bui2007model,carlberg2011efficient} first formulated this residual minimization problem by (1) restricting the state to live in a low-dimensional \textit{spatial} trial subspace and (2) sequentially minimizing the discrete-in-time residual (i.e., the residual arising after discretizing the FOM ODE in time) in a weighted $\ell^2$-norm at each time step. This formulation is now best known as the least-squares Petrov--Galerkin (LSPG) approach. LSPG has been demonstrated to be more robust than the classical Galerkin approach and has been successfully used for model reduction of numerous nonlinear dynamical systems~\cite{carlberg_thesis,bui_thesis,carlberg2013gnat,carlberg2017galerkin}. However, LSPG has several drawbacks. In addition to displaying a complex dependence on the time-discretization scheme~\cite{carlberg2017galerkin,parish2019windowed}, LSPG only reduces the \textit{spatial} (i.e., the state-space) dimension of a nonlinear dynamical system (this is a limitation of spatial-reduction-only ROMs in general) and has \emph{a posteriori} error bounds that grow exponentially in time~\cite{carlberg2017galerkin}. Thus, for a problem requiring a large number of temporal degrees of freedom (due to, e.g., disparate time scales, long time horizons), LSPG projection fails to yield an accurate ROM of an adequately small dimension. 

Space--time ROMs aim to overcome these latter two limitations by (1) reducing the spatial \textit{and} temporal dimensions of the FOM, and (2) executing a space--time projection or residual minimization process. A variety of space--time ROM formulations have been proposed in the literature, including approaches based on the reduced-basis method~\cite{urban2012new, urban2014improved, yano2014space, doi:10.1142/S0218202514500110}, the POD-Galerkin method~\cite{baumann2016space, volkwein2006algorithm}, and residual minimization principles~\cite{constantine2012residual, choi2019space}. The space--time LSPG (ST-LSPG) approach by Choi and Carlberg~\cite{choi2019space} is of particular interest to this work. ST-LSPG operates by (1) applying a time-discretization technique to the FOM ODE, (2) restricting the time-discrete space--time state (e.g., the discrete state at every time instance) to live in low-dimensional \textit{space--time} trial subspace, and (3) executing a space--time residual minimization process. In the space--time residual minimization problem, the goal is to obtain an approximate space--time state that minimizes the space--time residual in a weighted $\ell^2$-norm. ST-LSPG is preferable to LSPG due to its ability to (1) reduce both spatial and temporal degrees of freedom and (2) have more favorable theoretical properties; for example, error bounds do not grow exponentially in time and the space--time residual decreases monotonically with increasing basis dimension. Another appealing aspect of ST-LSPG is that, by virtue of its fully discrete formulation, it does not require the FOM to be a space--time model. Despite these advantages, ST-LSPG (and other space--time methods) suffers from several practical limitations. First, the computational cost of solving the system of equations emerging from a space--time discretization scales cubically with the number of space--time degrees of freedom; in contrast, the cost of standard spatial-reduction-only ROMs scales linearly in the number of temporal degrees of freedom. Thus, even though the temporally-global space--time dimension may be low, the computational cost associated with ST-LSPG can be high. Similarly, another disadvantage of ST-LSPG is the memory overhead. The ST-LSPG approach to solve the residual minimization problem requires computing the product between an $\nspacedof \ntimedof \times \nspacedof \ntimedof$ sparse matrix and an $\nspacedof \ntimedof \times n_{st}$ dense matrix, where $\nspacedof, \ntimedof$ and $n_{st}$ are the number of spatial degrees of freedom, temporal degrees of freedom, and space--time bases, respectively. This matrix-matrix product is cumbersome to compute and store for large high-fidelity problems. As a result, ST-LSPG is generally an impractical method without applying hyper-reduction in both space and time.           

Recently, the windowed least-squares (WLS) approach was developed to address several of these shortcomings. In contrast to (1) LSPG projection, which minimizes the FOM O$\Delta$E residual over each time step and (2) ST-LSPG projection, which minimizes the FOM O$\Delta$E residual over the entire time domain, WLS sequentially minimizes the FOM ODE residual over arbitrarily defined time windows. By sequentially minimizing the residual over arbitrarily defined time windows, WLS allows for fine-grained trade offs between computational cost and error. In Ref.~\cite{parish2019windowed}, WLS was formulated for two kinds of solution techniques (discretize then optimize and optimize then discretize) and two types of space–time trial subspaces: \textit{S-reduction} subspaces that associate with spatial dimension reduction, and \textit{ST-reduction} subspaces that associate with space–time dimension reduction. Ref.~\cite{parish2019windowed} showed that the limiting cases of WLS recover existing model reduction approaches. For instance, WLS with \textit{S-reduction} subspaces recovers LSPG when a discretize-then-optimize solution approach is employed, and the window size is set to be equivalent to the time step. Similarly, WLS with \textit{ST-reduction} subspaces recovers ST-LSPG when a discretize-then-optimize approach is employed, and the window size is set to be equivalent to the time domain. In Ref.\cite{parish2019windowed}, theoretical analyses and numerical experiments were carried out using WLS with \textit{S-reduction} subspaces. Critically, it was found that sequentially minimizing the residual over larger window sizes led to lower space--time residuals and more robust solutions, but not necessarily lower space--time $\ell^2$-errors.

Although WLS was formulated for space--time trial subspaces (ST-reduction subspaces) in Ref.~\cite{parish2019windowed}, the focus of the analyses and experiments was on spatial reduction (S-reduction) subspaces and on the impact of the window size over which the residual was minimized. A detailed investigation into the space--time model reduction---theoretical analyses, numerical experiments, and practical methods for constructing windowed space--time subspaces--- was not performed. This work on WST-LSPG fills this gap by extending the ST-LSPG approach to incorporate the concept of windows. The proposed WST-LSPG approach operates by minimizing the discrete-in-time space--time residual over arbitrarily defined time windows within low-dimensional space--time trial subspaces. WST-LSPG distinguishes itself from existing work in several ways. First, like WLS, WST-LSPG operates by minimizing a residual over arbitrarily defined time windows. However, contrary to Ref.~\cite{parish2019windowed}, the focus of WST-LSPG revolves around space--time model reduction. Second, in WST-LSPG, a sequence of time-local space--time bases are constructed within each window.
These space--time trial subspaces are constructed by (1) decomposing each window into sub-windows and (2) applying higher-order singular value decompositions~\cite{choi2019space} to time-local training data over each sub-window to generate a sequence of time-local space--time bases. A detailed study on the construction of these bases and the impact of the hyper-parameters that define them (e.g., window size, sub-window size) is conducted in this work. These time-local space--time bases enable WST-LSPG to produce accurate solutions at a low cost.

WST-LSPG displays commonalities with several existing efforts in the literature. First, as previously discussed, WST-LSPG is equivalent to the previously unexplored WLS with ST-reduction subspaces and is an extension of ST-LSPG. Second, the time-local space--time bases employed by ST-LSPG displays commonalities with the local reduced-order bases developed in Ref.~\cite{amsallem2012nonlinear,AmsZahWas15}. In Refs.~\cite{amsallem2012nonlinear,AmsZahWas15}, \emph{space-local} reduced bases are constructed by (1) executing an \textit{a priori} clustering step on the training data, and (2) performing POD on each cluster. The result of this process is a set of local trial subspaces. In the online stage, the state is restricted to live in a specific trial subspace, e.g., by identifying the cluster center nearest to the current state. The present work displays similarities to Refs.~\cite{amsallem2012nonlinear,AmsZahWas15} in that local bases are employed, but differs in that (1) each local basis is a space--time basis representing the state over a space--time (sub)window and (2) the local bases are assigned, \textit{a priori}, to represent the solution over a specific space--time interval; Refs.~\cite{amsallem2012nonlinear,AmsZahWas15} employ a distance-based algorithm to select the local basis at a given time instance. Lastly, the present work displays commonalities with domain decomposition ROMs (DD-ROMs)~\cite{maday_rbe,hoang2020domaindecomposition,IAPICHINO201263,phuong_huynh_knezevic_patera_2013}. In these DD-ROMs, the spatial domain is broken down into subdomains and bases tailored specifically to each subdomain are then employed. Compatibility between the various subdomains is then enforced via Lagrange multipliers. The present work displays commonalities to these approaches in that the time-domain is broken down into various subdomains, and time-local space--time bases are then employed over each subdomain. In the present case, however, solution algorithms are able to leverage the cylindrical nature of the space--time domain to avoid compatibility constraints.


Specific contributions of this work include:
\begin{itemize}
    \item Formulation of the WST-LSPG method for model reduction of dynamical systems
    \item A strategy for high-fidelity time-local space--time bases construction that incorporates both temporal domain decomposition and tensor decomposition
    \item Techniques to perform space--time hyper-reduction for WST-LSPG to further decrease the computational cost
    \item \emph{A priori} error bounds demonstrating that errors in WST-LSPG grow exponentially in the number of windows; in contrast, the Galerkin and LSPG methods are equipped with \textit{a priori} bounds that grow exponentially in the number of time steps.
    \item More favorable \emph{a posteriori} error bounds for WST-LSPG when compared to that of LSPG
    \item Numerical experiments demonstrating the effectiveness of WST-LSPG when applied to the one-dimensional Burgers' equation and the two-dimensional compressible Navier--Stokes equations
\end{itemize}
This paper is presented in the following way. Section~\ref{sec:FOM} describes the governing equations and the FOM. Section~\ref{sec:spaceROM} describes spatial projection-based model reduction and LSPG projection. Section~\ref{sec:spaceTimeLSPG} introduces the foundations of space--time projection-based model reduction. Section~\ref{sec:windowedSTLSPG} then introduces the proposed WST-LSPG projection and describes a method for constructing time-local bases over each window via tensor products. Section~\ref{sec:analysis} provides \emph{a priori} and \emph{a posteriori} error analysis for WST-LSPG, and Section~\ref{sec: experiments} provides numerical results. Section~\ref{sec: conclusions} provides conclusions and discusses future work.
\section{Full-order model}\label{sec:FOM}
This section introduces the FOM, which corresponds to a parameterized nonlinear dynamical system. Such a dynamical system can arise, for example, from the spatial discretization of a partial differential equation that depends on both space and time. Lastly, linear multistep methods are introduced and used to discretize the FOM in time. 

\subsection{Time-continuous representation}
The time-continuous FOM is defined as a parameterized set of nonlinear ODEs,
\begin{equation}
     \frac{d\state}{dt}\pa{t;\param} = \flux\pa{\state\pa{t;\param},t; \param},\quad\quad\state\pa{0;\param} =\solFullDiscreteArg{0}\pa{\param},\quad  t\in\bmat{0,\totaltime},
    \label{e:fom}
\end{equation}
where $\totaltime\in\RRplus{}$ is the final time, $\solFullDiscreteArg{0}:\paramDomain\rightarrow \RR{\nspacedof}$ are the initial conditions, and $\param\in\mathcal{D}$ are the parameters that belong to the input parameter domain, $\paramDomain\subseteq\RR{n_{\mu}}$. The number of parameters in the parameter domain is denoted $n_{\mu}$. as The time-dependent parameterized state is defined as $\state:\bmat{0,\totaltime}\times \paramDomain\rightarrow \RR{\nspacedof}$ and 
the velocity is defined as $\flux : \RR{\nspacedof} \times \bmat{0,\totaltime} \times \paramDomain \rightarrow\RR{\nspacedof}$. Lastly, in the remainder of this paper, the natural number domain is defined for any non-negative integer, $Y$, such that $\natNo\pa{Y}\defeq\Bmat{0,1,\ldots,Y\!\!-\!1}$.

\subsection{Time-discrete representation and linear multistep methods}
Linear multistep methods are considered to numerically solve the FOM ODE~\eqref{e:fom} in time. Linear multistep methods discretize the FOM ODE~\eqref{e:fom} in time such that the continuous-in-time state $\state\pa{t;\param}$ is approximated at $N_t$ time instances, $t^{n+1}, n \in \natNo\pa{N_t}$ by $\state\pa{t^{n+1};\param} \approx \solFullDiscreteArg{n+1}\pa{\param}$, where the discrete-in-time state is defined as $\solFullDiscreteArg{n+1} : \paramDomain \rightarrow \RR{\nspacedof}$, $n \in \mathbb{N}(N_t).$ The time step is defined as $\Delta \timeArg{n+1} = \timeArg{n+1}\!-\timeArg{n} \in \RR{+}$, $n \in \mathbb{N}(N_t)$. In this work, a fixed time step of constant $\dtArg{n+1}$ is used. 

Discretizing the FOM ODE~\eqref{e:fom} in time gives a set of O$\Delta$Es for each instance $n \in \mathbb{N}(N_t)$ that satisfy
\begin{align}
\resFullDiscreteArg{n+1}\pa{\solFullDiscreteArg{n+1}\pa{\param};\solFullDiscreteArg{n}\pa{\param},\ldots,\solFullDiscreteArg{n+1-\knArg{n+1}}\pa{\param},\param}= \zerobold,
\label{e:LMMresidual}
\end{align}
where $\kn$ denotes the width of the linear multistep stencil at the $n^{\mathrm{th}}$ time-instance. The discrete-in-time residual is defined as
\vspace{-3mm}
\begin{equation}
\begin{split}
  \resFullDiscreteArg{n}&:  \pa{\dummyFullDiscreteArg{n};\ldots,\dummyFullDiscreteArg{n - \kn},\param}
  \mapsto
	\sum_{j=0}^{\kn }\alpha_j^n\dummyFullDiscreteArg{n-j}
	-\dtArg{n}	\sum_{j=0}^{\kn }\beta_j^n\flux\pa{\dummyFullDiscreteArg{n - j},t^{n-j};\paramDummy},\\ 
&:\RR{\nspacedof}\times \cdots \times \RR{\nspacedof}\times\paramDomain\rightarrow\RR{\nspacedof}.
\end{split}
\label{e:residual}
\end{equation}
The coefficients $\alpha_j^n, \beta_j^n\!\in\!\RR{}$, $j \in \nat{ \kn + 1}$ define a particular linear multistep scheme where $\alpha_0^n \neq 0$ and $\sum\limits_{j=0}^{\kn}\alpha_j^n = 0$ is necessary for ensure consistency. 

\section{Spatial projection-based model reduction}\label{sec:spaceROM}

The spatial projection-based model reduction approach seeks to generate approximate solutions to the FOM O$\Delta$E\footnote{It is noted that projection-based model reduction approaches have been developed that seek to reduce the dimensionality of the FOM ODE~\eqref{e:fom}, as well as the FOM O$\Delta$E~\eqref{e:LMMresidual}. In this work, the discussion is restricted to methods that reduce the dimensionality of the FOM O$\Delta$E.} by (1) restricting the state at each time instance to live in a low-dimensional \textit{spatial trial subspace} and (2) performing a spatial projection or residual minimization process via a \emph{spatial test subspace}. This overarching method consists of two main phases: an \emph{offline} phase and an \emph{online} phase. First, this section describes the offline phase, which consists of describing how the trial subspaces are constructed. Next, this section outlines the standard LSPG projection and residual minimization problem, which is employed in the online phase. 


\subsection{Spatial trial subspaces for model reduction (offline phase)}
%
%
To reduce the dimensions of the FOM, spatial projection-based model reduction techniques seek to restrict the state solutions at each time instance to live in a low-dimensional \emph{affine trial subspace}.
This affine trial subspace is denoted as $\trialSSubspace$ and is given by
\begin{equation}\label{e:trialSsubspace}
    \trialSSubspace = \solref\pa{\param}+ \trialSSubspaceColumn\in \RR{N_s},
\end{equation}
where $\trialSSubspaceColumn$ is referred to as the invariant non-affine trial subspace or \emph{column space}, and $\solref : \paramDomain \rightarrow \RR{\nspacedof}$ is the spatial reference state, which describes the distance between the the column space and the affine trial subspace.\footnote{In projection-based model reduction, there is some freedom to choose which trial subspace to use, whether it be the column space, $\trialSSubspaceColumn$ or the affine subspace, $\trialSSubspace$, which is dependent on the chosen value for $\solref$. Employing $\trialSSubspace=\trialSSubspaceColumn$ (i.e., setting $\solref = \bm{0}$) as the trial subspace is a popular choice in model order reduction; however, employing an affine subspace with $\solref \neq\bm{0}$ has its benefits. For example, setting $\solref\pa{\param} =\solFullDiscreteArg{0}\pa{\param}$ guarantees that the initial conditions are contained in the trial subspace.}
The column space is taken to be spanned by a set of $\romdim$ orthonormal basis vectors,
\begin{equation*}
    \trialSSubspaceColumn \defeq
\mathrm{range}\pa{\sbasismat},\quad    \sbasismat \equiv \bmat{\basisvecspacei{0} & \cdots & \basisvecspacei{n_s-1}}\in \mathbb{V}_{\romdim}\pa{\RR{\nspacedof}},
\end{equation*}
where $\sbasismat$ is the spatial basis matrix and $\mathbb{V}_{\romdim}\pa{\RR{\nspacedof}}$ refers to the Stiefel manifold ($\mathrm{i.e.},\;\Bmat{\sbasismat\in \RR{N_s\times n_s} | \sbasismat^T\sbasismat =\bm{I}}$). In this work, the spatial basis is obtained by performing POD~\cite{berkooz_turbulence_pod, rathinam2003new}. The dimension of the affine trial subspace and the column space must be fewer than the total number of spatial degrees of freedom associated with the FOM, i.e., $\dim(\trialSSubspace)=\dim(\trialSSubspaceColumn) =\nbasisspace\leq\nspacedof$, where ideally, $\nbasisspace\ll\nspacedof$.
The spatial reference state is set to the initial conditions of the FOM, $\solref\pa{\param}=\solFullDiscreteArg{0}\pa{\param}$. One of the benefits of employing this reference state is that the initial conditions are ensured to live within the affine trial subspace and thus can be exactly enforced.

Spatial projection-based model reduction methods restrict the state to belong to the affine trial subspace~\eqref{e:trialSsubspace} and approximate the discrete-in-time state at each time instances $t^{n+1}$, $n\in \natNo\pa{N_t}$ as
\begin{equation}\label{e:spatial_approx_state}
    \solFullDiscreteArg{n+1}\pa{\param}\approx\solFullDiscreteApproxArg{n+1}\pa{\param}\equiv \solref\pa{\params}+\sbasismat\solFullDiscreteReducedArg{n+1}\pa{\param}\in \trialSSubspace,
\end{equation}
where $\solFullDiscreteReducedArg{n+1}: \mathcal{D}\rightarrow \RR{\romdim}$, $n\in \natNo\pa{N_t}$, are the discrete-in-time spatial reduced states and $\solFullDiscreteApproxArg{n+1}: \mathcal{D}\rightarrow \trialSSubspace$, $n\in \natNo\pa{N_t}$ are the approximate full states. Next, Section~\ref{s:spacetest} describes the test subspace and the spatial residual minimization problem used to solve for the spatial reduced states.

\subsection{Spatial test subspace for spatial model reduction (online phase)}\label{s:spacetest}
In order to describe spatial-projection-only reduced order models, the approximation of the state~\eqref{e:spatial_approx_state} is first substituted into the discrete-in-time residual of the O$\Delta$E~\eqref{e:LMMresidual} giving
\begin{equation*}
\resFullDiscreteArg{n+1}\pa{\solref\pa{\param}+\sbasismat\solFullDiscreteReducedArg{n+1}\pa{\param}; \solFullDiscreteApproxArg{n}\pa{\param},\ldots,\solFullDiscreteApproxArg{n+1-\kn}\,\pa{\param},\param}=\zerobold.
\label{eq:spatial_rom_overdetermined}
\end{equation*}
The above O$\Delta$E system comprises an over-determined system for the $\romdim$ reduced states.
To find a set of equations where a unique reduced state at each time instance is guaranteed, the O$\Delta$E residual is restricted to be $\weightmat^T \weightmat$ orthogonal to \emph{a test subspace}. In this work, $\weightmat \in \RR{z \times \nspacedof}$, $\nbasisspace\leq\weightmatRows\pa{\leq\ndof}$ is a spatial weighting matrix and is used primarily to enable hyper-reduction to further decrease the computational cost of model reduction techniques. Next, the test subspace, $\testSSubspaceColumn$, is defined as
\begin{equation}
\testSSubspaceColumn\defeq\mathrm{range}\pa{\stestmat},\quad \stestmat \in \RR{\nspacedof \times \romdim},
\end{equation}
where $\stestmat$ is referred to as the test space basis matrix. Restricting the O$\Delta$E residual to be orthogonal to the test subspace leads to the system,
\begin{equation}
 \stestmat^T \weightmat^T \weightmat \resFullDiscreteArg{n+1}\pa{\solref\pa{\params}+\sbasismat\solFullDiscreteReducedArg{n+1}\pa{\param}; \solFullDiscreteApproxArg{n}\pa{\param},\ldots, \solFullDiscreteApproxArg{n+1-k\pa{t^n}}\pa{\param},\param}=\zerobold.
\end{equation}
The model reduction ideology described above is the foundation of various reduced-order modeling methods. Arguably, the two most popular model reduction approaches are the Galerkin method and the LSPG method. In the Galerkin approach, the test subspace is set to be equivalent to the column space, $\testSSubspaceColumn = \trialSSubspaceColumn$. However, although the Galerkin method yields accurate approximate solutions for symmetric systems~\cite{carlberg_thesis,bui_thesis}, the approach is known to lack robustness for non-symmetric and non-coercive systems. The LSPG method offers a more robust alternative and is now described.

\subsubsection{Least-squares Petrov--Galerkin method}\label{sec:LSPGprojection}
LSPG projection~\cite{carlberg2011efficient, carlberg2017galerkin, carlberg2011model, carlberg2013gnat, bui2007model, bui2008model, bui2008parametric} differs from Galerkin projection in that the test subspace is no longer set to be equivalent to the column space. Instead, for $n\in \nat{N_t}$, LSPG employs the test basis,
\begin{equation*}
    \stestmat^{n+1}_{\mathrm{lspg}}\defeq \dfrac{\partial \resFullDiscreteArg{n+1}}{\partial \dummyFullDiscreteArg{n+1}}\pa{\solref(\params)+\sbasismat\solFullDiscreteReducedArg{n+1}\pa{\param}; \solFullDiscreteApproxArg{n}\pa{\param},\ldots,\solFullDiscreteApproxArg{n+1-\kn}\pa{\param},\param}\sbasismat.
    \label{e:testlspg}
\end{equation*}
LSPG projection is thus defined by
\begin{equation}\label{eq:lspg_projection}
     \left[\stestmat^{n+1}_{\mathrm{lspg}}\right]^T \weightmat^T \weightmat \resFullDiscreteArg{n+1}\pa{\solref\pa{\param}+\sbasismat\solFullDiscreteReducedArg{n+1}\pa{\param}; \solFullDiscreteApproxArg{n}\pa{\param},\ldots,\solFullDiscreteApproxArg{n+1-\kn}\pa{\param},\param}=\zerobold.
\end{equation}
%
Critically, LSPG projection can alternatively be expressed as a discrete residual minimization problem arising at each time instance,
\begin{equation} \label{eq:lspggnatReduced}
\solFullDiscreteReducedArg{n+1}\pa{\param}=\arg\min{\dummyFullDiscreteReduced\in \RR{n_s}}\Vmat{ \weightmat\resFullDiscreteArg{n+1}\pa{\solref\pa{\params}+\sbasismat\dummyFullDiscreteReduced;\solFullDiscreteApproxArg{n}\pa{\param},\ldots,\solFullDiscreteApproxArg{n+1-\kn}\pa{\param},\param}}_2^2,\quad n\in\natNo\pa{N_t}.
\end{equation}
The solution to the system~\eqref{eq:lspggnatReduced} is equivalent to that of~\eqref{eq:lspg_projection}. As a result, LSPG is optimal in the sense that at each time step, it computes the solution that minimizes the residual of the FOM O$\Delta$E at $n+1$ time step in the weighted $\ell^2$-norm induced by the weighting matrix, $\weightmat$.

LSPG projection has been shown to yield accurate approximate solutions to complex nonlinear dynamical systems arising from, for example, turbulent fluid dynamics and structural mechanics. However, LSPG is limited in that (1) it only reduces the spatial dimension of the FOM and (2) it is equipped with \textit{a priori} error bounds that grow exponentially in the number of time steps. Thus, for problems requiring many temporal degrees of freedom, LSPG can fail to yield accurate solutions with a sufficiently low computational cost. Space--time model reduction techniques attempt to address these shortcomings.

\section{Space--time projection-based model reduction}\label{sec:spaceTimeLSPG}

While spatial-only model reduction techniques have been shown to reduce the computational cost of approximating the FOM, they do not reduce the temporal dimension of the problem\footnote{It should be noted that, as compared to the FOM, spatial projection ROMs can often employ much larger stable time steps~\cite{bach2018stability}, thus in this sense spatial projection ROMs implicitly reduce the temporal dimension of the FOM.}. To decrease the computational cost further, space--time model reduction approaches seek to reduce the number of spatial \emph{and} temporal degrees of freedom. Space--time model reduction approaches operate by (1) seeking a space--time reduced state solution in a \textit{space--time subspace} and (2) executing a space--time projection process. This section outlines space--time model reduction.

\subsection{Space--time vector formulation}
To outline the space--time model reduction, a space--time formulation of the FOM O$\Delta$E is first provided. The discretization of the FOM ODE~\eqref{e:fom} satisfies the vectorized space--time system defined as 
\begin{equation}\label{e:fom_spacetime_vector}
\resFullDiscreteVector\pa{\solFullDiscreteVector\pa{\param};\param} = \bz,
\end{equation}
where the corresponding space--time state vector is given as
$
\solFullDiscreteVector: \param \mapsto  \bmat{\solFullDiscreteArg{1}\pa{\param}^{T} & \cdots & \solFullDiscreteArg{N_t}\pa{\param}^{T}}^{T}$,
$\solFullDiscreteVector: \paramDomain \rightarrow \RR{\nspacedof \ntimedof}$ 
and the space--time residual is given as
\begin{equation*}\label{e:fom_timediscrete_vector}
\begin{split}
	\resFullDiscreteVector &: \pa{\dumFullDiscreteVector;\param} \mapsto
\bmat{
\resFullDiscreteArg{1}\pa{ \dumFullDiscreteArg{1} ; \solFullDiscreteArg{0}(\params),\params}^{T} 
&\cdots& 
\resFullDiscreteArg{\ntimedof}\pa{ \dumFullDiscreteArg{\ntimedof} ; \dumFullDiscreteArg{\ntimedof-1} ,\ldots,\dumFullDiscreteArg{\ntimedof-k(t^{\ntimedof})} ,\params }^{T}
}^{T}, \\
&: \RR{\nspacedof \ntimedof }  \times \paramDomain \rightarrow \RR{\nspacedof \ntimedof},
\end{split}
\end{equation*}
where  $\dumFullDiscreteVector \equiv \bmat{\bmat{\dumFullDiscreteArg{1}}^T & \cdots & \bmat{\dumFullDiscreteArg{\ntimedof}}}^T$. This space--time vector formulation is used throughout this paper to describe the space--time model reduction technique.

\subsection{Space--time trial subspace for model reduction (offline)}\label{sec:spacetime_trialspace}
Space--time projection-based ROMs generate approximate solutions to the space--time system~\eqref{e:fom_spacetime_vector} by restricting the state to a low-dimensional space--time trial subspace. Analogous to spatial-only model reduction, space--time ROMs restrict the states to an \emph{affine space--time trial subspace}. The space--time approaches considered in this work generate approximate solutions, $\solFullDiscreteApproxVector\pa{\param}\pa{\approx \solFullDiscreteVector\pa{\param}}\in \trialSTSubspaceVector$, where the affine space--time trial subspace $\trialSTSubspaceVector$ is defined as
\begin{equation}
    \trialSTSubspaceVector = \solrefSTVector\pa{\params}+ \trialSTSubspaceColumnVector  \subset \RR{\nspacedof \ntimedof}.
    \label{e:spacetimeaffine}
\end{equation}
The \emph{space--time column space} is denoted as $\trialSTSubspaceColumnVector$, and the space--time reference state is denoted as $\solrefSTVector :\paramDomain \rightarrow  \RR{\nspacedof \ntimedof}$. The space--time reference state is set to be $\solrefSTVector \equiv \bmat{\bmat{\solFullDiscreteArg{0}\pa{\param}}^T\cdots \bmat{\solFullDiscreteArg{0}\pa{\param}}^T}^T$. The space--time column space is described as
\begin{equation*}
   \trialSTSubspaceColumnVector\defeq\mathrm{range}\pa{\stbasismat} \subset \RR{\nspacedof \ntimedof},\quad \stbasismatWindowArg{} \equiv \bmat{ \wstBasisVectorArg{0}{} & \cdots & \wstBasisVectorArg{{\romdimWSTArg{}-1}}{}} \in \RR{\nspacedof \numStepsInWindowArg{} \times \romdimWSTArg{}}
\end{equation*}
where $\stbasismat \in \RR{\nspacedof \ntimedof \times \romdimST}$ is the basis matrix for the space--time column space, $\trialSTSubspaceColumn$, and is constructed using space--time POD; the construction of space--time subspaces is discussed later in the manuscript. The dimension of the space--time affine trial subspace and space--time column space must be fewer than the total number of space--time degrees of freedom, i.e., $\mathrm{dim}\pa{\trialSTSubspaceVector}$ and $\mathrm{dim}\pa{\trialSTSubspaceColumnVector} = n_{st}\leq N_s N_t$. 
%
The approximate space--time state can thus be expressed in vector form as
\begin{equation}\label{e:st_approxsol_vector}
\solFullDiscreteVector\pa{\param}\approx\solFullDiscreteApproxVector\pa{\params}   =  \solrefSTVector\pa{\params}  +  \stbasismat \solFullDiscreteReducedArg{}\pa{\params}\in  \trialSTSubspaceVector.
\end{equation}

\subsection{Space--time test subspace for space--time least-squares Petrov--Galerkin method (online phase)}

Substituting the approximation~\eqref{e:st_approxsol_vector} into Eqn.~\eqref{e:fom_timediscrete_vector} gives the following modified space--time residuals,
\begin{equation}\label{e:fom_timediscrete_vector_romstate}
   \resFullDiscreteVector\pa{\solrefSTVector(\params)+\stbasismat \solFullDiscreteReducedArg{}\pa{\param};\param}=\zerobold.
\end{equation}
%
Eqn.~\eqref{e:fom_timediscrete_vector_romstate} comprises an overdetermined space--time system. To find a set of space--time equations where the space--time reduced state is guaranteed, the space--time O$\Delta$E residual~\eqref{e:fom_timediscrete_vector_romstate} is restricted to be $\weightmatst^T\weightmatst$ orthogonal to a space--time test space, where $\weightmatst\in\RR{\weightmatstRows\times\ndof\ntimedof}$ is a space--time weighting matrix (e.g., that enables hyper-reduction) and $\nbasisst\leq\weightmatstRows(\leq\ndof\ntimedof)$. Analogous to LSPG, ST-LSPG defines the test subspace as,
\begin{equation}\label{e:stlspg_testspace}
    \testSTSubspaceColumn\defeq\mathrm{range}\pa{\stestmat_{\mathrm{stlspg}}},
\end{equation}
where $\stestmat_{\mathrm{stlspg}}$ is the full rank space--time test space basis matrix. The space--time test basis is given as
\begin{align}
\stestmat_{\mathrm{stlspg}}:=& \dfrac{\partial \resFullDiscreteVector}{\partial \dumFullDiscreteVector}\pa{\solrefSTVector(\params)+\stbasismat \solFullDiscreteReducedArg{}\pa{\param};\param}\stbasismat\in\RR{N_s N_t\times n_{st}},
\end{align}
%
where $\frac{\partial \resFullDiscreteVector}{\partial \dumFullDiscreteVector}$ is the space--time Jacobian, whose sparsity pattern depends on the type of linear multistep scheme employed for time discretization. Restricting the space--time residual~\eqref{e:fom_timediscrete_vector_romstate} to be $\weightmatst^T\weightmatst$ orthogonal to the test space~\eqref{e:stlspg_testspace} yields the system,
 \begin{equation}\label{e:rom_timediscrete_vector_romstate}
 \stestmat_{\mathrm{stlspg}}^T \weightmatst^T\weightmatst   \resFullDiscreteVector\pa{\solrefSTVector(\params)+\stbasismat \solFullDiscreteReducedArg{}\pa{\param} ;   \param}=\zerobold.
\end{equation}
%
This setup is referred to as the ST-LSPG method.
ST-LSPG can equivalently be written as a minimization problem in a weighted $\ell^2$-norm as
\begin{equation} \label{e:t-lspgReducedLargerFull}
 \solFullDiscreteReducedArg{}\pa{\param} =
	\underset{
\dummyFullDiscreteReduced\in\RR{\romDimSpaceTime}
		}{\arg\min}
     \Vmat{\weightmatst
		\resFullDiscreteVector\pa{\solrefSTVector(\params)+\stbasismat\dummyFullDiscreteReduced;\param}
		}_2^2.
\end{equation}
In practice, the space--time reduced states can be obtained, for example, via the Gauss--Newton method. It is emphasized that different choices of the weighting matrix $\weightmatst$ can enable different hyper-reduction techniques, e.g., collocation, Gauss--Newton with approximated tensors.

\subsection{Outstanding challenges}
ST-LSPG offers distinct advantages over LSPG in that it (1) allows for both spatial and temporal dimension reduction and (2) is equipped with more favorable \textit{a priori} and \textit{a posteriori} error bounds~\cite{choi2019space}. In practice, however, ST-LSPG faces several significant issues. First, if hyper-reduction is not employed, ST-LSPG requires solving a temporally-global residual minimization problem whose residual vector is of size $\nspacedof \ntimedof$; storing the associated Jacobian and residual vectors for this residual minimization problem is not practical for real-world problems. Second, even if hyper-reduction is employed, the cost of solving the global minimization problem scales cubically with the number of space--time reduced-states. Thus, even though the global space–time dimension may be low, the computational cost associated with ST-LSPG can be high. Lastly, ST-LSPG employs a single space--time trial subspace that has support over all space and time. For complex systems that explore numerous regions in state-space, the dimensionality of the space--time trial subspace required for accurate solutions may be prohibitively large. The following section introduces the WST-LSPG approach to address these shortcomings.

\section{Windowed space--time projection based model reduction}\label{sec:windowedSTLSPG}
This section presents the proposed WST-LSPG approach. WST-LSPG first operates by partitioning the space--time domain into a series of non-overlapping temporal windows.. Over each window, a space--time state is then computed within a time-local low-dimensional space--time trial subspace that minimizes the space--time residual over each window. WST-LSPG offers two immediate advantages over ST-LSPG. First, by replacing the time-global residual minimization statement in ST-LSPG with a sequence of time-local residual minimization problems defined over each window, WST-LSPG overcomes the high-storage and computational complexity requirements associated with ST-LSPG. Second, by enabling the use of a piecewise linear space--time bases, WST-LSPG is equipped with subspaces that can more accurately represent the solution with fewer degrees of freedom. Lastly, it is emphasized that, in addition to being viewed as an extension of ST-LSPG, WST-LSPG can alternatively be derived from the WLS framework by using a ``discretize-then-optimize" solution approach and ``ST-reduction" trial subspaces~\cite{parish2019windowed}. The remainder of this section outlines the WST-LSPG approach. 

\subsection{Windowed problem setup}\label{s:wstlspgsetup}
\begin{figure}
    \centering
    \includegraphics[width = 1\textwidth]{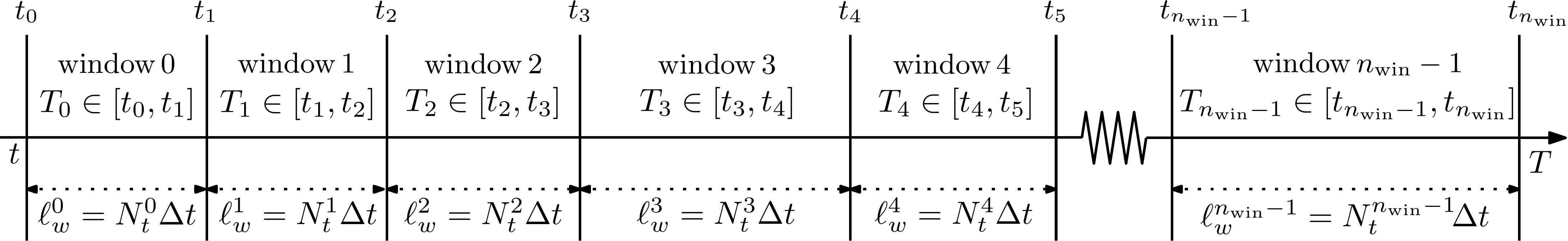}
    \caption{Window problem setup for the entire time simulation for $\numWindows$ where $N_t^k$ refers to the number of times steps, $T_k$ is the time interval in window $k$, and $t_k$ refers to the start (end) time of the $k^{\mathrm{th}}$ ($k-1^{\mathrm{st}}$) window.}
    \label{f:windowsImg}
\end{figure}
First, WST-LSPG partitions the discrete time domain $\{t^0,\ldots,t^{\numSteps}\}$ into $\numWindows$ windows of length $\windowLengthArg{k} = \numStepsInWindowArg{k}\Delta t$, $k \in \nat{\numWindows}$. Here, $\numStepsInWindowArg{k}$ denotes the number of time steps contained in the $k^{\mathrm{th}}$ window. 
WST-LSPG proceeds by (1) constructing low-dimensional space--time trial subspaces over each window and (2) sequentially solving a space--time residual minimization problem over each window. This decomposition of the time domain into windows is depicted in Figure~\ref{f:windowsImg}. Additionally, for notational purposes, a function that maps a window index to a time step index at the start of the window---excluding the initial conditions---is defined as
\begin{equation}\label{e:windowIndexMapper}
\windowIndexMapper: k \mapsto 
\begin{cases}
\sum\limits_{j=0}^{k-1} \numStepsInWindowArg{j} + 1 & 0 < k < \numWindows \\
1 & k = 0
\end{cases}.
\end{equation}

\subsection{Windowed space--time trial subspace}\label{sec:wst_trial_subspaces}
WST-LSPG employs piecewise linear space--time trial subspaces defined over each window. The time-global space--time subspace is thus defined as a direct sum of non-overlapping space--time trial subspaces defined for each window,
\begin{equation*}
    \trialSTSubspaceVector=\bigoplus\limits_{k=0}^{\numWindows-1}\trialSTSubspaceWinVector, 
\end{equation*}
where $\trialSTSubspaceWinVector \subset \RR{\nspacedof \numStepsInWindow}$ is the space--time subspace for the $k^{\mathrm{th}}$ window, $k\in \nat{\numWindows}$. 
%
The $k^{\mathrm{th}}$ windowed space--time subspace is defined as
\begin{equation}
   \trialSTSubspaceWinVector = \solrefSTVector^k\pa{\params}+\trialSTSubspaceColumnWinVector\in \RR{N_s N_t^k},
   \end{equation}
where  $\trialSTSubspaceColumnWinVector$ is the space--time column space on the $k^{\mathrm{th}}$ window, and $\solrefSTVector^k(\params) \in \RR{\nspacedof \numStepsInWindowArg{k}}$ is the windowed space--time reference state that is unique for each $k^{\mathrm{th}}$ window. The space--time column space for the $k^{\mathrm{th}}$ window, $k \in \nat{\numWindows}$, is defined as the range of the $k^{\mathrm{th}}$ space--time basis matrix,
\begin{equation*}
    \trialSTSubspaceColumnWinVector\defeq\mathrm{range}\pa{\stbasismatWindowArg{k}} \subseteq \RR{\nspacedof  \ntimedofWin} ,\quad \stbasismatWindowArg{k} \equiv \bmat{ \wstBasisVectorArg{0}{k} & \cdots & \wstBasisVectorArg{\romdimWSTArg{k}-1}{k}  } \in \RR{\nspacedof \numStepsInWindowArg{k} \times \romdimWSTArg{k}}
\end{equation*}
where $\stbasismatWindowArg{k}$ comprise $\romdimWSTArg{k}\leq \nspacedof \numStepsInWindowArg{k}$ basis vectors. 
Over the $k^{\mathrm{th}}$ window, $k \in \nat{\numWindows}$, WST-LSPG thus approximates the state as
\begin{equation}\label{e:romWstVectorDef}
\solFullDiscreteVector^k\pa{\param}\approx\solFullDiscreteApproxVectorArg{k}\pa{\param}= \solrefSTVector^k\pa{\params} +
\stbasismatWindowArg{k} \solFullDiscreteReducedArg{k}\pa{\param} \in \trialSTSubspaceWinVector,
\end{equation}
where $\solFullDiscreteReducedArg{k}: \paramDomain \rightarrow \RR{\romdimWSTArg{k}}$
denotes the space--time reduced states associated with the $k^{\mathrm{th}}$ window. Analogously to ST-LSPG, the space--time trial subspace for the temporal domain allows time dependence of the approximated solution to be moved from the space--time reduced states to the space--time bases.
\subsection{Construction of windowed space--time trial subspaces}
Various techniques exist for constructing the windowed space--time trial subspaces. The most straightforward approach for generating these subspaces is to apply space--time POD directly to the full space--time trajectories contained within the training data. However, as noted in~\cite{choi2019space}, this approach is limited in that (1) only a single space--time basis vector can be extracted from each training simulation, and (2) the space--time basis requires $\nspacedof \ntimedof \romdimST$ storage, which can be prohibitive for practical problems.  An alternative approach proposed in~\cite{choi2019space} is to construct space--time basis vectors via tensor products, in which each space--time basis vector is defined by a tensor product between a spatial basis vector and a temporal basis vector. While this approach enables extracting more basis vectors per training simulation and reduces the storage requirements, it is limited in that a single, global trial subspace is constructed. For problems whose dynamics trace through various regions in phase-space, as is often the case for parameterized nonlinear dynamical systems, the use of time-local subspaces can yield much accurate trial subspaces of a lower dimension~\cite{amsallem2012nonlinear}.  

This work develops time-local tailored windowed space--time trial subspaces via higher-order singular value decompositions (referred to here as tailored WST-HOSVD subspaces) to address these issues. Specifically, the techniques developed in Ref.~\cite{choi2019space} are extended by applying tensor-decomposition techniques to time-local training data to construct a piecewise-linear global space--time trial subspace. 

In what follows, it is helpful to view the windowed space--time state as a \textit{rank-2 tensor}, opposed to a \textit{rank-1 vector}. To this end, an ``unrolling" function that reshapes a rank-1 vector into a rank-2 tensor is defined as
\begin{equation}\label{e:unroll_func}
\begin{split}
\unrollfunc &: \dumFullDiscreteVector \mapsto
  	\bmat{\dumFullDiscreteArg{0} & \cdots & \dumFullDiscreteArg{N_2-1} },\\
&:\RR{N_1 N_2} \rightarrow \RR{N_1 \times N_2},
\end{split}
\end{equation}
for arbitrary $N_1$, $N_2$ and where $\dumFullDiscreteVector \equiv \bmat{ \bmat{\dumFullDiscreteArg{1}}^T & \cdots & \bmat{\dumFullDiscreteArg{N_2}}^T }^T$. The approximate space--time state over each window can be described in tensor form as
\begin{equation*}
    \solFullDiscreteApproxTensorArg{k}\pa{\param} \equiv \unrollfunc\pa{\solFullDiscreteApproxVectorArg{k}\pa{\params}} \in    \trialSTSubspaceWinTensor.
\end{equation*}
In tensor notation, the space--time affine trial subspace $\trialSTSubspaceWinTensor$ is given as the unrolling of $\trialSTSubspaceWinVector$,
\begin{equation*}
   \trialSTSubspaceWinTensor \defeq \solrefSTTensor^k\pa{\params}+\trialSTSubspaceColumnWinTensor \subseteq \RR{ \nspacedof \times \numStepsInWindowArg{k}},
\end{equation*}
with $ \solrefSTTensor^{k}(\params) \equiv \unrollfunc( \solrefSTVector^k(\params) )$ and $\trialSTSubspaceColumnWinTensor\defeq\mathrm{span}\!\Bmat{  \wstBasisTensorArg{i}{k}}_{i=0}^{n_{st,k}-1} \subseteq \RR{\nspacedof \times \ntimedofWin}$, where $ \wstBasisTensorArg{i}{k} \equiv \unrollfunc( \wstBasisVectorArg{i}{k})$. This rank-2 tensor formulation will be employed when describing the tailored WST-HOSVD subspaces in the next section.

%
\subsection{Construction of tailored WST-HOSVD subspaces}
\begin{figure}
    \centering
    \includegraphics[width = 1\textwidth]{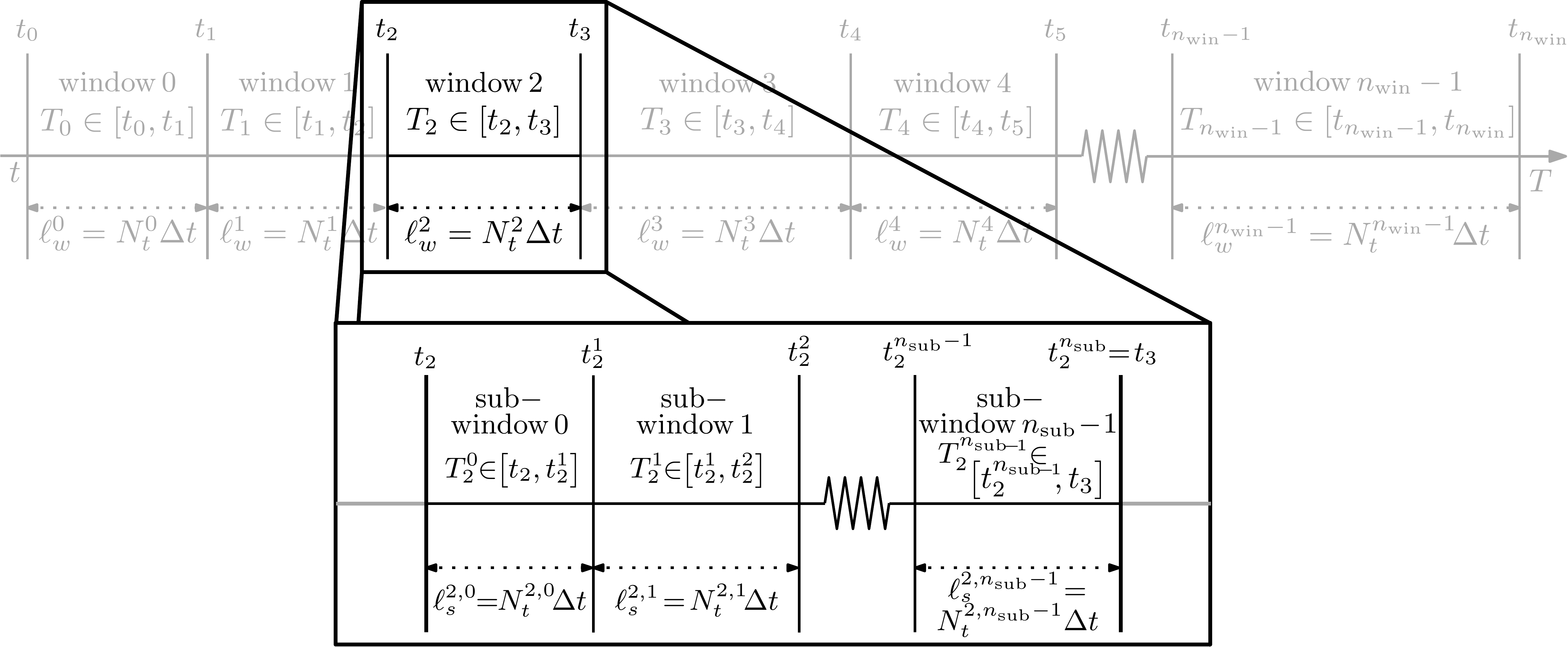}
    \caption{Sub-window problem setup for the entire time simulation for $\numWindows$, where $N_t^k$ refers to the number of times steps in window $k$, and $t_k^m$ refers to the time index at the start of the $m^{\mathrm{th}}$ sub-window of the $k$th window.}
    \label{f:windowsSubImg}
\end{figure}
First the tailored WST-HOSVD technique consists of decomposing the $k^{\mathrm{th}}$ window into $\numSubWindowsArg{k}$ sub-windows of length $\subWindowLengthArg{k}{m} = \numStepsInSubWindowArg{k}{m}\Delta t$, $k \in \nat{\numWindows}$, $m\in \nat{\numSubWindowsArg{k}}$;  $\numStepsInSubWindowArg{k}{m} \le \numStepsInWindowArg{k}$ is the number of time indices in the $m^{\mathrm{th}}$ sub-window of the $k^{\mathrm{th}}$ window. For simplicity, this work considers constant sub-window lengths. Figure~\ref{f:windowsSubImg} depicts this partitioning. Analogously to Eqn.~\eqref{e:windowIndexMapper}, a function that maps between a sub-window index and window index to the global time step index at the start of that sub-window is defined as
\begin{equation}\label{e:subWindowIndexMapper}
\subWindowIndexMapper: \pa{k,m} \mapsto  \windowIndexMapperArg{k}  +
\begin{cases}
\sum\limits_{n=0}^{m-1} \numStepsInSubWindowArg{k}{n} & 0 < m < \numStepsInWindowArg{k} \\
0 & m = 0 \\
\end{cases}.
\end{equation}
Next, the windowed space--time trial subspace over each time window is set to be the direct sum of non-overlapping space--time trial subspaces defined over each sub-window,
%
\begin{equation*}
    \trialSTSubspaceWinTensor=\bigoplus\limits_{m=0}^{\numSubWindows-1}\trialSTSubspaceSubWinTensor,
\end{equation*}
where $\trialSTSubspaceSubWinTensor$ is the affine space--time subspace for the $m^{\mathrm{th}}$ sub-window of the $k^{\mathrm{th}}$ window. 
Within each sub-window, the sub-windowed space--time affine subspace is defined as
\begin{equation}
    \trialSTSubspaceSubWinTensor\defeq \solrefST^{k,m}\pa{\params} +\trialSTSubspaceColumnSubWinTensor\subset \RR{\nspacedof \times \numStepsInSubWindowArg{k}{m}},
\end{equation}
where $\solrefST^{k,m}(\params) \in \RR{\nspacedof  \times \numStepsInSubWindowArg{k}{m}}$ and
$\trialSTSubspaceColumnSubWinTensor\defeq\mathrm{span}\!\Bmat{\overrightharpoon{\bm{\pi}}_i^{k,m}}_{i=1}^{n_{st}^{k,m}} \subset \RR{\nspacedof \times \ntimedofSubWin}$
are the sub-windowed space--time reference states and the column space over the $m^{\mathrm{th}}$ sub-window of the $k^{\mathrm{th}}$ window.
The discrete-in-time approximation for the state over the $m^{\mathrm{th}}$ sub-window on the $k^{\mathrm{th}}$ window is thus given as
\begin{equation}
\solFullDiscreteApproxTensorArg{k,m}\pa{\param}= \solrefWSTTensorArg{k}{m}\pa{\param} +
\sum\limits_{i=0}^{\romdimWSTSubWindowArg{k}{m}-1}\overrightharpoon{\bm{\pi}}^{k,m}_i\solFullDiscreteReducedArg{k,m}_i\pa{\param} \subset \RR{\nspacedof \times \numStepsInSubWindowArg{k}{m}},
\label{e:spacetimeExpansionnewST}
\end{equation}
where $\solFullDiscreteReducedArg{k,m} :\paramDomain \rightarrow  \RR{\romdimWSTSubWindowArg{k}{m}}$ are the space--time reduced states over the $m^{\mathrm{th}}$ sub-window of the $k^{\mathrm{th}}$ window.
%
\subsubsection{Tensor-product trial subspaces}\label{s:tailored_wsthosvd}
Building on the methods developed in~\cite{choi2019space}, the space--time bases over each sub-window, $\swstbasisvecTensorArg{k,m}{i}$, are obtained via \textit{tensor products}. Specifically, each basis tensor is given as a tensor product between a spatial basis vector and a temporal basis vector. These spatial and temporal basis vectors are obtained via higher-order singular value decompositions (HOSVD). To perform this process, first a training data set is required. Here, the training data comprise a collection of trajectories that are obtained by solving the FOM O$\Delta$E for a set of parameter instances $\paramDomainTrain \defeq \{\paramTrain{0}, \ldots,\paramTrain{\ntrain-1}\}\subset\paramDomain$. The set of parameter instances is chosen by, e.g., uniform sampling, Latin-hypercube sampling, greedy sampling. Thus, the training data comprise the solution snapshots $\solFullDiscreteArg{n+1}(\param^i_{\mathrm{train}})\in\RR{\nspacedof}$, $i \in \nat{\ntrain}$, $n \in \nat{N_t}$. 

Next, to construct bases over each sub-window, the training data are collected into a set of $ \sum\limits_{k=0}^{\numWindows-1} \numSubWindowsArg{k}$ tensors, defined here by $\stateTensorSubWindowArg{k}{m}$. Each of these three-way tensors contain training data shifted by the initial state over each window, i.e., the $a,b,c$ entry of the $k^{\mathrm{th}}$ tensor on the $m^{\mathrm{th}}$ sub-window is given by
\begin{equation}\label{e:threewaytensorInit}
\bmat{ \stateTensorSubWindowArg{k}{m}}_{abc} = 
\mathsl{u}^{\subWindowIndexMapperArg{k,m}+b}_a\pmat{\paramTrain{c}}-\mathsl{u}^{\subWindowIndexMapperArg{k,m}-1}_a\pmat{\paramTrain{c}},\quad
a\innat{\nspacedof},\ b\innat{\numStepsInSubWindowArg{k}{m}},\ c\innat{\ntrain},
\end{equation}
where $\mathsl{u}^b_a(\paramTrain{c}) \in \RR{}$ is used to denote the $a^{\mathrm{th}}$ entry of $\solFullDiscreteArg{b}$ evaluated at parameter instance $\paramTrain{c}$. 

Next, spatial and temporal basis vectors are found by applying HOSVDs to the three-way space--time snapshot tensor defined in Eqn~\eqref{e:threewaytensorInit}. Ref.~\cite{choi2019space} proposes three variations based on HOSVDs to calculate the (global) space--time basis. 
In each technique, the space--time basis is formed via the tensor product between a spatial and temporal basis. All techniques compute the spatial basis using POD~\cite{everson_sirovich_gappy}, but differ in terms of how the temporal basis is computed. The three approaches proposed in Ref.~\cite{choi2019space} for computing the temporal bases are: (1) fixed temporal subspace via Temporal-higher-order singular value decomposition (T-HOSVD), which performs POD directly on the mode-2 unfolding of the snapshot tensor, (2) fixed temporal subspace via space--time higher-order singular value decomposition (ST-HOSVD), which performs POD on a mode-2 tensor that is obtained by performing a tensor-matrix multiplication between the snapshot tensor and the basis matrix, and (3) tailored temporal subspace via space--time higher-order singular value decomposition (tailored ST-HOSVD), which performs POD on mode-2 tensors that are obtained by performing a tensor-matrix multiplication between the snapshot tensor and each spatial basiis vector. This work considers only the tailored ST-HOSVD method, and the following sections will detail their application for the construction of the tailored WST-HOSVD subspaces.  


\subsubsection{Spatial basis}\label{s:pod}
As noted previously, each space--time basis comprises a tensor product between a spatial basis and a temporal basis. Here, the spatial basis vectors are obtained via the standard POD approach. 
Also known as principle component analysis or the Karhunen-Lo\'{e}ve decomposition, POD computes a set of orthonormal basis vectors that enable optimal reconstruction of the training data in the least-squares sense. 
In the present context, POD is performed by applying the singular value decomposition (SVD) to the mode-1 unfolding of the three way snapshot tensor~\eqref{e:threewaytensorInit} in each sub-window. First, the mode-1 unfolding of the three way tensor~\eqref{e:threewaytensorInit} yields 
\begin{equation*}
\stateTensorUnfoldSubWindowArg{1}{k}{m}= \bmat{ \snapshotsSubWindowArg{k}{m}(\paramTrain{0}) & \cdots &  \snapshotsSubWindowArg{k}{m}(\paramTrain{\ntrain - 1}) }\in
\RR{\nspacedof\times\numStepsInSubWindowArg{k}{m}\ntrain},
\label{e:spacesnapshot}
\end{equation*}
where $\snapshotsSubWindowArg{k}{m}(\paramTrain{j}) = \bmat{ 
\solFullDiscreteArg{\subWindowIndexMapperArg{k,m}} (\paramTrain{j}) -  \solFullDiscreteArg{\subWindowIndexMapperArg{k,m}-1} (\paramTrain{j})& \cdots & \solFullDiscreteArg{\subWindowIndexMapperArg{k,m} + \numStepsInSubWindowArg{k}{m}}(\paramTrain{j}) - \solFullDiscreteArg{\subWindowIndexMapperArg{k,m}-1} (\paramTrain{j}) }$ are the state snapshots over the $m^{\mathrm{th}}$ sub-window of the $k^{\mathrm{th}}$ window.
Performing thin SVD on $\stateTensorUnfoldSubWindowArg{1}{k}{m}$ gives 
\vspace{-3mm}
\begin{equation}
\begin{split}
 \stateTensorUnfoldSubWindowArg{1}{k}{m} &= \leftsingmatspaceSubWindowArg{k}{m} \singvalmatspaceSubWindowArg{k}{m}	\bmat{\rightsingmatspaceSubWindowArg{k}{m}}^T \in \RR{N_s\times N_t^{k,m}\ntrain},\\
 \basisvecspaceSubWindowArg{i}{k}{m} &= \mathrm{col}_{i}\pa{\leftsingmatspaceSubWindowArg{k}{m}}, \quad i\in\natNo\pa{\romDimSpaceSubWindowArg{k}{m}},\\
 \sbasismatSubWindowArg{k}{m}  &=  \bmat{\basisvecspaceSubWindowArg{0}{k}{m} &\cdots & \basisvecspaceSubWindowArg{\romDimSpaceSubWindowArg{k}{m}-1}{k}{m}}\in \RR{N_s\times n_s^{k,m}},
 \end{split}
\end{equation}
where $\sbasismatSubWindowArg{k}{m}$ is the spatial basis for window $k$ and sub-window $m$. To determine the number of spatial bases required to represent the training data to a given error tolerance, one can look at the decay of the singular values, $\sigma_i^{k,m} = \bmat{\singvalmatspaceSubWindowArg{k}{m}}_{ii}$, in the following way:
\begin{equation*}
    e=\sum\limits_{i=0}^{\romDimSpaceSubWindowArg{k}{m}-1}\bmat{\sigma^{k,m}_i}^2\Big/\sum\limits_{i=0}^{ (\numStepsInSubWindowArg{k}{m}-1) n_{\mathrm{train}}}\bmat{\sigma_i^{k,m}}^2,
\end{equation*}
where $e$ refers to the fraction of statistical energy retained in the truncated basis. In practice, for a desired $e^*$, each column of $\leftsingmatspaceSubWindowArg{k}{m}$ is added to the bases and then the statistical energy fraction, $e$ is calculated and compared to the desired $e^*$. This process is continued until the following becomes true: $e\geq e^*$. Note that this fraction of statistical energy criterion is additionally used to calculate the number of temporal bases to use as well in the next sections.

\subsubsection{Temporal bases via tailored higher-order singular value decomposition}
After obtaining the spatial bases for the $m^{\mathrm{th}}$ sub-window, a corresponding temporal basis is required for each sub-window. Here the tailored temporal subspace via tailored ST-HOSVD. These approaches rely on mode-$p$ tensor-matrix products and mode-$p$ tensor-vector products. The mode-$p$ tensor-matrix product of a matrix $\bm{P} \in \RR{ I_p \times J}$ and tensor $\tensor \in \RR{I_1  \times \cdots \times I_p \times \cdots \times I_n}$ is denoted by $ \tensor \times_p \bm{P} \in \RR{I_1 \times \cdots \times I_{p-1} \times J \times I_{p+1} \times \cdots \times I_n}.$
Additionally, it is noted that the mode-$p$ unfolding of $\tensor \pa{\bm{P}}$ can be written as $\bm{X}_{(p)} \pa{\bm{P}} = \bm{P}^T \bm{X}_{(p)}.$ Similarly, the mode-$p$ tensor-vector product of a vector $\bm{p} \in \RR{I_p}$ and tensor $\tensor \in \RR{I_1 \times \cdots \times I_p \times \cdots \times I_n}$ is denoted by $\tensor \times_p \bm{p} \in \RR{I_1 \times \cdots \times I_{p-1} \times I_{p+1} \times \cdots \times I_n}.$

The tailored ST-HOSVD approach for sub-window $m$ creates a temporal snapshot matrix that is calculated for each spatial basis in sub-window $m$. The mode-$1$ tensor-vector product of $\stateTensorSubWindowArg{k}{m}$ and $\basisvecspaceSubWindowArg{i}{k}{m}$ for tailored ST-HOSVD is defined as $\stateTensorTSTHOSVDSubWindowArg{k}{m}{ \basisvecspaceSubWindowArg{i}{k}{m}} = \stateTensorSubWindowArg{k}{m} \times_1 \basisvecspaceSubWindowArg{i}{k}{m}$, $i \in \nat{\romDimSpaceSubWindowArg{k}{m}}$ such that the mode-1 unfolding of $\stateTensorTSTHOSVDSubWindowArg{k}{m}{ \basisvecspaceSubWindowArg{i}{k}{m}}$ is defined as
\begin{equation*}
\stateTensorSTHOSVDUnfoldSubWindowArg{1}{k}{m}{\basisvecspaceSubWindowArg{i}{k}{m} }  = \bmat{\bmat{ \basisvecspaceSubWindowArg{i}{k}{m}}^{T}\snapshotsSubWindowArg{k}{m}\pa{\paramTrain{0}} &
\cdots & \bmat{\basisvecspaceSubWindowArg{i}{k}{m}}^T\snapshotsSubWindowArg{k}{m}\pa{\paramTrain{\ntrain-1}} }\in
\RR{\numStepsInSubWindowArg{k}{m} \ntrain}.
\end{equation*}
Then the mode-2 unfolding of $\stateTensorTSTHOSVDSubWindowArg{k}{m}{\basisvecspaceSubWindowArg{i}{k}{m}}$, $i \in \nat{\romDimSpaceSubWindowArg{k}{m}}$ is defined as,
\begin{equation}
    \stateTensorTSTHOSVDUnfoldSubWindowArg{2}{k}{m}{\basisvecspaceSubWindowArg{i}{k}{m}}= \bmat{\bmat{\snapshotsSubWindowArg{k}{m}\pa{\paramTrain{0}}}^T \basisvecspaceSubWindowArg{i}{k}{m} &
\cdots &\bmat{ \snapshotsSubWindowArg{k}{m}\pa{\paramTrain{\ntrain-1}}}^T\basisvecspaceSubWindowArg{i}{k}{m}} \in
\RR{\numStepsInSubWindowArg{k}{m}\times \ntrain}.
\end{equation}
The temporal bases are then obtained via the singular value decomposition of $\stateTensorSTHOSVDUnfoldSubWindowArg{2}{k}{m}{\basisvecspaceSubWindowArg{i}{k}{m} }$, $i \in \nat{\romDimSpaceSubWindowArg{k}{m}}$, which gives
\vspace{-3mm}
\begin{equation}\label{e:tailoredSVD}
\begin{split}
\stateTensorSTHOSVDUnfoldSubWindowArg{2}{k}{m}{\basisvecspaceSubWindowArg{i}{k}{m} }
&=
\leftsingmattime\pa{\basisvecspaceSubWindowArg{i}{k}{m}} \singvalmattime\pa{\basisvecspaceSubWindowArg{i}{k}{m}}
\rightsingmattime\pa{\basisvecspaceSubWindowArg{i}{k}{m}}^T\in\RR{\numStepsInSubWindowArg{k}{m}\times\ntrain},\\
\basisvectimeij{i,k,m}{j} &=\mathrm{col}_j\left( \leftsingmattime\pa{\basisvecspaceSubWindowArg{i}{k}{m}}\right),\quad
	i\innat{\romDimSpaceSubWindowArg{k}{m}},\ j\innat{\nbasistime^{i,k,m}},\\
	 \bm{\psi}^{i,k,m}&=\bmat{\basispsi^{i,k,m}_0&\cdots& \basispsi^{i,k,m}_{n_t^{i,k,m}-1}}\in \mathbb{R}^{N_t \times n_t^{i,k,m}},
\end{split}
\end{equation}
where $n_t^{i,k,m}$ are the number of temporal basis vectors for the $i^{\mathrm{th}}$ spatial basis vector on the $m^{\mathrm{th}}$ sub-window of the $k^{\mathrm{th}}$ window.
This approach generates a tailored temporal subspace for each spatial basis vector. Additionally,
assuming $\numStepsInSubWindowArg{k}{m} \le \romDimSpaceSubWindowArg{k}{m} \ntrain$, the cost of computing the $\romDimSpaceSubWindowArg{k}{m}$ SVDs in \eqref{e:tailoredSVD} is significantly less than the cost incurred by ST-HOSVD. This reduction in cost results from the $\mathcal{O}\big(\text{min}(m^2 n, m n^2)\big)$ complexity of the SVD, where $m$ and $n$ are the number of rows and columns, respectively. However, the maximum
dimension of each temporal basis is limited to the number of
training parameter instances, i.e., $n_t^{i,k,m}\leq\ntrain$,
$i\innat{\nbasisspace}$.
This temporal subspace requires
$\sum\limits_{i=0}^{\nbasisspace^{k,m}-1}\nbasistime^{i,k,m}\numStepsInSubWindowArg{k}{m}=\nbasisst^{k,m}\ntimedof^{k,m}$ storage for each sub-window. Note that for each window $k$, the number of space--time bases is defined as $n^k_{st} = \sum\limits_{m=0}^{\numSubWindowsArg{k}-1}\sum\limits_{i=0}^{n_s^{k,m}-1} n_t^{i,k,m}$. Thus, for the tailored ST-HOSVD, each space--time basis is defined as
\begin{equation}
\overrightharpoon{\bm{\pi}}^{k,m}_i = \bm{\phi}_i^{k,m}\otimes\bm{\psi}^{i,k,m} \in \RR{\nspacedof \numStepsInSubWindowArg{k}{m}} \times \RR{n_t^{i,k,m}}.
\end{equation}
The space--time basis matrix over each window is then given as $\stbasismatSubWindowArg{k}{m} = \bmat{ \vectorizeNoArg( \overrightharpoon{\bm{\pi}}^{k,m}_0) & \cdots &  \vectorizeNoArg\pa{ \overrightharpoon{\bm{\pi}}^{k,m}_{\romdimWSTSubWindowArg{k}{m}-1}} },$ where $\vectorizeNoArg : \bmat{ \bm{v}^1 & \cdots & \bm{v}^{N_2} } \mapsto \bmat{\bmat{[\bm{v}^1}^T & \cdots & \bmat{\bm{v}^{N_2}}^T }^T$ is a vectorization function (i.e., the inverse of $\unrollfunc$).
%


\subsubsection{Selection of reference states and windowed basis}\label{s:reference_state}
In LSPG and ST--LSPG, the reference state is set to be the (parameterized) initial condition. This is made to be more challenging in WST-LSPG, as the initial condition into each window is not known \textit{a priori} when executing the ROM for a novel parameter instance. Here, with the exception of the first sub-window in the first window, the space--time reference state for each sub-window is set to be the solution from the previous time window, 
$$
\solrefWSTTensorArg{k}{m}\left(\bm{\mu}\right)  = \solFullDiscreteApproxArg{ \subWindowIndexMapperArg{k,m} -1}(\params) \otimes \onebold_{\numStepsInSubWindowArg{k}{m}},
$$
where $\onebold_{\numStepsInWindowArg{k}{m}} \in \{1\}^{\numStepsInSubWindowArg{k}{m}}$. This choice for the space--time reference state ensures continuity between the trial subspaces of each sub-window. In the first sub-window of the first window, the reference state is set to be the initial conditions,
$
\solrefWSTTensorArg{0}{0}\left(\bm{\mu}\right)  = \solFullDiscreteArg{0}(\params) \otimes \onebold_{\numStepsInSubWindowArg{0}{0}}
$.
The solution over each window is then expressed with the correct reference states as
\begin{equation}
\small
\renewcommand\arraystretch{2}
\bmat{
\begin{array}{c}
\\
\\
\solFullDiscreteApproxVector^k(\params)\\
\\
\\
\end{array}
}
=
\renewcommand\arraystretch{2}
\underbrace{\bmat{
\begin{array}{c}
\\
\\
\solFullDiscreteApproxArg{ \windowIndexMapperArg{k} -1}(\params) \otimes \onebold_{\numStepsInWindowArg{k}}\\
\\
\\
\end{array}
}}_{\let\scriptstyle\textstyle\solrefSTVector^{k}\textstyle(\params)
\setstacktabbedgap{2pt}
}
+
\renewcommand\arraystretch{2}
\underbrace{\bmat{
\begin{array}{ccccc}
    \hb \stbasismatSubWindowArg{k}{0}\h&\h& \h& \h\\
    \hb \stbasismatEndSubWindowArg{k}{0} \h&\hb \stbasismatSubWindowArg{k}{1} \h&  \h& \h \\
    \hb \stbasismatEndSubWindowArg{k}{0} \h& \hb \stbasismatEndSubWindowArg{k}{1} \h& \hb \stbasismatSubWindowArg{k}{2} \h&\h  \\
    \hb\vdots \h& \hb\vdots\h&\hb \vdots\h& \hb\ddots \h\\
    \hb \stbasismatEndSubWindowArg{k}{0} \h& \hb \stbasismatEndSubWindowArg{k}{1}\h&\hb \stbasismatEndSubWindowArg{k}{2} \h& \hb\cdots\h&\hb \stbasismatSubWindowArg{k}{\numSubWindowsArg{k}-1}\h\\ \end{array}
  }}_{\let\scriptstyle\textstyle{\stbasismat^{k}}}
  \renewcommand\arraystretch{1.65}
  \underbrace{\bmat{
  \begin{array}{c}
  \solFullDiscreteReducedArg{k,0}(\params)  \\
  \solFullDiscreteReducedArg{k,1}(\params)   \\
  \solFullDiscreteReducedArg{k,2}(\params) \\
  \vdots\\
  \solFullDiscreteReducedArg{k,\numSubWindowsArg{k}-1} \\
  \end{array}
  }}_{\let\scriptstyle\textstyle\solFullDiscreteReducedArg{k}\textstyle(\params)}.
  \label{e:wholethingFull}
\end{equation}
Setting the space--time reference state to be the solution from the previous sub-window leads to a modified basis matrix, $\stbasismat^k$, that is block lower triangular over a \textit{window} as seen in Eqn.~\eqref{e:wholethingFull} where
%
\begin{equation*}
\overline{\stbasismat}^{k,m} = \bmat{
\bmat{\text{row}_{\subWindowIndexMapperArg{k,m} + ( \numStepsInSubWindowArg{k}{m}-1) - \nspacedof }\pa{ \stbasismatSubWindowArg{k}{m}}}^T & \cdots &  
\bmat{\text{row}_{\subWindowIndexMapperArg{k,m} + \numStepsInSubWindowArg{k}{m} - 1}\pa{ \stbasismatSubWindowArg{k}{m}}}^T 
}^T \in \RR{\nspacedof \times \romdimWSTSubWindowArg{k}{m}}
\end{equation*}
comprise the space--time basis for the state at the last time step of each sub-window. 
It is important to note that due to this choice of an affine offset, the space--time basis for each window, $\stbasismat^k$ is not orthogonal. It is additionally noted that the number of bases contained within each window, $\romdimWSTArg{k}$, is potentially higher for $\ell_s<\ell_w$ than when $\ell_s=\ell_w$, making the problem more accurate at the expense of computational cost. This setup is explored in the numerical experiments.
\subsection{Windowed space--time least-squares Petrov--Galerkin method}
\begin{figure}[hp!]
\centering
\subfloat[$\ell_w=\ell_s$, $\numSubWindows=1$, space--time residual minimization occurs over window length of $\ell_s$\label{f:blue}]{\includegraphics[width = 0.9955\textwidth]{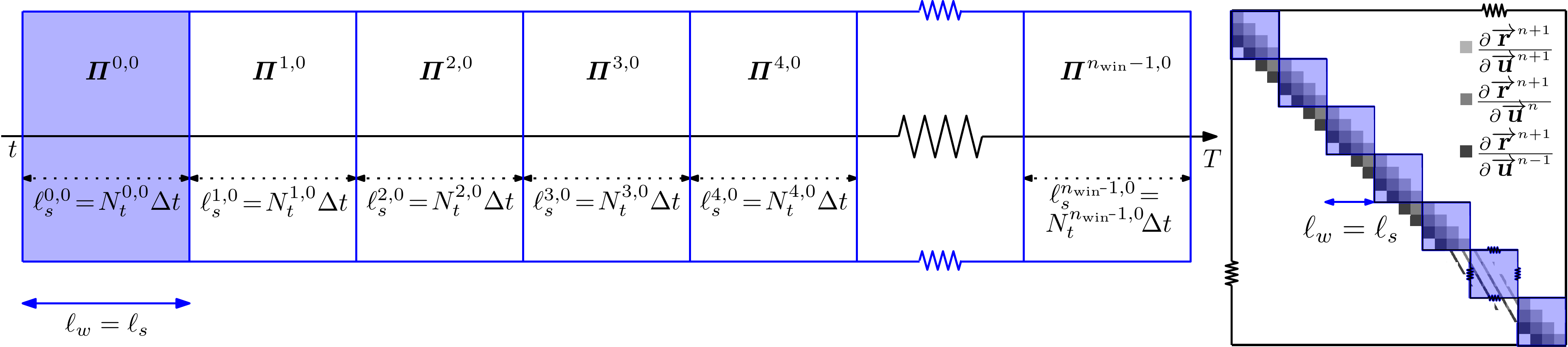}}\\\vspace{-3.25pt}
\subfloat[$\ell_w=2\ell_s$, $\numSubWindows=2$, space--time residual minimization occurs over window length of $2\ell_s$ \label{f:red}]{\includegraphics[width = 0.9955\textwidth]{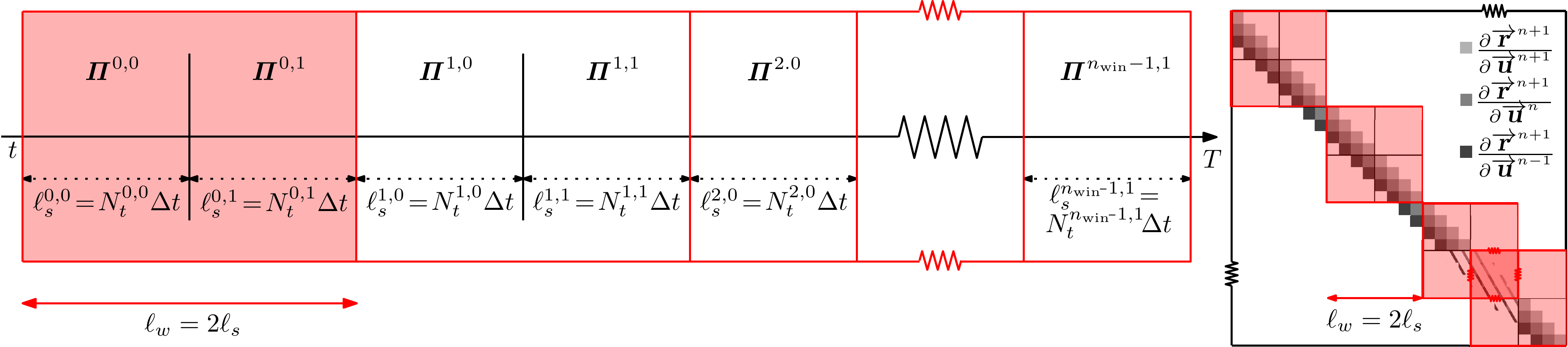}}\\\vspace{-3.25pt}
\subfloat[$\ell_w=3\ell_s$, $\numSubWindows=3$, space--time residual minimization occurs over window length of $3\ell_s$ \label{f:purple}]{\includegraphics[width = 0.9955\textwidth]{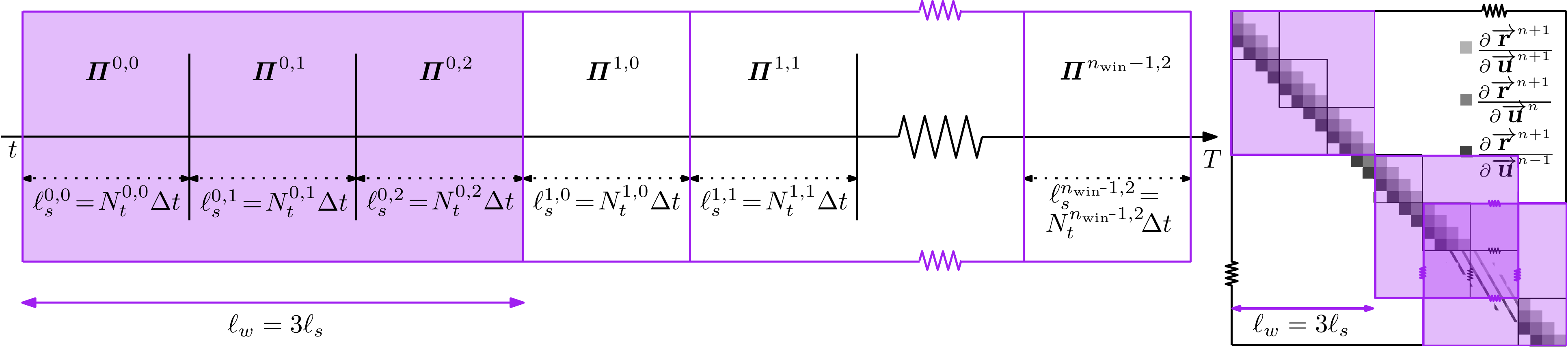}}\\\vspace{-3.25pt}
\subfloat[$\ell_w=4\ell_s$, $\numSubWindows=4$, space--time residual minimization occurs over window length of $4\ell_s$ \label{f:green}]{\includegraphics[width = 0.9955\textwidth]{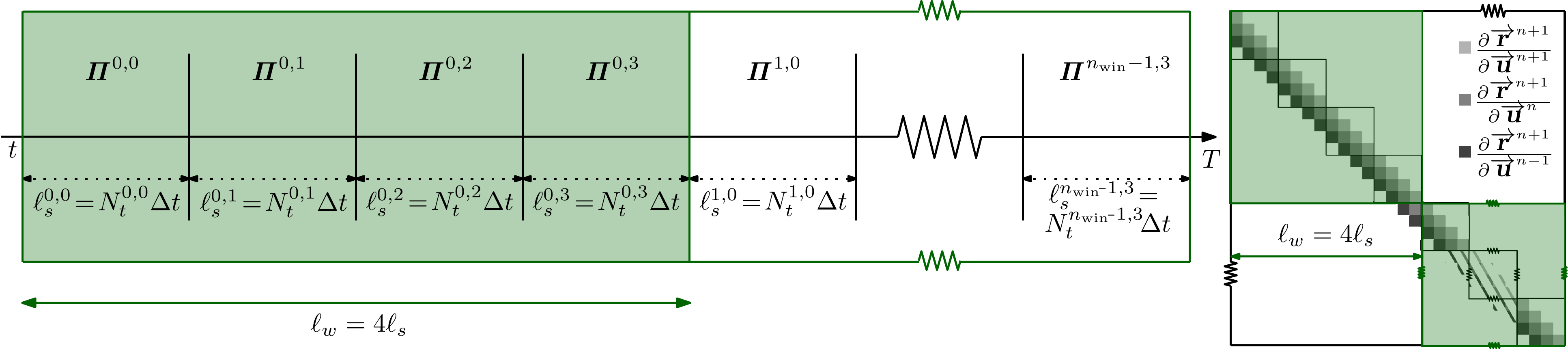}}\\\vspace{-3.25pt}
\subfloat[$\ell_w=T$, $\numSubWindows$ sub-windows,  space--time residual minimization occurs over window length of $T$ \label{f:orange}]{\includegraphics[width = 0.9955\textwidth]{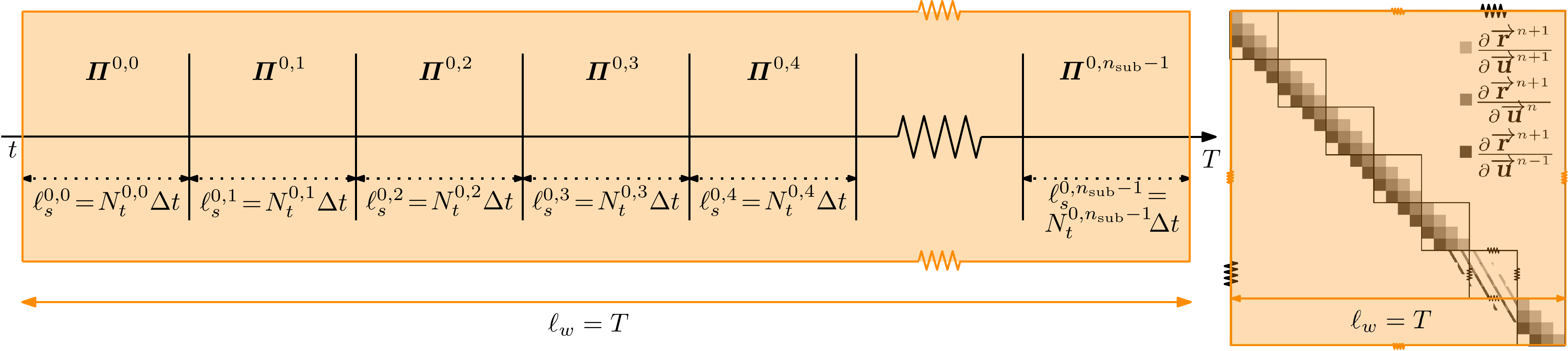}}
\caption{An example setup for a constant sub-window length, a varied window length $\ell_w$, and a set of Jacobians that are divided into $\numWindows$ space--time residual minimizations. As $\numWindows$ increases, the windowed space--time Jacobian becomes more decoupled or sparse compared to the full space--time Jacobian.}
\label{f:wstlspgsetup_window}
\end{figure}
%

Now that the trial subspaces have been defined for space--time model reduction, the proposed WST-LSPG approach proceeds by sequentially solving a system of algebraic equations defined by a Petrov--Galerkin projection on each window. To begin, the FOM O$\Delta$E solution satisfies\footnote{For simplicity, it is assumed that the size of the time stencil size at the start of each window is one, i.e., $\knArg{\windowIndexMapperArg{k}} = 1$.} 
\begin{equation}\label{e:fom_odeltae_spacetime_vector}
\resWindowDiscreteVectorArg{k} \pa{\solWindowDiscreteVectorArg{k}\pa{\param};\solFullDiscreteArg{\windowIndexMapperArg{k}-1}(\params) , \params} = \bz,
\end{equation}
over the $k^{\mathrm{th}}$ window, $k\in \nat{\numWindows}$, where the corresponding residual is defined as
\begin{equation}
\small{
\begin{aligned}
\resWindowDiscreteVectorArg{k} &: \pa{\dumWindowDiscreteVectorArg{k}; \dumFullDiscreteArg{0}, \params } \mapsto 
\bmat{
\resFullDiscreteArg{\windowIndexMapperArg{k}} \pa{\dumFullDiscreteArg{1}; \dumFullDiscreteArg{0}, \params} & \cdots & 
\resFullDiscreteArg{\windowIndexMapperArg{k} + \numStepsInWindowArg{k}}  \pa{\dumFullDiscreteArg{\numStepsInWindowArg{k}}; \dumFullDiscreteArg{\numStepsInWindowArg{k}-1},\ldots,\dumFullDiscreteArg{\numStepsInWindowArg{k} - \knArg{\windowIndexMapperArg{k}}}, \params}
} \\
&: \RR{\nspacedof \numStepsInWindowArg{k}} \times {\RR{\nspacedof}} \times \paramDomain \rightarrow \RR{\nspacedof  \numStepsInWindowArg{k}}
\end{aligned}
}
\end{equation}
with $\dumWindowDiscreteVectorArg{k} = \bmat{ \bmat{\dumFullDiscreteArg{1}}^T & \cdots & \bmat{\dumFullDiscreteArg{\numStepsInWindowArg{k}}}^T}^T$. Substituting Eqn.~\eqref{e:romWstVectorDef} into Eqn.~\eqref{e:fom_odeltae_spacetime_vector} yields the following overdetermined space--time system of algebraic equations over the $k^{\mathrm{th}}$ window, $k\in \nat{\numWindows}$, 
\begin{equation}\label{e:windowed_resid_at_rom_state}
  \resWindowDiscreteVectorArg{k}\pa{ \solrefSTVector^k\pa{\params}+\stbasismat^k\solFullDiscreteReducedArg{k}\pa{\params}; \solFullDiscreteApproxArg{ \windowIndexMapperArg{k} -1}\pa{\params} , \params} = \bz .
\end{equation}
To create a system where the reduced solution is unique, the residual over the $k^{\mathrm{th}}$ window~\eqref{e:windowed_resid_at_rom_state} is restricted to be ${\weightmatwstArg{k}}^T \weightmatwstArg{k}$ orthogonal to a space--time test space, where $\weightmatwstArg{k} \in \RR{\weightmatwstRowsArg{k} \times \nspacedof \numStepsInWindowArg{k}}$ with $\romdimWSTArg{k} \le \weightmatwstRowsArg{k} \le \nspacedof \numStepsInWindowArg{k}$ is a windowed space--time weighting matrix that enables hyper-reduction; the construction of such weighting matrices is outlined in Section~\ref{s:hyperreduction}. 

For this work, analogously to the spatial only model reduction technique, the space--time test subspace is defined as
\begin{equation}\label{e:wst_testspace}
    \testSTSubspaceColumn^k\defeq\mathrm{range}\pa{\stestmat^k_{\mathrm{wstlspg}}},
\end{equation}
where $\stestmat^k_{\mathrm{wstlspg}}$ is the full rank space--time test space basis matrix.
The space--time test basis is defined analogously to the LSPG test basis as
\begin{align}
\stestmat^k_{\mathrm{stlspg}}=& \dfrac{\partial \resWindowDiscreteVectorArg{k}}{\partial \dumWindowDiscreteVectorArg{k}}\pa{\solrefSTVector^k(\params)+\stbasismat^k \solFullDiscreteReducedArg{k}\pa{\param};   \solFullDiscreteApproxArg{ \windowIndexMapperArg{k} -1}(\params) , \params }\stbasismat^k\in\RR{N_s N_t\times n_{st}},
\end{align}
%
%
where $\frac{\partial \resFullDiscreteVector^k}{\partial \dumWindowDiscreteVectorArg{k}}$ is the space--time Jacobian whose non zero matrix structure is dependent on the linear multistep scheme that is used at all time steps. 
Restricting the windowed space--time residual~\eqref{e:windowed_resid_at_rom_state} to be ${\weightmatwstArg{k}}^T \weightmatwstArg{k}$ orthogonal to the space--time test space~\eqref{e:wst_testspace} yields the system of algebraic equations on the $k^{\mathrm{th}}$ window, $k \in \nat{\numWindows}$,
\begin{equation}\label{e:windowed_rom}
[ \stestmat^k_{\mathrm{wstlspg}}]^T {\weightmatwstArg{k}}^T \weightmatwstArg{k}  \resWindowDiscreteVectorArg{k}(   \solrefSTVector^k+\stbasismat^k\solFullDiscreteReducedArg{k}(\params) ; \solFullDiscreteApproxArg{ \windowIndexMapperArg{k} -1}(\params)  , \params) = \bz. 
\end{equation}
Thus, \methodAcronym\ operates by sequentially solving the system~\eqref{e:windowed_rom} over each window, $k \in \nat{ \numWindows }$. It is noted that the algebraic system defined by~\eqref{e:windowed_rom} can be equivalently written as a minimization problem in a weighted $\ell^2$-norm as
\begin{equation} \label{e:t-lspgReducedLargerFullST}
 \solFullDiscreteReducedArg{k}\pa{\param} =
	\underset{
\dummyFullDiscreteReduced\in\RR{\romdimWSTArg{k}}
		}{\arg\min}
    \left \| \weightmatwstArg{k} 
		\resWindowDiscreteVectorArg{k} \pa{\solrefSTVector^k(\params) + \wstbasismatArg{k}\dummyFullDiscreteReduced  ; \solFullDiscreteApproxArg{ \windowIndexMapperArg{k} -1}(\params) ,  \params } 
		\right \|_2^2 .
\end{equation}
%
Figure~\ref{f:wstlspgsetup_window} shows how the Jacobian is divided up for each window given a BDF2 temporal discretization and a constant sub-window length. The components of the Jacobian for window $k$ takes up a subset of the full space--time Jacobian for the original ST-LSPG. Figure~\ref{f:orange} refers to the space--time Jacobian setup for ST-LSPG where the black outlined boxes refer to the Jacobian of the states found within the designated sub-window length---adding sub-windows do not alter the structure of the space--time Jacobian. Figure~\ref{f:blue} shows what part of the ST-LSPG Jacobian is used when WST-LSPG is implemented for $\ell_w=\ell_s$. Figure~\ref{f:red} shows that the when $\ell_w=2\ell_s$, more of the ST-LSPG Jacobian in used and less information is loss between each time window. Figure~\ref{f:purple} and~\ref{f:green} show larger portions of the ST-LSPG Jacobian used for each window. Note that when $\ell_w<T$, the space--time Jacobian becomes more decoupled or sparse compared to the full ST-LSPG case. Finally, the procedure for WST-LSPG is presented in Algorithm~\ref{a:stlspg}.
\algdef{SE}[DOWHILE]{Do}{doWhile}{\algorithmicdo}[1]{\algorithmicwhile\ #1}%
\begin{figure}[t!]
\begin{center}
\begin{minipage}[t]{1\linewidth}
\begin{algorithm}[H]
\caption{Online: Gauss--Newton procedure for the WST-LSPG method}
\label{a:stlspg}
\textbf{Input:}
\multiline{for all time windows $k$: space--time basis, $\stbasismat^k\in \mathbb{R}^{N_s N_t \times n_{st}}$ \\
\hspace{37mm}initial guess for the reduced space--time states, $\solFullDiscreteReducedArg{0}\pa{\param}$ \\
\hspace{37mm}initial conditions, $\solrefSTTensor\pa{\param}$}\\
\textbf{Output:}
\multiline{for all time windows $k$: space--time reduced states,$\solFullDiscreteReducedArg{}\pa{\param}$\\
\hspace{37mm}total number of required Gauss--Newton iterations for $\param$, $N_{\mathrm{GN}}\pa{\param}$\\ 
\hspace{37mm}optional: rank-2 residual snapshot matrix for $\param$, $\bm{X}_{\bm{r}}$ }
\begin{algorithmic}[1]
\State Set the tolerance for convergence, $\epsilon$ 
\For{$k=0;\,k<\numWindows;\,k\scriptstyle^{++}$}
\State Set the Gauss--Newton iteration counter, $i=0$
\State Set $\Vmat{\bm{p}}_2= 0$ \Comment{Initialize the descent direction such that its $\ell^2$ norm is 0}
\State Set \texttt{converged}$ = 0$
\Do
\State $\solFullDiscreteApproxVector^{\pa{i}}\pa{\params}=  \solrefSTVector^k(\params)  +  \stbasismat^k \solFullDiscreteReducedArg{(i),k}(\params)$ \Comment{Calculate the approximate full space--time state}
\State Compute the space--time  
$\resFullDiscreteVector^{\pa{i}}\pa{\solFullDiscreteApproxVector^{(i)}(\params);\param} \in \mathbb{R}^z$
\State Compute the space--time $\bm{J}^{\,\pa{i}} = \frac{\partial \resFullDiscreteVector^{(i)}\pa{\solFullDiscreteApproxVector^{(i)}(\params);\param}}{\partial \solFullDiscreteApproxVector^{(i)}(\params)}\pa{\solFullDiscreteApproxVector^{(i)}(\params);\param)}\stbasismat^k \in\RR{z\times\nbasisst}$
\State Calculate $\overline{r}^{(i)}=\Vmat{\bmat{\bm{J}^{\,\pa{i}}}^T\resFullDiscreteVector^{\,\pa{i}}}_2$
\If{$ \overline{r}^{(i)}\Big/\overline{r}^{(0)} <\epsilon \;\mathrm{\mathbf{or}}\; \pa{\overline{r}^{(i)}<\epsilon\;\mathrm{\mathbf{and}}\;\Vmat{\bm{p}}_2<\epsilon}$}
\State \texttt{converged}=1
\Else
\State $\bm{J}^{\,\pa{i}} = \bm{Q}^{\,\pa{i}} \bm{R}^{\,\pa{i}} ,\; \bm{Q}^{\,\pa{i}} \in\mathbb{R}^{z\times n_{st}},\; \bm{R}^{\,\pa{i}} \in\mathbb{R}^{n_{st}\times n_{st}}$ \Comment{Compute thin qr decomposition}
\State $\bm{R}^{\,\pa{i}} \bm{p}= -\bmat{\bm{Q}^{\,\pa{i}} }^T\resFullDiscreteVector^{\,\pa{i}}$ \Comment{Solve the linear least-squares problem}
\State Perform line search to find $\lineSearchParamArg{i}$ or set $\lineSearchParamArg{i}=1$
\State $\solFullDiscreteReducedArg{\,\pa{i+1},k}\pa{\param}=\solFullDiscreteReducedArg{\,\pa{i},k}\pa{\param}+\lineSearchParamArg{i}\bm{p}$\Comment{Update the guess to the space--time reduced states}
\State $i\leftarrow i+1$ \Comment{Update the total number of Gauss--Newton iteration }
\EndIf
\doWhile{\texttt{converged} $==0$} \Comment{Continue with Gauss--Newton method until convergence criteria is met}
\State Set the total number of Gauss--Newton iterations, $N_{\mathrm{GN}}\pa{\param}=i+1$ 
\State Optional: Construct residual snapshots, \\ $\quad\quad\quad \bm{X}_{\bm{r}}\pa{\param}=\bmat{\resFullDiscreteTensor^{\pa{0}}\pa{\solFullDiscreteApproxTensor^{(0)}(\params);\param}& \cdots &\resFullDiscreteTensor^{(N_{\mathrm{GN}}\,\pa{\param}-1)}\pa{\solFullDiscreteApproxTensor^{(N_{\mathrm{GN}}\,\pa{\param}-1)}(\params);\param}}\in\RR{N_s\times N_t N_{\mathrm{GN}}\,\pa{\param}}$
\EndFor
\end{algorithmic}
\end{algorithm}
\end{minipage}
\end{center}
\end{figure}

\begin{remark}
(WST-LSPG recovers ST-LSPG) WST-LSPG recovers ST-LSPG when $\numWindows=1$ and $\numSubWindows=1$, which is equivalent to $\ell_w=T$ and $\ell_s = T$.
\end{remark} 
\subsubsection{Initial guess for Gauss--Newton Solver in WST-LSPG}
Employing the Gauss--Newton algorithm to solve the nonlinear least-squares problem requires the specification of an initial guess to the solution. In LSPG, the initial guess at a given time instance $t^n$ can be
set to the solution from the previous time instance, i.e.,
$\solFullDiscreteReducedArg{n(0)}\pa{\param} = \solFullDiscreteReducedArg{n-1}\pa{\param}$.
This choice typically leads
to rapid convergence due to the fact that the state undergoes limited
variation between time instances, particularly for small time steps $\dtArg{n+1}$.


This work employs a similar approach to that proposed in Ref.~\cite{choi2019space} and uses least-squares regression for the initial guess for the residual minimization problem over each window. However, the choice of space--time reference state detailed in Section~\ref{s:reference_state} results in several complications. First, unlike the ST-LSPG, the space--time basis is not necessarily orthogonal, such that $\stbasismat^{+}\neq\stbasismat^T$. Taking the pseudo inverse of a large windowed-space time basis can be expensive for large systems. The second issue is that, due to the choice of affine offset, the windowed space--time basis depends on the ROM solution throughout the window. As a result, the projection onto the trial subspace depends on the ROM solution. Thus, for the purpose of finding an initial guess for WST-LSPG, the training data are \textit{approximately} projected onto the trial space by employing the approximate space--time basis over each window $\widetilde{\stbasismat}^k = \mathrm{blockdiag}\pa{\stbasismat^k}$
where $\mathrm{blockdiag}\pa{\stbasismat^k}$ is a function that returns only the block diagonal entries of $\stbasismat^k$.  Thus, the training data for the initial guess is given by
\begin{equation}
\bmat{
\solFullDiscreteReducedArg{k}\pa{\paramTrain{0}} & \cdots & \solFullDiscreteReducedArg{k}\pa{\paramTrain{\ntrain-1}}
}
 = \bmat{\widetilde{\stbasismat}^{k}}^+\bmat{\snapshotmatparam^{\!\!\!\!k}}^T
\end{equation}
where $\snapshotmatparam^{\!\!\!\!k}$ denotes the mode-3 unfolding of the three-way tensor over the $k^{\mathrm{th}}$ window~\eqref{e:threewaytensorInit} (with one sub-window, i.e., $\snapshotmatparam^{\!\!\!\!k}$ contain all the training data over the $k^{\mathrm{th}}$ window). In the online phase, the initial guess $\solFullDiscreteReducedArg{k(0)}\pa{\param}$ is then computed via least-squares regression.
This approximation was found to produce adequate initial guesses for the Gauss--Newton solver.
\subsection{Hyper-reduction for WST-LSPG}\label{s:hyperreduction}
When the space--time weighting matrix is set to be the identity matrix, $\weightmatwstArg{k} = \bm{I}$, the model reduction process, for example, the Gauss--Newton solver, requires that all $\nspacedof \ntimedof$ elements of the space--time residual and space--time Jacobian be computed. Although $\weightmatwstArg{k} = \bm{I}$ may produce the most accurate solution, implementing this for high-fidelity nonlinear problems will lead to high computational costs. 
In order to fully realize the benefits of model reduction for space--time, it is best to look at hyper-reduction techniques that offer alternative definitions for the weighting matrix, $\weightmatwstArg{k}$. In this work, the Gauss--Newton with Approximated Tensors (GNAT) method~\cite{carlberg2013gnat} is employed for hyper-reduction in each window $k$. GNAT operates by sampling the space--time residual at a subset of the $\nspacedof \ntimedof^k$ elements (this subset is commonly referred to as the space--time sample mesh, $\bm{Z}^k$, ~\cite{carlberg_thesis}) and performing a least-squares reconstruction of the space--time residual via a space--time residual basis and the Gappy POD method~\cite{everson_sirovich_gappy}.
The hyper-reduction approach employed here for WST-LSPG using GNAT (WST-GNAT) involves the following steps: (1) compute a snapshot matrix of the residuals for each sub-window, (2) construct a basis for the residual over each sub-window by applying the tailored WST-HOSVD approach to the residual snapshot matrix, (3) compute the sample mesh $\bm{Z}^k$ via the greedy sampling of temporal then spatial indices method~\cite{carlberg2013gnat}. These ingredients are now described.

\subsubsection{Construction of the windowed space-time residual basis}

The GNAT method requires the calculation of a space--time residual basis for the least-squares reconstruction of the space--time residual. This calculation is done similarly to the way the spatial and temporal state bases are calculated. For the space--time residual basis, the tailored WST-HOSVD method used previously to calculate the space--time basis for the states is used.  

The space--time residual basis for window $k$ and sub-window $m$, $\stbasismat_r^{k,m}$, is found by collecting the Gauss--Newton space--time residuals for $n^r_{\mathrm{train}}$ parameters in the parameter design space to create the space--time residual snapshots. The parameter design space for the residual snapshots, $\paramTrain{r}$ can be the same design parameter space used for the state snapshots, $\paramTrain{r}=\paramTrain{}$. To find the residual snapshots, Algorithm~\ref{a:stlspg} is executed for all $n^r_{\mathrm{train}}$ parameters. Applying Algorithm~\ref{a:stlspg} to each $q^{th}$ parameter where $q\innat{\ntrain^r}$, outputs the number of Gauss--Newton iterations required to solve for the reduced space--time states at $\paramTrain{r,q}$, denoted as $N_{\mathrm{GN}}\pa{\paramTrain{r,q}}$ and the corresponding space--time residual snapshots over each sub-window, $\bm{X}^{k,m}_{\bm{r}}\pa{\paramTrain{r,q}}$.

After acquiring the space--time residual snapshots, $\bm{X}^{k,m}_{\bm{r}}\pa{\paramTrain{r,q}}$, at every training parameter, the residual snapshots can be written as a three-way residual tensor over each sub-window, $\tensorSub{\bm{r}}^{k,m} \in \RR{\nspacedof \times \numStepsInSubWindowArg{k}{m} \times n_{\mathrm{res}} }$,
%
where the final number of total Gauss--Newton iterations required for convergence, for all parameter instances, is denoted as $n_{\mathrm{res}}=\sum\limits_{q=1}^{n^r_{\mathrm{train}}}N_{\mathrm{GN}}\pa{\param_{\mathrm{train}}^{r,q}}$. Next, the POD of the mode-1 unfolding of $\tensorSub{\bm{r}}^{k,m}$ 
gives the spatial residual basis, $\sbasismat_r^{k,m}=\bmat{\basisvecspacei{r,0}^{k,m}&\cdots & \basisvecspacei{r,\romDimResidSpaceSubWindowArg{k}{m}-1}^{k,m}}$, where $\romDimResidSpaceSubWindowArg{k}{m}$ are the number of spatial residual basis vectors over the $m^{\mathrm{th}}$ sub-window of the $k^{\mathrm{th}}$ window. To find the temporal residual basis, the tailored ST-HOSVD method is used again, which acts upon the mode-$2$ of $\tensorSub{\bm{r}}^{k,m}$, $\snapshotsResWindowSubWindow\pa{\basisvecspacei{r,i}^{k,m}}$ for $i \in \nat{\romDimResidSpaceSubWindowArg{k}{m}}$. Performing POD on $\snapshotsResWindowSubWindow\pa{\basisvecspacei{r,i}^{k,m}}$ gives
%
\vspace{-3mm}
\begin{equation}
\begin{split}
\snapshotsResWindowSubWindow\pa{\basisvecspacei{r,i}^{k,m}} &=
\bm{U}_{r,t}^{k,m}(\basisvecspacei{r,i}^{k,m}) \bm{\Sigma}_{r,t}^{k,m}(\basisvecspacei{r,i}^{k,m})
\bm{V}_{r,t}^{k,m}(\basisvecspacei{r,i}^{k,m})^T\in\RR{\numStepsInSubWindowArg{k}{m}\times n_{\mathrm{res}}}
	\\
         \basisvectimeij{i,k,m}{r,j} &= \mathrm{col}_j\!\pa{\bm{U}_{r,t}^{k,m}\pmat{\basisvecspacei{r,i}}},\quad
	j\innat{\romDimResidTimeSubWindowArg{i}{k}{m}},\\
	 \bm{\psi}^{i,k,m}_{r,j}&=
	 \bmat{\basispsi^{i,k,m}_{r,0} &\ldots&\basispsi^{i,k,m}_{r,n_t^{i,k,m}-1}}
	 \in \mathbb{R}^{\numStepsInSubWindowArg{k}{m} \times \romDimResidTimeSubWindowArg{i}{k}{m}},
\end{split}
\end{equation}
where $\romDimResidTimeSubWindowArg{i}{k}{m}$ are the number of temporal basis vectors retained for the $i^{th}$ spatial vector over the $m^{th}$ sub-window of the $k^{\mathrm{th}}$ window.
The space--time residual basis is then calculated by first performing
\begin{equation*}
    \overrightharpoon{\bm{\pi}}_{r,i}^{k,m} = \basisvecspace_{r,i}^{k,m}\otimes \bm{\psi}_{r}^{i,k,m}\in \RR{\nspacedof \numStepsInSubWindowArg{k}{m}}\times \RR{n_{t,r}^{i,k,m}}.
\end{equation*}
Next the vectorization of the components of the space--time residual basis is found to be
\begin{equation*}
 \overrightarrow{\bm{\pi}}_{r,i}^{k,m} =\mathrm{vec}\pa{\overrightharpoon{\bm{\pi}}_{r,i}^{k,m} },\quad
    \bm{\Pi}_r^{k,m} =
       \bmat{
            \overrightarrow{\bm{\pi}}_{r,0}^{k,m}&\ldots& \overrightarrow{\bm{\pi}}_{r,n_{st,r}^{k,m}-1}^{k,m}
        }
    \in \mathbb{R}^{N_s N_t\times n_{st,r}^{k,m}}.
\end{equation*}
Next, the orthonormal space--time residual basis, $\stbasismat_r^{k,m}$, is found by taking the $QR$ factorization of $\bm{\Pi}_r^{k,m}$, $\stbasismat_r^{k,m}\rightarrow \bm{\Pi}_r^{k,m} = \stbasismat_r^{k,m} \bm{R}^{k,m}$. Finally, the full space--time residual basis for the entire window $k$ is defined as $\stbasismat_r^{k} = \mathrm{blockdiag}\pa{ \stbasismat_r^{k,0} , \ldots , \stbasismat_r^{k,\numSubWindows-1} } \in \mathbb{V}_{ \residualBasisDimWSTArg{k} }( \RR{\nspacedof \numStepsInWindowArg{k}} )$, where $n^k_{st,r}=\sum\limits_{m=0}^{\numSubWindows^k-1} n^{k,m}_{st,r}$. The space--time residual basis, $\stbasismat_r^{k}$ is used to calculate the sample mesh and perform hyper-reduction for each window.

\subsubsection{Sample mesh}
Hyper-reduction techniques such as GNAT operate by computing the space--time residual at a subset of the $\nspacedof \ntimedof^{\!\!k}$ elements for each window; this subset of elements is commonly referred to as the sample mesh. Mathematically, the sample mesh can be described by the sampling matrix, $\bm{Z}^k \in \{0,1\}^{ z \times \nspacedof \ntimedof^{\!\!k}}$, with $z \le \nspacedof \ntimedof^{\!\!k}$. In this work, the sequential greedy sampling of temporal then spatial indices (SGSTSI) algorithm is used to construct space--time sampling matrices defined over each window; SGSTSI is one of three methods proposed by Choi et. al~\cite{choi2019space} for space--time hyper-reduction. SGSTSI computes the space--time sample mesh as a Cartesian product of spatial indices and temporal indices, both of which are obtained in a greedy manner. The result of this process is a space--time sampling matrix defined over each window, $\bm{Z}^k$, $k \in \nat{\numWindows}$. 

 \subsubsection{Windowed space--time Gauss--Newton with Approximated Tensors}
GNAT is a hyper-reduction approach that operates by (1) sampling space--time elements of the residual on the sample mesh and (2) performing a least-squares reconstruction of the full space--time residual over the entire window $k$ via Gappy POD~\cite{everson_sirovich_gappy,eim,deim}. WST-GNAT approximates the space--time residual over the $k^{\mathrm{th}}$ window as 
\begin{equation*}
    \resFullDiscreteVector^k\pa{\solFullDiscreteApproxVectorArg{k}\pa{\param}; \solFullDiscreteApproxArg{\windowIndexMapperArg{k}-1},\param} \approx \resFullDiscreteApproxVector^k\pa{\solFullDiscreteApproxVectorArg{k}\pa{\param};\solFullDiscreteApproxArg{\windowIndexMapperArg{k}-1},\param} = \stbasismatResidArg{k} \residualReducedWindowArg{k} \in \RR{\nspacedof \numStepsInWindowArg{k}},
\end{equation*}
where $\residualReducedWindowArg{k}  \in \RR{\residualBasisDimWSTArg{k}}$ are the space--time reduced states associated with the gappy reconstruction over the $k^{\mathrm{th}}$ window. The reduced states are given by the solution to the minimization problem
\begin{equation}
\residualReducedWindowArg{k} = 	\underset{
\dummyFullDiscreteReduced\in \RR{\residualBasisDimWSTArg{k}}
		}{\arg\min}\Vmat{\sampleMatrixWindowArg{k}\stbasismatResidArg{k} \dummyFullDiscreteReduced - \sampleMatrixWindowArg{k} \resFullDiscreteVector^k\pa{\solFullDiscreteApproxVectorArg{k}\pa{\param}; \solFullDiscreteApproxArg{\windowIndexMapperArg{k}-1},\param} }^2_2 = \left(\sampleMatrixWindowArg{k}\stbasismatResidArg{k}\right)^+\sampleMatrixWindowArg{k} \resFullDiscreteVector^k\pa{\solFullDiscreteApproxVectorArg{k}\pa{\param}; \solFullDiscreteApproxArg{\windowIndexMapperArg{k}-1},\param},
\end{equation}
where $\sampleMatrixWindowArg{k} \in \{0,1\}^{z^k \times \nspacedof \numStepsInWindowArg{k}}$, with $z^k \le \nspacedof\numStepsInWindowArg{k}$, is the sampling matrix with one non-zero element per column that selects $z^k$ rows of the residual. Thus, GNAT solves the minimization problem~\eqref{e:t-lspgReducedLargerFullST} with the space--time weighting matrix $\weightmatwstArg{k} =  \left(\sampleMatrixWindowArg{k} \stbasismatResidArg{k}\right)^+\sampleMatrixWindowArg{k} $. This space--time weighting matrix is calculated offline and forces the least squares problem in the Gauss--Newton procedure to be the same size as the number of space--time residual bases.

\section{Analysis}\label{sec:analysis}
This section presents numerical analyses for the proposed WST-LSPG method. First, \textit{a posteriori} bounds are derived, where it is shown that WST-LSPG is equipped with a more favorable \textit{a posteriori} error bound than LSPG. Second, \textit{a priori} error bounds are derived. Critically, it is shown that the error in WST-LSPG grows exponentially in the number of windows; this is in direct contrast to the traditional Galerkin and LSPG methods, in where the \textit{a priori} error bounds grow exponentially in the number of time steps~\cite{carlberg2017galerkin}. For simplicity, analyses are presented for the case where there is no hyper-reduction, i.e., $\weightmatwstArg{k} = \mathbf{I}$; it is noted that the following results can be extended to include hyper-reduction. For subsequent exposition, several definitions are made. 
First, the space--time velocity over the $k^{\mathrm{th}}$ window is defined as
\vspace{-3mm}
\begin{equation}
\begin{split}
\fluxWSTArg{k}&: \pa{\stateWSTDumAArg{k};\params} \mapsto \bmat{\bmat{ \flux\pa{\stateDumA^1,t^{\windowIndexMapperArg{k}};\params} } ^T & \cdots & \bmat{\flux\pa{\stateDumA^{\numStepsInWindowArg{k}},t^{\windowIndexMapperArg{k} + \numStepsInWindowArg{k}-1};\params}}^T}^T \\
&: \RR{\nspacedof \numStepsInWindowArg{k}} \times \paramDomain \rightarrow \RR{\nspacedof \numStepsInWindowArg{k}},
\end{split}
\end{equation}
where $\stateSTDumA = \bmat{ \bmat{\stateSTDumA^1}^T & \cdots & \bmat{\stateSTDumA^{\numStepsInWindowArg{k}}}^T }^T$. Next, the time-discrete space--time residual over the $k^{\mathrm{th}}$ window can be expressed as
\begin{equation}
\resFullWSTDiscreteArg{k} : \pa{ \stateWSTDumAArg{k} ; \dumFullDiscreteArg{0}, \params} \mapsto  \AWSTArg{k} \stateWSTDumAArg{k} - \Delta  t \BWSTArg{k}\fluxWSTArg{k}\pa{ \stateWSTDumAArg{k};\params} +  \AICWSTArg{k} \dumFullDiscreteArg{0} - \Delta t \BICWSTArg{k}\flux\pa{ \dumFullDiscreteArg{0},t^{\windowIndexMapperArg{k}-1};\params},
\end{equation}
where $\AWSTArg{k}, \BWSTArg{k} \in \RR{\nspacedof \numStepsInWindowArg{k} \times \nspacedof \numStepsInWindowArg{k}}$, $k \in \nat{\numWindows}$ are matrices that contain the linear multistep coefficients pertaining to the unknowns over the window, and $\AICWSTArg{k} , \BICWSTArg{k} \in \RR{\nspacedof \numStepsInWindowArg{k} \times \nspacedof}$, $k \in \nat{\numWindows}$ contain the linear multistep coefficients pertaining to the input state $\dumFullDiscreteArg{0}$. Lastly, the initial states into the $k^{th}$ window will be denoted as $\fomStateICWSTArg{k}(= \solFullDiscreteArg{\windowIndexMapperArg{k}-1})$ and $\approxStateICWSTArg{k}(= \solFullDiscreteApproxArg{\windowIndexMapperArg{k}-1})$ for the FOM and the WST-LSPG ROM, respectively. 
%
%

Following Ref.~\cite{choi2019space}, the following assumptions are employed in the analysis:
\begin{itemize}
\item \textbf{A1}. The velocity is assumed to be Lipshitz continuous, i.e., there exists a constant $\lipschitzVelocity  > 0$ such that
\begin{equation*}
\Vmat{\flux\pa{\stateDumA,t;\params} - \flux\pa{\stateDumB,t;\params} }_2 \le \lipschitzVelocity \Vmat{ \stateDumA - \stateDumB }_2 \quad \forall \stateDumA,\stateDumB \in \RR{\nspacedof}.
\end{equation*}
\item \textbf{A2}. The time step is sufficiently small such that 
$
\sigmaMinArg{\AWSTArg{k}} - \sigmaMaxArg{\BWSTArg{k}} \Delta t \lipschitzVelocity > 0 ,
$
where $\sigmaMinArg{\cdot}$ and $\sigmaMinArg{\cdot}$ denote the minimum and maximum singular values, respectively.\footnote{While this bound is additionally employed in a similar analysis for ST-LSPG~\cite{choi2019space}, it is noted here that, for linear multistep schemes, $\sigmaMinArg{\AWSTArg{k}}$ decreases as the number of steps in the window grows. Thus, for large window sizes, this assumption becomes strong, and the resulting bounds may become less sharp.}
\end{itemize}
For the remainder of this section, the dependence of functions on parameters is suppressed when possible. 

\subsection{A posteriori error bounds}
\textit{A posteriori} error bounds are first derived. These bounds enable the error in the WST-LSPG approximate solution to be bounded by the FOM residual evaluated about the WST-LSPG solution.  
\begin{lemma}
Under assumption A1, the space--time velocity is Lipshitz continuous with the Lipshitz constant $\lipschitzVelocity$, i.e., there exists a constant $\lipschitzVelocity$ such that, $\forall \stateSTDumA,\stateSTDumB \in \RR{\nspacedof \numStepsInWindowArg{k}}$,
\begin{equation}                           
\Vmat{\fluxWSTArg{k}\pa{\stateSTDumA ;\params} - \fluxWSTArg{k}\pa{\stateSTDumB;\params}}_2 \le \lipschitzVelocity   \Vmat{\stateSTDumA - \stateSTDumB}_2, \qquad 
\end{equation}
where $\stateSTDumA \equiv \bmat{ \bmat{\stateDumA^1}^T & \cdots & \bmat{\stateDumA^{\numStepsInWindowArg{k}}}^T }^T$ and $\stateSTDumB \equiv \bmat{ \bmat{\stateDumB^1}^T & \cdots & \bmat{\stateDumB^{\numStepsInWindowArg{k}}}^T }^T$. 
\end{lemma}
\begin{proof}
The result follows trivially from expanding the definition of the norm and applying assumption A1, which yields 
\vspace{-3mm}
\begin{equation*}
\Vmat{\fluxWSTArg{k}\pa{\stateSTDumA  ;\params} - \fluxWSTArg{k}\pa{\stateSTDumB;\params} }_2^2 \le \lipschitzVelocity^2  \sum_{i=0}^{\numStepsInWindowArg{k} -1} \Vmat{ \stateDumA^{i+1} - \stateDumB^{i+1} }_2^2.
\end{equation*}
Noting that the norm on the right-hand side is the square of the $\ell^2$-norm of $\stateSTDumA - \stateSTDumB$, and taking the square root yields the desired result.
\end{proof}

\begin{theorem}\label{th:traj_bound}
(Local bound between two trajectories) Under assumptions A1 and A2, the difference in two space--time trajectories $\stateWSTDumAArg{k},\stateWSTDumBArg{k} \in \RR{\nspacedof \numStepsInWindowArg{k}}$---associating with starting points $\stateWSTDumAICArg{k},\stateWSTDumBICArg{k} \in \RR{\nspacedof}$---over the $k^{\mathrm{th}}$ window is bounded by
\vspace{-3mm}\begin{equation}\label{eq:local_error_bound}
 \Vmat{  \stateWSTDumAArg{k} - \stateWSTDumBArg{k}  }_2 \le
\frac{1}{\constantAWSTArg{k}} \Vmat{ \resFullWSTDiscreteArg{k}(\stateWSTDumAArg{k};\stateWSTDumAICArg{k}) -  \resFullWSTDiscreteArg{k}(\stateWSTDumBArg{k};\stateWSTDumBICArg{k})  }_2 +
\frac{\constantBWSTArg{k}}{\constantAWSTArg{k}}  \Vmat{  \stateWSTDumAICArg{k} -\stateWSTDumBICArg{k}   }_2,
\end{equation}
where $\constantAWSTArg{k} = \pa{\sigmaMinArg{\AWSTArg{k}} - \sigmaMaxArg{\BWSTArg{k}} \Delta t \lipschitzVelocity }$ and $ 
\constantBWSTArg{k} = \sigmaMaxArg{ \AICWSTArg{k}} +  \Delta t \lipschitzVelocity \sigmaMaxArg{\BICWSTArg{k}}.$
\end{theorem}
\begin{proof}
The difference between the windowed space--time residual evaluated at  trajectories $\stateWSTDumAArg{k}$ and $\stateWSTDumBArg{k}$ can be written as
\vspace{-1mm}
\begin{multline*}
 \underbrace{\resFullWSTDiscreteArg{k}\pa{\stateWSTDumAArg{k};\stateWSTDumAICArg{k}}  - \resFullWSTDiscreteArg{k}\pa{\stateWSTDumBArg{k};\stateWSTDumBICArg{k}}}_{\alpha}  =   \underbrace{\AWSTArg{k}\pa{ \stateWSTDumAArg{k} - \stateWSTDumBArg{k} }}_{\beta}  -  \underbrace{\Delta t \BWSTArg{k} \pa{ \fluxWSTArg{k}\pa{\stateWSTDumAArg{k}} - \fluxWSTArg{k}\pa{\stateWSTDumBArg{k}} }}_{\gamma}
+ \\
 \underbrace{\AICWSTArg{k}\pa{ \stateWSTDumAICArg{k} - \stateWSTDumBICArg{k} }}_{\eta}  -  \underbrace{\Delta t \BICWSTArg{k} \pa{ \flux\pa{\stateWSTDumAICArg{k},t^{\windowIndexMapperArg{k}-1};\params} - \flux\pa{\stateWSTDumBICArg{k},t^{\windowIndexMapperArg{k}-1};\params} }}_{\theta}.
\end{multline*}
Rearranging terms, taking the $\ell^2$-norm of each side, and applying the triangle inequality yields
\begin{equation*}
    \Vmat{\beta}_2=\Vmat{\alpha +\gamma-\eta+\theta}_2\longrightarrow\Vmat{\beta}_2\leq\Vmat{\alpha}_2 +\Vmat{\gamma}_2+\Vmat{\eta-\theta}_2.
\end{equation*}
Subtracting $\Vmat{\gamma}_2$ on both sides gives $\Vmat{\beta}_2-\Vmat{\gamma}_2\leq\Vmat{\alpha}_2 + \Vmat{\eta-\theta}_2.$
Bounding $\Vmat{\beta}_2$ from below and $\Vmat{\gamma}_2$ from above yields
\vspace{-2mm}
\begin{equation*}
\sigmaMinArg{\AWSTArg{k}} \Vmat{  \stateWSTDumAArg{k} - \stateWSTDumBArg{k}  }_2 - \sigmaMaxArg{\BWSTArg{k}} \Delta t \Vmat{  \fluxWSTArg{k}\pa{\stateWSTDumAArg{k}} - \fluxWSTArg{k}\pa{\stateWSTDumBArg{k}} }_2 \leq \Vmat{\alpha}_2 + \Vmat{\eta-\theta}_2.
\end{equation*}
Applying Assumption A1 (Lipshitz continuity of the velocity) results in
\begin{equation*}\label{e:analysis_tmp1}
\pa{ \sigmaMinArg{\AWSTArg{k}} - \sigmaMaxArg{\BWSTArg{k}} \Delta t \lipschitzVelocity}  \Vmat{  \stateWSTDumAArg{k} - \stateWSTDumBArg{k} }_2\leq \Vmat{\alpha}_2 + \Vmat{\eta-\theta}_2.
\end{equation*}
Next, with Assumption A1, $\Vmat{\eta-\theta}_2$ can be bounded as
\begin{equation}\label{e:analysis_tmp2}
    \Vmat{\eta-\theta}_2 \leq \\
\bmat{ \sigmaMaxArg{ \AICWSTArg{k}} +  \Delta t \lipschitzVelocity \sigmaMaxArg{\BICWSTArg{k}} }\Vmat{ \stateWSTDumAICArg{k} - \stateWSTDumBICArg{k}}_2.
\end{equation}
Inserting the inequality~\eqref{e:analysis_tmp2} into the inequality~\eqref{e:analysis_tmp1} yields 
\begin{equation*}
\pa{ \sigmaMinArg{\AWSTArg{k}} - \sigmaMaxArg{\BWSTArg{k}} \Delta t \lipschitzVelocity} \Vmat{ \stateWSTDumAArg{k} - \stateWSTDumBArg{k}}_2 \le
\Vmat{ \alpha }_2 + 
\bmat{ \sigmaMaxArg{ \AICWSTArg{k}} +  \Delta t \lipschitzVelocity \sigmaMaxArg{\BICWSTArg{k}} } \Vmat{ \stateWSTDumAICArg{k} - \stateWSTDumBICArg{k} }_2.
\end{equation*}
Dividing  through by $\pa{  \sigmaMinArg{\AWSTArg{k}} - \sigmaMaxArg{\BWSTArg{k}} \Delta t \lipschitzVelocity}$ and employing assumption A2 gives the desired result.
\end{proof}
\begin{corollary}\label{th:local_apost}
(Local a posteriori error bounds) The error between the WST-LSPG solution and the FOM solution over the $k^{th}$ window is bounded by 
\vspace{-1mm}
\begin{equation}\label{e:local_apost_bound_corr}
\Vmat{ \errorWSTArg{k} }_2 \le  \frac{1}{\constantAWSTArg{k}} \Vmat{ \resFullWSTDiscreteArg{k}\pa{\approxStateWSTArg{k};\approxStateICWSTArg{k}}}_2 + \frac{\constantBWSTArg{k}}{\constantAWSTArg{k}} \Vmat{\errorICWSTArg{k} }_2,
\end{equation}
where $\errorWSTArg{k} = \approxStateWSTArg{k} - \fomStateWSTArg{k}$ and $\errorICWSTArg{k} = \approxStateICWSTArg{k} - \fomStateICWSTArg{k}$.
\end{corollary}
\begin{proof}
The proof follows directly from evaluating Theorem~\ref{th:traj_bound} for the trajectories $\approxStateWSTArg{k},\fomStateWSTArg{k}$, starting from $\approxStateICWSTArg{k}$ and $\fomStateICWSTArg{k}$, and noting that  $\resFullWSTDiscreteArg{k}\pa{ \fomStateWSTArg{k}; \fomStateICWSTArg{k}} = \bz$.
\end{proof}
Theorem~\ref{th:local_apost} demonstrates that, if using the same initial condition into the $k^{th}$ window, the solution obtained via WST-LSPG is equipped with a lower a posteriori bound than the solution obtained via LSPG; WST-LSPG guarantees the solution over the window with a minimum residual. Next, global error bounds are presented.
\begin{corollary}\label{th:global_apost}
(Global \textit{a posteriori} error bounds) The error in the WST-LSPG solution over the $k^{th}$ window is bounded by
\vspace{-2mm}
\begin{equation}\label{e:global_apost_error_bound}
\Vmat{  \errorWSTArg{k} }_2 \le
\mathlarger{\sum}\limits_{i=0}^{k}\pa{\, \prod\limits_{j=0}^{k-i}\constantBWSTArg{k-j}\Bigg/\prod\limits_{j=0}^{k - i+1}\constantAWSTArg{k-j}}\Vmat{ \resFullWSTDiscreteArg{i}\pa{\approxStateWSTArg{i};\approxStateICWSTArg{i}} }_2,
\end{equation}
where, for notational simplicity, $\Pi$ denotes the product that is non-inclusive in its upper limit.\footnote{For example, $\prod\limits_{j=1}^{4} j = 1\times 2\times3$.}  
\end{corollary}
\begin{proof}
First, note that $\Vmat{\errorICWSTArg{k}}_2 \le \Vmat{ \approxStateWSTArg{k-1} - \fomStateWSTArg{k-1}  }_2$, $k \in \nat{\numWindows}$.
Applying the local error bound~\eqref{eq:local_error_bound} to the last term on the right hand side of~\eqref{eq:local_error_bound} gives
\begin{equation*}
\Vmat{\errorWSTArg{k}}_2 \le
\dfrac{1}{\constantAWSTArg{k}} \Vmat{ \resFullWSTDiscreteArg{k}\pa{\approxStateWSTArg{k};\approxStateICWSTArg{k}}  }_2 +
\frac{\constantBWSTArg{k}}{\constantAWSTArg{k}} \bmat{ \dfrac{1}{\constantAWSTArg{k-1}} \Vmat{ \resFullWSTDiscreteArg{k-1}\pa{\approxStateWSTArg{k-1};\approxStateICWSTArg{k-1}}}_2 +
\dfrac{\constantBWSTArg{k-1}}{\constantAWSTArg{k-1}}  \Vmat{ \errorWSTArg{k-1} }_2
}.
\end{equation*}
%
Applying the same process recursively yields the desired result.
%
%
\end{proof}
The global error bound~\eqref{e:global_apost_error_bound} shows how the error in WST-LPSG can be bounded by a sum of weighted residuals over each window. As WST-LPSG yields a lower residual $\ell^2$ norm over a window than LSPG, WST-LSPG is likely equipped with a more favorable \textit{a posteriori} error bound than LSPG. Analogously, as ST-LSPG guarantees the lowest total space--time residual, ST-LSPG is likely equipped with a more favorable \textit{a posteriori} error bound than WST-LSPG. 

\subsection{A priori bounds}
\textit{A priori} error bounds are now derived. These bounds enable the error in the WST-LSPG solution to be bounded, \textit{a priori} by the FOM solution and demonstrate how errors are accumulated in WST-LSPG.
\begin{theorem}\label{th:local_ariori}
(Local a priori error bounds) The error between the WST-LSPG solution and the FOM solution onto the trial subspace over the $k^{th}$ window can be bounded, \textit{a priori}, by
\begin{equation*}
\Vmat{ \errorWSTArg{k}}_2 \le  \frac{1}{\constantAWSTArg{k}} \Vmat{ \resFullWSTDiscreteArg{k}\pa{\fomStateProjectWSTArg{k};\fomStateProjectICWSTArg{k}}}_2 + \frac{2 \constantBWSTArg{k}}{\constantAWSTArg{k}} \Vmat{\errorICWSTArg{k} }_2.\end{equation*}
where $\fomStateProjectWSTArg{k}$ denotes the $\ell^2$-orthogonal projection of the FOM solution onto the trial space, i.e., 
\begin{equation*}
\fomStateProjectWSTArg{k} = \underset{\stateWSTDumAArg{k}\in \trialWSTSubspaceVectorArg{k}}{\arg\min}\!\Vmat{ \fomStateWSTArg{k} - \stateWSTDumAArg{k} }_2.
\end{equation*}
Analogously, $\fomStateProjectICWSTArg{k}$ corresponds to the $\ell^2$-orthogonal projection of the FOM initial conditions for the $k^{th}$ window onto the trial subspace.
\end{theorem}
\begin{proof}
Using the bound~\eqref{e:local_apost_bound_corr} and leveraging the residual minimization property of WST-LSPG
$$\Vmat{ \errorWSTArg{k} }_2 \le  \frac{1}{\constantAWSTArg{k}} \Vmat{ \resFullWSTDiscreteArg{k}(\fomStateProjectWSTArg{k};\approxStateICWSTArg{k}) }_2 + \frac{\constantBWSTArg{k}}{\constantAWSTArg{k}} \Vmat{\errorICWSTArg{k} }_2.$$
Expanding the definition of the space--time residual 
$$\Vmat{ \errorWSTArg{k} }_2 \le  \frac{1}{\constantAWSTArg{k}} \Vmat{  \AWSTArg{k} \fomStateProjectWSTArg{k} - \Delta  t \BWSTArg{k}\fluxWSTArg{k}( \fomStateProjectWSTArg{k};\params) +  \AICWSTArg{k} \approxStateICWSTArg{k} - \Delta t \BICWSTArg{k}\flux( \approxStateICWSTArg{k},t^{\windowIndexMapperArg{k}-1};\params)  }_2 + \frac{\constantBWSTArg{k}}{\constantAWSTArg{k}} \Vmat{\errorICWSTArg{k}}_2.$$
Adding and subtracting $\AICWSTArg{k} \fomStateProjectICWSTArg{k} - \Delta t \BICWSTArg{k}\flux( \fomStateProjectICWSTArg{k},t^{\windowIndexMapperArg{k}-1};\params)$, and gathering terms, $$\Vmat{ \errorWSTArg{k}}_2 \le  \!\frac{1}{\constantAWSTArg{k}} \!\Vmat{ \resFullWSTDiscreteArg{k}\pa{\fomStateProjectWSTArg{k};\fomStateProjectICWSTArg{k}} \!+\! \AICWSTArg{k}\pa{ \approxStateICWSTArg{k} \!-\! \fomStateProjectICWSTArg{k}} \!-\! \Delta t \BICWSTArg{k}\!\bmat{\flux\pa{\approxStateICWSTArg{k},t^{\windowIndexMapperArg{k}-1};\params} \!-\! \flux\pa{ \fomStateProjectICWSTArg{k},t^{\windowIndexMapperArg{k}-1};\params} }  }_2 \!+ \frac{\constantBWSTArg{k}}{\constantAWSTArg{k}} \Vmat{\errorICWSTArg{k} }_2.$$
Using the triangle inequality gives%
\begin{multline*}
\Vmat{ \errorWSTArg{k}}_2 \le  \frac{1}{\constantAWSTArg{k}} \Vmat{ \resFullWSTDiscreteArg{k}\pa{\fomStateProjectWSTArg{k};\fomStateProjectICWSTArg{k}}}_2 + \frac{1}{\constantAWSTArg{k}} \Vmat{ \AICWSTArg{k}\pa{ \approxStateICWSTArg{k} - \fomStateProjectICWSTArg{k}} }_2 + \\ \frac{1}{\constantAWSTArg{k}} \Vmat{ \Delta t \BICWSTArg{k}\bmat{\flux\pa{ \approxStateICWSTArg{k},t^{\windowIndexMapperArg{k}-1};\params} - \flux\pa{ \fomStateProjectICWSTArg{k},t^{\windowIndexMapperArg{k}-1};\params} }  }_2 + \frac{\constantBWSTArg{k}}{\constantAWSTArg{k}} \Vmat{\errorICWSTArg{k} }_2.
\end{multline*}%
Employing assumption A1 (Lipshitz continuity of the velocity),%
$$\Vmat{ \errorWSTArg{k}}_2 \le  \frac{1}{\constantAWSTArg{k}} \Vmat{ \resFullWSTDiscreteArg{k}\pa{\fomStateProjectWSTArg{k};\fomStateProjectICWSTArg{k}}}_2 + \frac{1}{\constantAWSTArg{k}} \Vmat{  \AICWSTArg{k} \pa{ \approxStateICWSTArg{k} - \fomStateProjectICWSTArg{k}} }_2 + \frac{\lipschitzVelocity }{\constantAWSTArg{k}} \Vmat{ \Delta t \BICWSTArg{k} }_2  \Vmat{ \approxStateICWSTArg{k} - \fomStateProjectICWSTArg{k} }_2 +  \frac{\constantBWSTArg{k}}{\constantAWSTArg{k}} \Vmat{\errorICWSTArg{k} }_2.$$%
Collecting terms and noting that $\Vmat{\approxStateICWSTArg{k} - \fomStateProjectICWSTArg{k} }_2 \le \Vmat{\approxStateICWSTArg{k} - \fomStateICWSTArg{k} }_2$ gives the desired result.
\end{proof}

\begin{corollary}
(Global \textit{a priori} error bounds) The error in the WST-LSPG solution over the $k^{th}$ window is bounded, \textit{ a priori}, by
\vspace{-2mm}
\begin{equation}\label{e:global_apriori_error_bound}
\Vmat{  \errorWSTArg{k} }_2 \le \mathlarger{\sum}_{i=0}^{k} \pa{\,\prod\limits_{j=0}^{k-i}2 \constantBWSTArg{k-j}\Bigg/\prod\limits_{j=0}^{k - i+1}\constantAWSTArg{k-j}} \Vmat{ \resFullWSTDiscreteArg{i}\pa{\fomStateProjectWSTArg{i};\fomStateProjectICWSTArg{i}} }_2. 
\end{equation}
\end{corollary}
\begin{proof}
The proof follows that that provided for Theorem~\ref{th:global_apost}.
\end{proof}

\begin{corollary}
(Simplified global \textit{a priori} error bounds) In the case that the window sizes are equivalent, and the time schemes are equivalent on each window such that $\AWSTArg{k} = \AWSTArg{}$, $\BWSTArg{k} = \BWSTArg{}$, $\AICWSTArg{k} = \AICWSTArg{}$, and $\BICWSTArg{k} = \BICWSTArg{}$, $k \in \nat{\numWindows}$, the error in the WST-LSPG solution over the $k^{th}$ window is bounded, \textit{ a priori}, by
\begin{equation}\label{e:global_apriori_error_bound_simple}
\Vmat{  \errorWSTArg{k} }_2 \le  \frac{k}{\constantAWSTArg{}}   \pa{ 2 \frac{  \sigmaMaxArg{\AICWSTArg{}}} {  \sigmaMinArg{\AWSTArg{}} } }^{k} \exp \pa{ k \frac{  \Delta t \lipschitzVelocity \pa{ \frac{ \sigmaMaxArgSmall{\BICWSTArg{}} }{ \sigmaMaxArgSmall{ \AICWSTArg{}} } + \frac{  \sigmaMaxArgSmall{\BWSTArg{}} }{ \sigmaMinArgSmall{\AWSTArg{}}  } } }{1 - \Delta t \lipschitzVelocity  \frac{  \sigmaMaxArgSmall{\BWSTArg{}} }{ \sigmaMinArgSmall{\AWSTArg{}}  } } }\max_{i \in \nat{k+1} } \Vmat{ \resFullWSTDiscreteArg{i}\pa{\fomStateProjectWSTArg{i};\fomStateProjectICWSTArg{i}} }_2. 
\end{equation}
where $\constantAWSTArg{k}  = \constantAWSTArg{}$, $k \in  \nat{\numWindows}$.
\end{corollary}
\begin{proof}
Employing the identity~\cite{carlberg2017galerkin}
\vspace{-3mm}
$$\pa{ \frac{ \sigmaMaxArg{ \AICWSTArg{n}} +  \Delta t \lipschitzVelocity \sigmaMaxArg{\BICWSTArg{n}} } { \sigmaMinArg{\AWSTArg{n}} - \sigmaMaxArg{\BWSTArg{n}} \Delta t \lipschitzVelocity } }^n = \pa{ \frac{  \sigmaMaxArg{\AICWSTArg{n}}} {  \sigmaMinArg{\AWSTArg{n}} } }^n \pa{ 1 + \frac{  \Delta t \lipschitzVelocity \pa{ \frac{ \sigmaMaxArgSmall{\BICWSTArg{n}} }{ \sigmaMaxArgSmall{ \AICWSTArg{n}} } + \frac{  \sigmaMaxArgSmall{\BWSTArg{n}} }{ \sigmaMinArgSmall{\AWSTArg{n}}  } } }{1 - \Delta t \lipschitzVelocity  \frac{  \sigmaMaxArgSmall{\BWSTArg{n}} }{ \sigmaMinArgSmall{\AWSTArg{n}}  } } }^n,$$
along with $(1+x)^n \le \exp(nx)$, the bound~\eqref{e:global_apriori_error_bound} becomes\vspace{-2mm}
\begin{equation}
\Vmat{  \errorWSTArg{k} }_2 \le \mathlarger{\sum}_{i=0}^{k}  \frac{1}{ \constantAWSTArg{} } \pa{ 2 \frac{  \sigmaMaxArg{\AICWSTArg{}}} {  \sigmaMinArg{\AWSTArg{}} } }^{k - i} \exp \pa{(k - i) \frac{  \Delta t \lipschitzVelocity \pa{ \frac{ \sigmaMaxArgSmall{\BICWSTArg{}} }{ \sigmaMaxArgSmall{ \AICWSTArg{}} } + \frac{  \sigmaMaxArgSmall{\BWSTArg{}} }{ \sigmaMinArgSmall{\AWSTArg{}}  } } }{1 - \Delta t \lipschitzVelocity  \frac{  \sigmaMaxArgSmall{\BWSTArg{}} }{ \sigmaMinArgSmall{\AWSTArg{}}  } } } \Vmat{ \resFullWSTDiscreteArg{i}\pa{\fomStateProjectWSTArg{i};\fomStateProjectICWSTArg{i}} }_2. 
\end{equation}
Noting that $2 \sigmaMaxArg{\AICWSTArg{}} \ge   \sigmaMinArg{\AWSTArg{}} $, and employing assumption A2, the desired result is obtained by using\vspace{-2mm}
$$\Vmat{ \resFullWSTDiscreteArg{i}\pa{\fomStateProjectWSTArg{i};\fomStateProjectICWSTArg{i}} }_2 \le \max_{j \in \nat{k+1} } \Vmat{ \resFullWSTDiscreteArg{j}\pa{\fomStateProjectWSTArg{j};\fomStateProjectICWSTArg{j}} }_2,\quad \pa{ 2 \frac{  \sigmaMaxArg{\AICWSTArg{}}} {  \sigmaMinArg{\AWSTArg{}} } }^{k - i}  \le  \pa{ 2 \frac{  \sigmaMaxArg{\AICWSTArg{}}} {  \sigmaMinArg{\AWSTArg{}} } }^{k},$$ and, for all $x > 0$, $\exp((k-i)x) \le \exp(kx)$, $i \in \nat{k}$.
\end{proof}
Due to the repeated use of inequalities, the bound~\eqref{e:global_apriori_error_bound_simple} is less sharp than the bound~\eqref{e:global_apriori_error_bound}. However, the bound~\eqref{e:global_apriori_error_bound_simple} highlights that WST-LSPG is equipped with a global \textit{a priori} error bound that grows exponentially in the number of windows. In contrast, LSPG and Galerkin are equipped with a priori error bounds that grow exponentially in the number of time steps. In the case where only one window is employed (i.e., ST-LSPG), there is no exponential growth in the error.

\section{Numerical experiments}\label{sec: experiments}
This section analyzes the performance of the WST-LSPG/WST-GNAT method via numerical experiments. 
In this work, only predictive model reduction simulations are performed due to the way the space--time bases are created. The following parameters will be examined: dimension of the space--time trial subspace, time window length ($\ell_w$) or number of time windows ($\numWindows$), sub-window length ($\ell_s$) or number of sub-windows ($\numSubWindows$), and hyper-reduction parameters such as the number of spatial and temporal nodes in the sample mesh. For this work, only the case where $\numSubWindows=1$ will be considered when constructing the space--time residual basis for GNAT. This numerical experiment will focus on the following methods:
\begin{itemize}
    \item Full-order model $(\bm{\mathrm{FOM}})$: satisfies the governing equations (O$\Delta$E) set by Eqn.~\ref{e:LMMresidual}, which is characterized by a set of design parameters.
    \item Space--time least-squares Petrov--Galerkin ROM ({\bf{ST-LSPG}}): Unweighted space--time model reduction technique that solves the residual  minimization problem with $\weightmatst = \mathbf{I}$.
    \item Space--time least-squares Petrov--Galerkin GNAT ({\bf{ST-GNAT}}): Weighted space--time model reduction technique that solves the residual minimization problem by setting the space--time weighted matrix to $\weightmatst=\pa{\bm{Z}\stbasismat_r}^+\bm{Z}$.
    \item Windowed space--time least-squares Petrov--Galerkin ROM ({\bf{WST-LSPG}}): Unweighted windowed space--time model reduction technique that solves the residual minimization problem for each time window, $k$, with $\weightmatwstArg{k} = \mathbf{I}$
    \item Windowed space--time least-squares Petrov--Galerkin GNAT ({\bf{ST-GNAT}}): Weighted windowed space--time model reduction technique that uses a space--time weighted matrix set to $\weightmatwstArg{k}=\pa{\bm{Z}^k\stbasismat_{r}^k}^+\bm{Z}^k$ for each window, $k$. 
\end{itemize}
One metric of interest used to assess the performance of the ROMs is the relative mean squared error (MSE), which is defined as
\vspace{-2mm}
\begin{equation}
    \mathrm{mean\;squared\:error}= \sqrt{\sum_{n=1}^{N_t}\Vmat{\solFullDiscreteApproxArg{n}\pa{\param}-\solFullDiscreteArg{n}\pa{\param}}_2^2}\Bigg/\sqrt{\sum_{n=1}^{N_t}\Vmat{\solFullDiscreteArg{n}\pa{\param}}_2^2}.
\end{equation}
Another measure of accuracy used to determine the performance of these methods is the residual $\ell^2$ norm of the ROM. For each problem introduced in this section, other output errors of interest will be introduced to assess the performance of the windowed space--time model reduction methods.

\subsection{Parameterized Burgers' equation}\label{s:Burgers}
The first prototypical problem to consider for WST-LSPG is the parameterized Burgers' equation, which is defined as
\begin{equation} \label{e:burgers}
\begin{split}
\frac{\partial u(x,t;\param)}{\partial t}+ \frac{1}{2} \frac{\partial (u(x,t;\param)^2)}{\partial x}&=0.02e^{\mu_2 x}\quad\forall x\in [0,100],\quad\forall t\in [0,T], \\
u(x,0)&=1.0, \quad \forall x \in [0,100],
\end{split}
\end{equation}
where the spatial domain is defined to be $x\in[0,100]$, and the temporal domain is defined to be $t\in[0,T]$. The trajectory is defined as a conserved quantity, $u:[0,100]\times [0,T]\times \mathcal{D}\rightarrow \RR{}$. The parameter design set is chosen to be $\param\equiv (\mu_1,\mu_2) \in \mathcal{D}=[2, 4.1]\times [0.013, 0.02]$, where $\Delta \mu_1=0.105$ for $N_{\mu_1}=20$ indices and  $\Delta \mu_2=0.0035$ for $N_{\mu_2}=5$ indices. Overall, the number of training data sets is $n_{\mathrm{train}}=100$. The test parameters used for the predictive model reduction model are set to $\mu_{\mathrm{test},1}=4.0714$ and $\mu_{\mathrm{test},2}=0.0185$. For the creation of the space--time residual bases, the parameter design set is chosen to be the same as the parameter design set for the states; however, the number of indices for $\mu_{r,1}$ changes to $N_{\mu_{r,1}}=10$, resulting in $\Delta \mu_{r,1}=0.21$.

The parameterized Burgers' equation is implemented in the \texttt{pymortestbed} python code, which uses the finite volume method to solve the governing equation and the \texttt{NumPy} library for numerical linear algebra~\cite{zahr2015progressive}. Godunov's method is used to spatially discretize the spatial domain into $N_s=200$ control volumes. This leads to a spatial step of $\Delta x = 0.5$. For temporal integration, the linear multistep method also known as the backward differentiation formulae (BDF) is used. For this problem, first order BDF1 (also known as backward Euler) is chosen as the time discretization scheme for the FOM and all model reduction simulations. The final time is set to $T=25.6$ with $N_t=256$ time steps, such that the time step is $\Delta t = 0.1$. For each ROM and corresponding FOM case, the online simulations were executed five times each to find the average relative wall times. All timings for the Burgers' equation are obtained by performing serial calculations on an Intel(R) Xeon(R) CPU E5-2680 0 @ 2.70GHz 128 GB RAM.

\subsubsection{Model Predictions}
This section presents and compares the performances of the WST-LSPG method over varying window lengths $\ell_w$ and varying sub-window lengths $\ell_s$. The results for this comparison additionally include results for the ST-LSPG method, which corresponds to when the window lengths and the sub-window lengths are both equal to the full time simulation length, $\ell_w=\ell_s=T$. 

Figure~\ref{f:window1} shows the MSE, integrated mean squared error (IMSE), and the residual $\ell^2$ norm as a function of the window length, $\ell_w$ and the sub-window length, $\ell_s$. The IMSE is found by calculating the relative difference between the definite integral of the ROM and the FOM over the entire time domain. The goal of this study is to gain an understanding on how introducing windows and sub-windows affects the performances of the resulting ROM as measured by the different error metrics. Note, that the solution for $\ell_s=256\Delta t$ in these figures refers to the the case where the WST-LSPG method is equivalent to the ST-LSPG method.
\begin{figure}[t!]
\subfloat[mean squared error: $e_s = 0.99, e_t = 0.99$][\centering \hspace{1.5mm}mean squared error:\par\hspace{5.5mm}$e_s = 0.99, e_t = 0.99$\label{f:burgersMSE0}]{\includegraphics[height=0.264\textwidth]{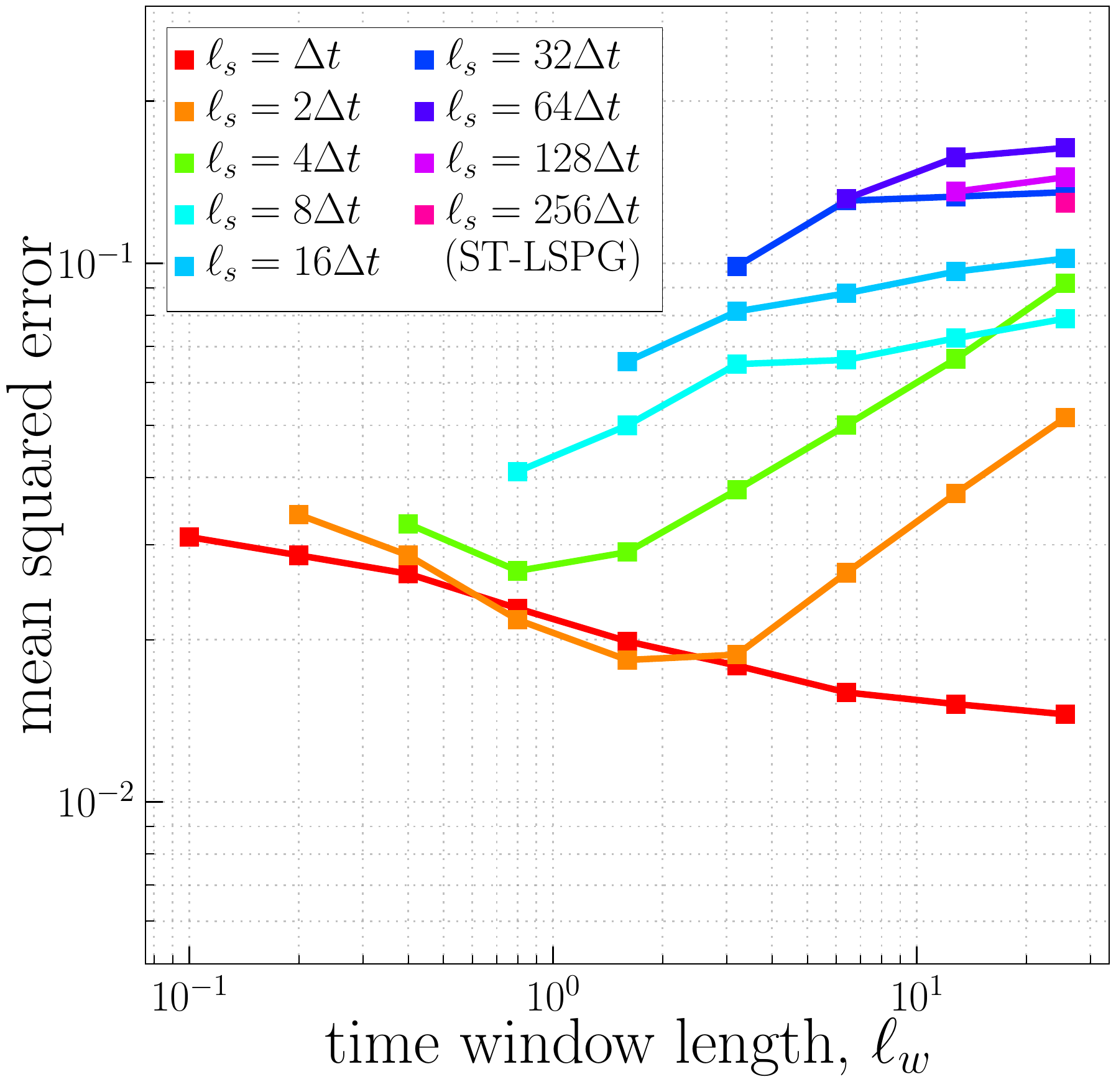}}\hspace{1mm}\nolinebreak
\subfloat[mean squared error: $e_s = 0.99, e_t = 0.999$][\centering \hspace{1.5mm}mean squared error:\par\hspace{5.5mm}$e_s = 0.99, e_t = 0.999$\label{f:burgersMSE2}]{\includegraphics[height=0.264\textwidth]{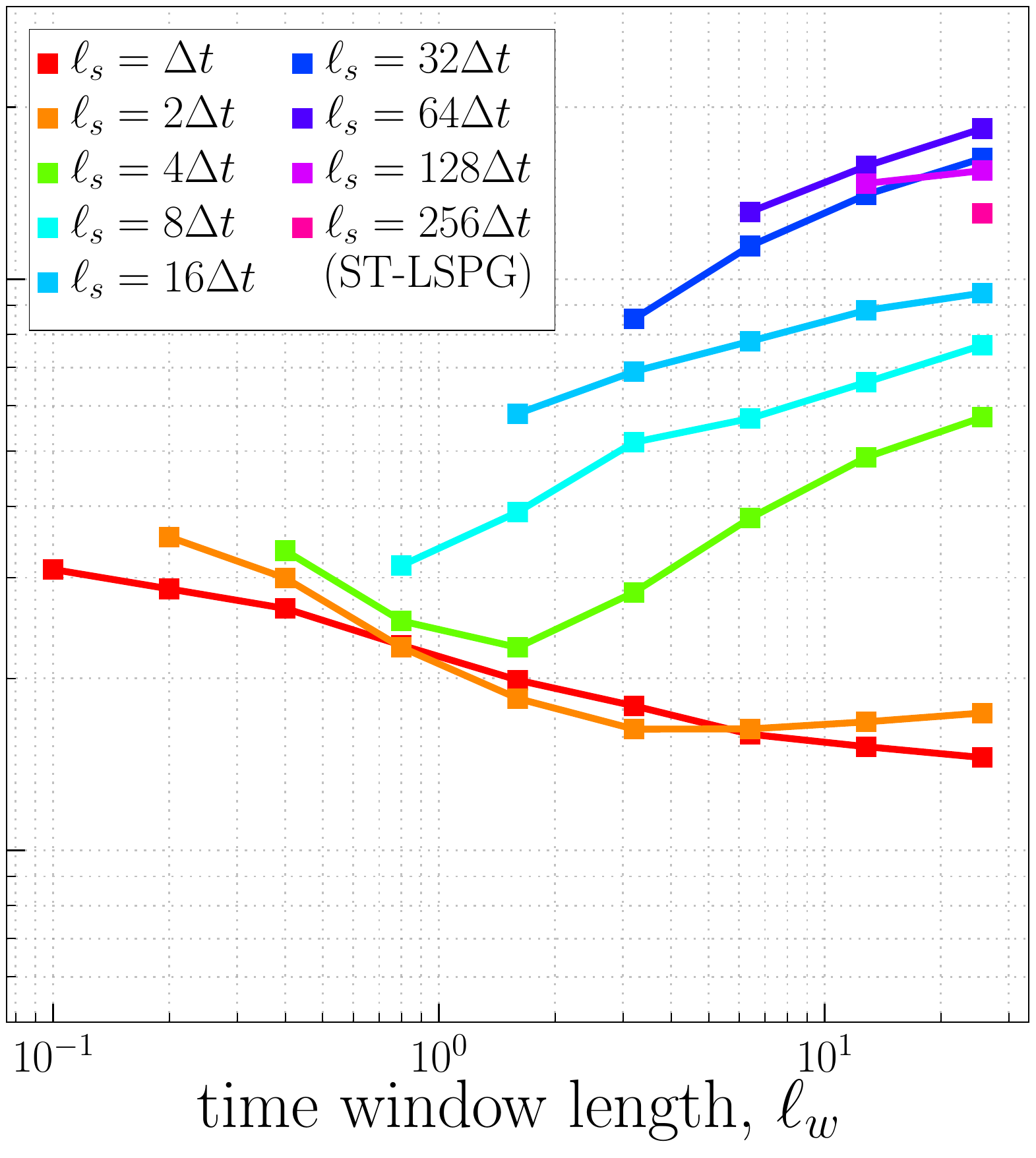}}\hspace{1mm}\nolinebreak
\subfloat[mean squared error: $e_s = 0.999, e_t = 0.99$][\centering \hspace{1.5mm}mean squared error:\par\hspace{5.5mm}$e_s = 0.999, e_t = 0.99$\label{f:burgersMSE4}]{\includegraphics[height=0.264\textwidth]{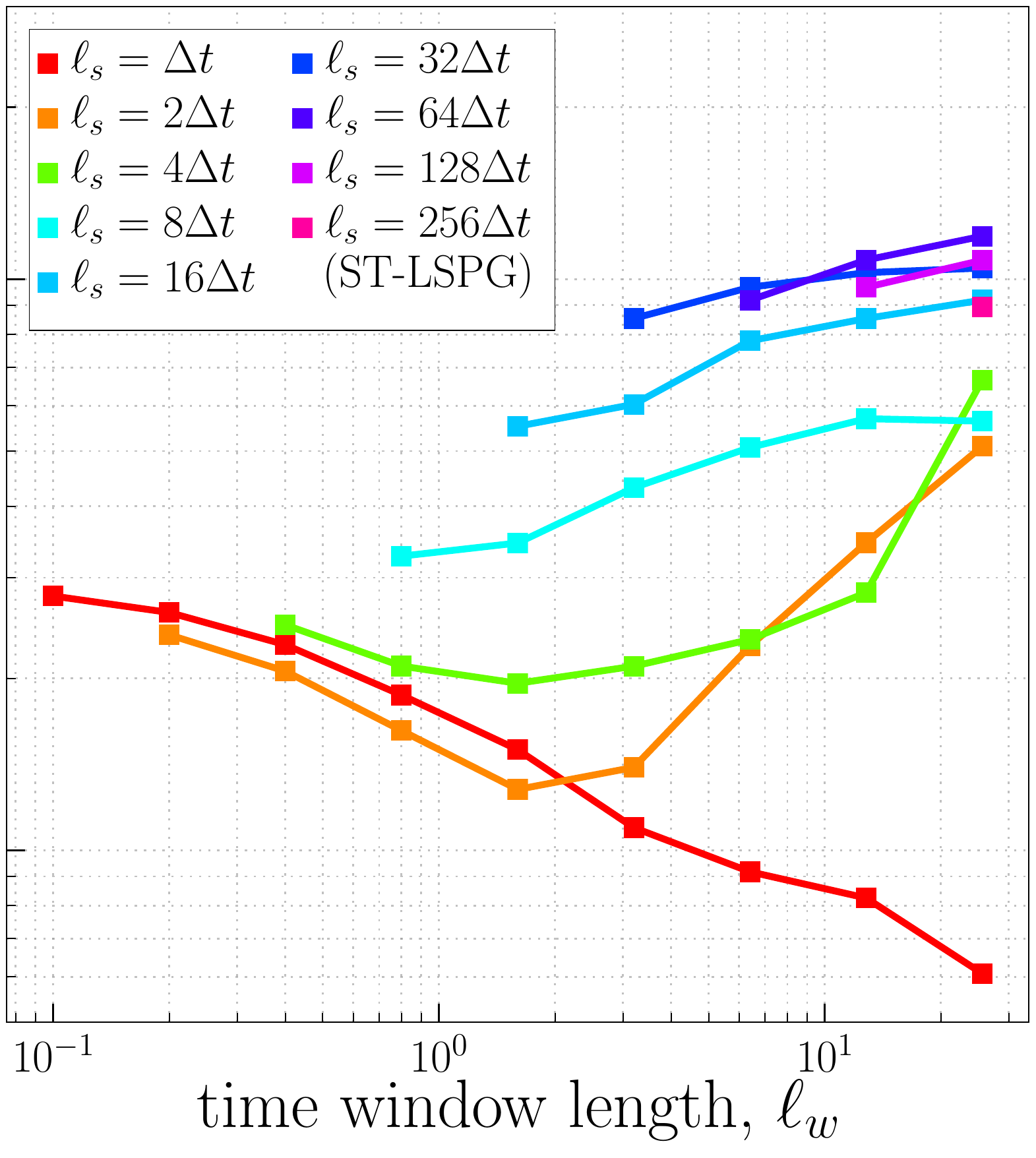}}\hspace{1mm}\nolinebreak
\subfloat[mean squared error: $e_s = 0.999, e_t = 0.999$][\centering \hspace{1.5mm}mean squared error:\par\hspace{5.5mm}$e_s = 0.999, e_t = 0.999$\label{f:burgersMSE6}]{\includegraphics[height=0.264\textwidth]{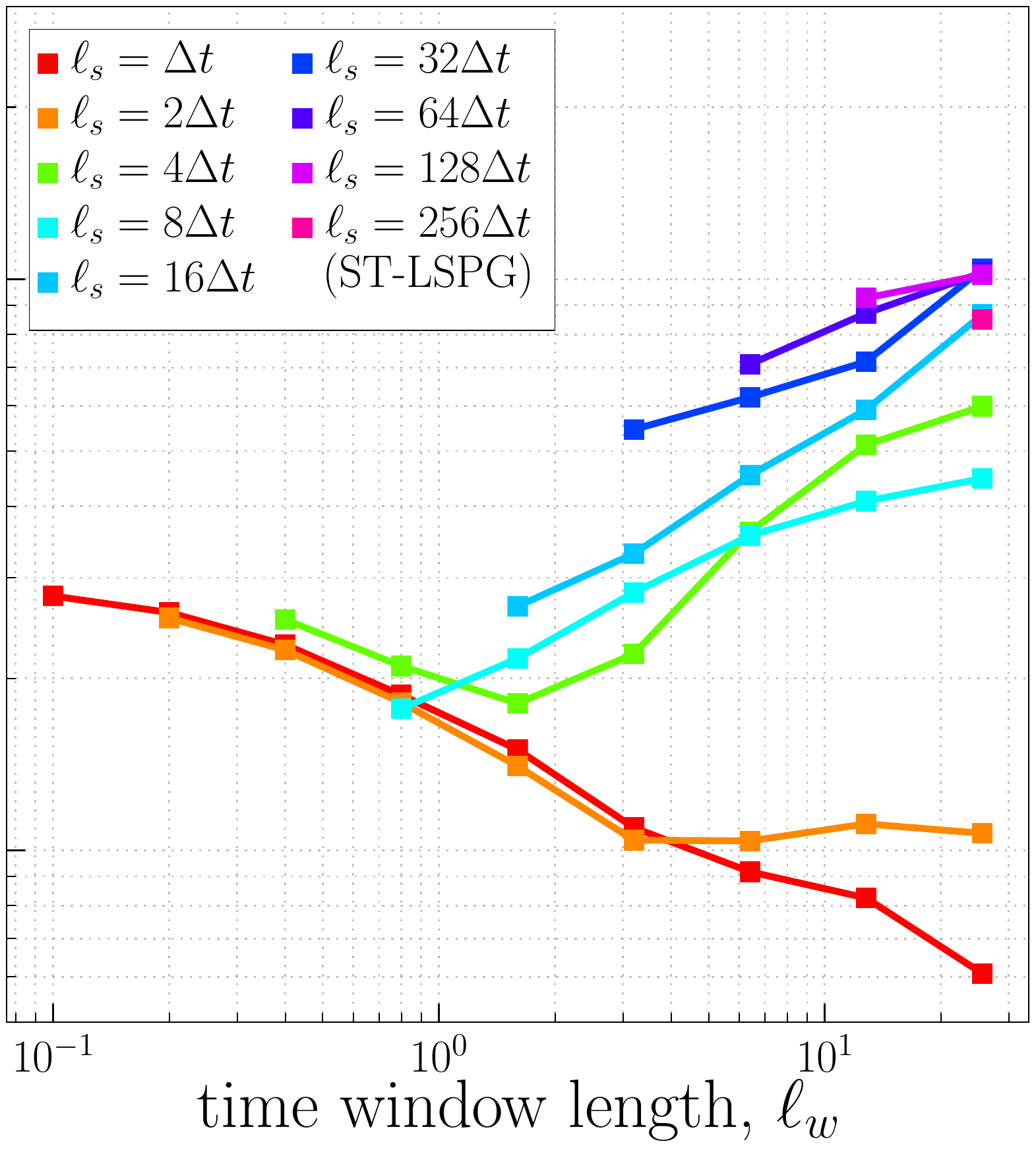}}\vspace{-2mm}\\
\subfloat[residual $\ell^2$ norm: $e_s = 0.99, e_t = 0.99$][\centering \hspace{1.5mm}residual $\ell^2$ norm:\par\hspace{5.5mm}$e_s = 0.99, e_t = 0.99$\label{f:burgersResidual0}]{\includegraphics[height=0.264\textwidth]{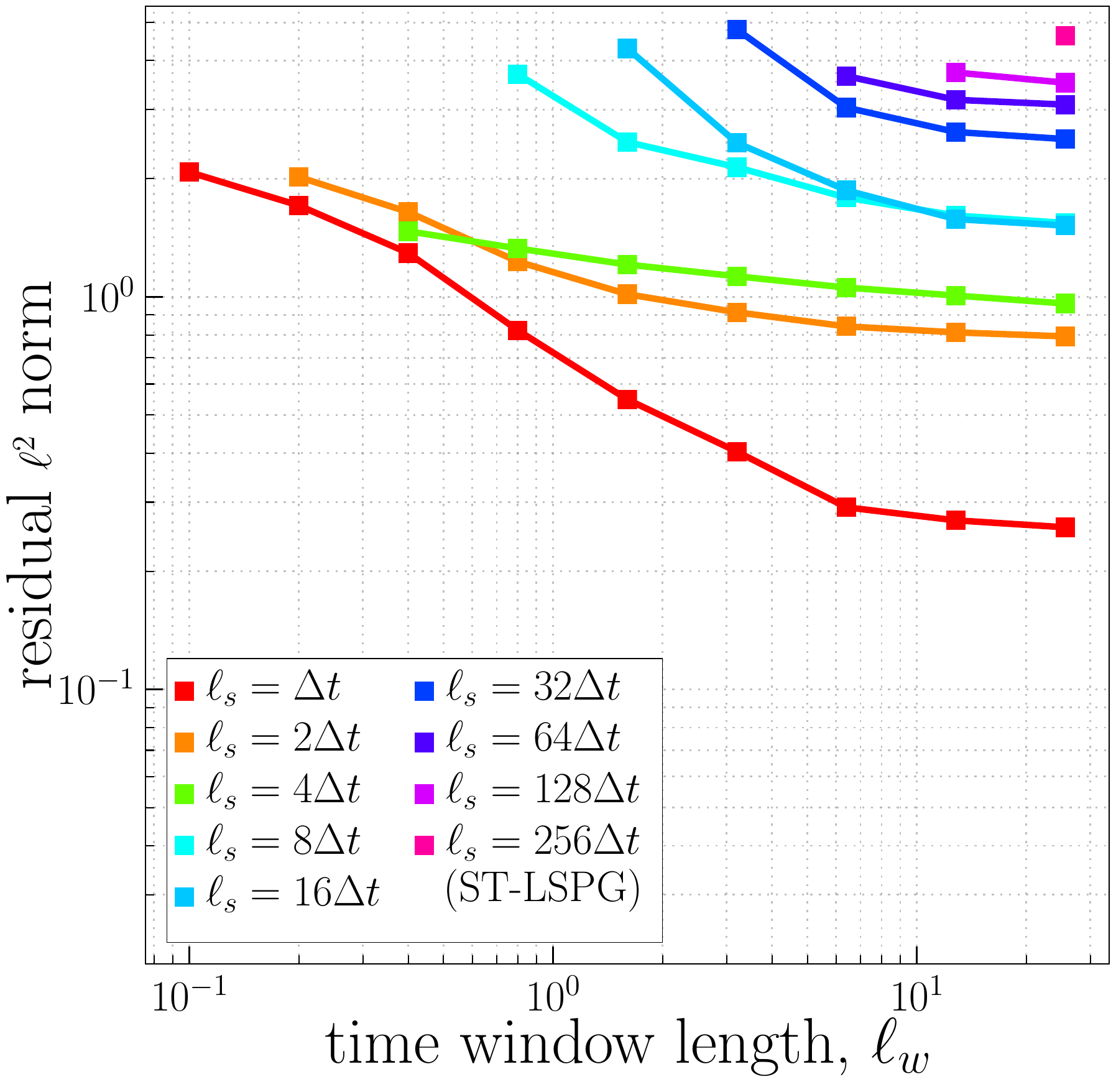}}\hspace{1mm}\nolinebreak
\subfloat[residual $\ell^2$ norm: $e_s = 0.99, e_t = 0.999$][\centering \hspace{1.5mm}residual $\ell^2$ norm:\par \hspace{5.5mm}$e_s = 0.99, e_t = 0.999$\label{f:burgersResidual2}]{\includegraphics[height=0.264\textwidth]{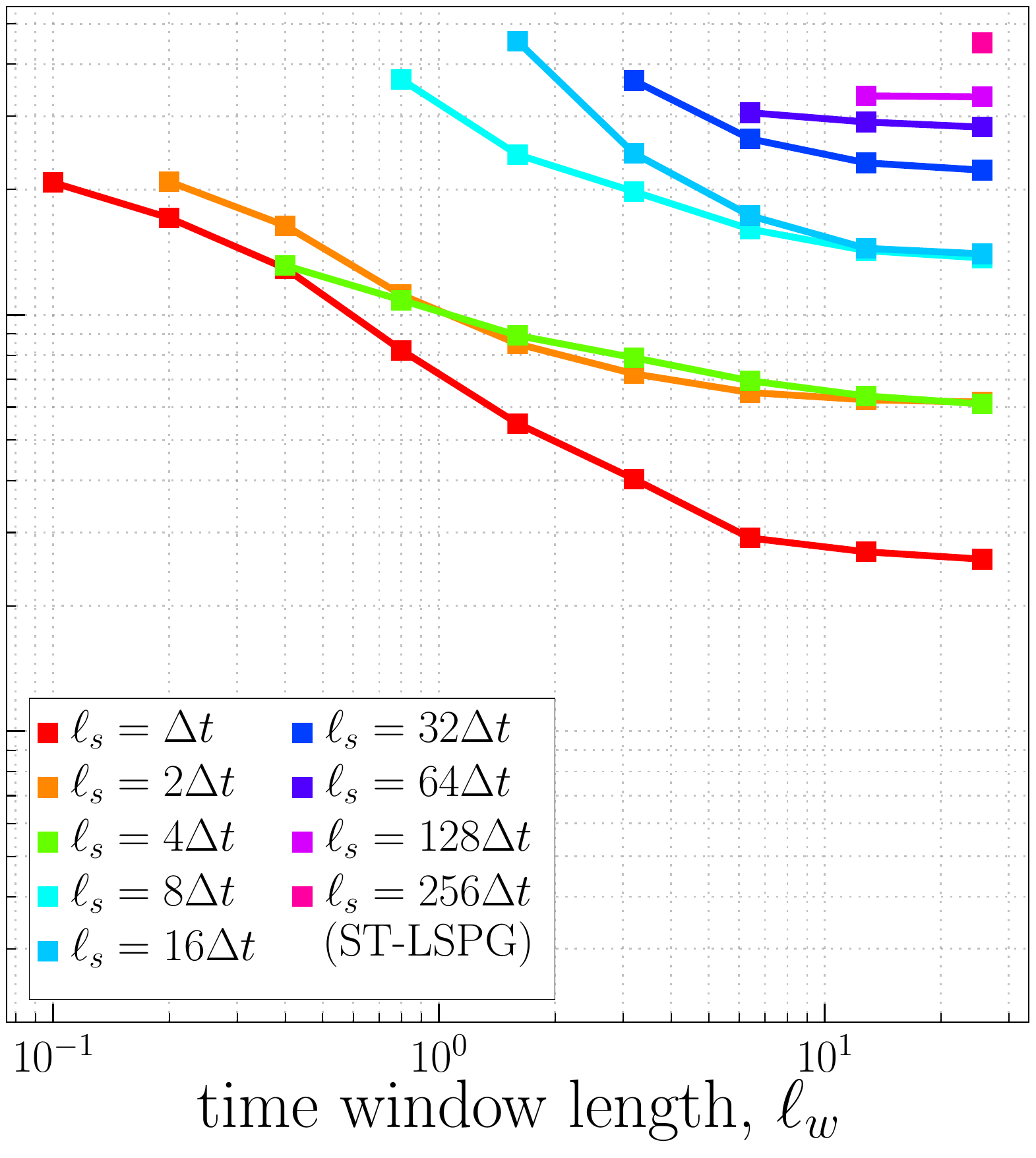}}\hspace{1mm}\nolinebreak
\subfloat[residual $\ell^2$ norm: $e_s = 0.999, e_t = 0.99$][\centering \hspace{1.5mm}residual $\ell^2$ norm:\par 
\hspace{5.5mm}$e_s = 0.999, e_t = 0.99$\label{f:burgersResidual4}]{\includegraphics[height=0.264\textwidth]{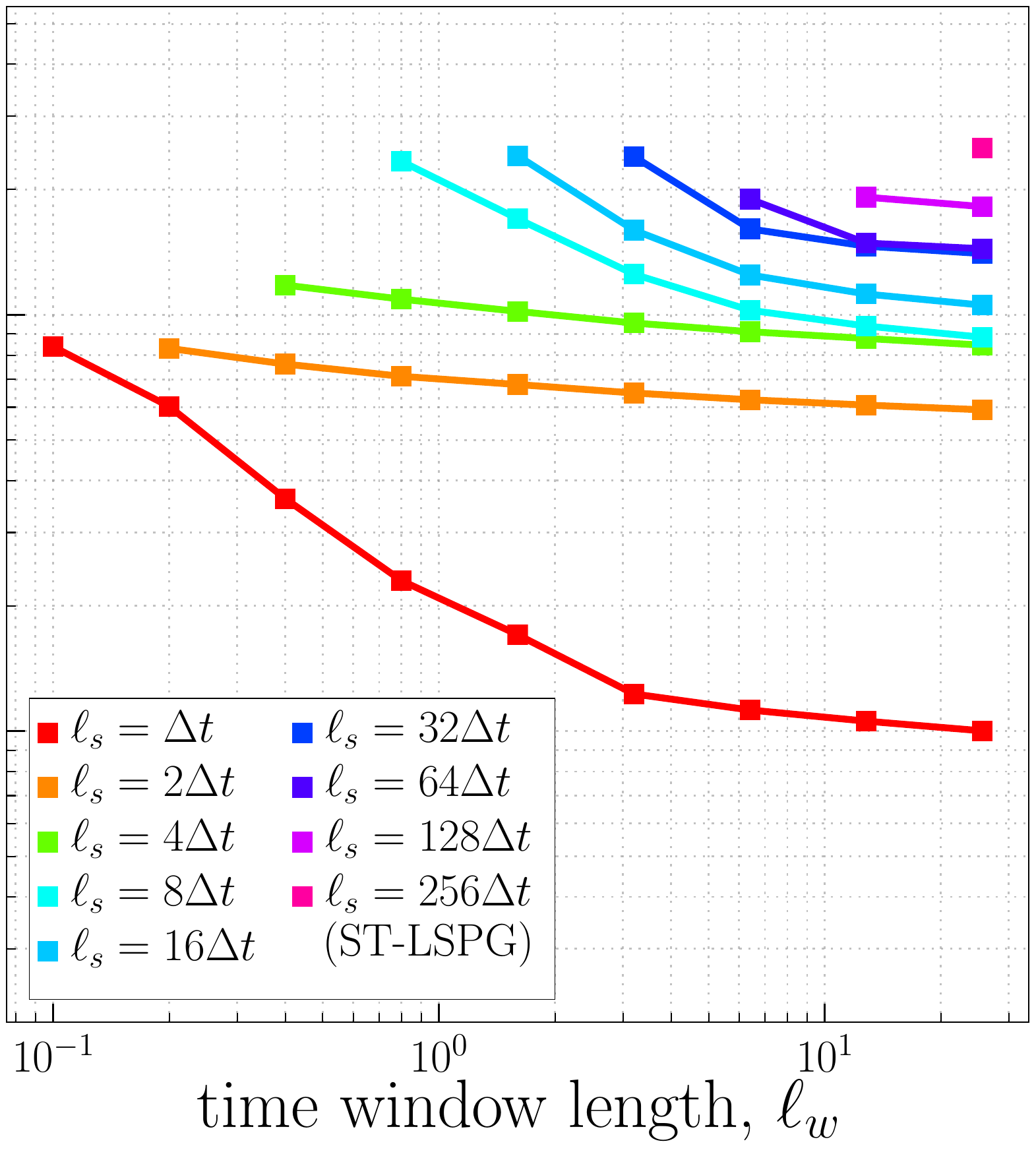}}\hspace{1mm}\nolinebreak
\subfloat[residual $\ell^2$ norm: $e_s = 0.999, e_t = 0.999$][\centering \hspace{1.5mm}residual $\ell^2$ norm:\par 
\hspace{5.5mm}$e_s = 0.999, e_t = 0.999$\label{f:burgersResidual6}]{\includegraphics[height=0.264\textwidth]{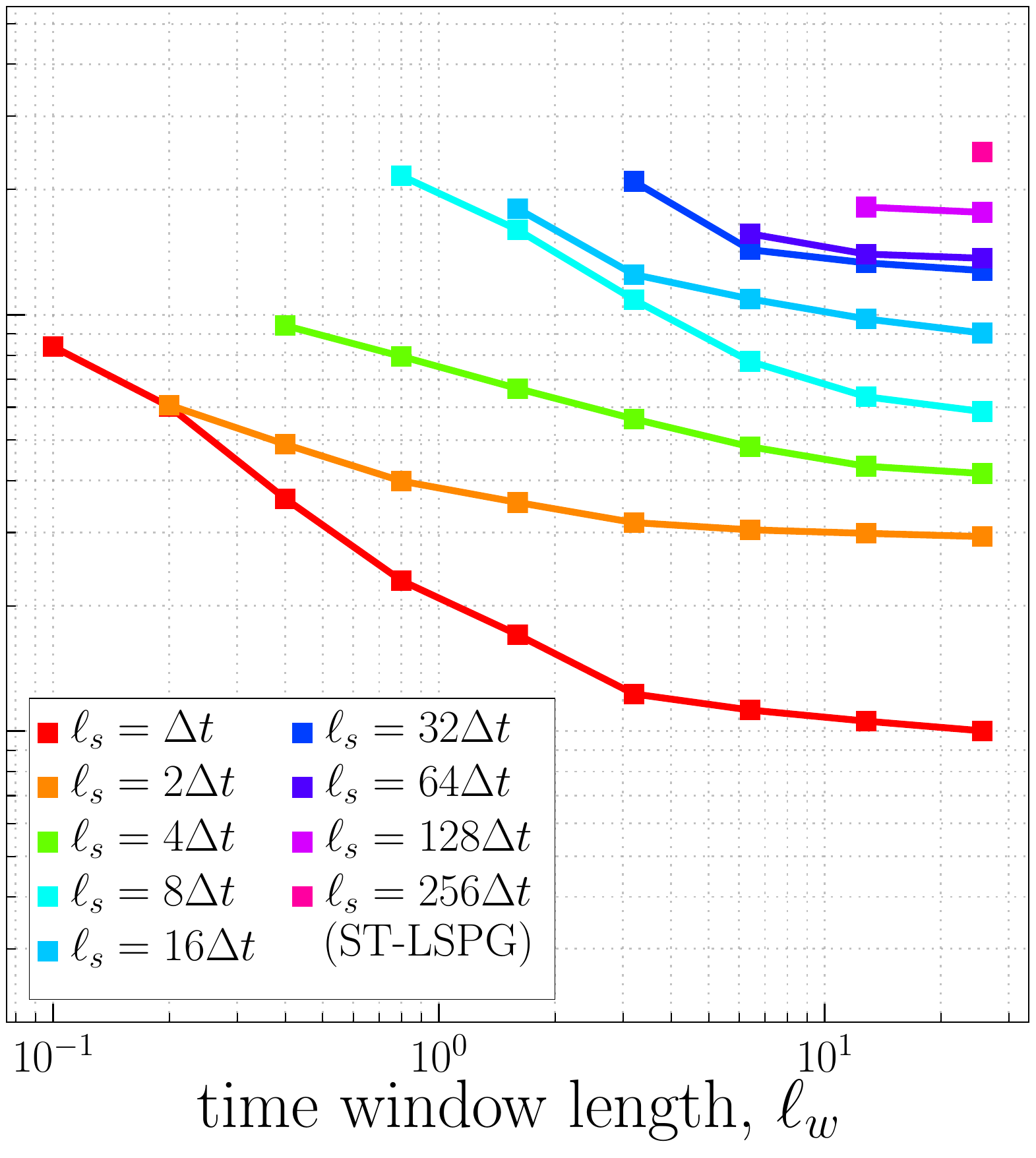}}
\caption{\emph{Burgers' equation.} WST-LSPG MSE and  residual $\ell^2$ norms for basis set $e_s = 0.99, e_t = 0.99$, $e_s = 0.99, e_t = 0.999$, $e_s = 0.999, e_t = 0.99$, and $e_s = 0.999, e_t = 0.999$. The results are presented as a function of the window length, $\ell_w$, for various sub-window lengths, $\ell_s$. The MSE and the residual $\ell^2$ norm exhibit opposite behavior with increasing time window lengths.}
\label{f:window1}
\end{figure}

For each set of temporal window lengths and each set of temporal sub-window lengths, Figure~\ref{f:window1} shows the MSE, IMSE, and the  residual $\ell^2$ norm for four different sets of space--time bases. The four different sets of space--time bases are based on two different spatial energy basis parameter values, $e_s =0.99$ and $e_s=0.999$, and two different temporal energy basis parameter values, $e_t=0.99$ and $e_t=0.999$. The window length for this study is set to be $\ell_w=\{1,2,4,8,16,32,64,128,256\}$ and the sub-window length is set to be $\ell_s=\{1,2,4,8,16,32,64,128,256\}$. To study how the error metrics ---MSE, IMSE, residual $\ell^2$ norm--- are affected by the WST-LSPG parameters, the experiments consist of running different ROMs with different combinations of window lengths and sub-window lengths. In this experimental setup, $\ell_s\leq \ell_w$ is a requirement.

Figures~\ref{f:burgersMSE0},~\ref{f:burgersMSE2},~\ref{f:burgersMSE4},~\ref{f:burgersMSE6} show how the MSEs behave with different window and sub-window lengths. It is observed that when the window length increases, the MSE generally first decreases, but then increases to create a minimum in the MSE at approximately $\ell_w\in [0.4,0.8]$. However, when $\ell_w=\Delta t$, as the window length increases, the MSE continues to decrease. The lowest MSE can be seen when $\ell_w=\Delta t$, which shows great improvement when compared to the $\ell_s=256\Delta t$, which corresponds to ST-LSPG. It can already be seen that, even with only one window, introducing sub-windows in the construction of the space--time bases leads to a more accurate space--time ROM as compared to that of the ST-LSPG. It can be seen that larger sub-window lengths lead to higher MSEs.  Figures~\ref{f:burgersMSE0},~\ref{f:burgersMSE2},~\ref{f:burgersMSE4},~\ref{f:burgersMSE6} additionally show that increasing $e_s$ and $e_t$ decreases the MSE. 

Next, Figures~\ref{f:burgersResidual0},~\ref{f:burgersResidual2},~\ref{f:burgersResidual4},~\ref{f:burgersResidual6} show how the  residual $\ell^2$ norms behave with different window and sub-window lengths. It can be seen that as the window length increases, the residual $\ell^2$ norms decrease for all values of sub-window length. The highest residual $\ell^2$ norms are found for the case where $\ell_s=256\Delta t$, which corresponds to the ST-LSPG technique. These results show that introducing sub-windows can greatly lower the residual $\ell^2$ norms by an order of magnitude across the four different set of space--time bases. However, the MSE results behave contrary to the way the residual $\ell^2$ norm results behave, especially at larger time window lengths. This behavior is consistent with what was observed in~\cite{parish2019windowed}.
Some possible explanations for this behavior include the following. In model reduction, the accuracy of the ROM is dependent on the trial subspace data and its corresponding subspace, hence, it is expected that the MSE will improve with increased fidelity of the bases. However, residual $\ell^2$ norms measure how well a solution, whether it be a ROM or FOM satisfies the governing equations. Since the goal of projection-based model reduction is to minimize the residual $\ell^2$ norms given the trial tensor data obtained offline, it can be deduced that without a physical constraint on the original PDE problem, the residual $\ell^2$ norms can increase with decreased time windowing length. An additional study was conducted to assess if the type of problem and/or the fidelity of the basis chosen is responsible for the opposing behaviors of the MSE and the residual $\ell^2$ norms at longer window lengths. To see if the phase error contributes to this opposing behavior, the IMSE is plotted in \ref{s:imse}; it was found to be a non-contributing factor.

Overall, Figure~\ref{f:window1} shows that there is some set of parameters (time window length, $\ell_w$ and sub-window length $\ell_s$) that gives the best MSE, IMSE and residual $\ell^2$ norm, a solution that minimizes the error relative to the space--time subspace while still following the physics of the problem as much as possible. In the next section, another experiment is made in order to study the effect of windowing and hyper-reduction on the MSE and the residual $\ell^2$ norm.

\subsubsection{Windowed parameter study with hyper-reduction}
This section compares the performances of the following space--time model reduction methods with varying method parameters: ST-LSPG, WST-LSPG, ST-GNAT, and WST-GNAT for the Burgers' equation. The goal of these comparisons is to gain insight and understanding on how using windows and sub-windows for the space--time bases affects the accuracy of the space--time ROMs. This goal is achieved by creating a study that collects space--time model reduction results over varying ROM parameters (i.e., window size, sub-window size, trial basis dimension, residual basis dimension) and constructing Pareto plots from these results. These Pareto plots highlight the performances of the ROM given the error output of interest and the relative wall time. For the Burgers' equation, two sets of Pareto plots are presented corresponding to two different error output of interest: (1) MSE and (2) residual $\ell^2$ norm. 
For each window size $\ell_w$ and sub-window size $\ell_s$, smaller localized Pareto plots are presented. An overall Pareto plot and its corresponding Pareto front for each error output of interest and windowed space--time parameter of interest are presented as well. The Pareto fronts for a given error output and given windowed space--time parameter of interest are characterized by the given set of parameters that minimize the competing relative error and relative wall time. For the Burgers' equation, the number of windows and sub-windows is varied such that $\numWindows=\Bmat{1,2,4,8,16,32,64,128,256}$ and $\numSubWindows=\Bmat{1,2,4,8,16,32,64,128,256}$. The relative wall time is calculated by dividing the time it takes to execute the method of interest (e.g. WST-LSPG) by the time it takes to complete the FOM. The number of spatial elements in the sample mesh for hyper-reduction is defined as $z_s =\bmat{10, 20, 30 &\ldots& 200}$, and the number of temporal indices in the space--time simulation for each window $k$ is defined as $z_t=\bmat{2&\ldots &N^k_t}$. 
%

%
\begin{figure}[!t]
\centering
\subfloat[$\numWindows=1,\ell_{w}=256\Delta t$\label{f:mse8}]{\pareto{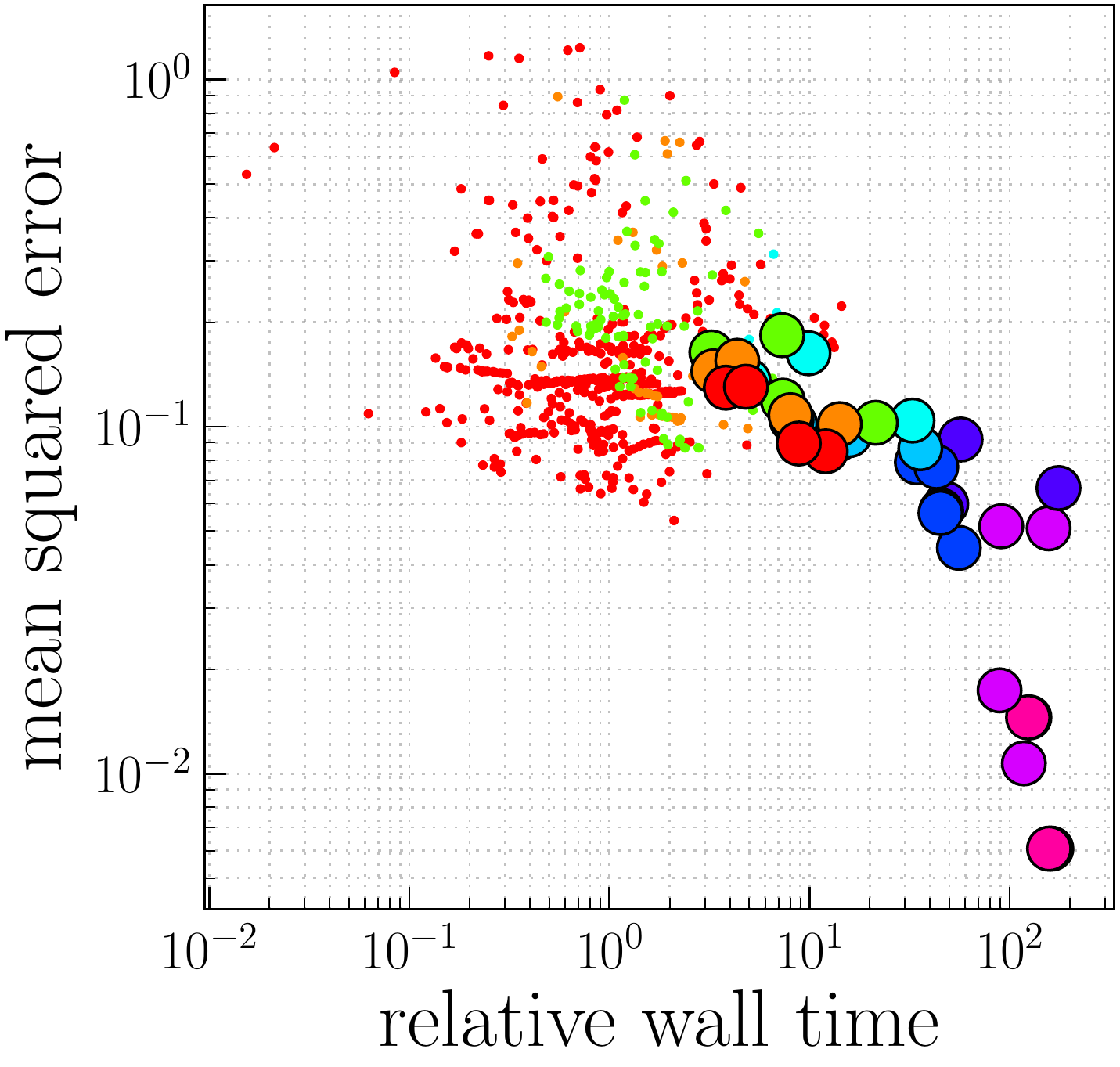}}\hspace{1mm}\nolinebreak
\subfloat[$\numWindows=2,\ell_{w}=128\Delta t$\label{f:mse7}]{\pareto{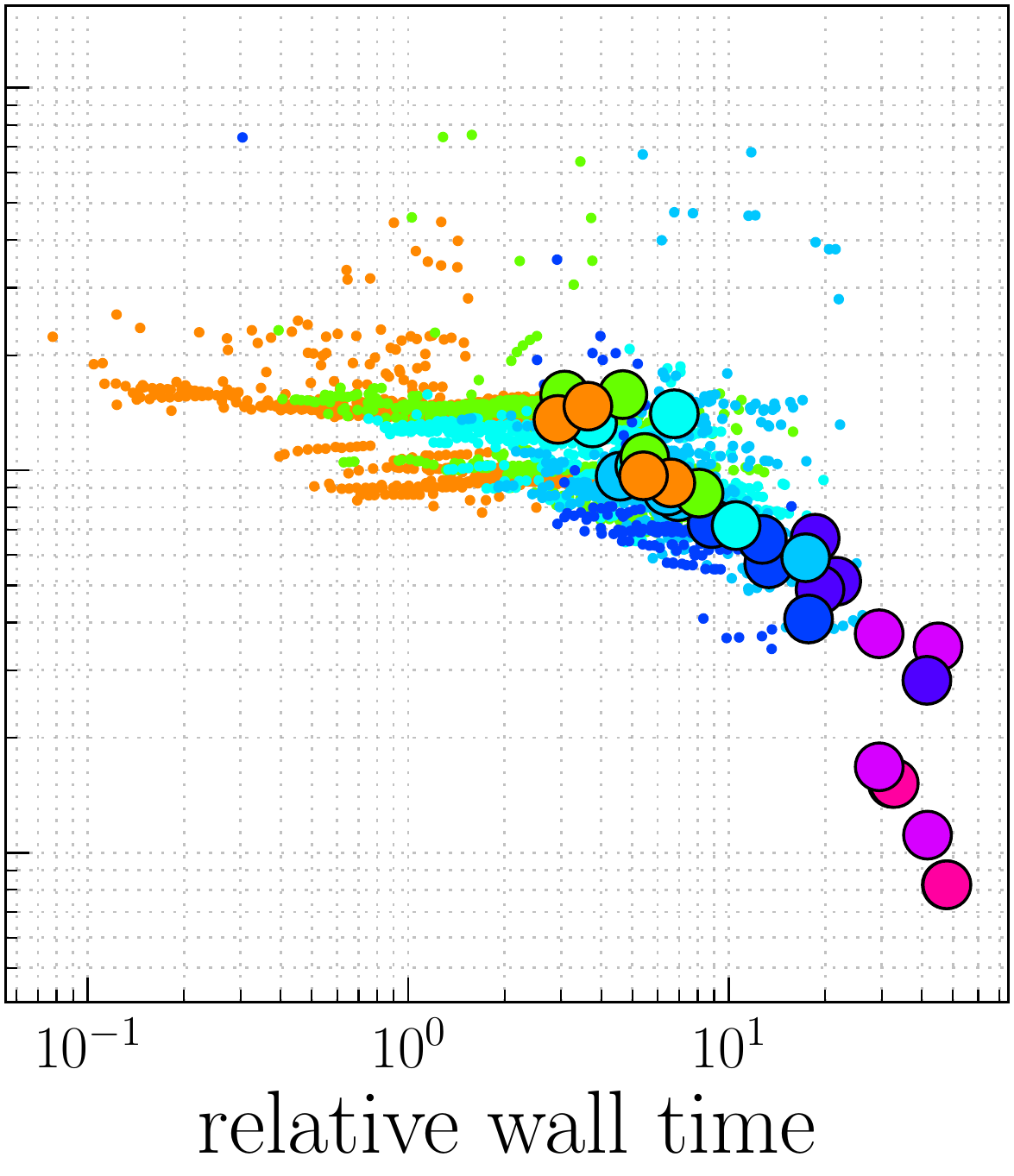}}\hspace{1mm}\nolinebreak
\subfloat[$\numWindows = 4,\ell_{w}=64\Delta t$\label{f:mse6}]{\pareto{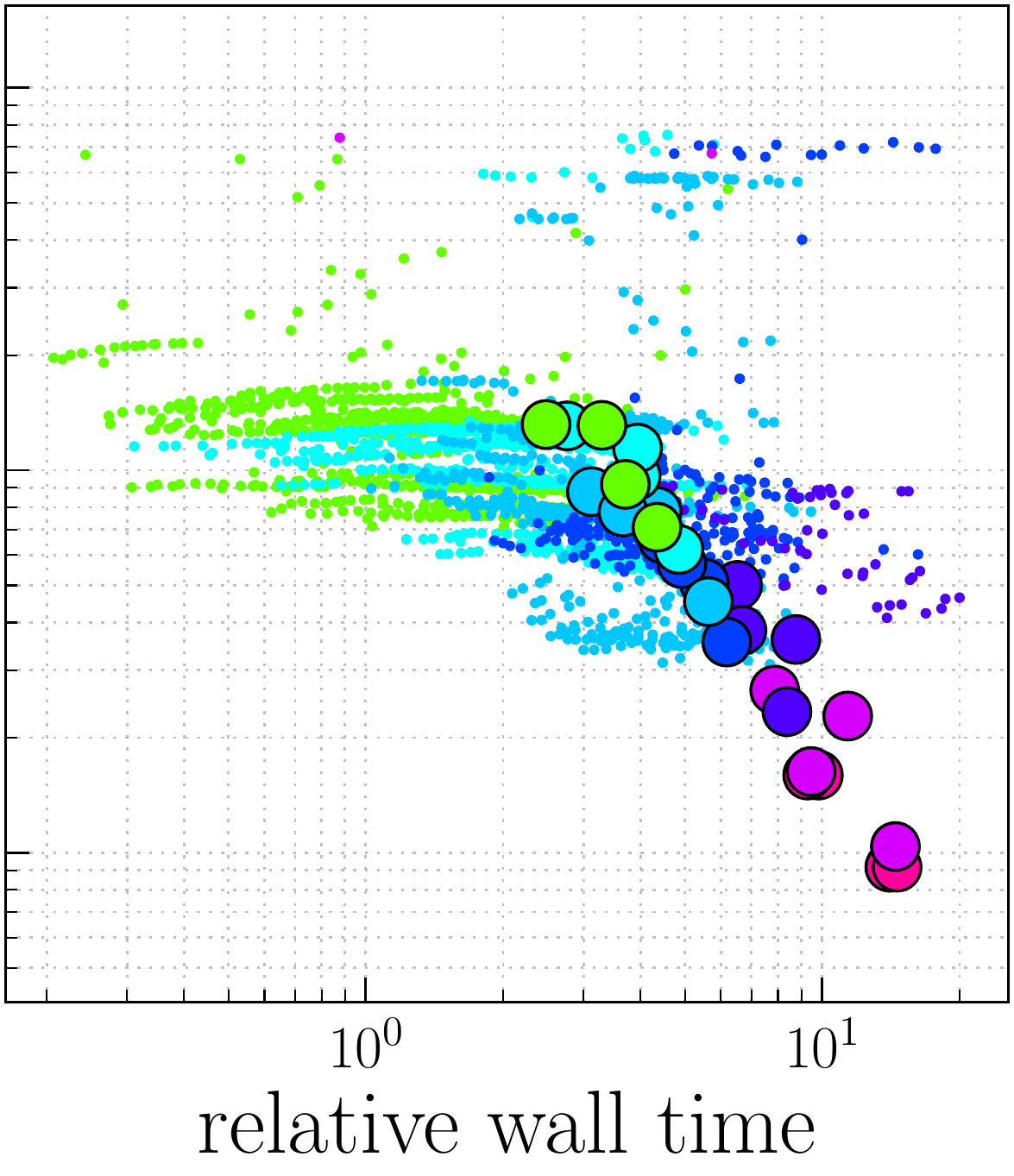}}\hspace{1mm}\nolinebreak
\subfloat[$\numWindows = 8,\ell_{w}=32\Delta t$\label{f:mse5}]{\pareto{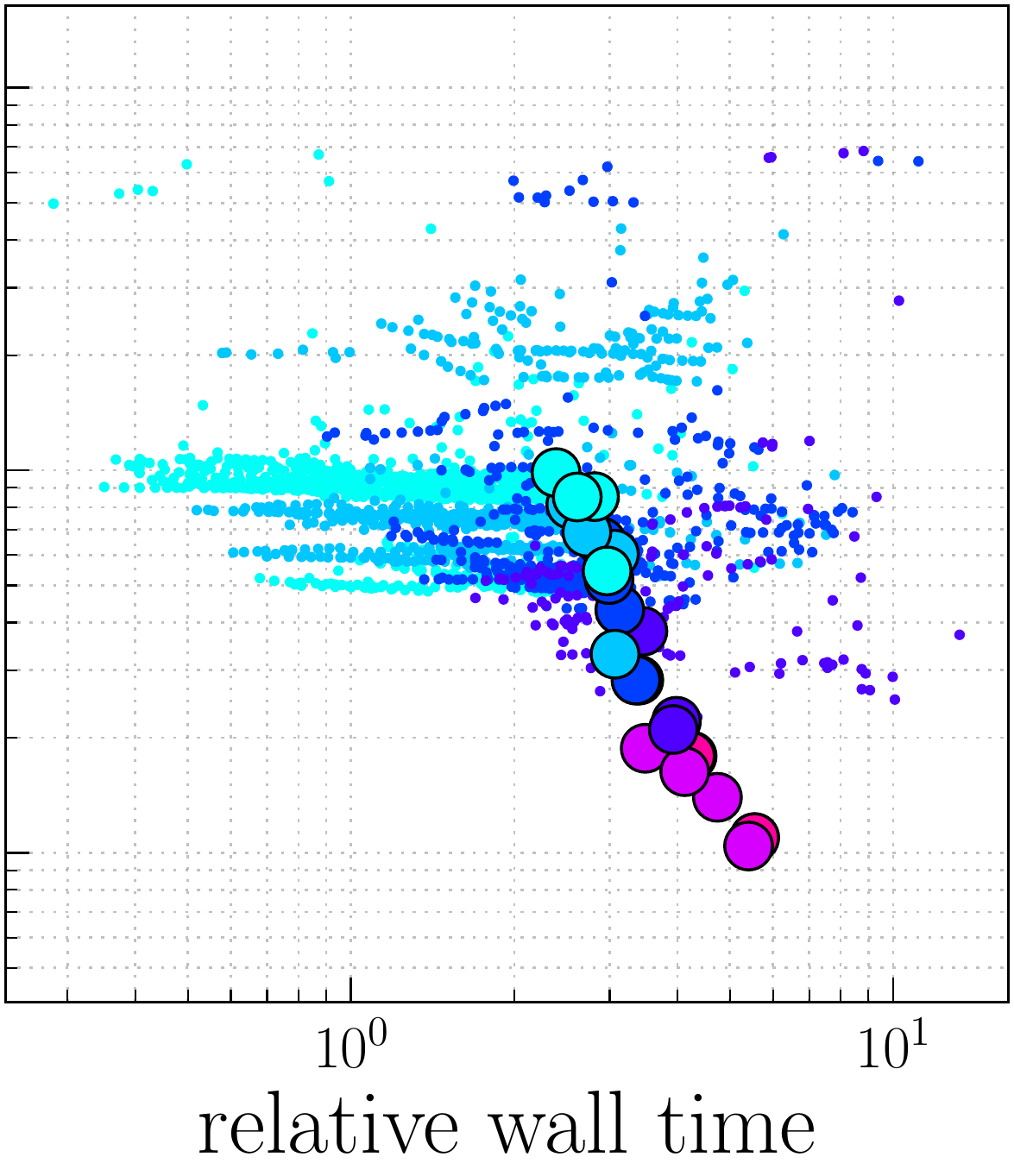}}\hspace{1mm}\nolinebreak
\subfloat[$\numWindows = 16,\ell_{w}=16\Delta t$\label{f:mse4}]{\pareto{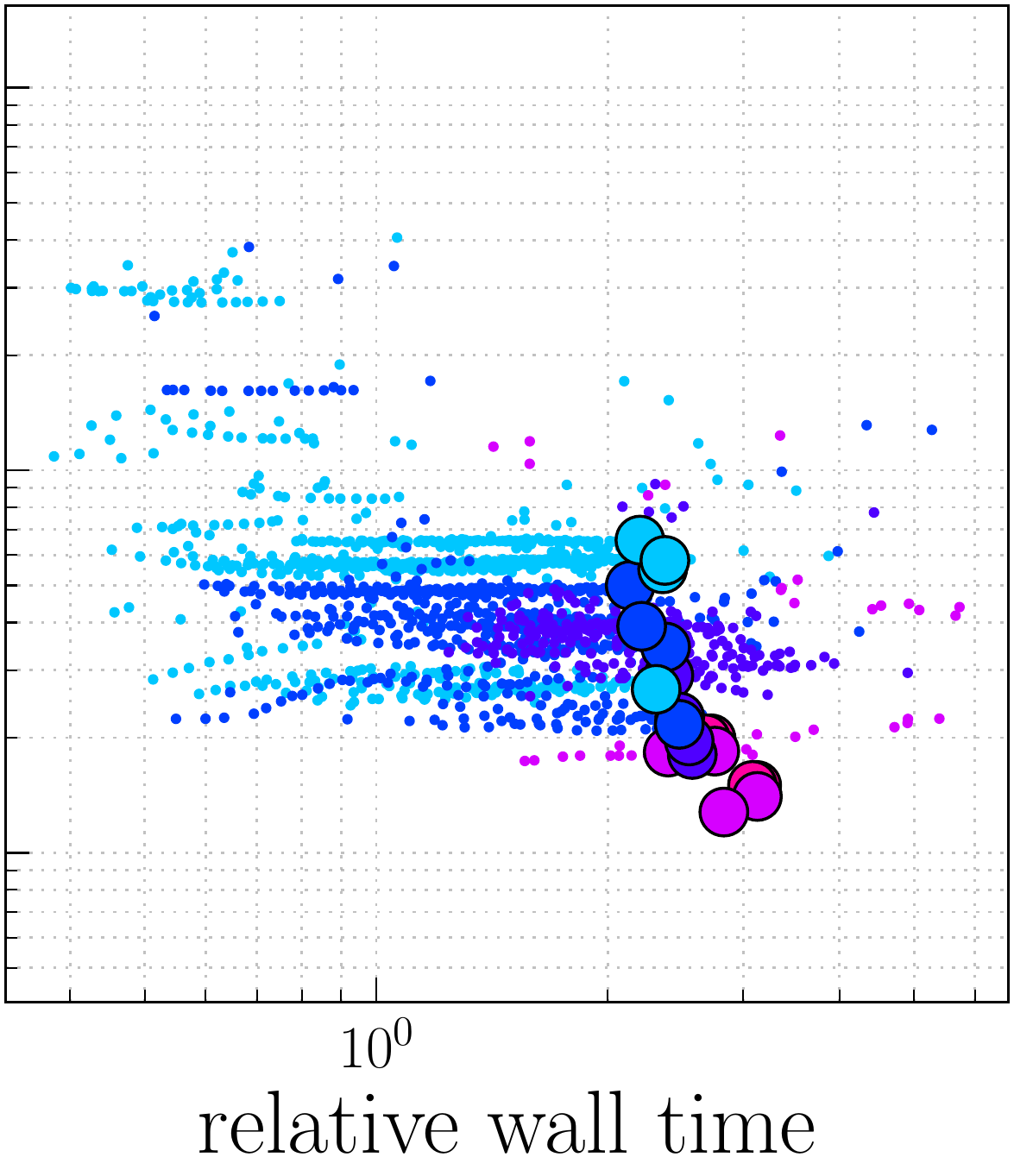}}\vspace{-1mm}\\
\subfloat[$\numWindows = 32,\ell_{w}=8\Delta t$\label{f:mse3}]{\pareto{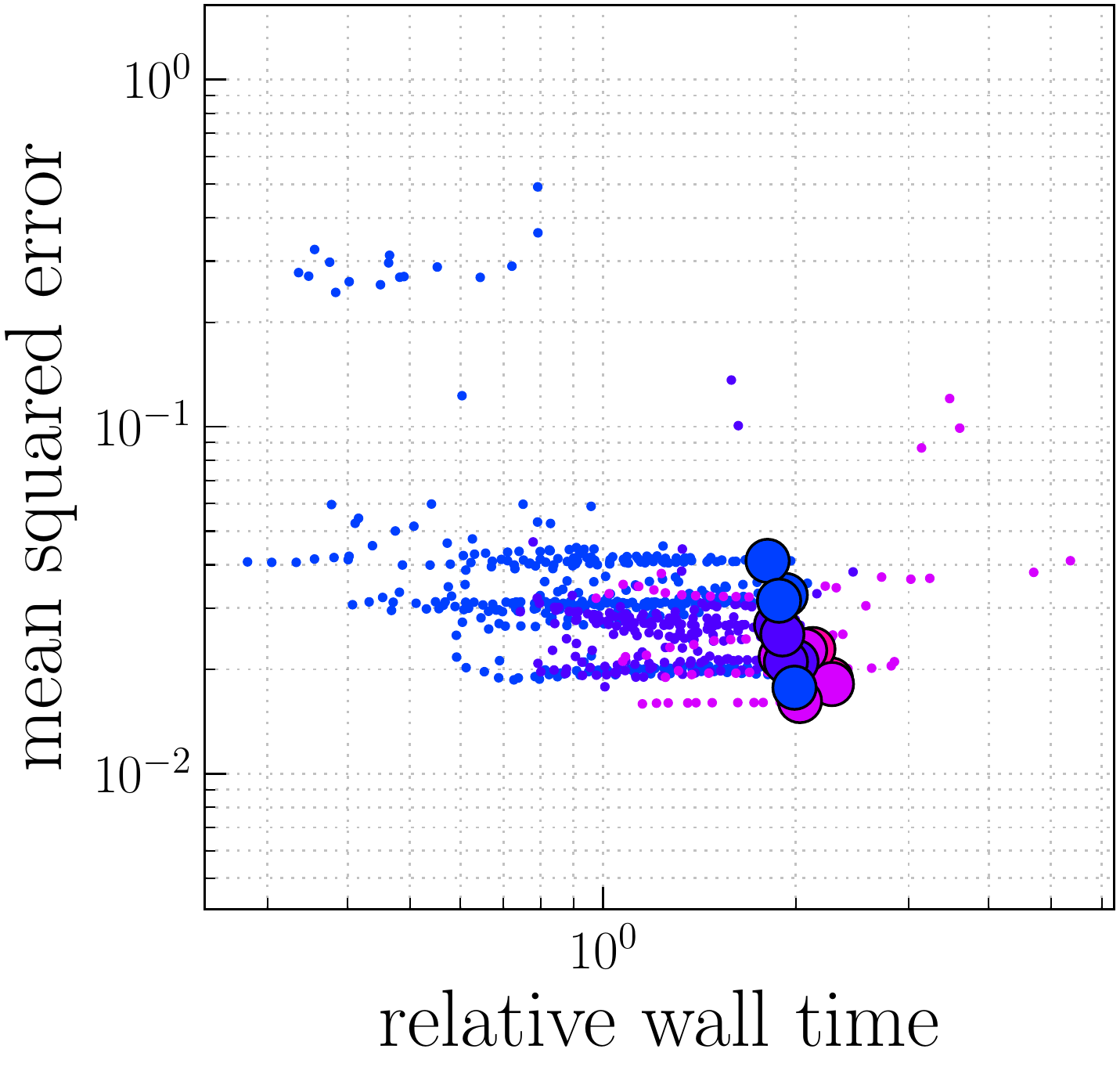}}\hspace{1mm}\nolinebreak
\subfloat[$\numWindows = 64,\ell_{w}=4\Delta t$\label{f:mse2}]{\pareto{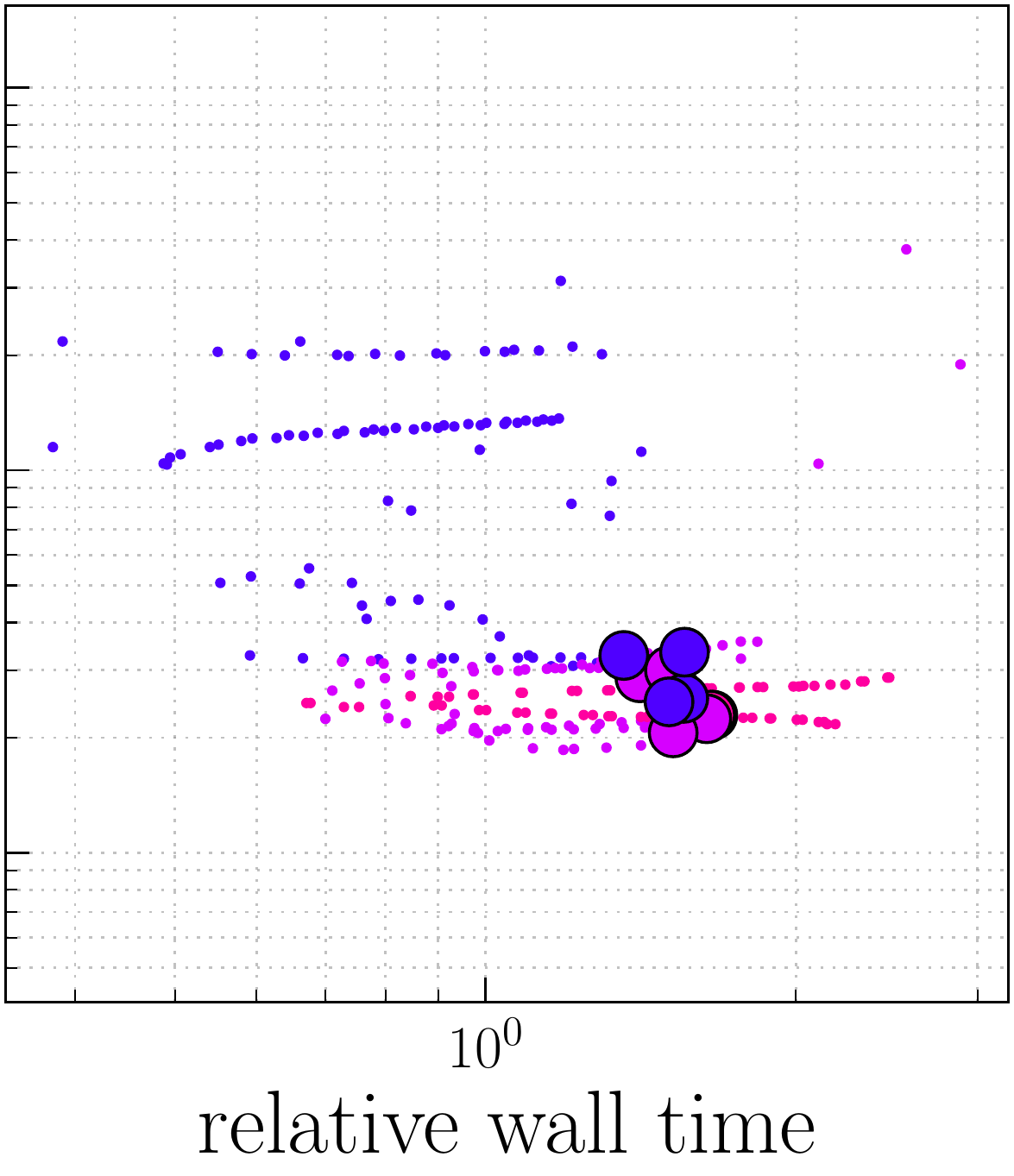}}\hspace{1mm}\nolinebreak
\subfloat[$\numWindows = 128,\ell_{w}=2\Delta t$\label{f:mse1}]{\pareto{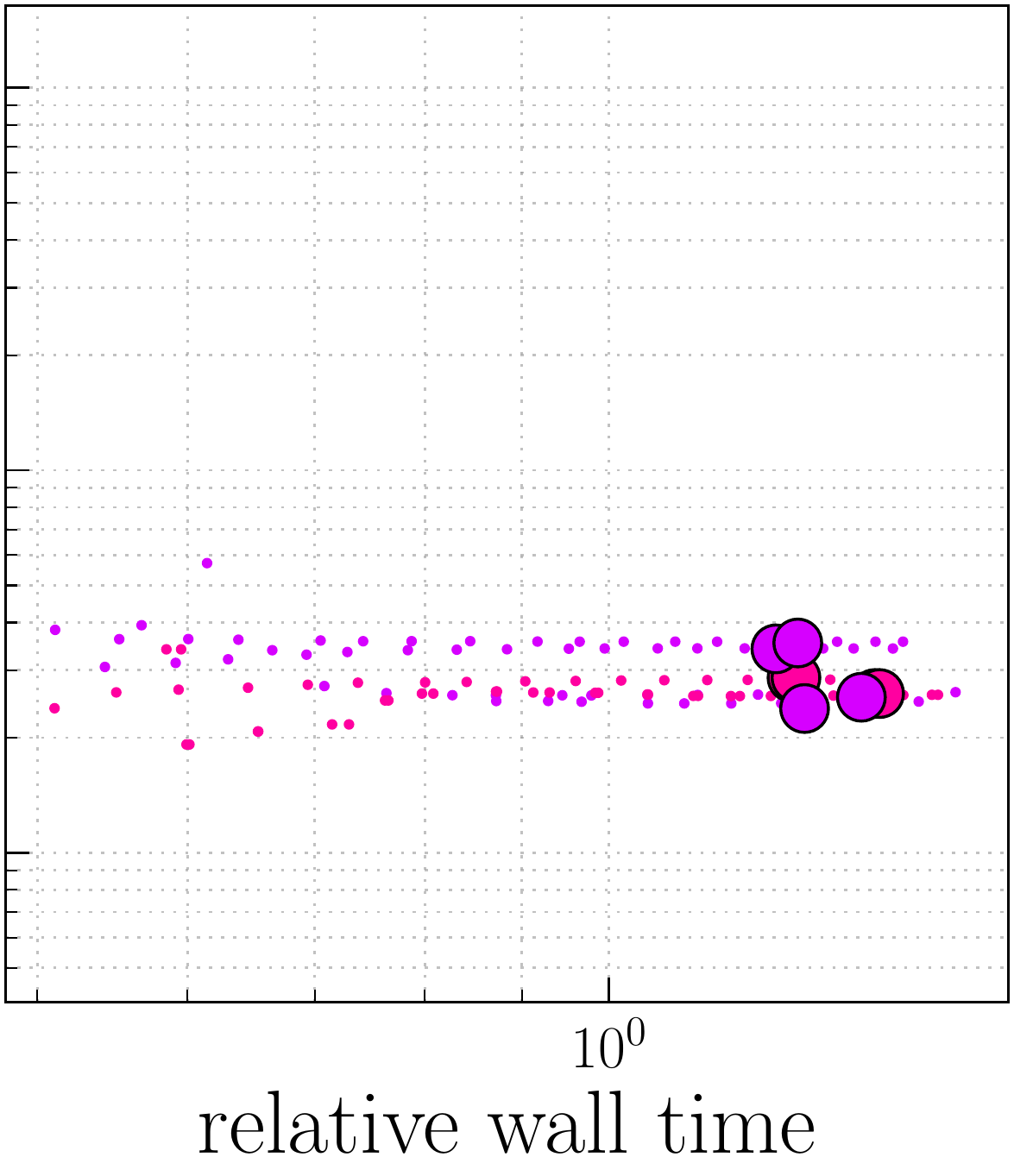}}\hspace{1mm}\nolinebreak
\subfloat[$\numWindows = 256,\ell_{w}=\Delta t$\label{f:mse0}]{\pareto{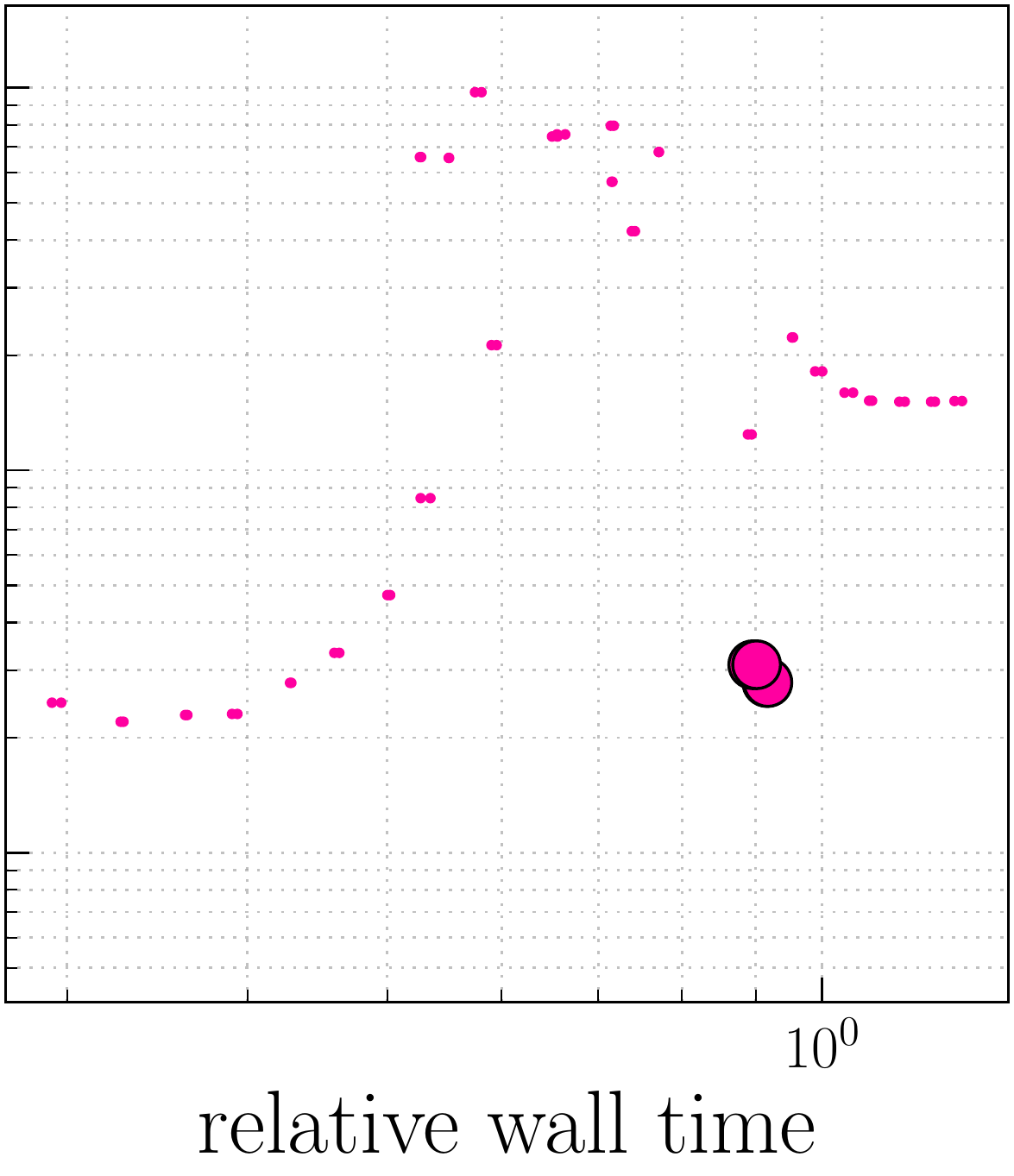}}\hspace{1mm}\nolinebreak
\legend{legend.pdf}\vspace{-1mm}\\
\subfloat[Combined Pareto Plot with Pareto Front\label{f:MSE_ALL}]{\includegraphics[width=1\textwidth]{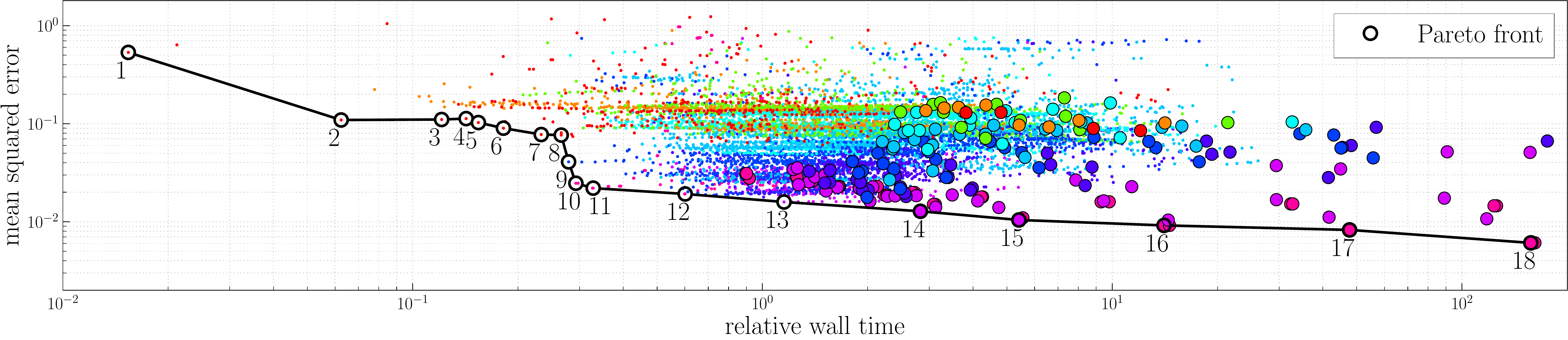}}
\caption{\emph{Burgers' equation}. Pareto plots showing the MSE vs. relative wall time for varying sub-window lengths. These Pareto plots highlight how the sub-window length $\ell_s$ affect the mean squared error for each time window length $l_w$ for WST-LSPG (ST-LSPG when $\ell_w=\ell_s$) and WST-GNAT (ST-GNAT when $\ell_w=\ell_s$). Each hyper-reduction point refers to a unique set of spatial sample nodes and temporal sample nodes. The last Pareto plot combines all the results and shows the Pareto front.}
\label{f:paretoMSE}
\end{figure}
\begin{figure}[!hp]
\centering
\subfloat[$\ell_{s}=256\Delta t$\label{f:mse_e8}]{\pareto{Pareto_MSE_8_e.pdf}}\hspace{1mm}\nolinebreak
\subfloat[$\ell_{s}=128\Delta t$\label{f:mse_e7}]{\pareto{Pareto_MSE_7_e.pdf}}\hspace{1mm}\nolinebreak
\subfloat[$\ell_{s}=64\Delta t$\label{f:mse_e6}]{\pareto{Pareto_MSE_6_e.pdf}}\hspace{1mm}\nolinebreak
\subfloat[$\ell_{s}=32\Delta t$\label{f:mse_e5}]{\pareto{Pareto_MSE_5_e.pdf}}\hspace{1mm}\nolinebreak
\subfloat[$\ell_{s}=16\Delta t$\label{f:mse_e4}]{\pareto{Pareto_MSE_4_e.pdf}}\vspace{-1mm}\\
\subfloat[$\ell_{s}=8\Delta t$\label{f:mse_e3}]{\pareto{Pareto_MSE_3_e.pdf}}\hspace{1mm}\nolinebreak
\subfloat[$\ell_{s}=4\Delta t$\label{f:mse_e2}]{\pareto{Pareto_MSE_2_e.pdf}}\hspace{1mm}\nolinebreak
\subfloat[$\ell_{s}=2\Delta t$\label{f:mse_e1}]{\pareto{Pareto_MSE_1_e.pdf}}\hspace{1mm}\nolinebreak
\subfloat[$\ell_{s}=\Delta t$\label{f:mse_e0}]{\pareto{Pareto_MSE_0_e.pdf}}\hspace{1mm}\nolinebreak
\legend{legend2.pdf}\vspace{-1mm}\\
\subfloat[Combined Pareto Plot with Pareto Front\label{f:MSE_e_ALL}]{\includegraphics[width=1\textwidth]{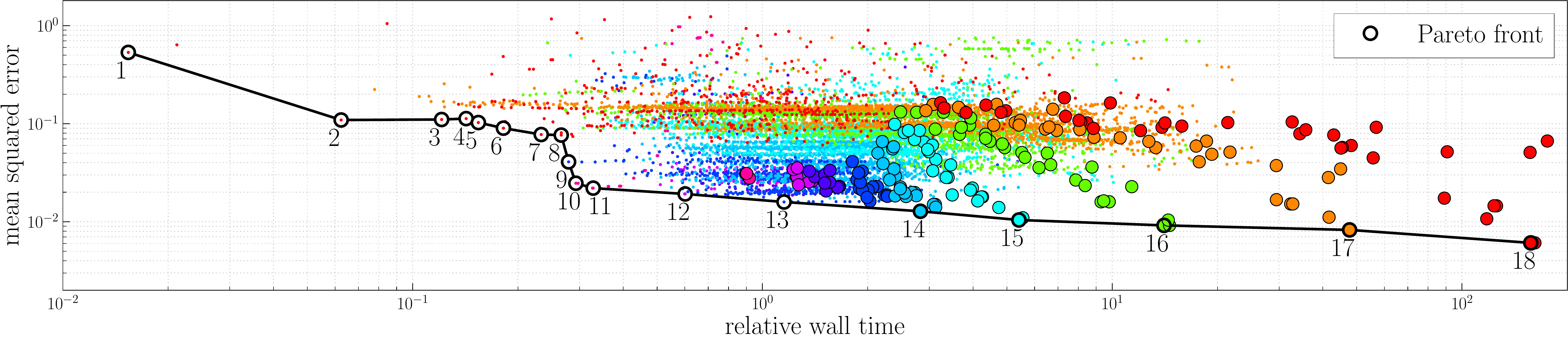}}
\caption{\emph{Burgers' equation}. Pareto plots showing the MSE vs. relative wall time for varying window lengths. These Pareto plots highlight how the window length $\ell_w$ affect the mean squared error for each sub-window length $\ell_s$ for WST-LSPG (ST-LSPG when $\ell_w=\ell_s$) and WST-GNAT (ST-GNAT when $\ell_w=\ell_s$). Each hyper-reduction point refers to a unique set of spatial sample nodes and temporal sample nodes. The last Pareto plot combines all the results and shows the Pareto front (same as Figure~\ref{f:MSE_ALL}).}
\label{f:paretoMSE_opposite}
\vspace{-2mm}
\begin{center}
\captionof{table}{\label{t:MSE}\emph{Burgers' equation}. Method and parameter values for labeled data points in Figures~\ref{f:paretoMSE} and~\ref{f:paretoMSE_opposite} that yield Pareto-optimal performance for MSE. Highlighted rows refer to test cases that additionally exist on the optimal Pareto Front for the residuals.}
\resizebox{\columnwidth}{!}{%
\begin{tabular}{ ccccccccccccccc } \toprule
case& method & $\ell_w$ & $\ell_s$ & \thead{total\\$n_{st}$} &\thead{total\\$n^r_{st}$} &$e_s$ & $e_t$ & $e_{r_s}$ & $e_{r_t}$ & $z_t$ & $z_s$ & MSE & \thead{x-LSPG relative\\ wall time} &\thead{relative wall\\time}\\
\midrule \midrule
\rowcolor{matcha}1 & ST-GNAT&25.6&	25.6& 98& 148&	0.999&	0.999& 0.999&	0.999&	2	&130 & 0.533224345 & 12.45653562&0.015386589\\ 
\rowcolor{matcha} 2 & ST-GNAT&	25.6&	25.6&22&76&	0.99&	0.99& 0.999&	0.999&	18	&20 & 0.109034729 &3.786097827&0.062465228\\ 
 3  & ST-GNAT	& 25.6&	25.6& 22&76&	0.99&	0.99 &0.999	&0.999&	18&	30& 0.110290589& 3.857666540&0.121033087\\ 
 4  & ST-GNAT	&25.6&	25.6&22&76&	0.99&	0.99  & 0.999	&0.999&	34&	20&0.112667597& 3.852059257&0.142223121\\ 
 5  & ST-GNAT&	25.6&	25.6&22& 76&	0.99&	0.99 &  0.999&	0.999&	18&	40& 0.102668917 &3.846951908&0.154143598\\ 
\rowcolor{matcha} 6  & ST-GNAT	&25.6&	25.6&	68& 190&0.999&	0.99 &0.999	&0.999&	18&	50&  0.089993176& 9.092380722& 0.181795453 \\ 
\rowcolor{matcha}7  & ST-GNAT&25.6&	25.6& 68&190&	0.999&	0.99& 0.999&	0.999&	18&	60& 0.077490715& 8.935311415&0.233166894\\  
\rowcolor{matcha} 8  & ST-GNAT&	25.6&	25.6&68&190&	0.999&	0.99& 0.999&	0.999&	18&	70& 0.076497668& 9.036366568&0.265945585\\ 
 9  & WST-GNAT&	0.8&	0.8& 360 & 1563&	0.99&	0.99 &0.999&	0.999&	5&	40& 0.040785017& 1.313484544&0.279190563\\
 10  &  WST-GNAT&	0.1&	0.1&1730&4010&	0.99&	0.99& 0.999&	0.999&	2&	20& 0.024680093& 0.894093204&0.292897573 \\
 11  & WST-GNAT&	0.1&	0.1& 1730&4010&	0.99&	0.999 &0.999&	0.999&	2&	30 &0.022005269&0.904674125&0.328264012 \\
\rowcolor{matcha}12  &  WST-GNAT& 	0.2&	0.1& 2490& 4186 &	0.999&	0.999 & 0.999&	0.999&	3&	50 &0.019193423&1.383185815& 0.601325305\\
 13  &  WST-GNAT&0.8&	0.2& 1496&5247&	0.999&	0.99& 0.9999&	0.9999&	9&	70 &0.015893315&2.003705125&  1.152805599\\
 14  &  WST-LSPG&	1.6&	0.2	&1496& \NA&0.999&	0.99  &\NA &\NA &\NA &\NA &0.012777172& 2.829880654 &\NA \\
 15  & WST-LSPG&	3.2& 	0.2	&2205&\NA& 0.999	& 0.999& \NA &\NA &\NA &\NA & 0.010414663 &  5.411230826&\NA \\
 16  & WST-LSPG&	6.4& 	0.1& 2490&\NA&	0.999& 	0.999&  \NA &\NA &\NA &\NA & 0.009168400&  14.06306610&\NA \\
 17  & WST-LSPG&	12.8& 	0.1& 2490& \NA &	0.999& 	0.999 & \NA &\NA &\NA &\NA & 0.008247672&  47.71383365&\NA \\
 18  &WST-LSPG&	25.6& 	0.1& 2490& \NA& 	0.999& 	0.99&\NA &\NA &\NA &\NA & 0.006083036 & 157.3006332&\NA \\
 \bottomrule
\end{tabular}
}
\end{center}
\end{figure}
The first set of Pareto plots can be seen in Figure~\ref{f:paretoMSE}, which shows nine different Pareto plots for MSE vs. relative wall time for varying values of $\ell_w$. Each Pareto plot shows the WST-LSPG and WST-GNAT results for all four different sets of space--time bases. Note that the data points that refer to $\ell_w=\ell_s=256\Delta t$ is equivalent to the ST-LSPG and ST-GNAT method. Figure~\ref{f:paretoMSE} shows that incorporating hyper-reduction into WST-LSPG (i.e., WST-GNAT) decreases the relative wall time compared to its non-hyper-reduction counter part (i.e., WST-LSPG). For each window length $\ell_w$, the MSE decreases as the sub-window lengths decrease for both WST-LSPG and WST-GNAT. The MSE is the lowest when $\ell_s=\Delta t$ and $\ell_w=\Delta t$; however, these cases are associated with longer relative wall times. To obtain the lowest relative wall time, fewer number of windows and longer sub-window lengths are needed as exhibited in the limiting case of ST-LSPG. However, this reduction in relative wall time is at the expense of a higher MSE. Figure~\ref{f:MSE_ALL} combines the results from Figure~\ref{f:mse8}-\ref{f:mse0} and shows the Pareto front for all the MSE vs. relative wall time results. From this overall Pareto plot, the worst performing  methods correspond to larger sub-window lengths and the ST-GNAT methods, while the optimally performing methods on the Pareto front refer to WST-LSPG methods with smaller sub-window lengths.

%
%
The next set of Pareto plots presents the same data as shown in Figure~\ref{f:paretoMSE}, but highlights how the MSE and relative wall time are affected by varying the window length for a constant sub-window length. The results for this study can be seen in Figure~\ref{f:paretoMSE_opposite}. Figure~\ref{f:mse_e8} shows results for only the ST-LSPG and the ST-GNAT method (when the sub-window length and the time window length are both equal to the length of the entire time simulation). Next, in Figures~\ref{f:mse_e7}-\ref{f:mse_e3} is is seen that, for a given sub-window length, the lowest MSE and relative wall time occurs when the the window length is equal to the sub-window length. Figures~\ref{f:mse_e2}-\ref{f:mse_e0}, however, show different behaviors: For smaller sub-window lengths, the lowest MSE is associated with larger window lengths, $\ell_w>\ell_s$. Figure~\ref{f:MSE_e_ALL} combines the results from Figures~\ref{f:mse_e8}-\ref{f:mse_e0} and shows the same Pareto front as seen in Figure~\ref{f:MSE_ALL}. The difference between Figure~\ref{f:MSE_e_ALL} and Figure~\ref{f:MSE_ALL} is that Figure~\ref{f:MSE_ALL} highlights how the MSE decreases with smaller sub-window length, and Figure~\ref{f:MSE_e_ALL} highlights how the MSE decreases with smaller window length.  

The parameters for MSE that lie on the Pareto front, labeled as cases $1-18$ in Figures~\ref{f:MSE_ALL} and~\ref{f:MSE_e_ALL}, are listed in Table~\ref{t:MSE}. Cases 1-8 refer to the original ST-GNAT method with hyper-reduction (i.e.,  $\ell_w=\ell_s=T$). 
Case 1 refers to when only two temporal nodes are used out of the possible 256 temporal nodes in temporal hyper-reduction. Case 1 additionally refers to the use of the (spatial) sample mesh where 130 spatial nodes out of the possible 200 spatial nodes are used. Cases 9--13 refer to WST-GNAT, where windowing and hyper-reduction is used to achieve low MSE and low relative times. The number of windows range from 32 to 256 where the associated time window lengths range from 0.1 to 0.8. Additionally, the sub-window length ranges from 0.1 to 0.8. The sample mesh for these Pareto front cases includes as little as 20 spatial nodes to as many as 70 spatial nodes, much is less than half of the total number spatial nodes. Cases 14--18 refer to the WST-LSPG method without hyper-reduction. As expected, the results produced by the WST-LSPG method give the lowest MSE results, but require longer relative wall times.  Overall, these MSE results show that WST-GNAT can decrease the relative wall time by two to three orders of magnitude. Lastly, longer window lengths lead to higher computational costs, and shorter window lengths lead to shorter relative wall times without substantially affecting the MSE. 
%
%
%
\begin{figure}[!t]
\centering
\subfloat[$\numWindows=1,\ell_{w}=256\Delta t$\label{f:res8}]{\pareto{Pareto_Residual_8.pdf}}\hspace{1mm}\nolinebreak
\subfloat[$\numWindows=2,\ell_{w}=128\Delta t$\label{f:res7}]{\pareto{Pareto_Residual_7.pdf}}\hspace{1mm}\nolinebreak
\subfloat[$\numWindows=4,\ell_{w}=64\Delta t$\label{f:res6}]{\pareto{Pareto_Residual_6.pdf}}\hspace{1mm}\nolinebreak
\subfloat[$\numWindows=8,\ell_{w}=32\Delta t$\label{f:res5}]{\pareto{Pareto_Residual_5.pdf}}\hspace{1mm}\nolinebreak
\subfloat[$\numWindows=16,\ell_{w}=16\Delta t$\label{f:res4}]{\pareto{Pareto_Residual_4.pdf}}\vspace{-1mm}\\
\subfloat[$\numWindows=32,\ell_{w}=8\Delta t$\label{f:res3}]{\pareto{Pareto_Residual_3.pdf}}\hspace{1mm}\nolinebreak
\subfloat[$\numWindows=64,\ell_{w}=4\Delta t$\label{f:res2}]{\pareto{Pareto_Residual_2.pdf}}\hspace{1mm}\nolinebreak
\subfloat[$\numWindows=128,\ell_{w}=2\Delta t$\label{f:res1}]{\pareto{Pareto_Residual_1.pdf}}\hspace{1mm}\nolinebreak
\subfloat[$\numWindows=256,\ell_{w}=\Delta t$\label{f:res0}]{\pareto{Pareto_Residual_0.pdf}}\hspace{1mm}\nolinebreak
\legend{legend.pdf}\vspace{-1mm}\\
\subfloat[Combined Pareto Plot with Pareto Front\label{f:res0_ALL}]{\includegraphics[width=1\textwidth]{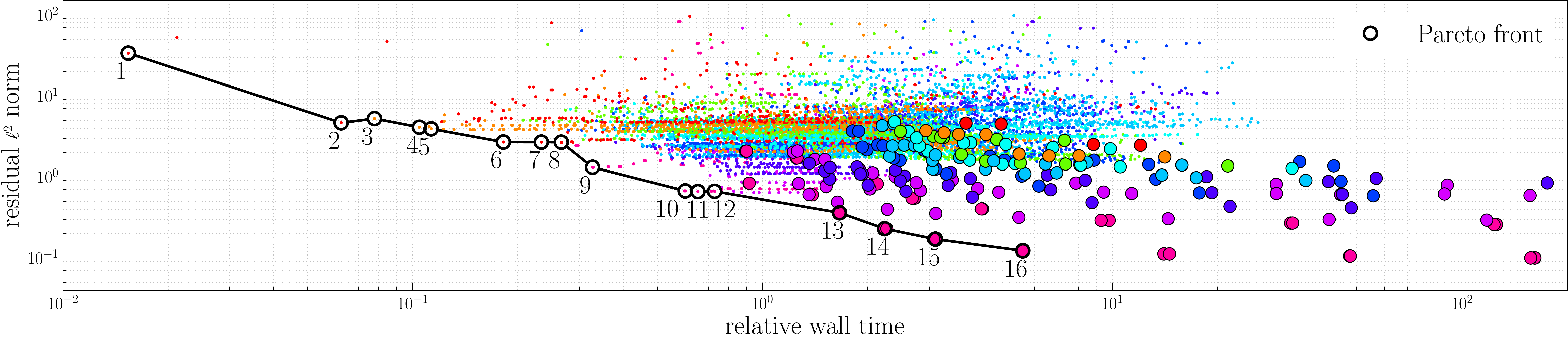}}
\caption{\emph{Burgers' equation}. Pareto plots showing the residual $\ell^2$ norm vs. relative wall time at varying sub-window lengths. These Pareto plots highlight how the sub-window length $\ell_s$ affect the residual $\ell^2$ norm for each time window length $l_w$ for WST-LSPG (ST-LSPG when $\ell_w=\ell_s$) and WST-GNAT (ST-GNAT when $\ell_w=\ell_s$). Each hyper-reduction point refers to a unique set of spatial sample nodes and temporal sample nodes. The last Pareto plot combines all the results and shows the Pareto front.}
\label{f:paretoResidual}
\end{figure}
\begin{figure}[!hp]
\centering
\subfloat[$\ell_{s}=256\Delta t$\label{f:res_e8}]{\pareto{Pareto_Residual_8_e.pdf}}\hspace{1mm}\nolinebreak
\subfloat[$\ell_{s}=128\Delta t$\label{f:res_e7}]{\pareto{Pareto_Residual_7_e.pdf}}\hspace{1mm}\nolinebreak
\subfloat[$\ell_{s}=64\Delta t$\label{f:res_e6}]{\pareto{Pareto_Residual_6_e.pdf}}\hspace{1mm}\nolinebreak
\subfloat[$\ell_{s}=32\Delta t$\label{f:res_e5}]{\pareto{Pareto_Residual_5_e.pdf}}\hspace{1mm}\nolinebreak
\subfloat[$\ell_{s}=16\Delta t$\label{f:res_e4}]{\pareto{Pareto_Residual_4_e.pdf}}\vspace{-1mm}\\
\subfloat[$\ell_{s}=8\Delta t$\label{fres_e3}]{\pareto{Pareto_Residual_3_e.pdf}}\hspace{1mm}\nolinebreak
\subfloat[$\ell_{s}=4\Delta t$\label{f:res_e2}]{\pareto{Pareto_Residual_2_e.pdf}}\hspace{1mm}\nolinebreak
\subfloat[$\ell_{s}=2\Delta t$\label{f:res_e1}]{\pareto{Pareto_Residual_1_e.pdf}}\hspace{1mm}\nolinebreak
\subfloat[$\ell_{s}=\Delta t$\label{f:res_e0}]{\pareto{Pareto_Residual_0_e.pdf}}\hspace{1mm}\nolinebreak
\legend{legend2.pdf}\vspace{-1mm}\\
\subfloat[Combined Pareto Plot with Pareto Front\label{f:res_e0ALL}]{\includegraphics[width=1\textwidth]{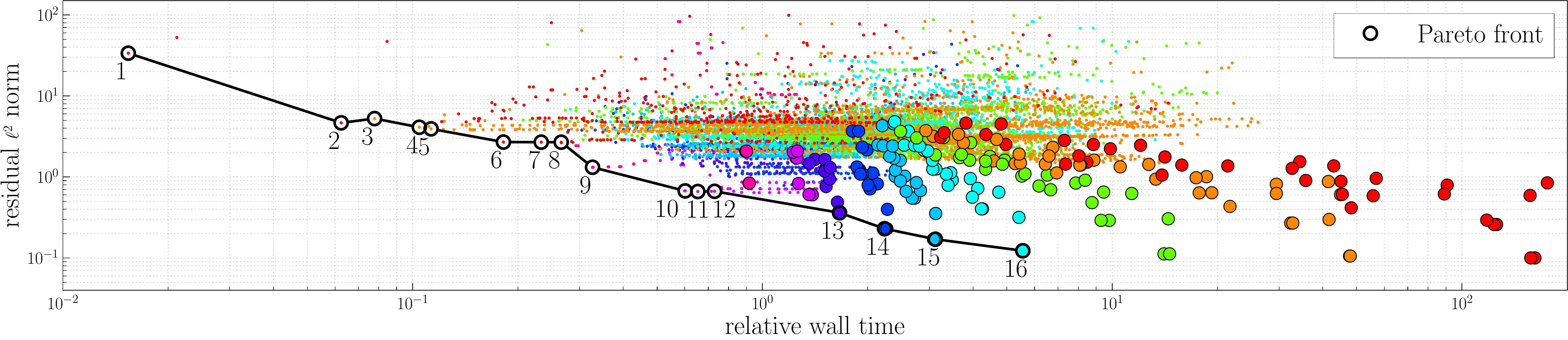}}
\caption{\emph{Burgers' equation}. Pareto plots for residual $\ell^2$ norm vs. relative wall time for varying sub-window lengths. These Pareto plots highlight how the window length, $\ell_w$, affect the residual $\ell^2$ norm for each sub-window length $\ell_s$ for WST-LSPG/ST-LSPG and WST-GNAT/ST-GNAT. Each hyper-reduction case refers to a unique set of spatial sample nodes and temporal sample nodes. The last Pareto plot combines all the results for the residual $\ell^2$ norm and shows the Pareto front (same as Figure~\ref{f:MSE_ALL}).}
\label{f:residual_opposite}
\vspace{-2mm}
\begin{center}
\captionof{table}{\label{t:residuals}\emph{Burgers' equation}. Method and parameter values for labeled data points in Figures~\ref{f:paretoMSE} and~\ref{f:paretoMSE_opposite} that yield Pareto-optimal performance for the general Windowed model reduction method in terms of the residual $\ell^2$ norm. Highlighted rows refer to test cases that additionally exist on the optimal Pareto Front for MSE.}
\resizebox{\columnwidth}{!}{%
\begin{tabular}{ ccccccccccccccc } \toprule
case& method & $\ell_w$ & $\ell_s$ & \thead{total\\$n_{st}$} &\thead{total\\$n^r_{st}$} & $e_s$ & $e_t$ & $e_{r_s}$ & $e_{r_t}$ & $z_t$ & $z_s$ & \thead{residual\\$\ell^2$ norm} & \thead{x-LSPG relative\\ wall time} &\thead{Relative wall\\time}\\
 \midrule \midrule
 \rowcolor{matcha}1 & ST-GNAT&25.6&	25.6& 98& 148&	0.999&	0.999& 0.999&	0.999&	2	&130 & 33.68670968 & 12.45653562&0.015386589\\ 
\rowcolor{matcha} 2 & ST-GNAT&	25.6&	25.6&22&76&	0.99&	0.99& 0.999&	0.999&	18	&20 & 4.655537161 &3.786097827&0.062465228\\ 
3&WST-GNAT  &12.8 &	12.8 & 39& 148&	0.99 &	0.99 &	0.999 &	0.999 &	10 &	20 &5.253386860 &2.950033721&	0.077836879 \\ 
4&WST-GNAT &12.8 &	12.8 &39& 148&	0.99 &	0.99 &	0.999 &	0.999 &	10 &	40 &	4.104538792 &	2.932780389&0.104491508\\ 
5&WST-GNAT  &	12.8 &	12.8 &	39&148 &0.99 &	0.99 &	0.999 &	0.999 &	10 &	50 &	3.921332144	 &2.927011265&0.112900133\\ 
\rowcolor{matcha}6&ST-GNAT &	25.6 &	25.6 &	68 & 190&0.999 &	0.99 &	0.999 &	0.999 &	18 &	50 &	2.694921068	 &9.092380722&0.181795453 \\ 
\rowcolor{matcha}7&ST-GNAT &25.6 &	25.6 &	68 & 190 & 0.999 &	0.99 &	0.999 &	0.999 &	18 &	60 &	2.687869219 &8.935311415&	0.233166894\\  
\rowcolor{matcha}8&ST-GNAT&	25.6 &	25.6 & 68 & 190&	0.999 &	0.99 &	0.999 &	0.999 &	18 &	70 &	2.685009426	 & 9.036366568&0.265945585\\ 
9&WST-GNAT  &0.1 &	0.1 & 1730& 4010&	0.99 &	0.99 &	0.999 &	0.999 &	2 &	30 &	1.316808304 &0.895191869&	0.326798440\\
\rowcolor{matcha}10&WST-GNAT	 &0.2 &	0.1 & 2490 &4186&	0.999 &	0.999 &	0.999 &	0.999 &	3 &	50 &	0.672582791 &1.383185815&	0.601325305\\
11&WST-GNAT &	0.2 &	0.1	 &2490 &4186& 0.999	 &0.999 &	0.999 &	0.999 &	3 &	60 &	0.660878763	 &1.383908211&0.653675049\\
12&WST-GNAT  &	0.2 &	0.1 &2490 &4186& 	0.999 & 	0.999  & 	0.999 &	0.999 & 	3 &	70 &	0.664936770 &1.403074748&	0.729775602\\
13&WST-LSPG &	0.4	 &0.1 & 2490& \NA &	0.999 &	0.999 &\NA &\NA &\NA &\NA & 	0.361554579	 &\NA &1.661458611\\
14&WST-LSPG  &	0.8	 &0.1 &	2490 & \NA &0.999 &	0.999 &	\NA &\NA &\NA &\NA & 0.229856161 &	\NA &2.234523461\\
15&WST-LSPG &1.6 &	0.1 &	2490 & \NA&0.999 &	0.999 &	\NA &\NA &\NA &\NA & 0.170444365	 &\NA &3.121685483\\
16&WST-LSPG &3.2 &	0.1 &2490& \NA& 	0.999 &	0.99 &\NA &\NA &\NA &\NA & 	0.122859453	 &\NA &5.551955888\\
\bottomrule
\end{tabular}
}
\end{center}
\end{figure}
The next set of Pareto plots can be seen in Figure~\ref{f:paretoResidual}, which shows nine different Pareto plots for residual $\ell^2$ norms vs. relative wall time for varying values of $\ell_s$. These residual $\ell^2$ norm calculations come from the same set of data shown in Figures~\ref{f:paretoMSE} and~\ref{f:paretoMSE_opposite}. Like the MSE results in Figure~\ref{f:paretoMSE}, the hyper-reduction results show that the WST-GNAT simulations result in relative wall times that are orders of magnitude less than those of WST-LSPG. However, for the residual $\ell^2$ norms, the lower relative wall times are only obtained up to a certain sub-window length for each window length. For example, in Figure~\ref{f:res7}, when the window length and sub-window length is the same $\ell_w=\ell_s$, WST-GNAT reveals that for the same amount of accuracy, the relative wall time decreases. This is true for when the sub-window length is further halved. However, past this point, the residual $\ell^2$ norms do not follow this downward trend and increase with decreasing sub-window length as seen when $\ell_s = 16\Delta t$ and $\ell_s = 8\Delta t$. Thus, the optimal $\ell_s$ values give low relative wall times and low residual $\ell^2$ norms. Figure~\ref{f:res0_ALL} combines the results from Figure~\ref{f:res8}-\ref{f:res0} to create a Pareto front showing the optimal set of parameters that give lower residual $\ell^2$ norms vs relative wall time. 

The next set of Pareto plots presents the same data as shown in Figure~\ref{f:paretoResidual}, but highlights how the residual $\ell^2$ norm and relative wall time are affected by varying window lengths for a constant sub-window length. The results for this can be seen in Figure~\ref{f:residual_opposite}. Figures~\ref{f:res_e7}-\ref{f:res_e4} show that with hyper-reduction, WST-GNAT does decrease the relative wall time without increasing the residual $\ell^2$ norm for many of the test cases. However, as the sub-window length decreases, the WST-GNAT results begin to deviate from the initial trend. It is apparent in Figure~\ref{f:res_e2} that at a certain value of sub-window length, adding hyper-reduction increases the residual $\ell^2$ norms regardless of the set of bases used or the number of temporal and spatial nodes used. This shows that in order to achieve low MSE and low residual $\ell^2$ norms, the sub-window length for this experiment need to be set such that $\ell_s>16\Delta t$. If the sub-window length is too small, the ROM will produce larger residual $\ell^2$ norms. 
Figure~\ref{f:residual_opposite} compiles the results from Figure~\ref{f:res_e8}-\ref{f:res_e0} to create a Pareto front showing the optimal set of parameters that gives the lowest residual $\ell^2$ norms vs. relative wall time. This Pareto front is the same one shown in Figure~\ref{f:res0_ALL}, but highlights how the residual $\ell^2$ norms and relative wall time change with different window lengths. 
The parameters for residual $\ell^2$ norm that lie on the Pareto front can be seen in Table~\ref{t:residuals}. Cases 1--2 and cases 6--8, refer to the original ST-LSPG method with hyper-reduction (i.e., ST-GNAT). These ST-GNAT data points correspond to low relative wall time, but high residual norms. Cases 3--5 and 9--12 refer to the WST-LSPG method with hyper-reduction. These cases, especially cases 9--12, exhibit results that have a combination of both lower residual $\ell^2$ norms and lower relative wall times. Cases 13--16 refer to the WST-LSPG method without hyper-reduction and are characterized by lower residual $\ell^2$ norms, but longer relative wall times.

Lastly, Table~\ref{t:MSE} and Table~\ref{t:residuals} have certain rows highlighted. The highlighted rows refer to a specific set of parameters that refer to the Pareto optimal parameters for both the MSE and residual $\ell^2$ norm; these test cases exist in both Pareto front. These highlighted test cases are adequate starting points to find a windowed space--time ROM solution that optimizes both the MSE and residual $\ell^2$ norm. 

\subsubsection{Summary of results}
WST-LSPG was performed on the one-dimensional Burgers' equation. To assess this proposed method, the MSE, IMSE, and the residual $\ell^2$ norms were calculated over two different window parameters, $\ell_w$ and $\ell_s$. First, the results for this study were shown in Figure~\ref{f:window1}. Figure~\ref{f:window1} showed that the MSE and the residual $\ell^2$ norms exhibited opposite behaviors as the time window length increased. The IMSE was plotted as well in~\ref{s:imse} in order to discount the influence of phase error on MSE. Next, the WST-LSPG results were shown in Pareto plots for the MSE in Figure~\ref{f:paretoMSE} and Figure~\ref{f:paretoMSE_opposite} and for the residual $\ell^2$ norm in Figure~\ref{f:paretoResidual} and Figure~\ref{f:residual_opposite}. The optimal cases were plotted on Pareto fronts. These fronts showed that the cases that gave the lowest MSE and residual $\ell^2$ norms were characterized by a larger number of windows and sub-windows. Finally, with hyper-reduction, WST-GNAT was able to reduce the computational cost while achieving a MSE as low as $0.015893315$ and low residual $\ell^2$ norms as low as $0.664936770$. Note that without hyper-reduction, the WST-LSPG achieved MSEs as low as $0.006083036$ and residual $\ell^2$ norms as low as $0.122859453$, highlighting that hyper-reduction did not achieve the lowest MSE and residual $\ell^2$ norms. 


\subsection{Parameterized compressible Navier--Stokes equations}
This section presents results for WST-LSPG and WST-GNAT applied to the two-dimensional compressible Naiver--Stokes equations. The compressible Navier--Stokes equations describe the motion of a viscous fluid and is made up of the conservation of mass, conservation of momentum, and the conservation of energy equations. Together, the equations are defined as
\begin{equation}
    \frac{\partial \state}{\partial t} +\nabla\cdot\bm{F} = \nabla \cdot \bm{Q},\quad \forall x\in [-100,100],\quad y\in [-100,100],\quad \forall t\in[0,T],
\end{equation}
where $\bm{u}$ is the state vector, $\bm{F}$ is the convective flux, and $\bm{Q}$ is the diffusive flux. The state and fluxes are defined as
\begin{equation*}
    \state =
    \bmat{
    \rho\\
    \rho u_1\\
    \rho u_2\\
    \rho E
    },\quad
    \bm{F} =
    \bmat{
    \rho u_1\\
    \rho u_1^2 +p\\
    \rho u_1 u_2\\
    \rho u_1 H
    }
    \widehat{\bm{i}}\;+\;
    \bmat{
    \rho u_2\\
    \rho u_1 u_2\\
    \rho u_2^2+p\\
    \rho u_2 H
    }
    \widehat{\bm{j}},\quad
    \bm{Q} =
    \bmat{
    0\\
    \tau_{11}\\
    \tau_{12}\\
    u_1\tau_{11}+u_2\tau_{12}+q_1
    }\widehat{\bm{i}}\;+\;
    \bmat{
    0\\
    \tau_{12}\\
    \tau_{22}\\
    u_1\tau_{12}+u_2\tau_{22}+q_2
    }\widehat{\bm{j}},
\end{equation*}
where $\widehat{\bm{i}}$ and $\widehat{\bm{j}}$ refer to the spatial dimension, $\rho$ is the density, $u_1, u_2$ are the components of the velocities, $E$ is the specific total energy per unit mass, and $H$ is the specific total enthalpy per unit mass. The heat transfer term and the viscous shear and normal stresses for a Newtonian fluid are defined as
\begin{equation*}
  q_i = \kappa_T\partial_i T,\quad \tau_{ij} = \mu_{\mathrm{NS}}(\partial_iu_j+\partial_ju_i)+\delta_{ij}\lambda\partial_mu_m,
\end{equation*}
where $\mu_{\mathrm{NS}}$ is the dynamic viscosity, $\lambda$ is the bulk viscosity, $\delta_{ij}$ is the Kronecker delta function, and $\kappa_T$ thermal conductivity. These quantities are defined as
\begin{equation}
\mu = \mu_{\mathrm{ref}}\left(\frac{T}{T_{ref}}\right)^{1.5}\left(\frac{T_{ref}+T_s}{T+T_s}\right),\quad \lambda = -\frac{2}{3}\mu,\quad \kappa_T = \frac{\gamma\mu R}{(\gamma-1)Pr}.
\end{equation}
Important properties for this problem include the specific heat ratio $\gamma = 1.4$, Reynolds number $\mathrm{Re} = 5000$, Prandtl number $Pr = 0.71$, and the gas constant $R = 287.05\;J/(kg\cdotp K)$. The free-stream state is initialized to $\bm{u}_{\text{FS}} =\bmat{\rho & \rho u_1 & \rho u_2 & \rho E}^T=\bmat{1 & \cos(\alpha)& \sin(\alpha) & \frac{1}{\gamma(\gamma-1)M^2}+\frac{1}{2}}^T$, and the parameter space is defined as $\param\equiv \left(\alpha,M\right) \in \mathcal{D}=[8, 9]\times [0.6, 0.7])$, where $\alpha$ is the angle of attack and $M = \sqrt{\gamma p/\rho}$ is the Mach number. The set of design parameters for $\alpha$ and $M$ were found by uniformly sampling such that $\Delta \mu_{\alpha} =0.025$, and $\Delta \mu_{M} =0.025$. This leads to $N_{\mu_{\alpha}}=5$ and $N_{\mu_{M}}=5$. Overall, the number of training data sets is $n_{\mathrm{train}} = 25$. The test parameters used in the predictive model reduction simulations are set to $\mu_{\mathrm{test},\alpha} = 8.37$ and $\mu_{\mathrm{test},M}=0.67$. The design parameter space for the space--time residual basis is set to be the same as the one used for the states. This example considers four different quantities of interest to evaluate the accuracy and behavior of the WST-LSPG ROMs: MSE, residual $\ell^2$ norm, relative drag error, and relative lift error. The FOM wall time was found by taking the average of 100 simulation wall times. For each ROM case, the online simulation was executed five times to find the average relative wall times.

The geometry chosen for this work is the two-dimensional NACA 0012 airfoil. The boundary conditions are full-state on the far-field and adiabatic no-slip wall on the airfoil. The final time is set to $T=12.8$. The state is initialized to free-stream quantities and advanced forward in time using BDF2 time marching with a constant time step of $\Delta t = 0.1$ for a total of 128 time steps. The far-field is approximately 100 chord lengths away from the airfoil. To discretize in space, the discontinuous Galerkin (DG) method is used. The basis used in the DG method are Tri-Lagrange polynomial basis functions defined by $p+1$ order. For this work, the interpolation order is set to $p=1$ for a second order accurate state solution.

The mesh used for the NACA 0012 airfoil is an unstructured mesh that can be seen in Figure~\ref{f:airfoil}. The unstructured mesh was created by performing steady mesh refinement via output-based error estimation on a coarse mesh~\cite{Fidkowski_2015_VKI}. Five iterations of mesh adaptation were performed based on the output-based adjoint error to find a mesh that was fine enough at locations such as the wake of the flow. The outputs for mesh adaption were set to be the drag and lift. The final unstructured mesh used in these simulations contain $32,424$ spatial nodes and $10808$ triangular elements. The compressible two-dimensional Navier--Stokes equations are implemented in the \texttt{xflow} program, a high-order discontinuous Galerkin finite element solver in \texttt{ANSI C}~\cite{xflow}. All model reduction routines for the two-dimensional model reduction are implemented in the \texttt{ykflow} program in \texttt{ANSI C}, which is a library of functions that performs model reduction routines using the \texttt{PETSc} and \texttt{SLEPc} numerical linear algebra libraries. All timings are obtained by performing serial calculations on an Intel(R) Xeon(R) CPU E5-2695 v4 @ 2.10GHz, 128 GB RAM.
\begin{figure}[!t]
\begin{center}
\includegraphics[width = 0.243\textwidth,frame]{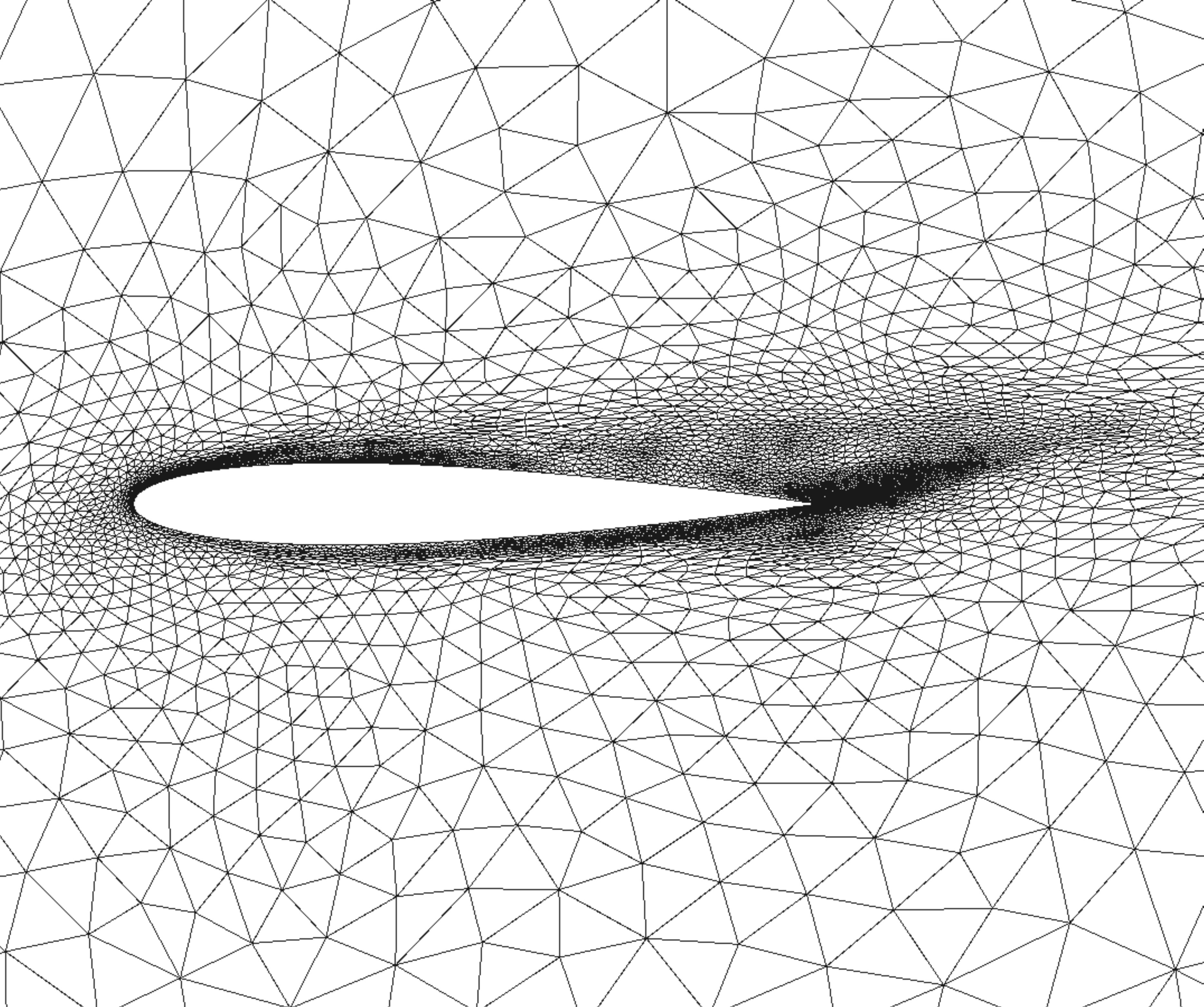} %
\includegraphics[width = 0.243\textwidth,frame]{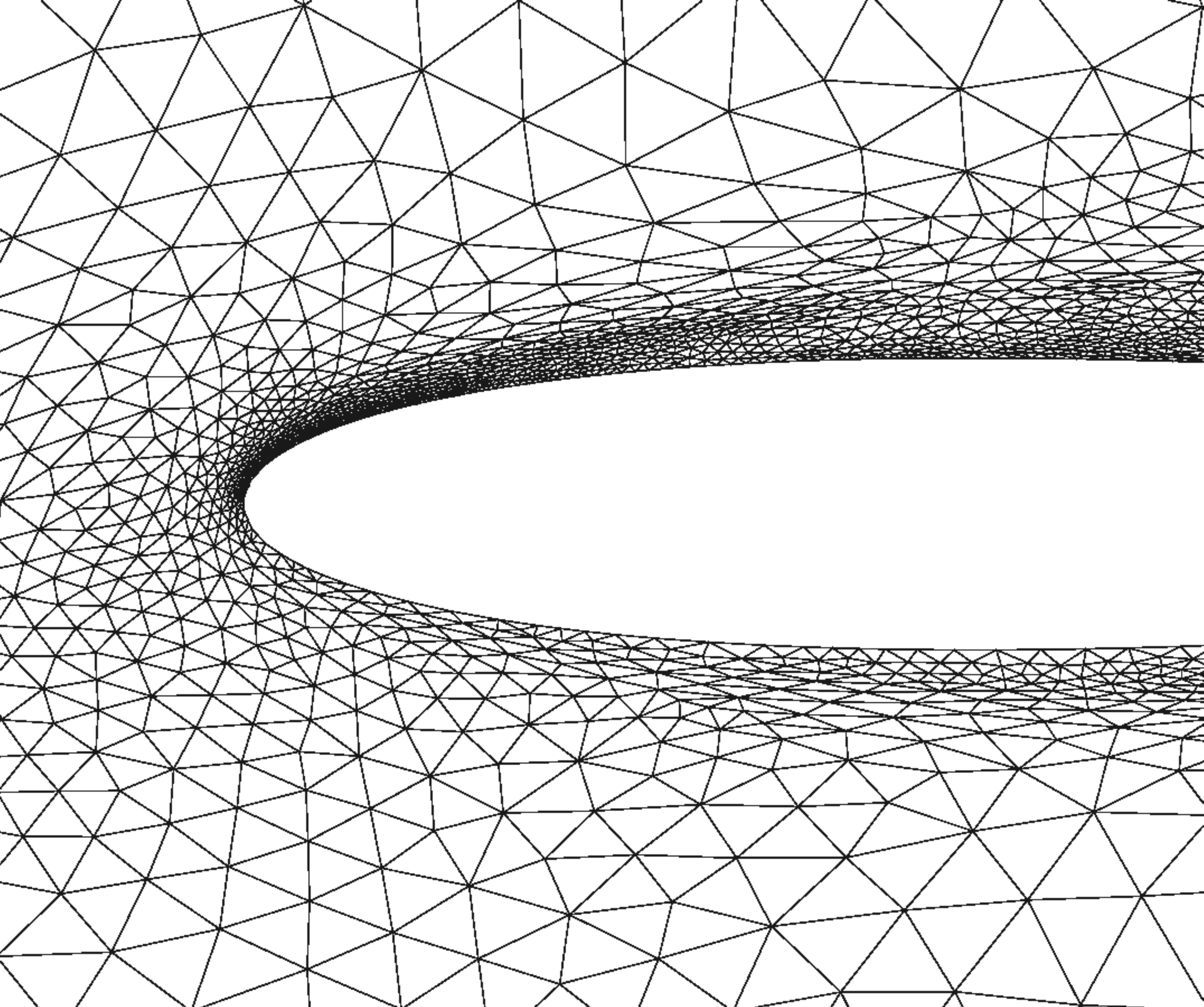}
\includegraphics[width = 0.243\textwidth,frame]{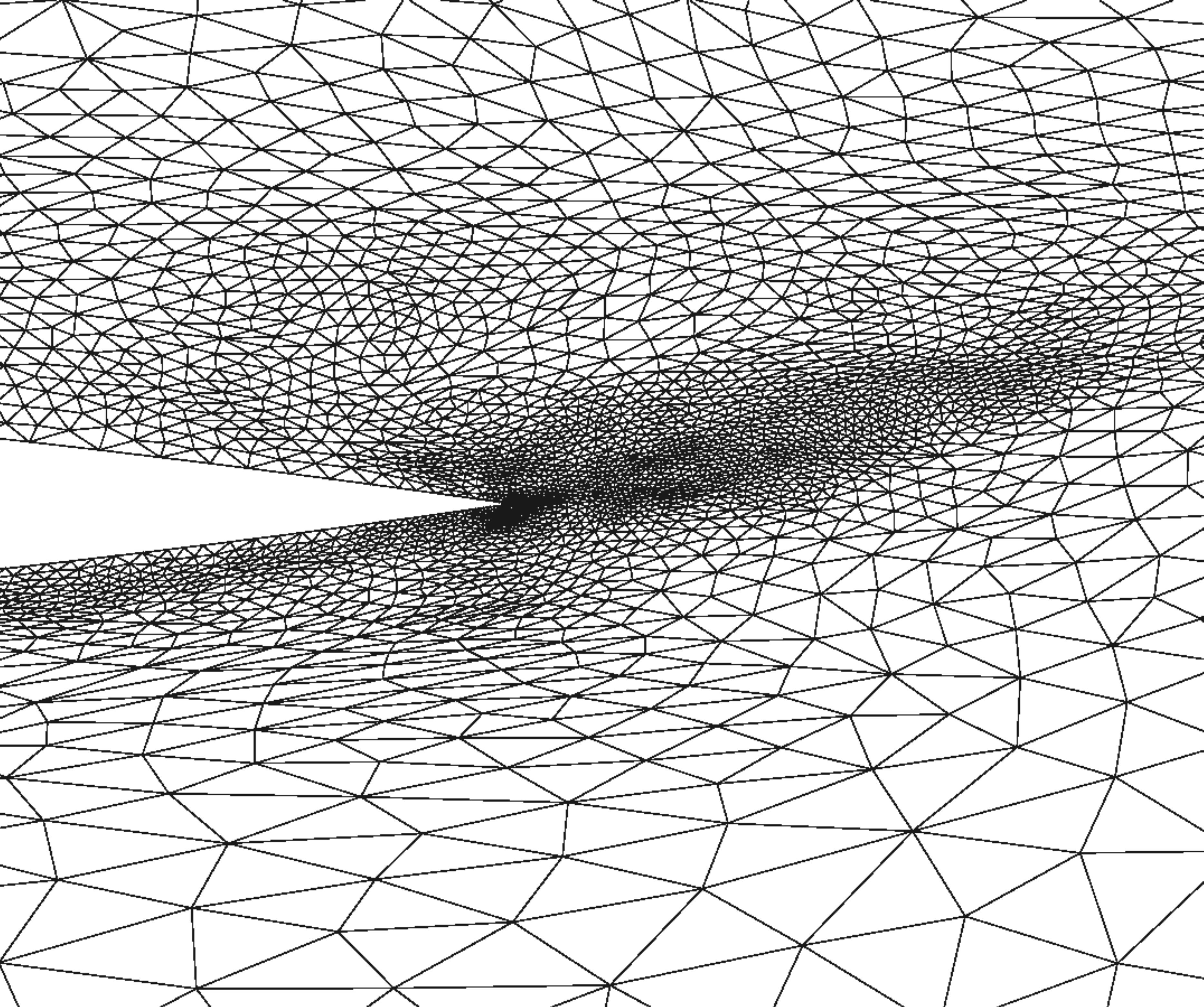}
\includegraphics[width = 0.243\textwidth,frame]{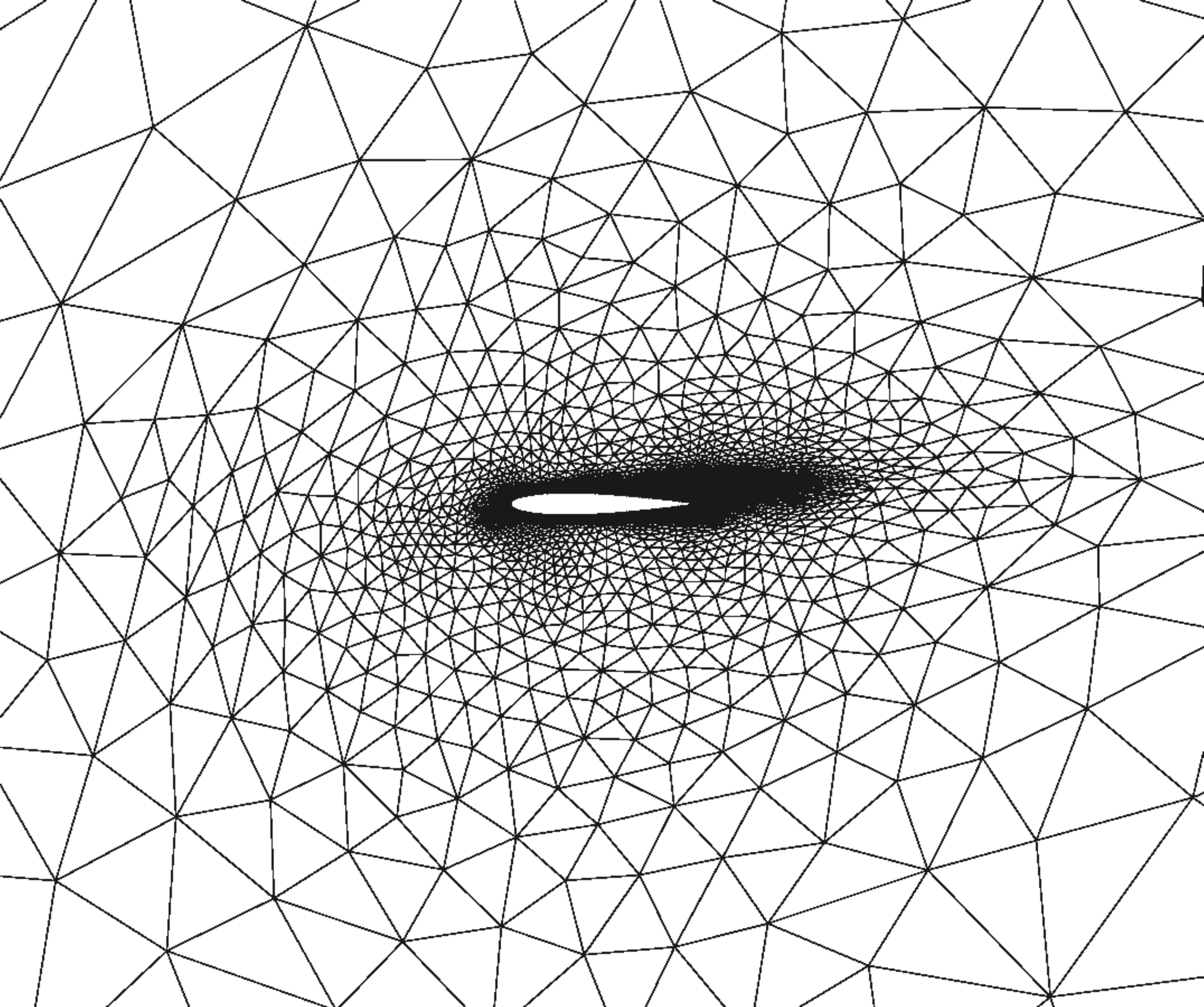}
\caption{NACA 0012 geometry and its unstructured mesh}
\label{f:airfoil}
\end{center}
\end{figure}
\subsubsection{Model predictions}

\begin{figure}[p!]
\subfloat[mean squared error: $e_s = 0.99, e_t = 0.99$][\centering \hspace{1.5mm}mean squared error:\par\hspace{5.5mm}$e_s = 0.99, e_t = 0.99$\label{f:NSMSE0}]{\includegraphics[height=0.264\textwidth]{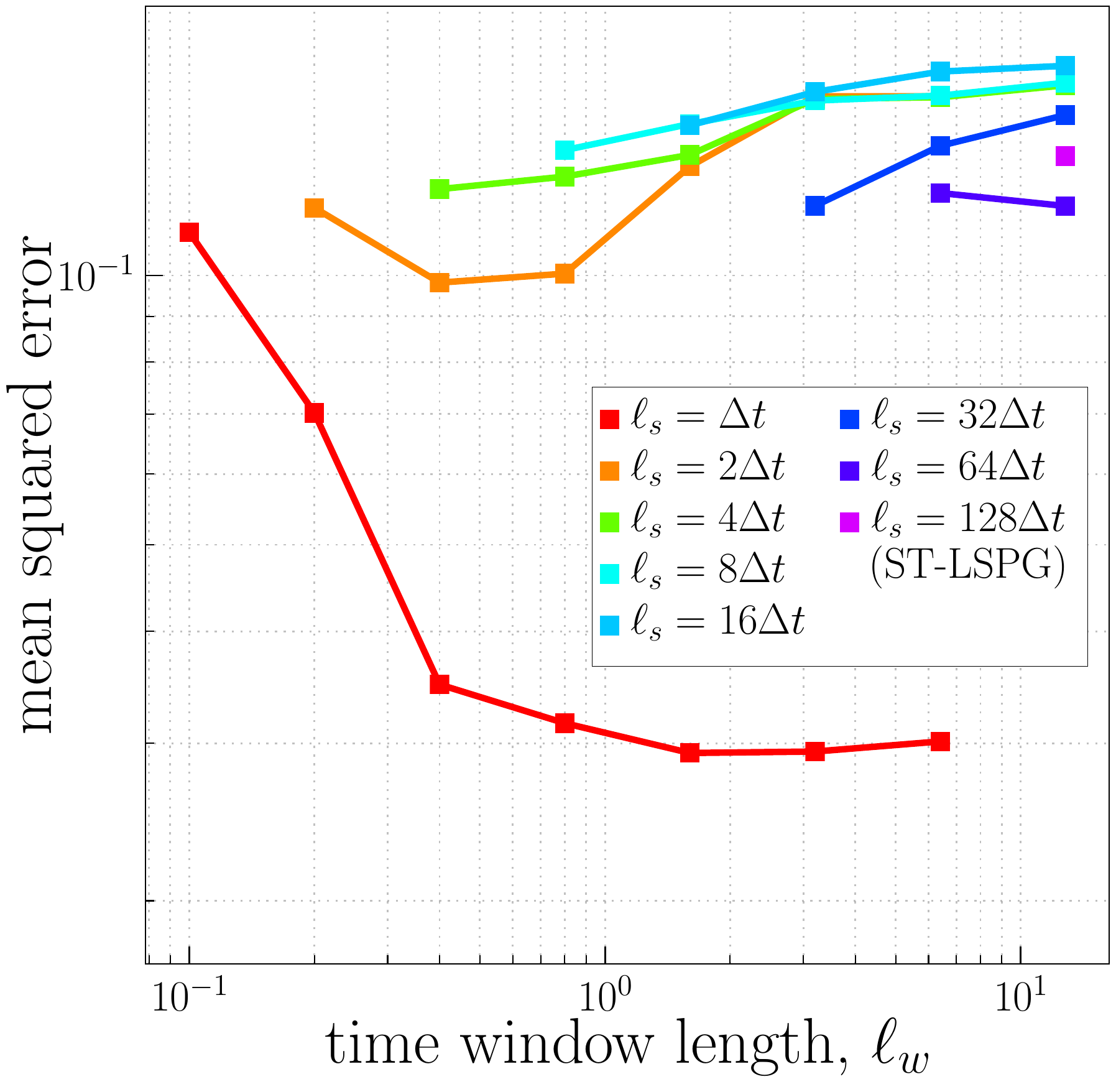}}\hspace{1mm}\nolinebreak
\subfloat[mean squared error: $e_s = 0.99, e_t = 0.999$][\centering \hspace{1.5mm}mean squared error:\par\hspace{5.5mm}$e_s = 0.99, e_t = 0.999$\label{f:NSMSE2}]{\includegraphics[height=0.264\textwidth]{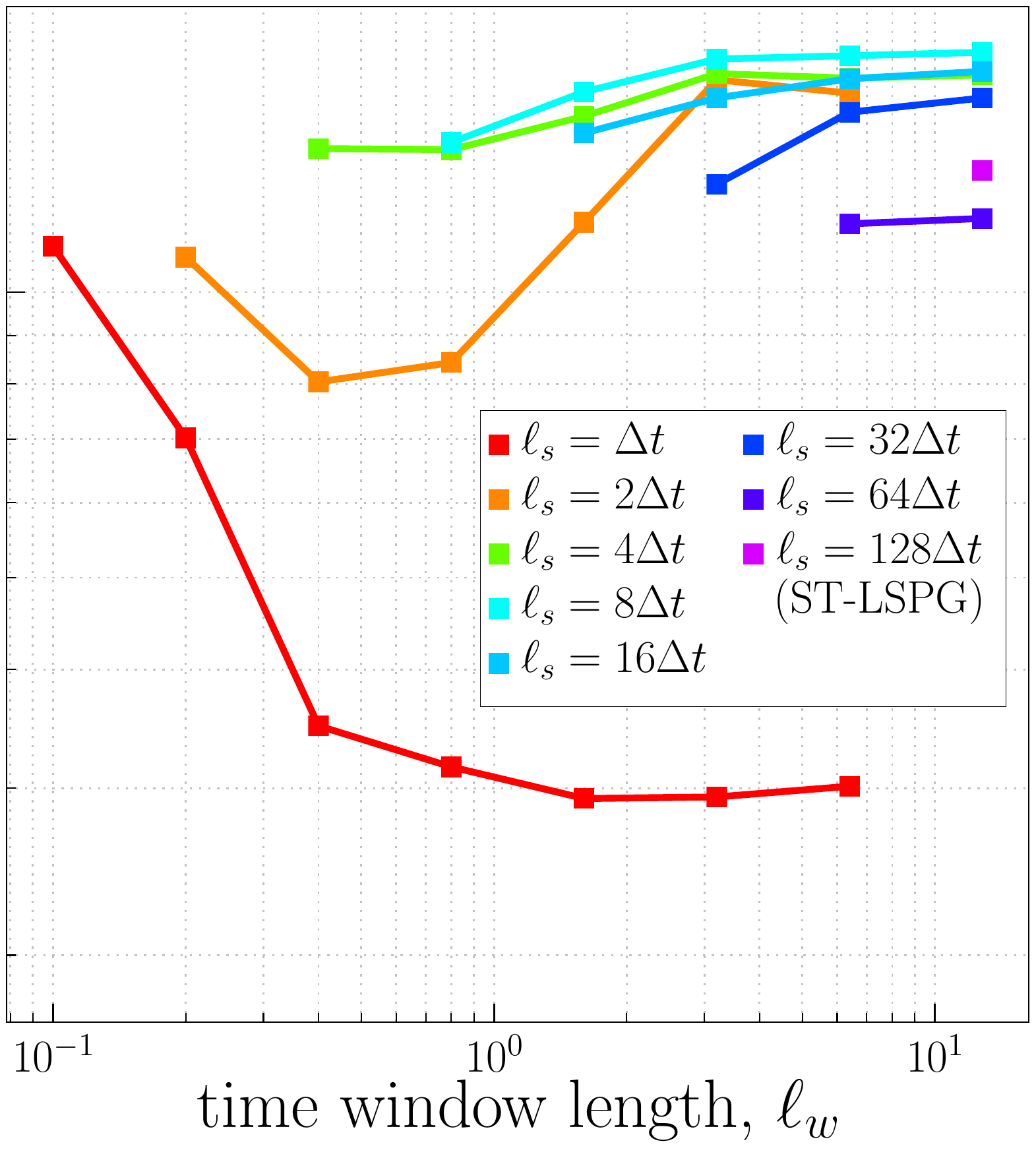}}\hspace{1mm}\nolinebreak
\subfloat[mean squared error: $e_s = 0.999, e_t = 0.99$][\centering \hspace{1.5mm}mean squared error:\par\hspace{5.5mm}$e_s = 0.999, e_t = 0.99$\label{f:NSMSE4}]{\includegraphics[height=0.264\textwidth]{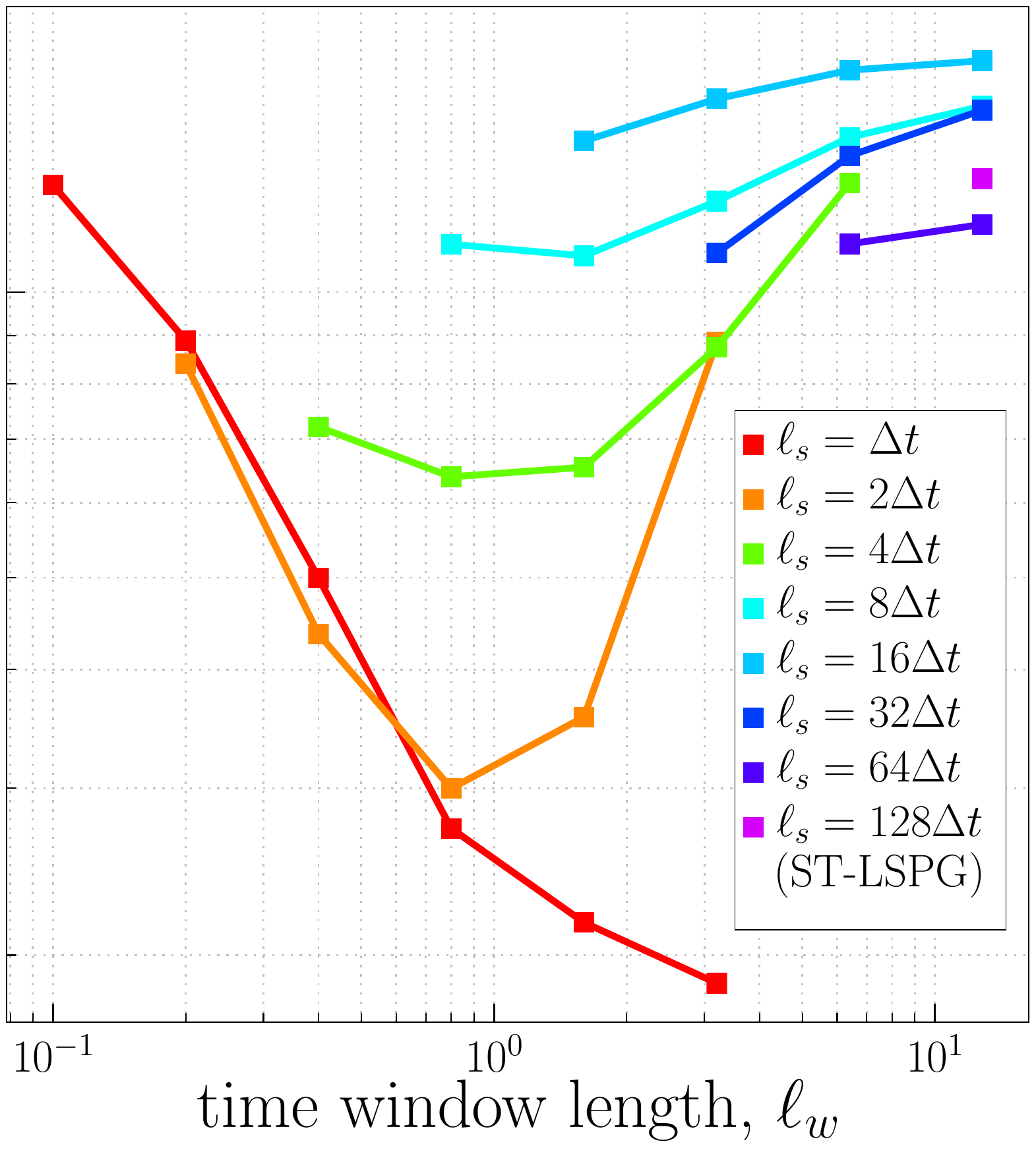}}\hspace{1mm}\nolinebreak
\subfloat[mean squared error: $e_s = 0.999, e_t = 0.999$][\centering \hspace{1.5mm}mean squared error:\par\hspace{5.5mm}$e_s = 0.999, e_t = 0.999$\label{f:NSMSE6}]{\includegraphics[height=0.264\textwidth]{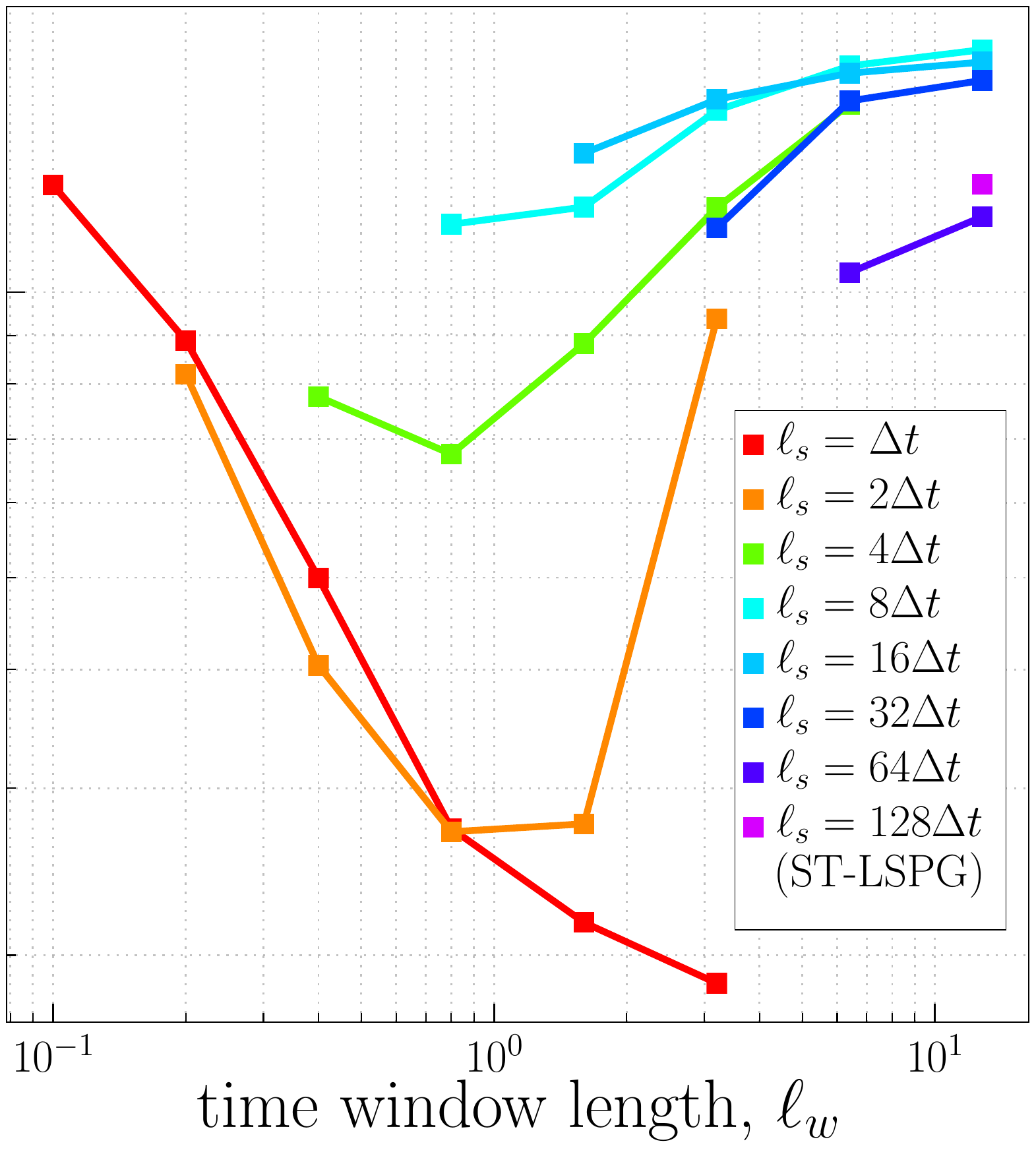}}\vspace{-2mm}\\ 
\subfloat[residual $\ell^2$ norm: $e_s = 0.99, e_t = 0.99$][\centering \hspace{1.5mm}residual $\ell^2$ norm:\par\hspace{5.5mm}$e_s = 0.99, e_t = 0.99$ \label{f:NSResidual0}]{\includegraphics[height=0.264\textwidth]{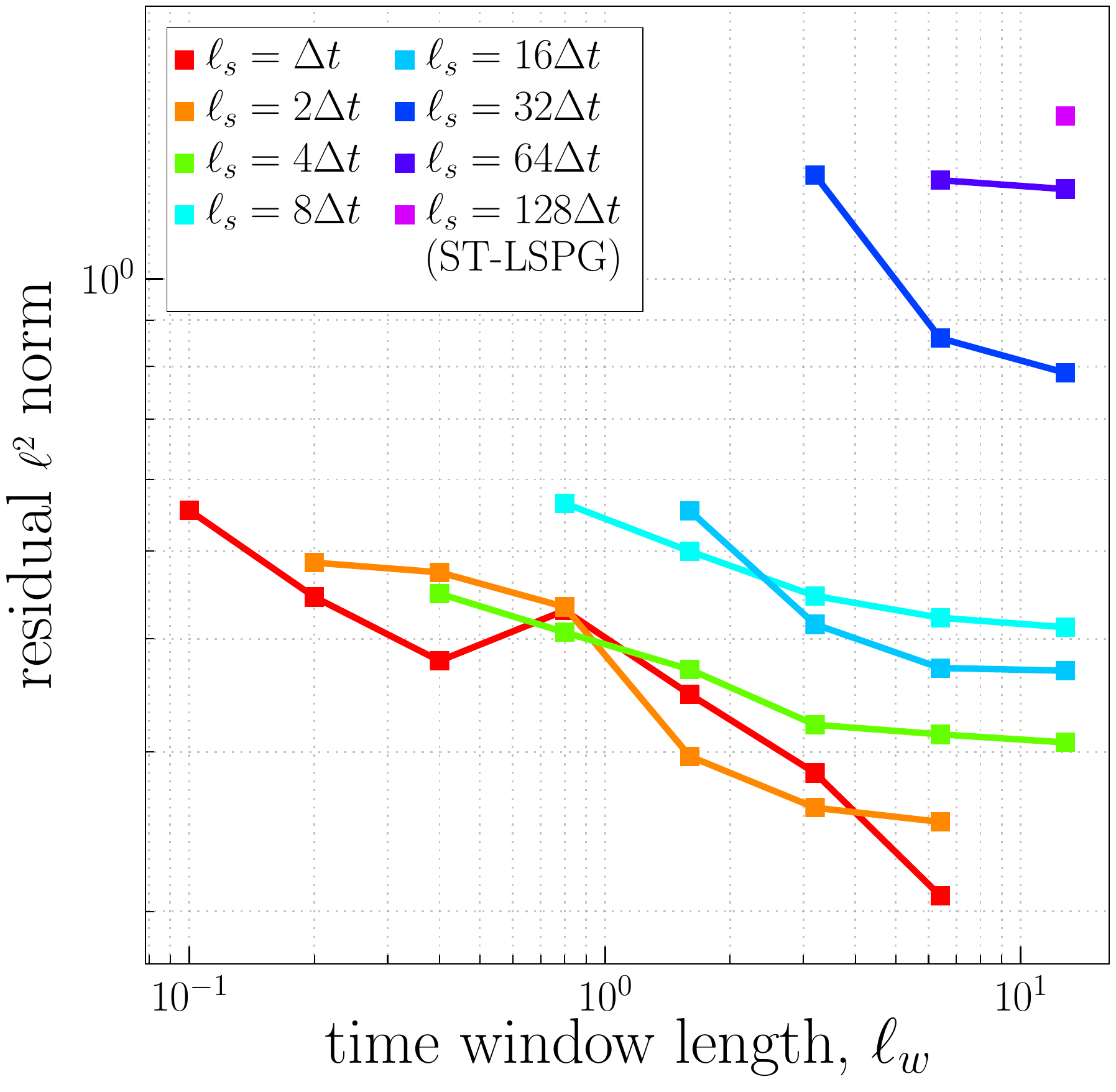}}\hspace{1mm}\nolinebreak
\subfloat[residual $\ell^2$ norm: $e_s = 0.99, e_t = 0.999$][\centering \hspace{1.5mm}residual $\ell^2$ norm:\par\hspace{5.5mm}$e_s = 0.99, e_t = 0.999$ \label{f:NSResidual2}]{\includegraphics[height=0.264\textwidth]{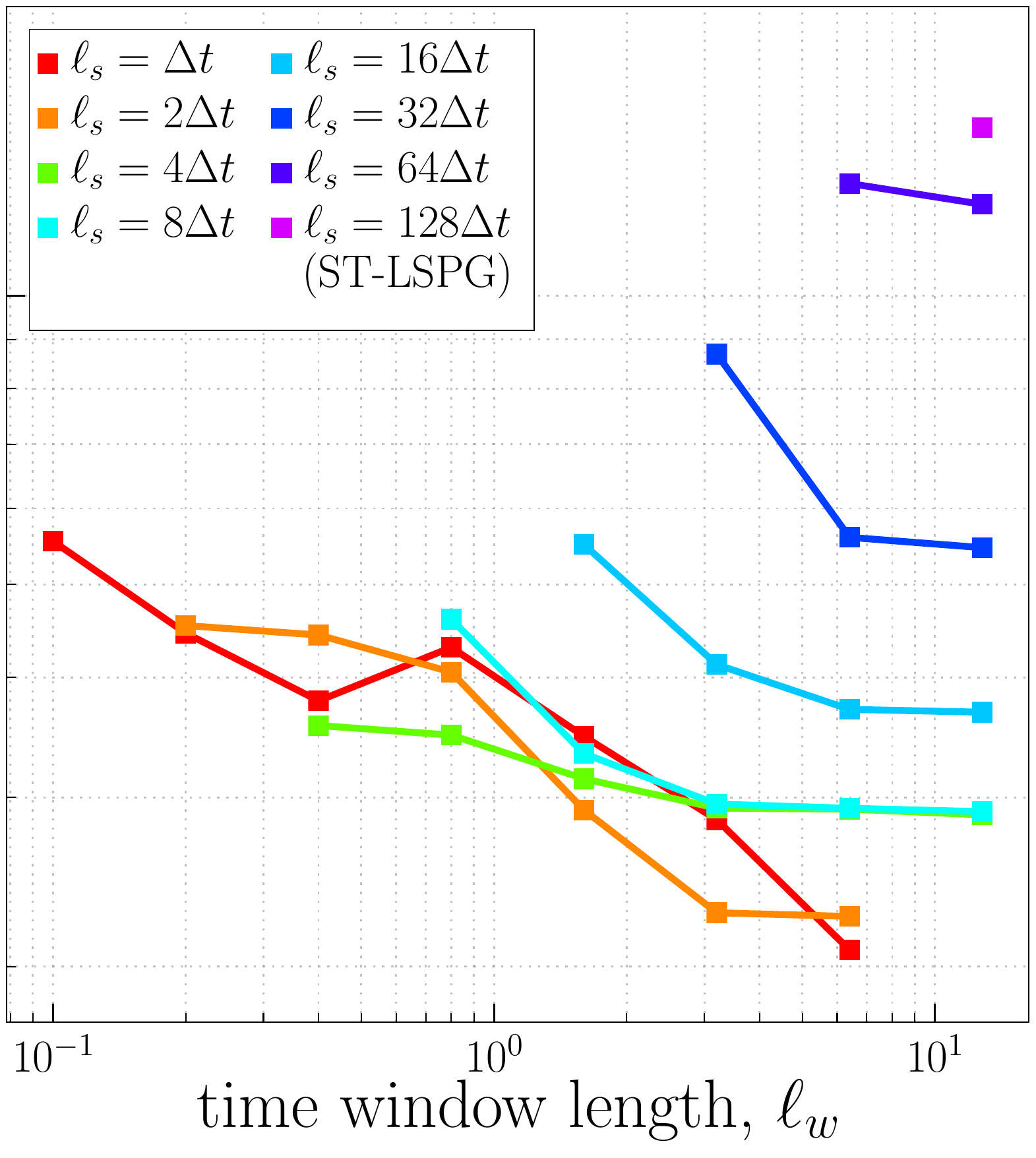}}\hspace{1mm}\nolinebreak
\subfloat[residual $\ell^2$ norm: $e_s = 0.999, e_t = 0.99$][\centering \hspace{1.5mm}residual $\ell^2$ norm:\par\hspace{5.5mm}$e_s = 0.999, e_t = 0.99$\label{f:NSResidual4}]{\includegraphics[height=0.264\textwidth]{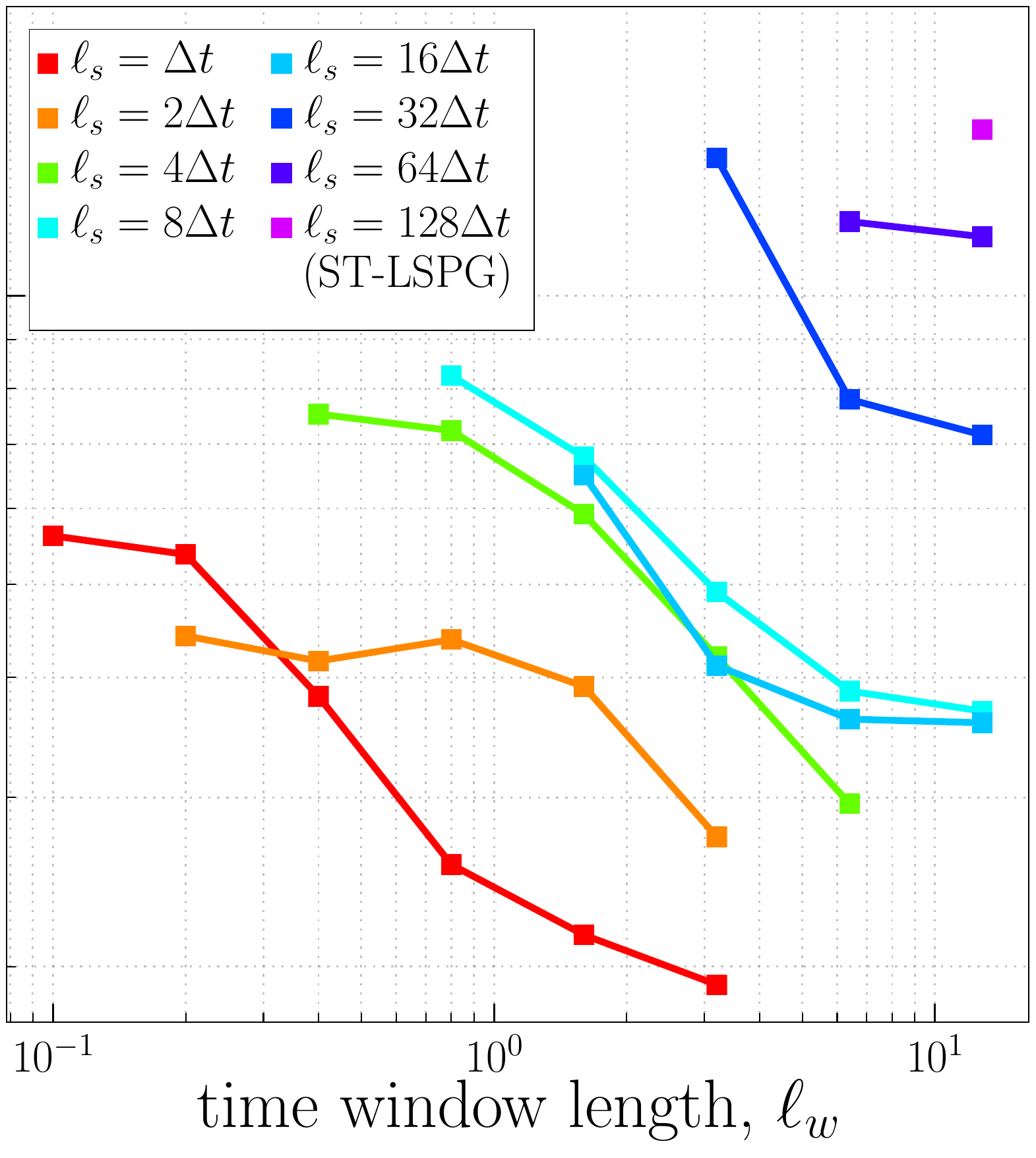}}\hspace{1mm}\nolinebreak
\subfloat[residual $\ell^2$ norm: $e_s = 0.999, e_t = 0.999$][\centering \hspace{1.5mm}residual $\ell^2$ norm:\par\hspace{5.5mm}$e_s = 0.999, e_t = 0.999$\label{f:NSResidual6}]{\includegraphics[height=0.264\textwidth]{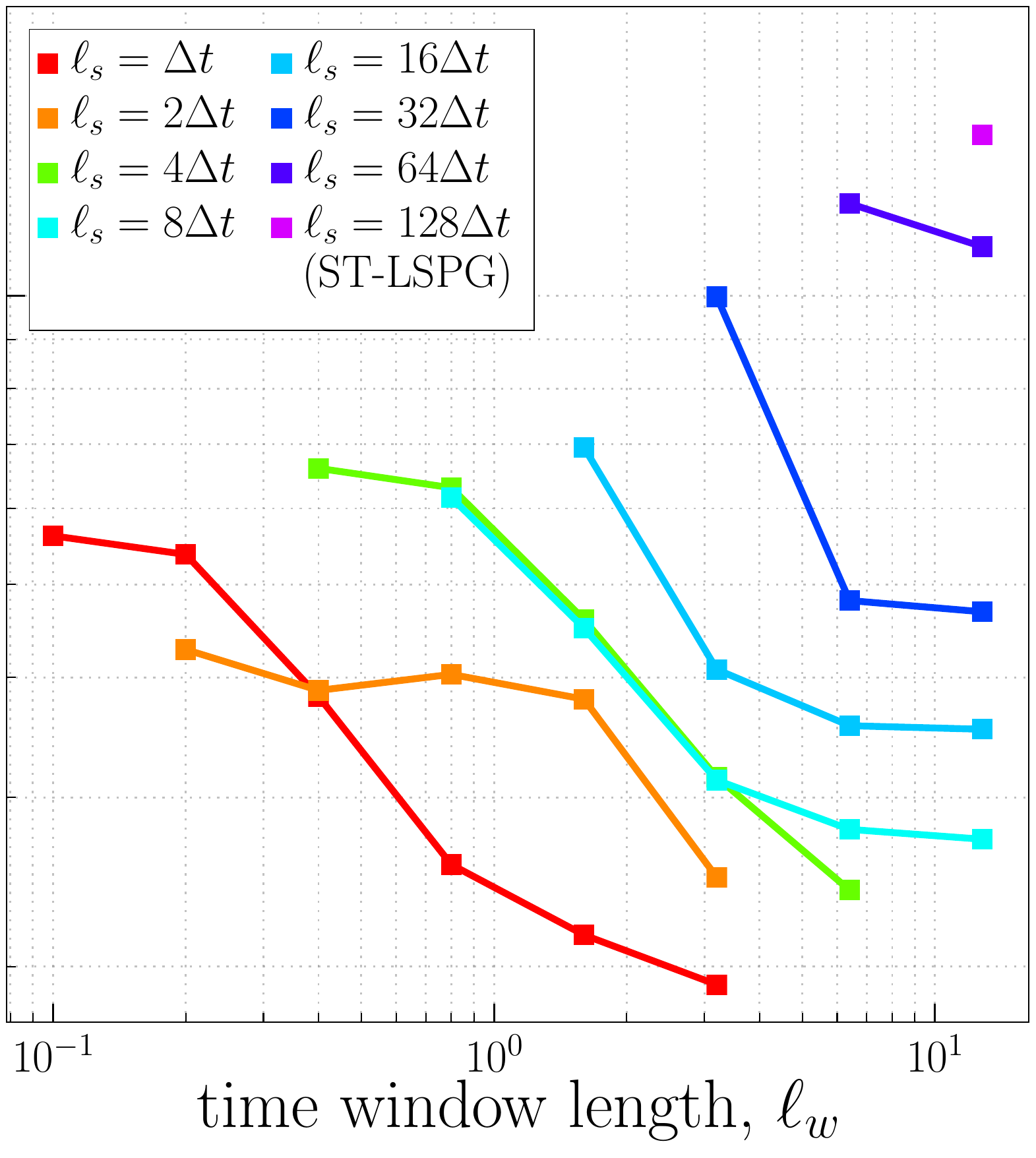}}\vspace{-2mm}\\
\subfloat[relative drag error: $e_s = 0.99, e_t = 0.99$][\centering \hspace{1.5mm}relative drag error:\par\hspace{5.5mm}$e_s = 0.99, e_t = 0.99$\label{f:NSDrag0}]{\includegraphics[height=0.264\textwidth]{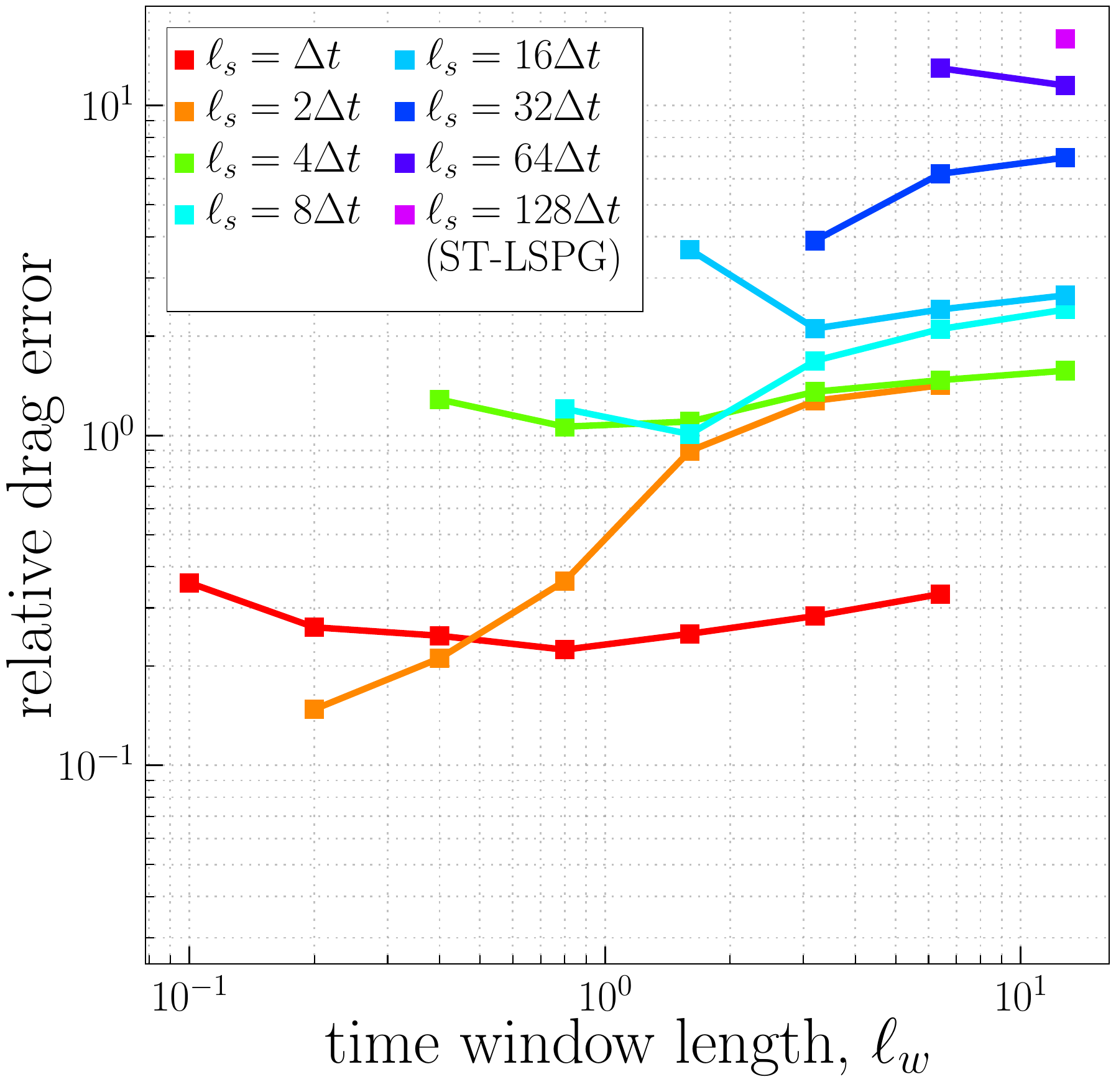}}\hspace{1mm}\nolinebreak
\subfloat[relative drag error: $e_s = 0.99, e_t = 0.999$][\centering \hspace{1.5mm}relative drag error:\par\hspace{5.5mm}$e_s = 0.99, e_t = 0.999$\label{f:NSDrag2}]{\includegraphics[height=0.264\textwidth]{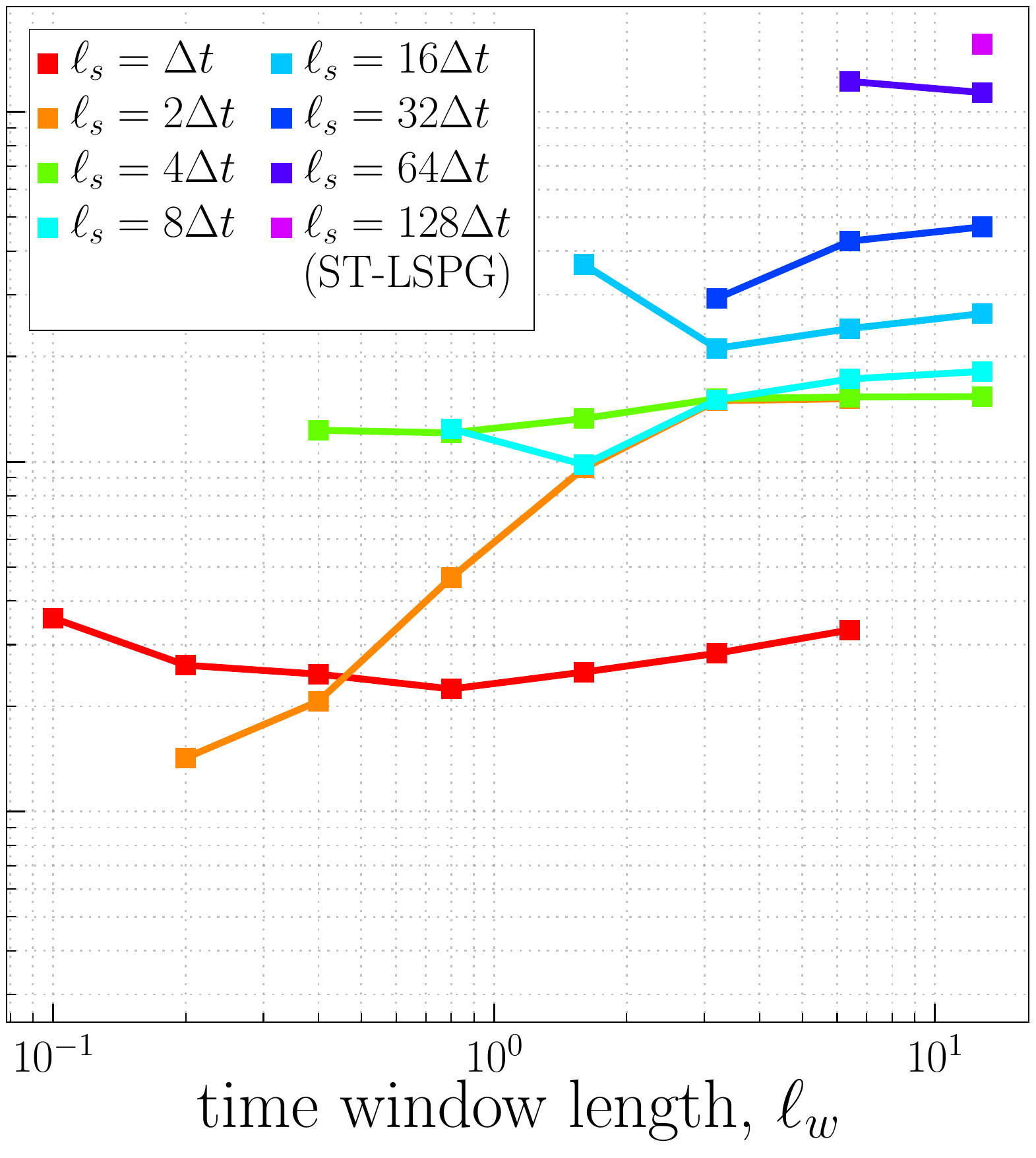}}\hspace{1mm}\nolinebreak
\subfloat[relative drag error: $e_s = 0.999, e_t = 0.99$][\centering \hspace{1.5mm}relative drag error:\par\hspace{5.5mm}$e_s = 0.999, e_t = 0.99$\label{f:NSDrag4}]{\includegraphics[height=0.264\textwidth]{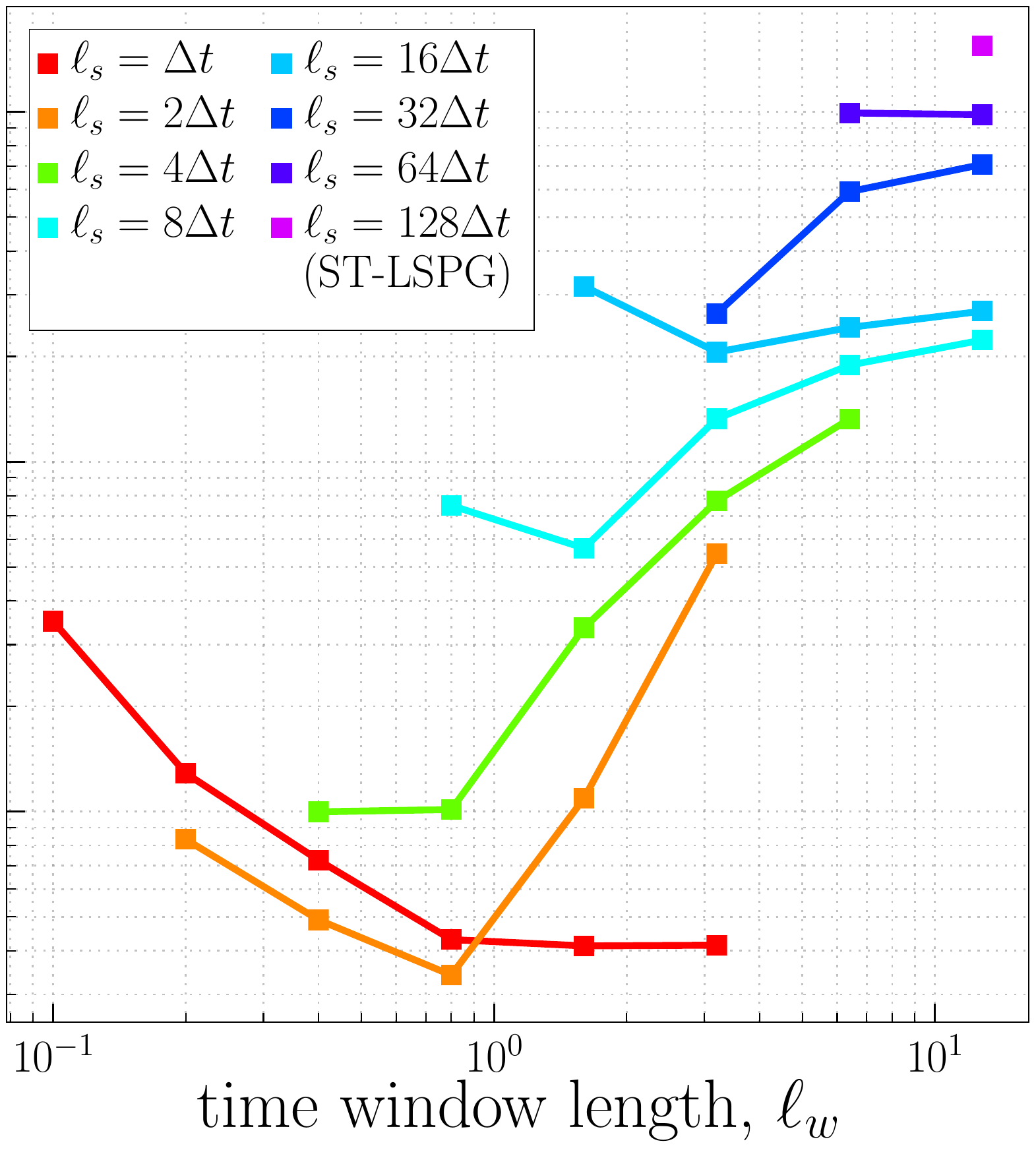}}\hspace{1mm}\nolinebreak
\subfloat[relative drag error: $e_s = 0.999, e_t = 0.999$][\centering \hspace{1.5mm}relative drag error:\par\hspace{5.5mm}$e_s = 0.999, e_t = 0.999$\label{f:NSDrag6}]{\includegraphics[height=0.264\textwidth]{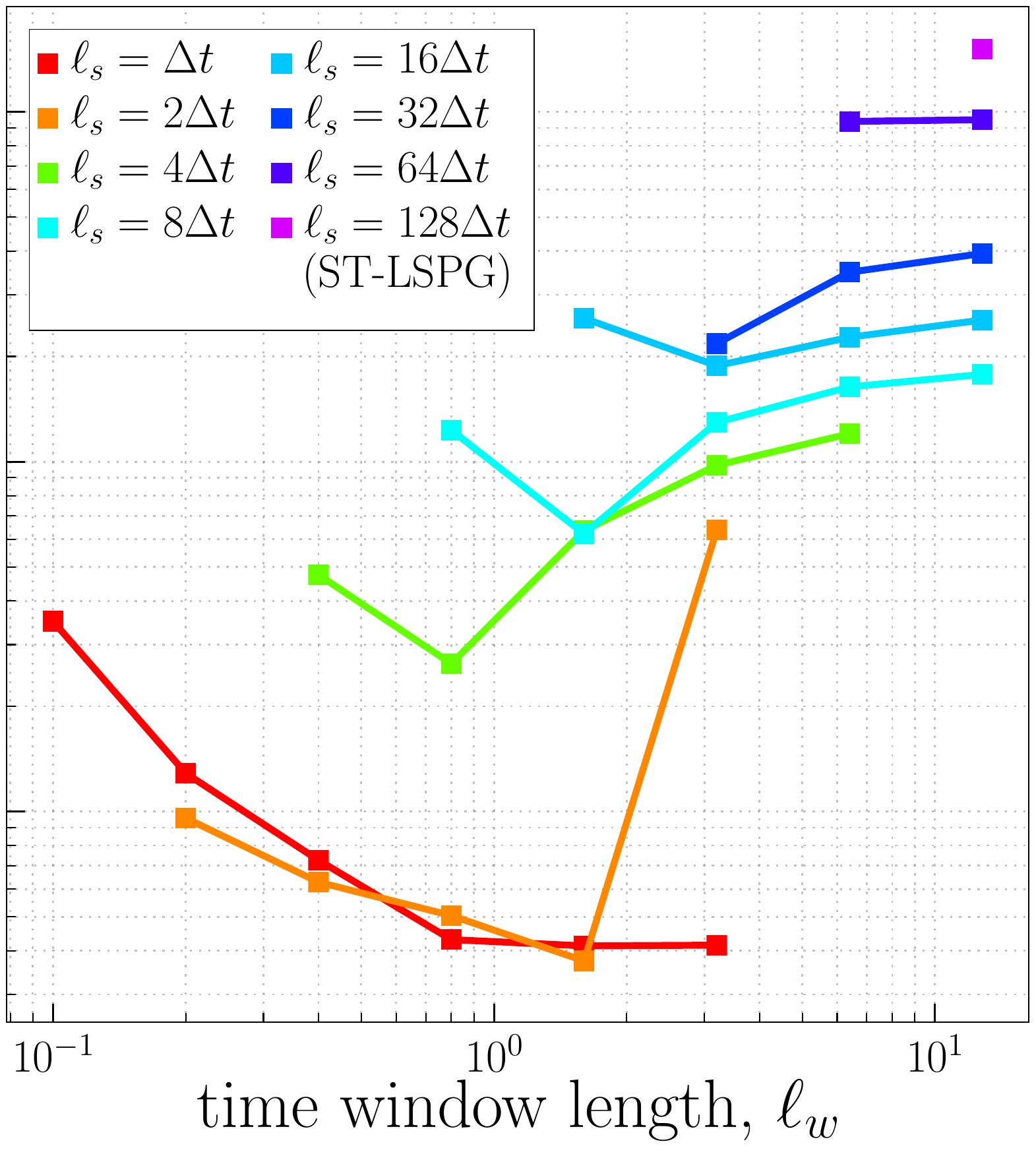}} \vspace{-2mm}\\
\subfloat[relative lift error: $e_s = 0.99, e_t = 0.99$][\centering \hspace{1.5mm}relative lift error:\par\hspace{5.5mm}$e_s = 0.99, e_t = 0.99$\label{f:NSLift0}]{\includegraphics[height=0.264\textwidth]{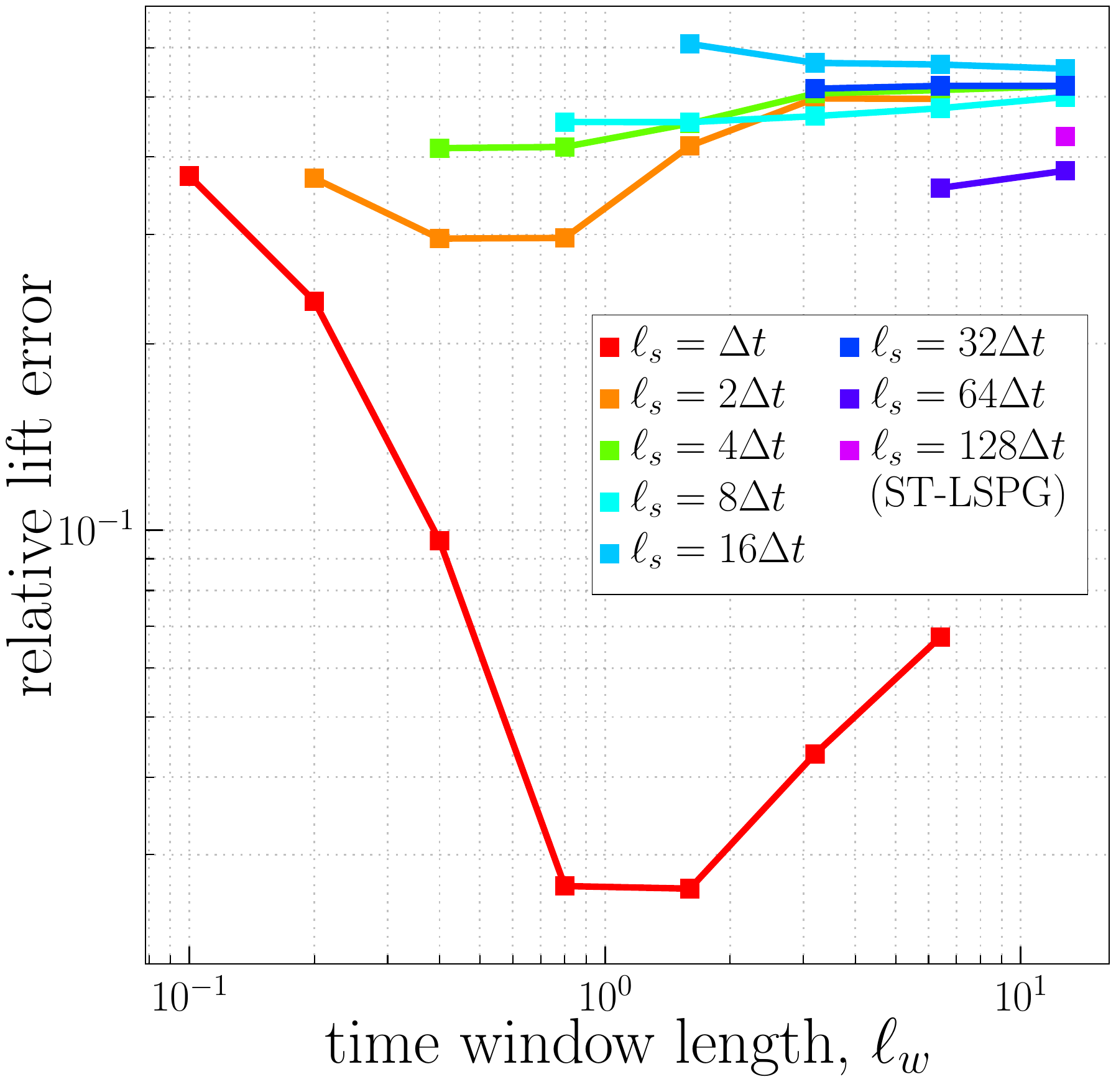}}\hspace{1mm}\nolinebreak
\subfloat[relative lift error: $e_s = 0.99, e_t = 0.999$][\centering \hspace{1.5mm}relative lift error:\par\hspace{5.5mm}$e_s = 0.99, e_t = 0.999$\label{f:NSLift2}]{\includegraphics[height=0.264\textwidth]{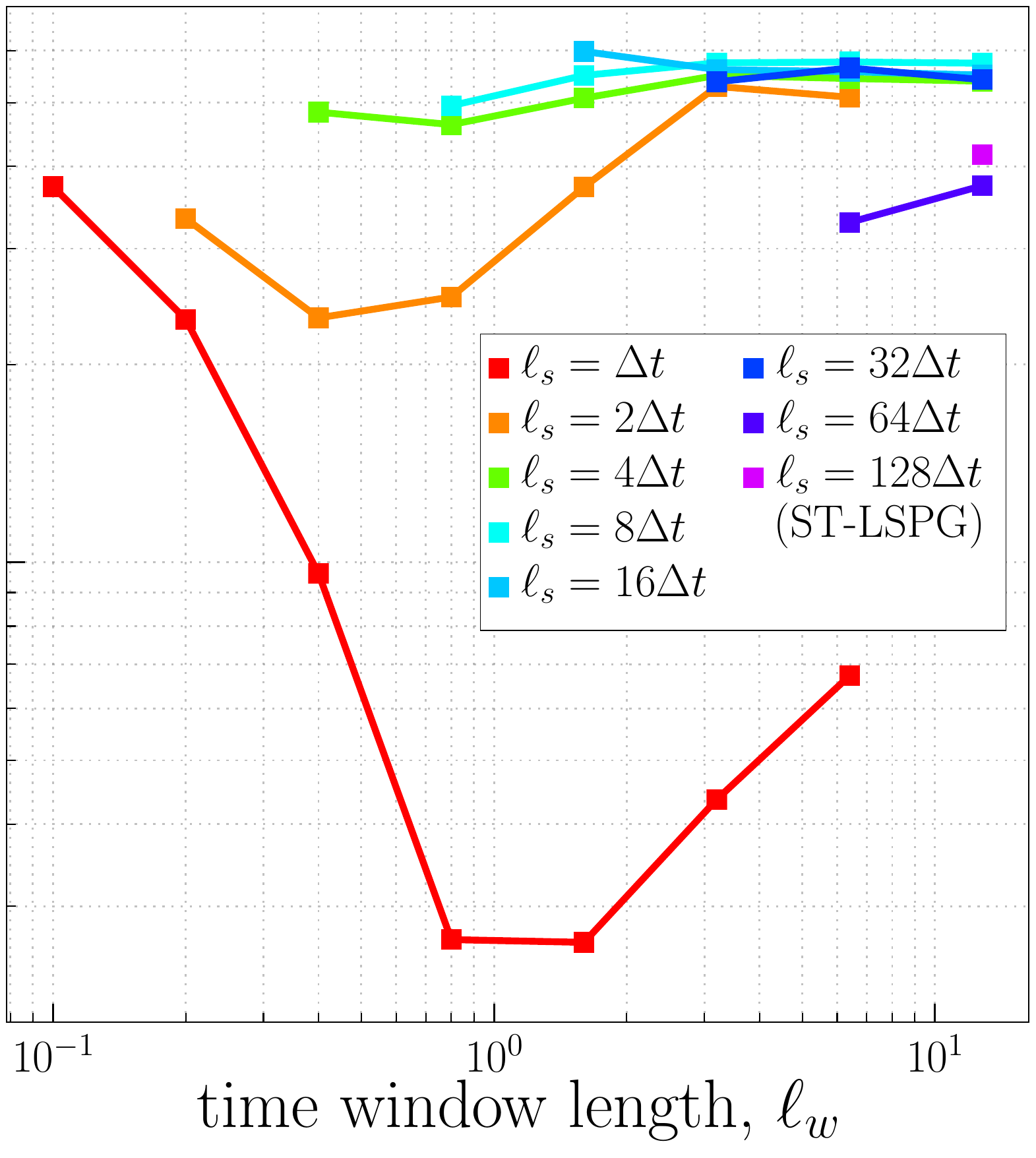}}\hspace{1mm}\nolinebreak
\subfloat[relative lift error: $e_s = 0.999, e_t = 0.99$][\centering \hspace{1.5mm}relative lift error:\par\hspace{5.5mm}$e_s = 0.999, e_t = 0.99$\label{f:NSLift4}]{\includegraphics[height=0.264\textwidth]{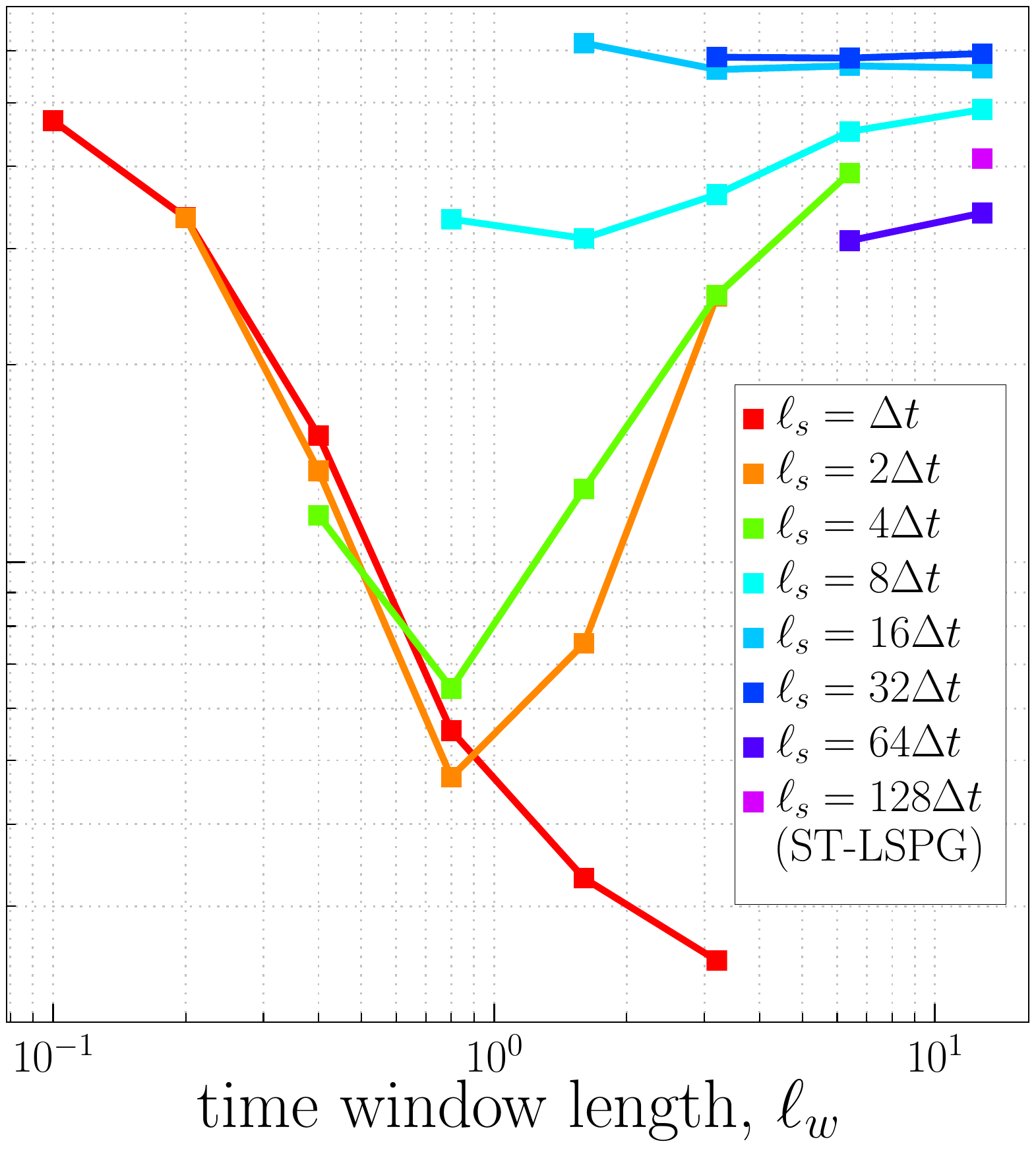}}\hspace{1mm}\nolinebreak
\subfloat[relative lift error: $e_s = 0.999, e_t = 0.999$][\centering \hspace{1.5mm}relative lift error:\par\hspace{5.5mm}$e_s = 0.999, e_t = 0.999$\label{f:NSLift6}]{\includegraphics[height=0.264\textwidth]{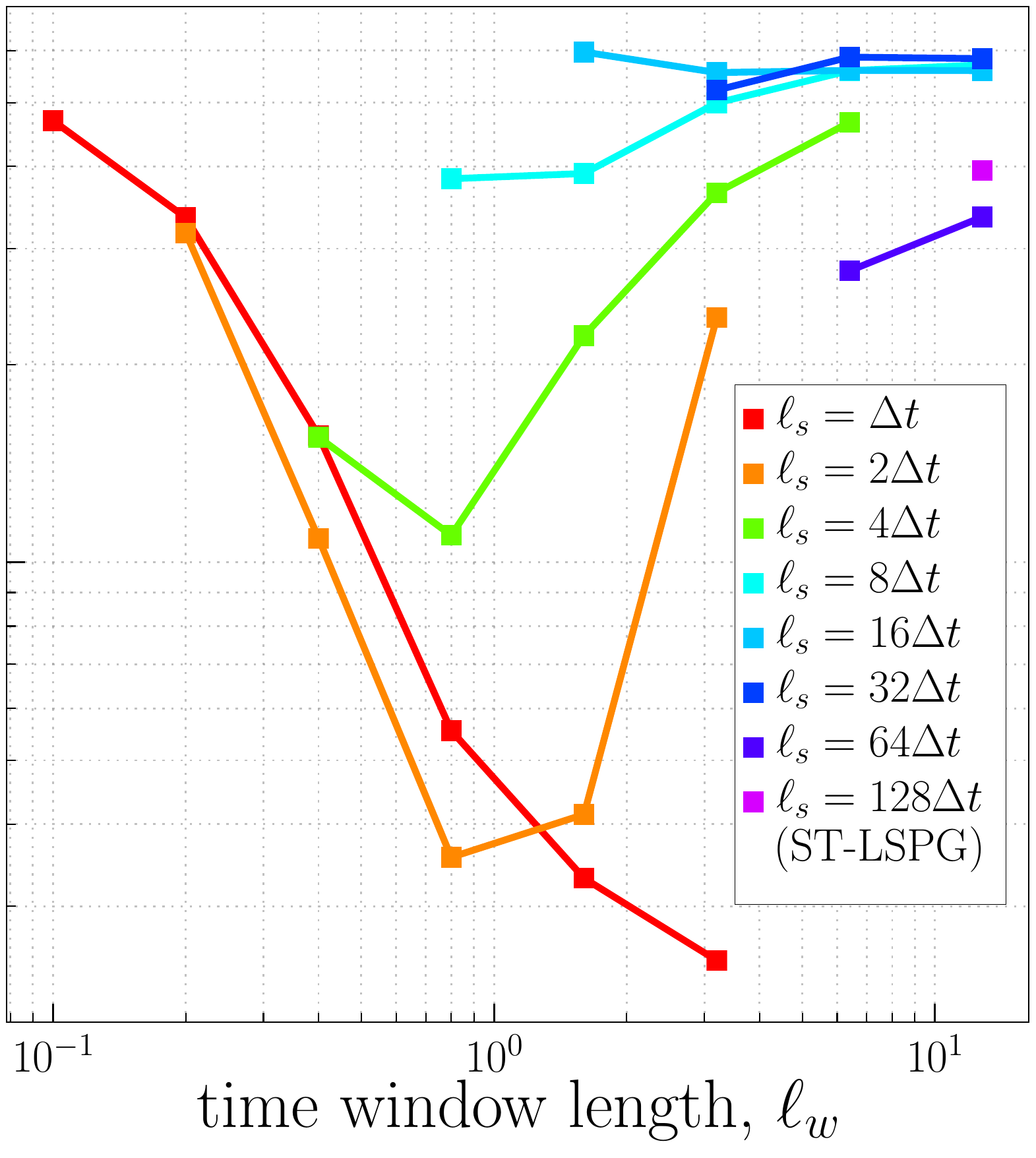}}
\caption{\emph{Compressible Navier--Stokes equations.} MSE, residual $\ell^2$ norms, relative drag error, and relative lift error for WST-LSPG for the following set of bases: $e_s = 0.99, e_t = 0.99$ $e_s = 0.99, e_t = 0.999$, $e_s = 0.999, e_t = 0.99$, $e_s = 0.999, e_t = 0.999$. All results are shown for various time window size, $\ell_w$, and various sub-window size, $\ell_s$.}
\label{f:NSLinePlots}
\end{figure}
This section presents and compares the performance of the WST-LSPG method over varying window lengths and varying sub-window lengths for the two-dimensional Navier--Stokes equations.  Figure~\ref{f:NSLinePlots} shows the MSE, residual $\ell^2$ norm, relative drag error, and relative lift error as a function of the window length, $\ell_w$ and the sub-window length, $\ell_s$. The number of time windows is set to be $\numWindows=\Bmat{1,2,4,8,16,32,64,128}$, and the number of sub-windows is set to be $\numSubWindows=\Bmat{1,2,4,8,16,32,64,128}$.

First, Figures~\ref{f:NSMSE0},~\ref{f:NSMSE2},~\ref{f:NSMSE4},~\ref{f:NSMSE6} show the MSE vs. window length, $\ell_w$ for varying sub-window lengths. These results show that the MSE decreases initially and then increases as the time window length increases, except when, $\ell_s=\Delta t$. Overall larger sub-window and window lengths lead to higher MSE. It is interesting to note that the ST-LSPG test case ($\ell_w=128\Delta t,\ell_s=128\Delta t$) does not give the lowest MSE when compared to the test cases where $\ell_s=\Bmat{8\Delta t, 16\Delta t, 32\Delta t, 64\Delta t}$. Lastly, it can be seen that higher number of bases for space and time lead to overall lower MSE, especially for lower sub-window lengths, $\ell_s$. 

Second, Figures~\ref{f:NSResidual0},~\ref{f:NSResidual2},~\ref{f:NSResidual4},~\ref{f:NSResidual6} show the residual $\ell^2$ norms vs. window length, $\ell_w$ for varying sub-window lengths. These figures show that as the window lengths increase for all sets of space--time bases, the residual $\ell^2$ norm decreases for all sub-window lengths, contrary to the results shown for MSE. This is consistent with the results found for the Burgers' equation in Figure~\ref{f:window1}. Note that unlike the MSE results, the ST-LSPG case has the largest residual $\ell^2$ norm for all set of space--time bases, showing that introducing windows and sub-windows improves the residual $\ell^2$ norm, especially for larger time-window lengths.

Next, Figures~\ref{f:NSDrag0},~\ref{f:NSDrag2},~\ref{f:NSDrag4},~\ref{f:NSDrag6} show the relative drag error vs. window length, $\ell_w$ for varying sub-window lengths. Unlike the results for the residual $\ell^2$ norm, the relative drag error error increases with the time window length, $\ell_w$ except for when $\ell_s=\Delta t$ and $e_s=0.999$. Again, the ST-LSPG case has the largest relative drag error for all set of space--time bases, which shows that windowing can increase the accuracy of the ROM when it comes to the drag. For a larger number of spatial and temporal bases, the relative drag error does decrease with increasing time window length when $\ell_s=\Delta t$. Of the four different types of output errors, relative drag error is the largest in this study.

Lastly, Figures~\ref{f:NSLift0},~\ref{f:NSLift2},~\ref{f:NSLift4},~\ref{f:NSLift6} show the relative lift error vs. window length for varying sub-window lengths. The relative lift error behaves similarly to the MSE. For $e_s=0.99, e_t=0.99$ and $e_s=0.99, e_t=0.999$, the relative lift error decreases when $\ell_s=\Delta t$. For $\ell_s> \Delta t$, the relative lift error is an order of magnitude greater than when $\ell_s=\Delta t$, and increases slowly as the time window length increases. Note that the ST-LSPG case does not give the lowest relative lift error---true for MSE as well. It can be seen that when $\ell_s>2\Delta t$, the WST-LSPG relative lift error can be greater than the relative lift error for ST-LSPG. However, when $e_s=0.999, e_t=0.99$ and $e_s=0.999, e_t=0.999$, the relative lift error improves, showing that more spatial bases are needed in order to achieve lower relative lift errors. 
\subsection{Windowed parameter study with hyper-reduction}
\begin{figure}[p!]
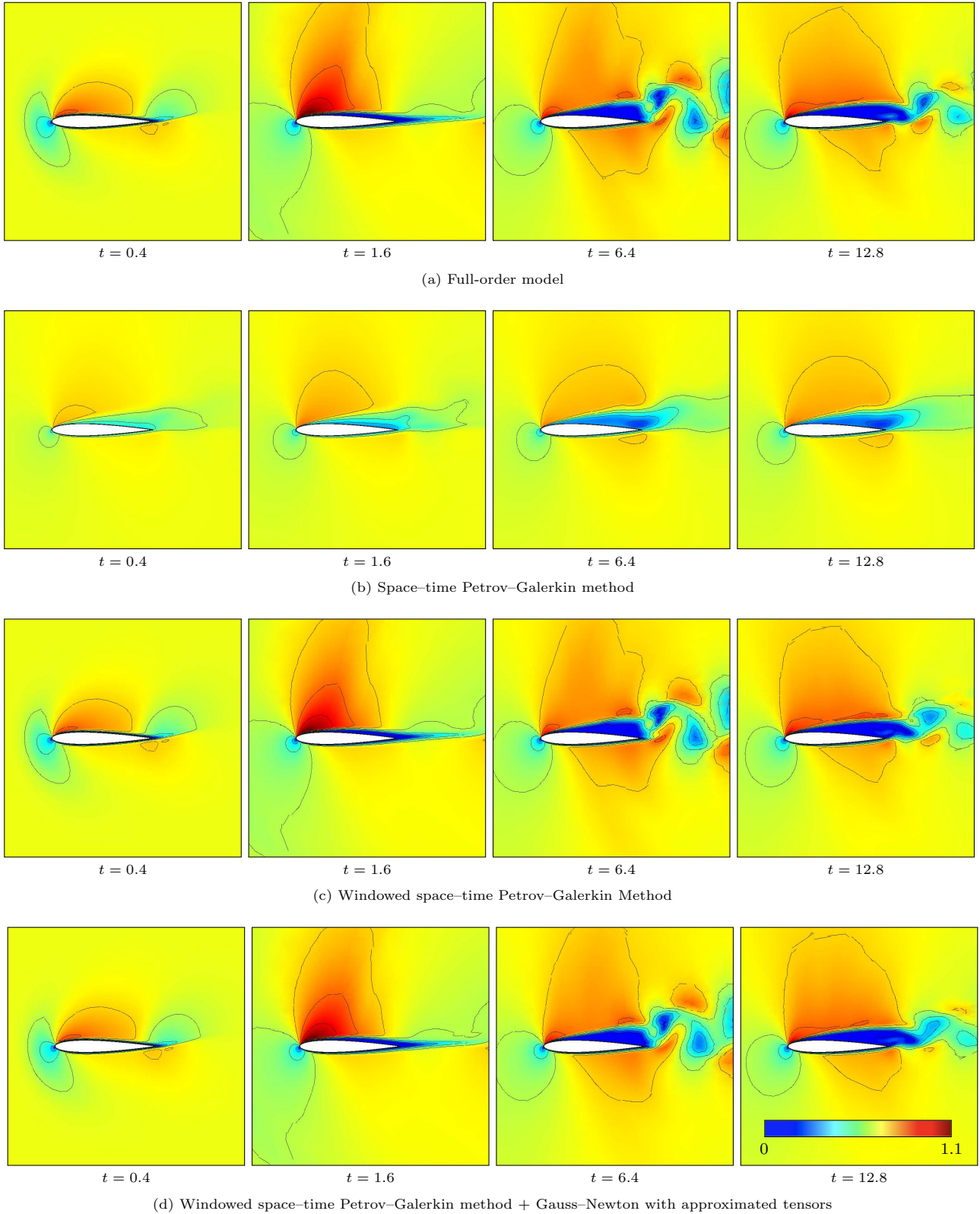

\begin{center}
\subfloat[Full-order model\label{f:fomAirfoilA}]{\nolinebreak
\airfoil{naca_U00004.pdf}{{\scriptstyle t\;=\;0.4}}\nolinebreak
\airfoil{naca_U00016.pdf}{{\scriptstyle t\;=\;1.6}}\nolinebreak
\airfoil{naca_U00064.pdf}{{\scriptstyle t\;=\;6.4}}\nolinebreak
\airfoil{naca_U00128.pdf}{{\scriptstyle t\;=\;12.8}}}\\
\subfloat[Space--time Petrov--Galerkin method\label{f:stlspgAirfoilA}]{\nolinebreak
\airfoil{naca_state_1_2_STROM_approximate_U00004.pdf}{{\scriptstyle t\;=\;0.4}}\nolinebreak
\airfoil{naca_state_1_2_STROM_approximate_U00016.pdf}{{\scriptstyle t\;=\;1.6}}\nolinebreak
\airfoil{naca_state_1_2_STROM_approximate_U00064.pdf}{{\scriptstyle t\;=\;6.4}}\nolinebreak
\airfoil{naca_state_1_2_STROM_approximate_U00128.pdf}{{\scriptstyle t\;=\;12.8}}}\\
\subfloat[Windowed space--time Petrov--Galerkin Method\label{f:wstlspgAirfoilA}]{\nolinebreak
\airfoil{WST-LSPG-4-32-naca_state_1_2_STROM_approximate_U00004.pdf}{{\scriptstyle t\;=\;0.4}}\nolinebreak
\airfoil{WST-LSPG-4-32-naca_state_1_2_STROM_approximate_U00016.pdf}{{\scriptstyle t\;=\;1.6}}\nolinebreak
\airfoil{WST-LSPG-4-32-naca_state_1_2_STROM_approximate_U00064.pdf}{{\scriptstyle t\;=\;6.4}}\nolinebreak
\airfoil{WST-LSPG-4-32-naca_state_1_2_STROM_approximate_U00128.pdf}{{\scriptstyle t\;=\;12.8}}}\\
\subfloat[Windowed space--time Petrov--Galerkin method + Gauss--Newton with approximated tensors\label{f:wstgnatAirfoilA}]{\nolinebreak
\airfoil{naca_state_1_2_STGNAT_approximateFinal_U00004.pdf}{{\scriptstyle t\;=\;0.4}}\nolinebreak
\airfoil{naca_state_1_2_STGNAT_approximateFinal_U00016.pdf}{{\scriptstyle t\;=\;1.6}}\nolinebreak
\airfoil{naca_state_1_2_STGNAT_approximateFinal_U00064.pdf}{{\scriptstyle t\;=\;6.4}}\nolinebreak
\airfoilbar{naca_state_1_2_STGNAT_approximateFinal_U00128.pdf}{{\scriptstyle t\;=\;12.8}}}
\caption{\emph{Compressible Navier--Stokes equations.} NACA 0012 airfoil Mach number plots for FOM, ST-LSPG, WST-LSPG, and WST-GNAT for $t=\Bmat{0.4, 1.6, 6.4, 12.8}$}
\label{f: MachContours1}
\end{center}
\end{figure}
%

\begin{figure}[t!]
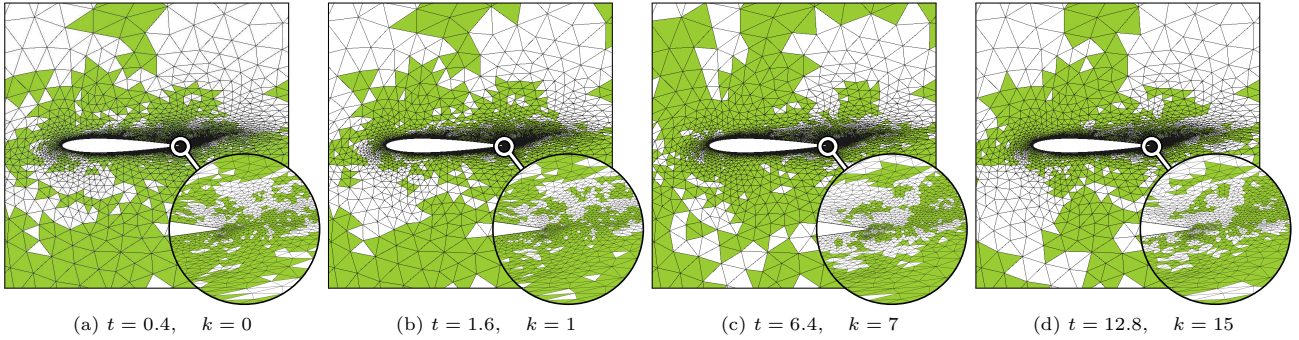

\subfloat[$t = 0.4,\quad k=0$\label{f:mesh_8}]{\samplemesh{hyperplot_8.pdf}{circle_8.pdf}}\hspace{1mm}\nolinebreak
\subfloat[$t = 1.6,\quad k=1$\label{f:mesh_16}]{\samplemesh{hyperplot_16.pdf}{circle_16.pdf}}\hspace{1mm}\nolinebreak
\subfloat[$t = 6.4,\quad k=7$\label{f:mesh_64}]{\samplemesh{hyperplot_64.pdf}{circle_64.pdf}}\hspace{1mm}\nolinebreak
\subfloat[$t = 12.8,\quad k=15$\label{f:mesh_128}]{\samplemesh{hyperplot_128.pdf}{circle_128.pdf}}
\caption{\emph{Compressible Navier--Stokes equations.} Sample mesh obtained by performing a space--time greedy algorithm for a time instance, $t$ of the $k^{\mathrm{th}}$ window used for the WST-GNAT case in Figure~\ref{f: MachContours1}.}
\label{f: po}
\end{figure}

In this section, the performances of WST-LSPG and WST-GNAT are compared with ST-LSPG and ST-GNAT across different window lengths and sub-window lengths. First, Figure~\ref{f: MachContours1} shows the Mach number contour plots for the compressible Navier--Stokes equations for four different methods: FOM, ST-LSPG, WST-LSPG, and WST-GNAT.

Figure~\ref{f:fomAirfoilA} shows the Mach contour plots for the FOM at different time steps. With a Reynolds number of $\mathrm{Re}=5000$, the flow tends to exhibit more pseudo-chaotic behaviors as time increases. Figure~\ref{f:stlspgAirfoilA} shows the Mach contour plots for the original ST-LSPG method at the same time steps as shown for the FOM. This test case refers to a ROM produced with $\numWindows=1$ and $\numSubWindows=1$ where $e_s=0.999$ and $e_t=0.999$. At $\mathrm{Re}=5000$, ST-LSPG is unable to pick up the intricacies involved in the wake of the NACA 0012 airfoil. 

Figure~\ref{f:wstlspgAirfoilA} shows the Mach contour plots for the proposed WST-LSPG method at the same time steps as shown for the FOM and ST-LSPG. This test case refers to a ROM produced with $\numWindows=4$, $\numSubWindows=32$, $e_s=0.999$, and $e_t=0.999$. This test case was chosen, because this set of WST-LSPG parameters produced the lowest residual $\ell^2$ norms. It can be seen that introducing windows and sub-windows improves the accuracy of the Mach number profiles. Figure~\ref{f:wstlspgAirfoilA} shows a transonic shock wave developing, accompanied by high Mach numbers, which ST-LSPG failed to qualify. Note that by $t=6.4$, the development of the transonic shock wave is hindered by the development of shedding vortices. Figure~\ref{f:wstlspgAirfoilA} at $t=6.4$ and $t=12.8$ shows that WST-LSPG is able to resolve the wake better than ST-LSPG. 

Figure~\ref{f:wstgnatAirfoilA} show the Mach contour plots for the WST-GNAT method at the same time steps as shown for WST-LSPG. This test case refers to a ROM produced with $\numWindows=16$, $\numSubWindows=8$, $e_s=0.999$, $e_t=0.999$, $e_{r_s}=0.999$, $e_{r_t}=0.999$, $z_s=3242$ (30\% of the total number of elements in the spatial sample mesh), and $z_t=3$. These plots show that WST-GNAT is able to resolve the wake, but at a lower reduced computational cost compared to WST-LSPG. A subset of the sample mesh used in Figure~\ref{f:wstgnatAirfoilA} can be seen in Figure~\ref{f: po}. 

Next, a more thorough comparison is made for the performances of WST-LSPG and WST-GNAT. Like the Burgers' equation experiment, the results of these methods over varying windowed parameters are presented in Pareto plots. There are three sets of Pareto plots for the compressible Navier--Stokes equations in this paper, one for each output error of interest: MSE, relative drag error, and relative lift error. Additionally, residual $\ell^2$ norm results were calculated and can be found in \ref{s:nsresidualss}, which showed similar residual $\ell^2$ norm results for Navier--Stokes equations as the residual $\ell^2$ norm results for Burgers' equation. For each output error of interest, this paper presents two Pareto plots, one for varying sub-window lengths, $\ell_s$ and one for varying window lengths, $\ell_w$. Note that not all possible combinations of window lengths and sub-window lengths for WST-GNAT were able to be performed due to three reasons: (1) when the window sizes are large (i.e., large space--time residual vectors) and many sub-windows are employed (i.e., high-fidelity basis), the Gauss--Newton method requires more memory than was available on the node used to perform the computations, (2) not all parameter combinations led to solutions that converged, and (3) for $\ell_w\in[\Delta t, 2\Delta t]$, e.g., Figures~\ref{f:NSParetoMSE7},~\ref{f:NSParetoMSE8},~\ref{f:NSParetoDrag_7},~\ref{f:NSParetoDrag_8},~\ref{f:NSParetoLift7},~\ref{f:NSParetoLift8}, performing hyper-reduction in time is limited, and performing hyper-reduction in space over a small time window yields unstable solutions due to lack of \emph{a priori} bounds~\cite{carlberg2017galerkin}.
\begin{figure}[!t]
\centering
\subfloat[$\numWindows=1,\ell_{w}=128\Delta t$\label{f:NSParetoMSE1}]{\pareto{NS_Pareto_MSE_7.pdf}}\hspace{1mm}\nolinebreak
\subfloat[$\numWindows=2,\ell_{w}=64\Delta t$\label{f:NSParetoMSE2}]{\pareto{NS_Pareto_MSE_6.pdf}}\hspace{1mm}\nolinebreak
\subfloat[$\numWindows=4,\ell_{w}=32\Delta t$\label{f:NSParetoMSE3}]{\pareto{NS_Pareto_MSE_5.pdf}}\hspace{1mm}\nolinebreak
\subfloat[$\numWindows=8,\ell_{w}=16\Delta t$\label{f:NSParetoMSE4}]{\pareto{NS_Pareto_MSE_4.pdf}}\hspace{1mm}\nolinebreak
\subfloat[$\numWindows=16,\ell_{w}=8\Delta t$\label{f:NSParetoMSE5}]{\pareto{NS_Pareto_MSE_3.pdf}}\\
\subfloat[$\numWindows=32,\ell_{w}=4\Delta t$\label{f:NSParetoMSE6}]{\pareto{NS_Pareto_MSE_2.pdf}}\hspace{1mm}\nolinebreak
\subfloat[$\numWindows=64,\ell_{w}=2\Delta t$\label{f:NSParetoMSE7}]{\pareto{NS_Pareto_MSE_1.pdf}}\hspace{1mm}\nolinebreak
\subfloat[$\numWindows=128,\ell_{w}=\Delta t$\label{f:NSParetoMSE8}]{\pareto{NS_Pareto_MSE_0.pdf}}\hspace{1mm}\nolinebreak
\legend{ns_legend.pdf}\hspace{1mm}\\
\subfloat[Combined Pareto plot with Pareto front for MSE\label{f:NSParetoMSEALL}]{\includegraphics[width=1.0\textwidth]{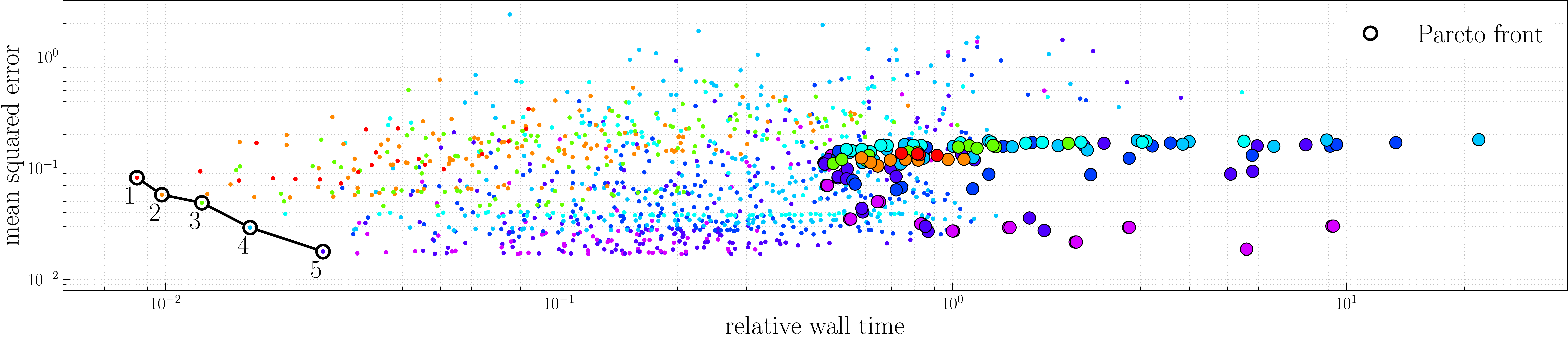}}
\caption{\emph{Compressible Navier--Stokes equations.} Pareto plots for mean squared error vs. relative wall time for varying sub-window lengths. These Pareto plots highlight how the sub-window length, $\ell_s$, affect the residual $\ell^2$ norm for each window length, $\ell_w$ for WST-LSPG and WST-GNAT. Each hyper-reduction case refers to a unique set of spatial sample nodes and temporal sample nodes. The last Pareto plot combines all the results for MSE and shows the Pareto front.}
\label{f:NSParetoMSE}
\end{figure}
\begin{figure}[!hp]
\centering
\subfloat[$\ell_{s}=128\Delta t$\label{f:NSParetoMSE_e1}]{\pareto{NS_Pareto_MSE_7_e.pdf}}\hspace{1mm}\nolinebreak
\subfloat[$\ell_{s}=64\Delta t$\label{f:NSParetoMSE_e2}]{\pareto{NS_Pareto_MSE_6_e.pdf}}\hspace{1mm}\nolinebreak
\subfloat[$\ell_{s}=32\Delta t$\label{f:NSParetoMSE_e3}]{\pareto{NS_Pareto_MSE_5_e.pdf}}\hspace{1mm}\nolinebreak
\subfloat[$\ell_{s}=16\Delta t$\label{f:NSParentoMSE_e4}]{\pareto{NS_Pareto_MSE_4_e.pdf}}\hspace{1mm}\nolinebreak
\subfloat[$\ell_{s}=8\Delta t$\label{f:NSParetoMSE_e5}]{\pareto{NS_Pareto_MSE_3_e.pdf}}\\
\subfloat[$\ell_{s}=4\Delta t$\label{f:NSParetoMSE_e6}]{\pareto{NS_Pareto_MSE_2_e.pdf}}\hspace{1mm}\nolinebreak
\subfloat[$\ell_{s}=2\Delta t$\label{f:NSParetoMSE_e7}]{\pareto{NS_Pareto_MSE_1_e.pdf}}\hspace{1mm}\nolinebreak
\subfloat[$\ell_{s}=\Delta t$\label{f:NSParetoMSE_e8}]{\pareto{NS_Pareto_MSE_0_e.pdf}}\hspace{1mm}\nolinebreak
\legend{ns_legend2.pdf}\\
\subfloat[Combined Pareto plot with Pareto front for MSE \label{f:NSParetoMSE_eALL}]{\includegraphics[width=1.0\textwidth]{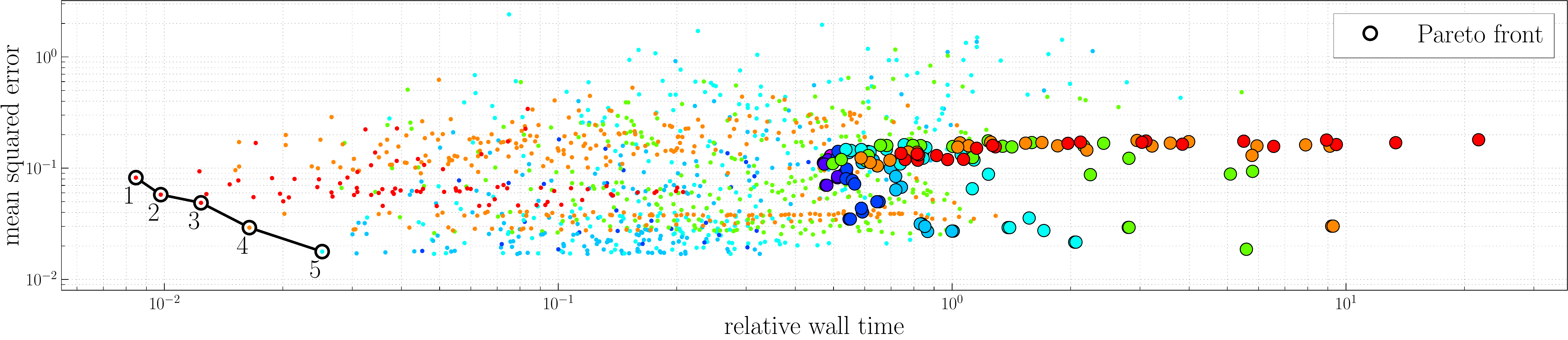}}
\caption{\emph{Compressible Navier--Stokes equations.} Pareto plots for MSE vs. relative wall time for varying window lengths. These Pareto plots highlight how the window length, $\ell_w$, affect the MSE for WST-LSPG and WST-GNAT. The last Pareto plot combines the results for the MSE and shows the Pareto front (same as Figure~\ref{f:MSE_ALL}).}
\label{f:NSParetoMSE_e}
\begin{center}
\vspace{-3mm}
\captionof{table}{\label{t:NSMSE}\emph{Compressible Navier--Stokes equations}. Method and parameter values for labeled data points in Figures~\ref{f:NSParetoMSE} and~\ref{f:NSParetoMSE_e} that yield Pareto-optimal performance for MSE. Highlighted rows refer to test cases that additionally exist on the Pareto Front for the residual $\ell^2$ norms.}
\resizebox{\columnwidth}{!}{%
\begin{tabular}{ ccccccccccccccc } \toprule
case& method & $\ell_w$ & $\ell_s$ & \thead{total\\$n_{st}$} &\thead{total\\$n^r_{st}$} &$e_s$ & $e_t$ & $e_{r_s}$ & $e_{r_t}$ & $z_t$ & $z_s$ & MSE & \thead{x-LSPG relative\\ wall time} &\thead{relative wall\\time}\\
\midrule \midrule
\rowcolor{matcha}1 & ST-GNAT	&12.8&	12.8&20 &16 &	0.999&	0.99& 0.999&	0.999&	9	&1080 & 0.082234972 &  0.839231102&0.008468197\\ 
2 & WST-GNAT&	12.8&	6.4& 35 & 10 &	0.999&	0.99 & 0.999&	0.999& 9	&1080& 0.057827486 & 0.842380016 &  0.009793139\\ 
 3  & WST-GNAT&	12.8&	3.2&71 & 6&	0.999&	0.99& 0.999&	0.999&	9&	1080& 0.048843402 & 1.306378678& 0.012390053\\ 
 \rowcolor{matcha}4 & WST-GNAT&	6.4	&0.8& 245 & 13 &	0.999&	0.99& 0.999	&0.999	&5	&1080& 0.029223707 &  2.261437990 &  0.016439454\\
 \rowcolor{matcha} 5& WST-GNAT &	1.6&	0.2	& 725 & 145& 0.999&	0.999 &0.999&	0.999&	3&	1080& 0.017791989 & 1.717868245 & 0.025160652\\\bottomrule
\end{tabular}
}
\end{center}
\end{figure}

Similarly to Figure~\ref{f:paretoMSE}, Figure~\ref{f:NSParetoMSE} shows how the MSE vs. relative wall time behaves with varying sub-window lengths for each time window length. Overall, hyper-reduction reduces the relative wall time for each $\ell_w$ considered. At the same time, the MSE improved with decreasing sub-window length. Figure~\ref{f:NSParetoMSEALL} combines the results for the MSE for varying sub-window length and shows the resulting Pareto front for MSE. It is seen that introducing hyper-reduction to WST-LSPG decreases the relative wall time by up to three orders of magnitude. 

Figure~\ref{f:NSParetoMSE_e} shows how the MSE vs. relative time wall time behaves with varying window lengths. These Pareto plots show the same data points as shown in Figure~\ref{f:NSParetoMSE}, but highlight how the MSE changes with varying window length for each sub-window length. It can be seen that decreasing the time window length leads to higher mean squared error. This means that an optimal set of window lengths and sub-window lengths exist for MSE: larger time window lengths, but shorter sub-window lengths. Figure~\ref{f:NSParetoMSE_eALL} combines the results and shows the Pareto front, which is the same as the Pareto front in Figure~\ref{f:NSParetoMSEALL}. Table~\ref{t:NSMSE} shows the parameters of the cases that lie on the Pareto front for MSE shown in both Figures~\ref{f:NSParetoMSEALL} and~\ref{f:NSParetoMSE_eALL}. Case 1 refers to the ST-GNAT technique, which has the lowest relative wall time. Cases 2--5 refer to the WST-GNAT technique, which exhibited higher relative wall times compared to Case 1, but, overall, exhibited lower MSEs. The highlighted rows refer to a specific set of parameters that lie on the Pareto front for MSE, residual $\ell^2$ norms, relative drag error, and relative lift error.
\begin{figure}[!t]
\begin{center}
\subfloat[$\numWindows=1,\ell_{w}=128\Delta t$\label{f:NSParetoDrag_1}]{\pareto{NS_Pareto_Drag_7.pdf}}\hspace{1mm}\nolinebreak
\subfloat[$\numWindows=2,\ell_{w}=64\Delta t$\label{f:NSParetoDrag_2}]{\pareto{NS_Pareto_Drag_6.pdf}}\hspace{1mm}\nolinebreak
\subfloat[$\numWindows=4,\ell_{w}=32\Delta t$\label{f:NSParetoDrag_3}]{\pareto{NS_Pareto_Drag_5.pdf}}\hspace{1mm}\nolinebreak
\subfloat[$\numWindows=8,\ell_{w}=16\Delta t$\label{f:NSParetoDrag_4}]{\pareto{NS_Pareto_Drag_4.pdf}}\hspace{1mm}\nolinebreak
\subfloat[$\numWindows=16,\ell_{w}=8\Delta t$\label{f:NSParetoDrag_5}]{\pareto{NS_Pareto_Drag_3.pdf}}\\
\subfloat[$\numWindows=32,\ell_{w}=4\Delta t$\label{f:NSParetoDrag_6}]{\pareto{NS_Pareto_Drag_2.pdf}}\hspace{1mm}\nolinebreak
\subfloat[$\numWindows=64,\ell_{w}=2\Delta t$\label{f:NSParetoDrag_7}]{\pareto{NS_Pareto_Drag_1.pdf}}\hspace{1mm}\nolinebreak
\subfloat[$\numWindows=128,\ell_{w}=\Delta t$\label{f:NSParetoDrag_8}]{\pareto{NS_Pareto_Drag_0.pdf}}\hspace{1mm}\nolinebreak
\legend{ns_legend.pdf}\\
\subfloat[Combined Pareto plot with Pareto front for relative drag error error\label{f:NSParetoDrag_All}]{\includegraphics[width=1.0\textwidth]{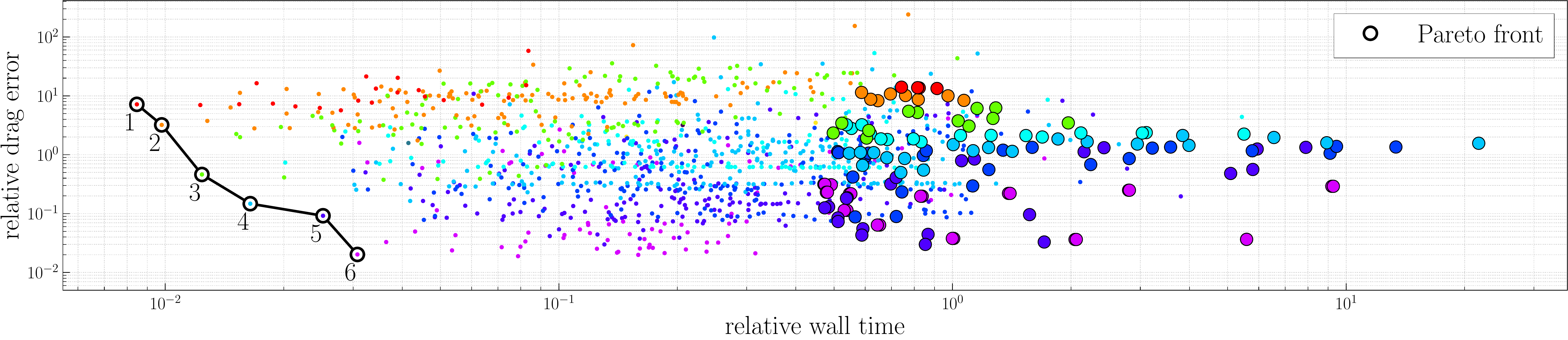}}
\caption{\emph{Compressible Navier--Stokes equations.} Pareto plots for relative drag error vs. relative wall time for varying sub-window lengths. These Pareto plots highlight how the sub-window length, $\ell_s$, affect the relative drag error for WST-LSPG and WST-GNAT. Each hyper-reduction case refers to a unique set of spatial sample nodes and temporal sample nodes. The last Pareto plot combines all the results for the relative drag error and shows the Pareto front.}
\label{f:NSParetoDrag}
\end{center}
\end{figure}
\begin{figure}[!hp]
\begin{center}
\subfloat[$\ell_{s}=128\Delta t$\label{f:NSParetoDrag_e1}]{\pareto{NS_Pareto_Drag_7_e.pdf}}\hspace{1mm}\nolinebreak
\subfloat[$\ell_{s}=64\Delta t$\label{f:NSParetoDrag_e2}]{\pareto{NS_Pareto_Drag_6_e.pdf}}\hspace{1mm}\nolinebreak
\subfloat[$\ell_{s}=32\Delta t$\label{f:NSParetoDrag_e3}]{\pareto{NS_Pareto_Drag_5_e.pdf}}\hspace{1mm}\nolinebreak
\subfloat[$\ell_{s}=16\Delta t$\label{f:NSParetoDrag_e4}]{\pareto{NS_Pareto_Drag_4_e.pdf}}\hspace{1mm}\nolinebreak
\subfloat[$\ell_{s}=8\Delta t$\label{f:NSParetoDrag_e5}]{\pareto{NS_Pareto_Drag_3_e.pdf}}\\
\subfloat[$\ell_{s}=4\Delta t$\label{f:NSParetoDrag_e6}]{\pareto{NS_Pareto_Drag_2_e.pdf}}\hspace{1mm}\nolinebreak
\subfloat[$\ell_{s}=2\Delta t$\label{f:NSParetoDrag_e7}]{\pareto{NS_Pareto_Drag_1_e.pdf}}\hspace{1mm}\nolinebreak
\subfloat[$\ell_{s}=\Delta t$\label{f:NSParetoDrag_e8}]{\pareto{NS_Pareto_Drag_0_e.pdf}}\hspace{1mm}\nolinebreak
\legend{ns_legend2.pdf}\\
\subfloat[Combined Pareto plot with Pareto front for relative drag error\label{f:NSParetoDrag_eAll}]{\includegraphics[width=1.0\textwidth]{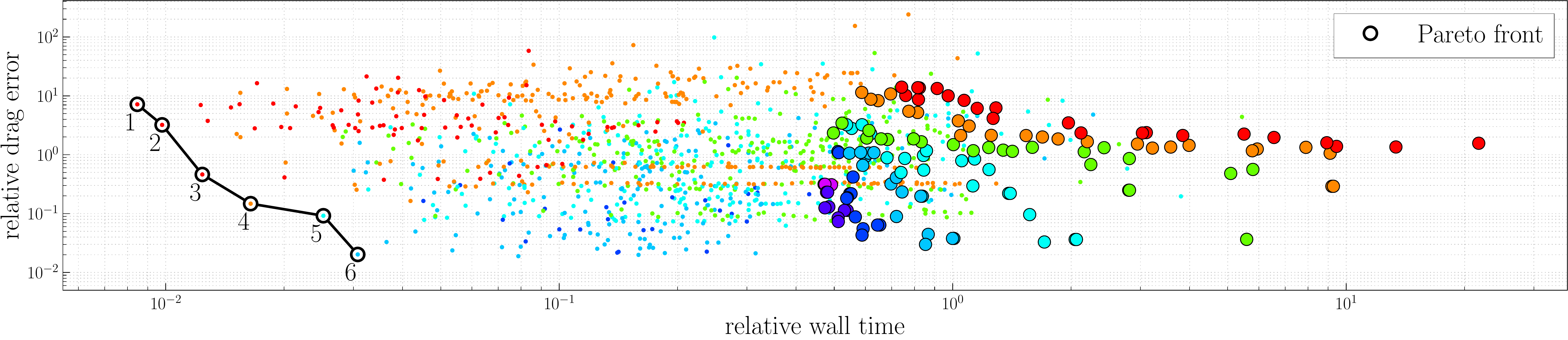}}
\caption{\emph{Compressible Navier--Stokes equations.} Pareto plots for relative drag error vs. relative wall time for varying window lengths. These Pareto plots highlight how the window length, $\ell_w$, affect the relative drag error for WST-LSPG and WST-GNAT. The last Pareto plot combines all the results for the relative drag error and shows the Pareto front.}
\label{f:NSParetoDrag_e}
\end{center}
\begin{center}
\vspace{-3mm}
\captionof{table}{\label{t:NSDrag}\emph{Compressible Navier--Stokes equations}. Method and parameter values for labeled data points in Figures~\ref{f:NSParetoDrag} and~\ref{f:NSParetoDrag_e} that yield Pareto-optimal performance for the relative drag error. Highlighted rows refer to test cases that additionally exist on the optimal Pareto Front for MSE, residual $\ell^2$ norm, and relative lift error.}
\resizebox{\columnwidth}{!}{%
\begin{tabular}{ ccccccccccccccc } \toprule
case& method & $\ell_w$ & $\ell_s$ & \thead{total\\$n_{st}$} &\thead{total\\$n^r_{st}$}& $e_s$ & $e_t$ & $e_{r_s}$ & $e_{r_t}$ & $z_t$ & $z_s$ & \thead{relative drag\\ error} & \thead{x-LSPG relative\\ wall time} &\thead{relative wall\\time}\\
\midrule \midrule 
\rowcolor{matcha}1 & ST-GNAT	&12.8&	12.8&20 &16 &	0.999&	0.99& 0.999&	0.999&	9	&1080 & 7.168969146 &  0.839231102&0.008468197\\ 
2 & WST-GNAT&12.8&	6.4& 35 & 10 &	0.999&	0.99 & 0.999&	0.999& 9	&1080& 3.210113397 & 0.842380016 &  0.009793139\\ 
 3  & WST-GNAT&	12.8&	3.2	&71& 6&0.999&	0.99& 0.999&	0.999&	9&	1080& 0.461654353& 1.306378678& 0.012390053\\ 
\rowcolor{matcha}4 & WST-GNAT&	6.4	&0.8& 245 &13&	0.999&	0.99& 0.999	&0.999	&4	&1080&0.147403667 &  2.261437990 &  0.016439454\\
 \rowcolor{matcha}5& WST-GNAT &	1.6&	0.2	&725 &145&0.999&	0.999 &0.999&	0.999&	3&	1080&0.091927101& 1.717868245 & 0.025160652\\
  6& WST-GNAT&0.8&	0.1& 774 &277&	0.999&	0.999& 0.999&	0.999&	3&	1080 &0.020251278& 1.006124284& 0.030765692\\ 
\bottomrule
\end{tabular}
}
\end{center}
\end{figure}
\begin{figure}[!t]
\begin{center}
\subfloat[$\numWindows=1,\ell_{w}=128\Delta t$ \label{f:NSParetoLift1}]{\pareto{NS_Pareto_Lift_7.pdf}}\hspace{1mm}\nolinebreak
\subfloat[$\numWindows=2,\ell_{w}=64\Delta t$ \label{f:NSParetoLift2}]{\pareto{NS_Pareto_Lift_6.pdf}}\hspace{1mm}\nolinebreak
\subfloat[$\numWindows=4,\ell_{w}=32\Delta t$ \label{f:NSParetoLift3}]{\pareto{NS_Pareto_Lift_5.pdf}}\hspace{1mm}\nolinebreak
\subfloat[$\numWindows=8,\ell_{w}=16\Delta t$ \label{f:NSParetoLift4}]{\pareto{NS_Pareto_Lift_4.pdf}}\hspace{1mm}\nolinebreak
\subfloat[$\numWindows=16,\ell_{w}=8\Delta t$ \label{f:NSParetoLift5}]{\pareto{NS_Pareto_Lift_3.pdf}}\\
\subfloat[$\numWindows=32,\ell_{w}=4\Delta t$ \label{f:NSParetoLift6}]{\pareto{NS_Pareto_Lift_2.pdf}}\hspace{1mm}\nolinebreak
\subfloat[$\numWindows=64,\ell_{w}=2\Delta t$ \label{f:NSParetoLift7}]{\pareto{NS_Pareto_Lift_1.pdf}}\hspace{1mm}\nolinebreak
\subfloat[$\numWindows=128,\ell_{w}=\Delta t$ \label{f:NSParetoLift8}]{\pareto{NS_Pareto_Lift_0.pdf}}\hspace{1mm}\nolinebreak
\legend{ns_legend.pdf}\\
\subfloat[Combined Pareto plot with Pareto front for relative lift error error\label{f:NSParetoLiftALL}]{\includegraphics[width=1.0\textwidth]{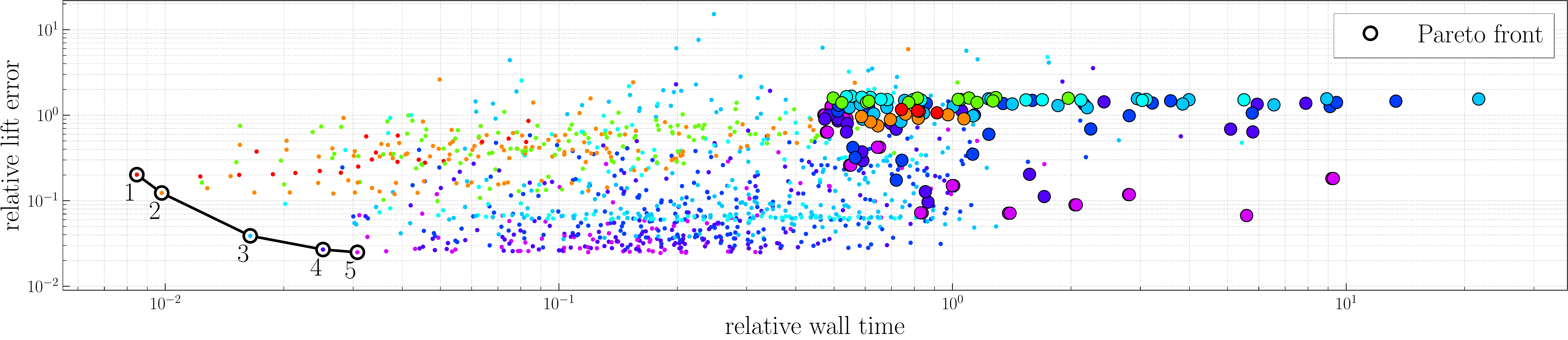}}
\caption{\emph{Compressible Navier--Stokes equations.} Pareto plots for relative lift error vs. relative wall time for varying sub-window lengths. These Pareto plots highlight how the sub-window length, $\ell_s$, affect the relative lift error for WST-LSPG and WST-GNAT. Each hyper-reduction case refers to a unique set of spatial sample nodes and temporal sample nodes. The last Pareto plot combines all the results for the relative lift error and shows the Pareto front.}
\label{f:NSParetoLift}
\end{center}
\end{figure}
\begin{figure}[!hp]
\begin{center}
\subfloat[$\ell_{s}=128\Delta t$ \label{f:NSParetoLift_e1}]{\pareto{NS_Pareto_Lift_7_e.pdf}}\hspace{1mm}\nolinebreak
\subfloat[$\ell_{s}=64\Delta t$ \label{f:NSParetoLift_e2}]{\pareto{NS_Pareto_Lift_6_e.pdf}}\hspace{1mm}\nolinebreak
\subfloat[$\ell_{s}=32\Delta t$ \label{f:NSParetoLift_e3}]{\pareto{NS_Pareto_Lift_5_e.pdf}}\hspace{1mm}\nolinebreak
\subfloat[$\ell_{s}=16\Delta t$ \label{f:NSParetoLift_e4}]{\pareto{NS_Pareto_Lift_4_e.pdf}}\hspace{1mm}\nolinebreak
\subfloat[$\ell_{s}=8\Delta t$ \label{f:NSParetoLift_e5}]{\pareto{NS_Pareto_Lift_3_e.pdf}}\\
\subfloat[$\ell_{s}=4\Delta t$ \label{f:NSParetoLift_e6}]{\pareto{NS_Pareto_Lift_2_e.pdf}}\hspace{1mm}\nolinebreak
\subfloat[$\ell_{s}=2\Delta t$ \label{f:NSParetoLift_e7}]{\pareto{NS_Pareto_Lift_1_e.pdf}}\hspace{1mm}\nolinebreak
\subfloat[$\ell_{s}=\Delta t$ \label{f:NSParetoLift_e8}]{\pareto{NS_Pareto_Lift_0_e.pdf}}\hspace{1mm}\nolinebreak
\legend{ns_legend2.pdf}\\
\subfloat[Combined Pareto plot with Pareto front for relative lift error\label{f:NSParetoLift_eAll}]{\includegraphics[width=1.0\textwidth]{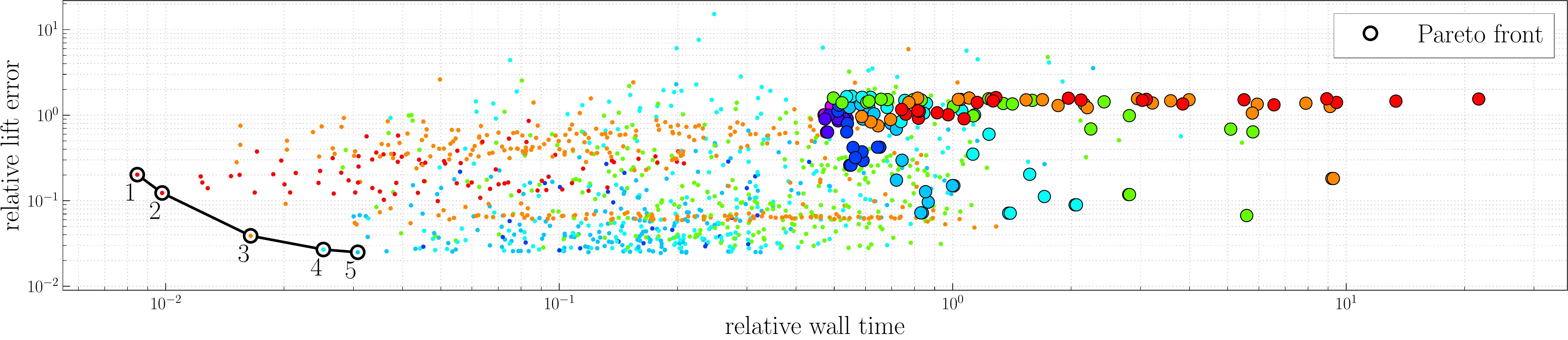}}
\caption{\emph{Compressible Navier--Stokes equations.} Pareto plots for relative lift error vs. relative wall time for varying window lengths. These Pareto plots highlight how the window length, $\ell_w$, affect the relative lift error for WST-LSPG and WST-GNAT. The last Pareto plot combines all the results for the relative lift error and shows the Pareto front.}
\label{f:NSParetoLift_e}
\end{center}
\begin{center}
\vspace{-3mm}
\captionof{table}{\label{t:NSLift}\emph{Compressible Navier--Stokes equations.} Method and parameter values for labeled data points in Figures~\ref{f:NSParetoLift} and~\ref{f:NSParetoLift_e} that yield Pareto-optimal performance for the relative lift error. Highlighted rows refer to test cases that additionally exist on the optimal Pareto Front for MSE, residual $\ell^2$ error, and the relative drag error.}
\resizebox{\columnwidth}{!}{%
\begin{tabular}{ ccccccccccccccc } \toprule
case& method & $\ell_w$ & $\ell_s$ & \thead{total\\$n_{st}$} &\thead{total\\$n^r_{st}$}&$e_s$ & $e_t$ & $e_{r_s}$ & $e_{r_t}$ & $z_t$ & $z_s$ & \thead{relative lift\\ error} & \thead{x-LSPG relative\\ wall time} &\thead{relative wall\\time}\\
\midrule \midrule
\rowcolor{matcha}1 & ST-GNAT	&12.8&	12.8&20 &16 &	0.999&	0.99& 0.999&	0.999&	9	&1080 &0.201771432&  0.839231102&0.008468197\\ 
 2 & WST-GNAT &	12.8&	6.4& 35 & 10&	0.999&	0.99& 0.999&	0.999&	9&	1080& 0.123280751& 0.842380016 &  0.009793139\\
 \rowcolor{matcha}3 & WST-GNAT&	6.4	&0.8& 245 & 13&	0.999&	0.99& 0.999	&0.999	&5	&1080& 0.038892137&  2.261437990 &  0.016439454\\
 \rowcolor{matcha} 4& WST-GNAT &	1.6&	0.2	& 725 & 145  &0.999&	0.999 &0.999&	0.999&	3&	1080& 0.026915803& 1.717868245 & 0.025160652\\
  5& WST-GNAT &	0.8&	0.1& 774 &277&	0.999&	0.999& 0.999&	0.999	&3&	1080 &0.024986254 &1.006124284 &0.030765692\\\bottomrule
\end{tabular}
}
\end{center}
\end{figure}
%
Next, Figure~\ref{f:NSParetoDrag} shows how the relative drag error vs. relative wall time behaves with varying sub-window lengths for each time window length. It is seen that, as the sub-window length increases, the relative drag error for WST-LSPG and WST-GNAT increases as well. With hyper-reduction, the relative wall time decreases approximately up to two orders of magnitude without affecting the relative drag error. Again, none of the WST-GNAT cases converged for $\ell_w=2\Delta t$ and $\ell_w=\Delta t$, due to the fact that stability guarantee decreases as the window length decreases. Figure~\ref{f:NSParetoDrag_All} combines the data for the relative drag error and shows the Pareto front. Along the Pareto front, it can be seen that the lowest relative wall time are obtained when $\ell_s=128\Delta t$, but that these parameter settings lead to high relative drag errors. Decreasing the sub-window length along the Pareto front leads to WST-GNAT cases where the relative drag error decreases up to three orders of magnitude with less than an order of magnitude in change for the relative wall time. These results show that, to obtain low relative drag error with low computational cost, using the smallest possible sub-window length is ideal for this problem.

Figure~\ref{f:NSParetoDrag_e} shows how the relative drag error vs. relative wall time behaves with varying window lengths for each sub-window length. The data shown in these plots are the same data shown in Figure~\ref{f:NSParetoDrag}. It can be seen that the relative drag error tends to increase with decreasing window lengths. This behavior can be seen more apparently for $\ell_s=64\Delta t$ and $\ell_s=32\Delta t$. Figure~\ref{f:NSParetoDrag_eAll} combines the data for relative drag error and shows the Pareto front, which is the same Pareto front shown in Figure~\ref{f:NSParetoDrag_All}. The WST-GNAT cases that lie on the Pareto front show that decreasing the time window length can lead to less relative drag error and relatively low relative wall times when compared to the WST-LSPG cases. Table~\ref{t:NSDrag} lists the parameters used in these cases that lie on the Pareto front for relative drag error shown in both Figures~\ref{f:NSParetoDrag_All} and~\ref{f:NSParetoDrag_eAll}. Case 1 refers to a ST-GNAT case, which exhibits the lowest relative wall time, but is accompanied by high relative drag error. Cases 2--6 show parameters for WST-GNAT methods that yield optimal relative drag error and relative wall time. Case 6 shows the most optimal data point that gives the lowest optimal relative drag error with less than an order of magnitude increase in the relative wall time. This particular case was executed with $\numWindows=16$ and $\numSubWindows=8$. The highlighted rows again refer to the most optimal cases on the Pareto front that are best for MSE, residual $\ell^2$ norm, relative drag error, and relative lift error. 
%

Figure~\ref{f:NSParetoLift} shows how the relative lift error vs. relative wall time behaves with varying sub-window length. The relative lift error behaves similarly as the other output errors; for a certain window length, the relative lift improves as the sub-window length decreases. Additionally, introducing hyper-reduction to WST-LSPG resulted in increased accuracy of the relative lift. As before, the WST-GNAT cases for $\ell_w=2\Delta t$ and $\ell_w=\Delta t$ did not converge due to lack of enough stability guarantee over the smaller windows. 
Figure~\ref{f:NSParetoLiftALL} shows the combined results for relative lift vs. relative wall time. It can be seen that introducing windowing to ST-LSPG leads to lower relative lift errors and lower relative time. These results additionally confirm that increasing the sub-window length leads to higher relative lift and lower relative wall times.

Figure~\ref{f:NSParetoLift_e} shows how the relative lift error vs. relative wall time behaves with varying window lengths. All data in this figure is the same as those found in Figure~\ref{f:NSParetoLift}. When $\ell_s=\Bmat{128\Delta t, 64\Delta t, 32\Delta t}$, the relative lift error decreases with decreasing window length. The opposite seems to be true for $\ell_s=16\Delta t$ and $8\Delta t$, where the relative lift error decreases with increasing window length. Figure~\ref{f:NSParetoLift_eAll} combines the results and shows the Pareto front, which is the same Pareto front shown in Figure~\ref{f:NSParetoLiftALL}. It is seen that longer window lengths result in higher relative lift along the Pareto front. Overall, including hyper-reduction reduces the wall time by approximately two orders of magnitudes. Table~\ref{t:NSLift} details the parameters for the cases that lie on the Pareto front for the relative lift error shown in both Figure~\ref{f:NSParetoLiftALL} and Figure~\ref{f:NSParetoLift_eAll}. Case 1 refers to the ST-GNAT method and is associated with higher relative lift errors but lower relative wall times. Cases 2--5 refer to the WST-GNAT method, which is associated with lower relative lift errors but higher relative wall times; however, the relative wall times are reduced by less than an order of magnitude. The highlighted rows again refer to the most optimal cases on the Pareto front that is best for MSE, residual $\ell^2$ norm, relative drag error, and relative lift error. 
\subsubsection{Summary of results}
The WST-LSPG method was implemented on the two-dimensional compressible Navier--Stokes equation. To additionally evaluate the behavior of WST-LSPG, the MSE, the residual $\ell^2$ norm, the relative drag error, and the relative lift error were calculated over two different window hyper-parameters: window length, $\ell_w$ and sub-window length, $\ell_s$. Applied to the two-dimensional problem, WST-LSPG was shown to produce ROMs with lower MSE and lower residual $\ell^2$ norms. This is due to higher time-local fidelity bases over windows that were able to better resolve complexities in the flow, especially at the wake of the NACA 0012 airfoil. As seen in the results for the Burgers' equation, the results for the two-dimensional Navier--Stokes problem additionally produced MSE that first decreases in window lengths and then increases in window lengths. The results for the residual $\ell^2$ norm showed different behaviors compared to the results of the MSE, where the residual $\ell^2$ norm decreased consistently as the window length increased over varied values of sub-window length. Note that unlike the Burgers' equation results, inclusion of hyper-reduction in the space--time model reduction process yielded the lowest MSE of $0.017791989$ and yielded approximately the same residual $\ell^2$ norms. The results for the relative drag error and relative lift error showed similar behaviors as the results for the MSE. Finally, it was shown that the addition of hyper-reduction for the windowed model reduction technique yielded orders of magnitude lower relative wall times without affecting the output errors of interest. Overall, applying the WST-LSPG technique on the two-dimensional compressible Navier--Stokes equation produced lower MSEs, residual $\ell^2$ norms, relative drag errors, and relative lift errors, when compared to the ST-LSPG technique. 

\section{Conclusions}\label{sec: conclusions}
This work introduced the WST-LSPG method for the space--time model reduction of parameterized nonlinear dynamical systems. WST-LSPG, which builds upon the space--time least-squares Petrov--Galerkin method, divides the time simulation into $\numWindows$ windows and sequentially minimizes the discrete-in-time residual within its own unique low dimensional space--time subspace in a weighted $\ell^2$-norm. Additionally, this work proposes a temporal domain decomposition technique to construct space--time trial subspaces via tensor decompositions by further dividing the $\numWindows$ into $\numSubWindows$ sub-windows. Advantages of WST-LSPG include:

\begin{itemize}
\item Ability to limit the exponential growth of the computational cost to windows instead of time steps, decreasing the overall cost of finding the reduced space--time states
\item Providing more efficient functionality for implementation of ROMs in high-fidelity codes by reducing the cost associated with calculating the product of the space--time Jacobian and the space--time bases
\item Ability to decrease the output error of interests when introducing windows and sub-windows to the space--time affine trial subspaces
\item Production of ROMs with lower relative wall times and better output errors when using smaller time window lengths and smaller sub-window lengths as shown in the numerical experiments for the Burgers' equation and the compressible Navier--Stokes equations
\item Overall reduction of the computational cost of creating the space--time ROM by using hyper-reduction via the Gauss--Newton with Approximated Tensors (GNAT) method 
\end{itemize} 
Future work for WST-LSPG consists of (1) developing a more intelligent way to determine where a window shall begin and end depending on the physical attributes of the problem in the time domain, (2) creating a general windowed space--time method for nonlinear manifolds using deep convolutions autoencoders, (3) integrating conservation techniques to improve the residual $\ell^2$ norms of more complicated problems such as compressible turbulent flows, and (4) formulating ways to better resolve complex events that happen in flows such as shocks with windowed space--time methods.

\section*{Acknowledgments}
The authors acknowledge Matthew Zahr for providing the \texttt{pymortestbed} code and Krzysztof Fidkowski for providing the \texttt{xflow} program, which were used and modified to help obtain the numerical results. Additionally, the authors acknowledge Kevin Carlberg, Irina Tezaur, Jaideep Ray, Kookjin Lee, and Francesco Rizzi for the productive discussions on the method and Guodong Chen for his advice on developing the two-dimensional geometry case. This work was performed at Sandia National Laboratories. Any subjective views or opinions that might be expressed in the paper do not necessarily represent the views of the U.S. Department of Energy or the United States Government. Sandia National Laboratories is a multimission laboratory managed and operated by National Technology \& Engineering Solutions of Sandia, LLC, a wholly owned subsidiary of Honeywell International Inc., for the U.S. Department of Energy’s National Nuclear Security Administration under contract DE-NA0003525.

\appendix
\section{\emph{Burgers' equation}. Integrated mean squared error}\label{s:imse}
An additional study was conducted to see if the type of problem and/or the fidelity of the basis chosen is responsible for the opposing behaviors of MSE and  residual $\ell^2$ norm at longer time window lengths. The ROM found by performing WST-LSPG on the Burgers' problem exhibits phase error over a longer period of time, which could possibly contribute to the MSE increasing with increased time window length. In order to remove any phase error that may exist and that may contribute to the MSE, the integrated MSE was calculated and is shown in Figures~\ref{f:burgersMSE0Integrated},~\ref{f:burgersMSE2Integrated},~\ref{f:burgersMSE4Integrated},~\ref{f:burgersMSE6Integrated}. These figures show that the overall integrated MSE is an order of magnitude less than the MSE, showing that for the Burgers' equation, phase error is a non inconsequential error in this particular setup. The IMSE exhibits a minimum like seen in Figures~\ref{f:burgersMSE0},~\ref{f:burgersMSE2},~\ref{f:burgersMSE4},~\ref{f:burgersMSE6} and increases with larger time window length. This behavior again is different than the way the residual $\ell^2$ norm behaves with increasing time window length. 
\begin{figure}[h!]
\subfloat[mean squared error: $e_s = 0.99, e_t = 0.99$][\centering \hspace{1.5mm}integrated MSE:\par\hspace{5.5mm}$e_s = 0.99, e_t = 0.99$\label{f:burgersMSE0Integrated}]{\includegraphics[height=0.264\textwidth]{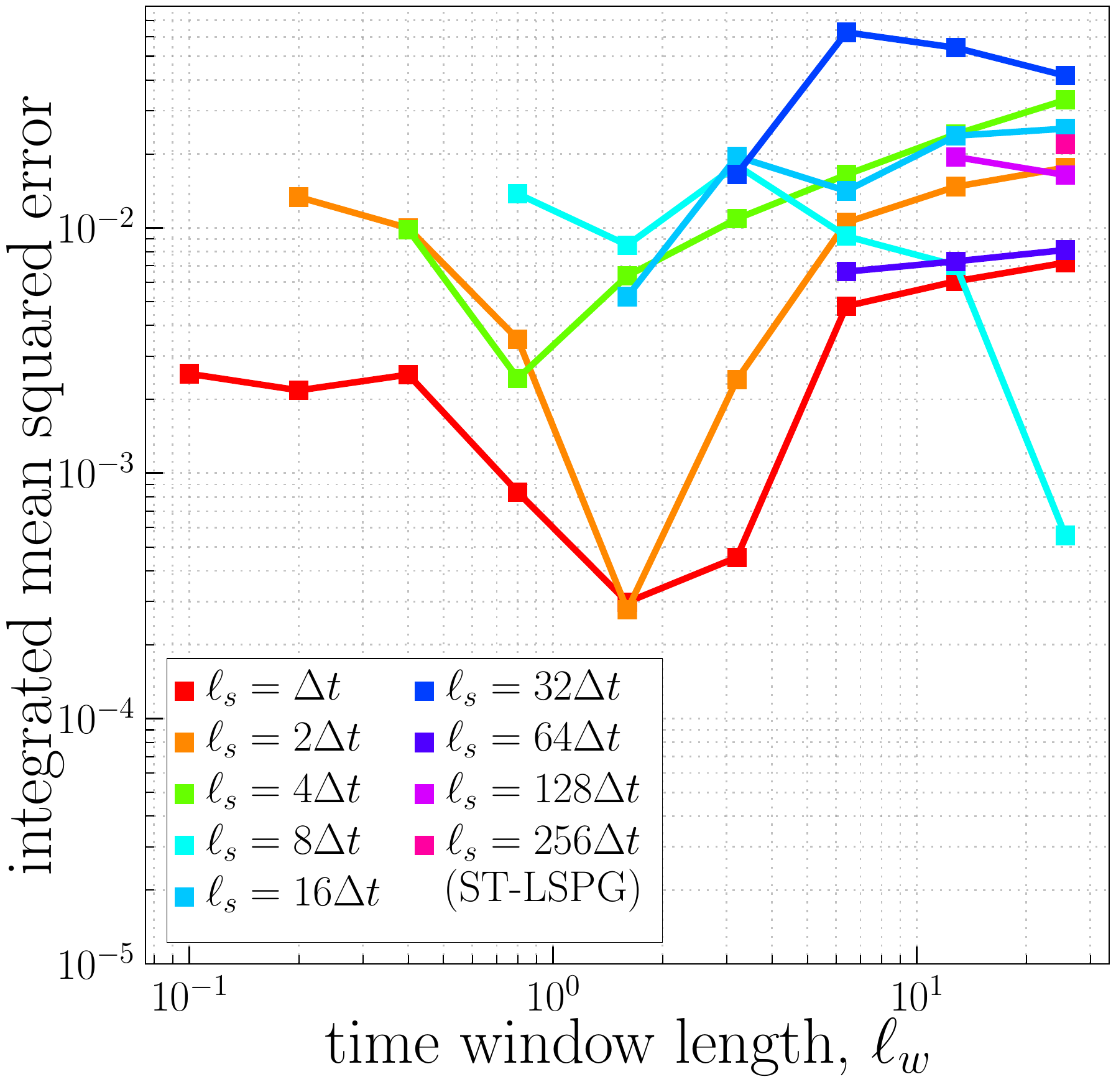}}\hspace{1mm}\nolinebreak
\subfloat[mean squared error: $e_s = 0.99, e_t = 0.999$][\centering \hspace{1.5mm}integrated MSE:\par\hspace{5.5mm}$e_s = 0.99, e_t = 0.999$\label{f:burgersMSE2Integrated}]{\includegraphics[height=0.264\textwidth]{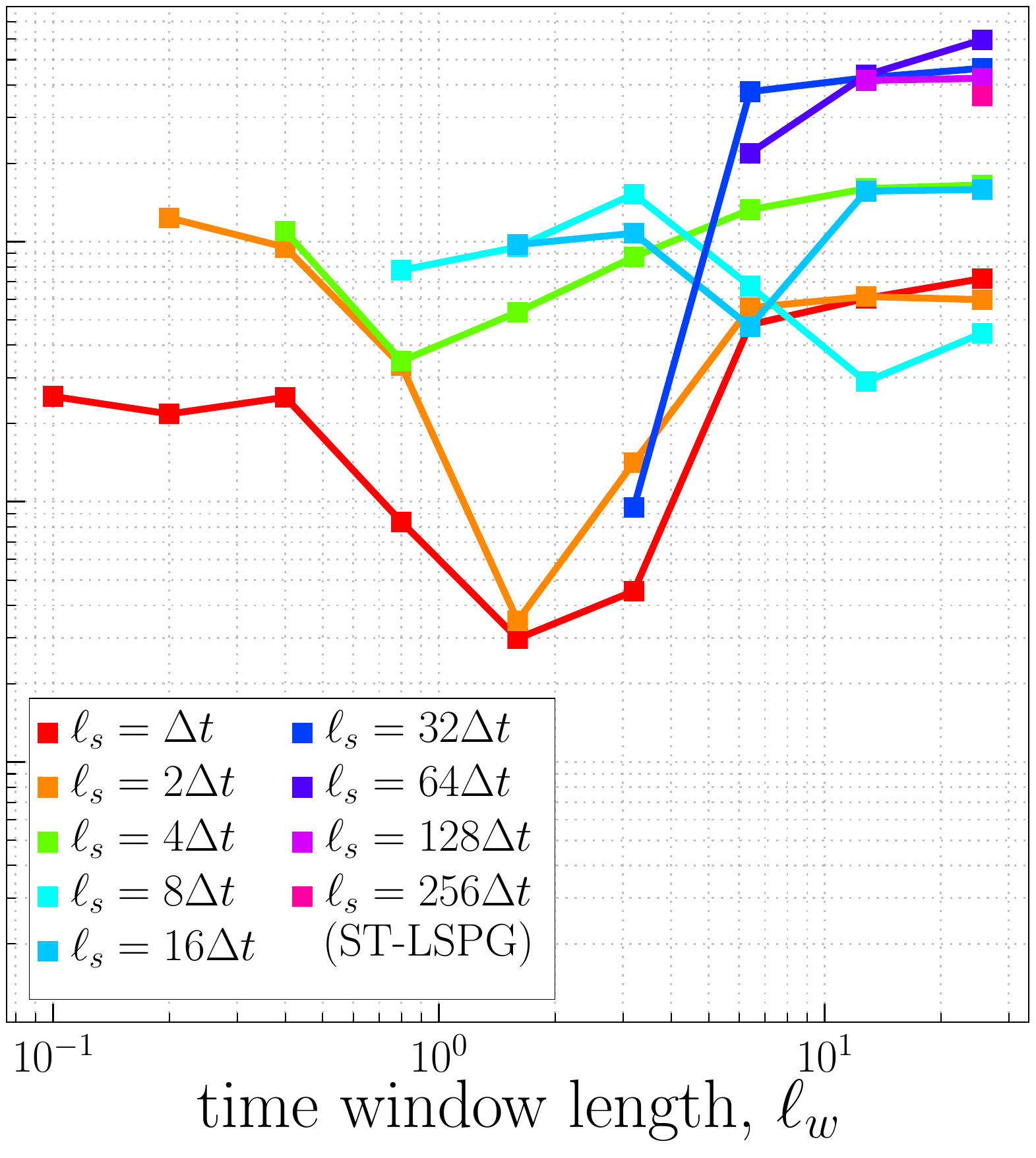}}\hspace{1mm}\nolinebreak
\subfloat[mean squared error: $e_s = 0.999, e_t = 0.99$][\centering \hspace{1.5mm}integrated MSE:\par\hspace{5.5mm}$e_s = 0.999, e_t = 0.99$\label{f:burgersMSE4Integrated}]{\includegraphics[height=0.264\textwidth]{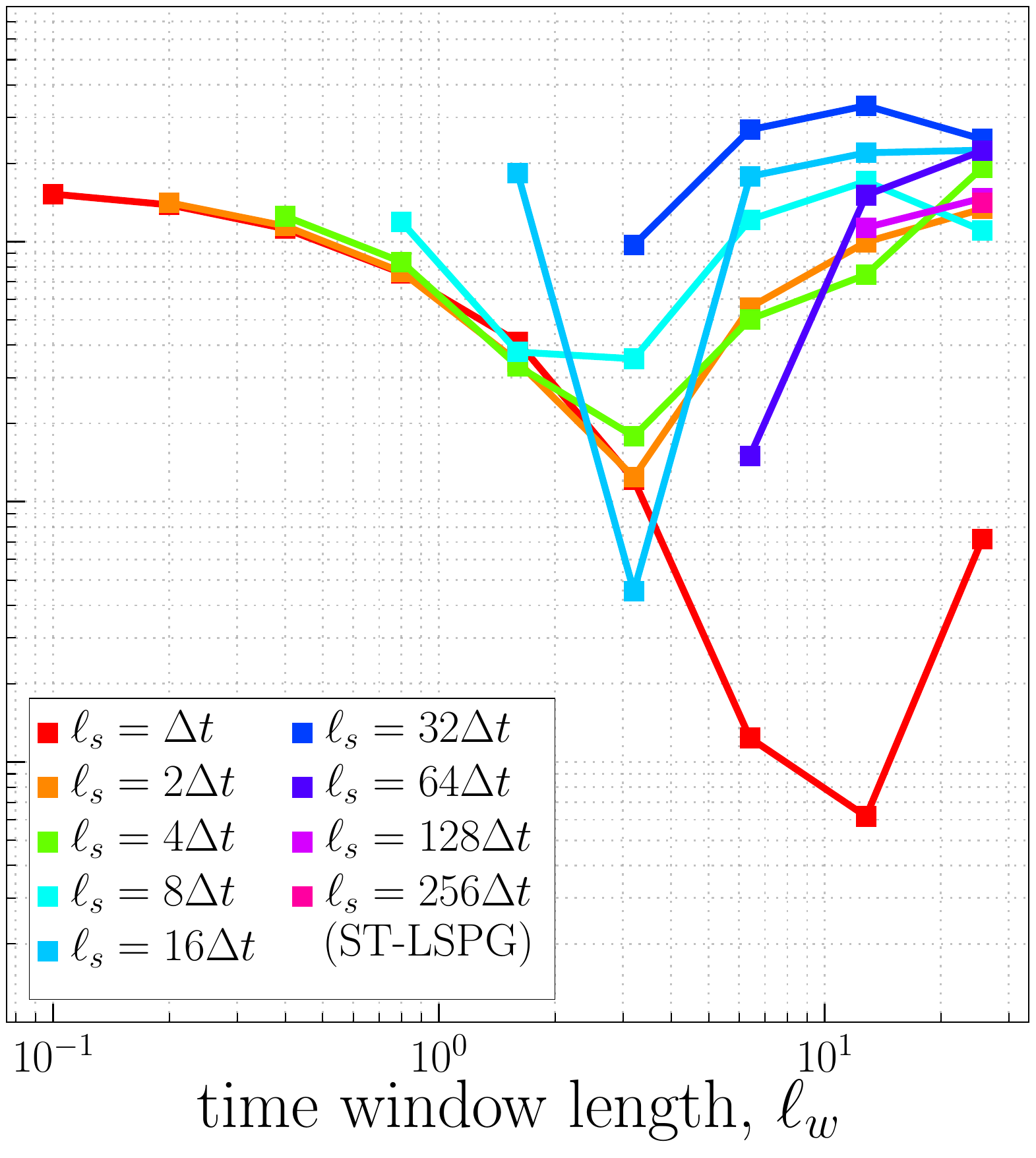}}\hspace{1mm}\nolinebreak
\subfloat[mean squared error: $e_s = 0.999, e_t = 0.999$][\centering \hspace{1.5mm}integrated MSE:\par\hspace{5.5mm}$e_s = 0.999, e_t = 0.999$\label{f:burgersMSE6Integrated}]{\includegraphics[height=0.264\textwidth]{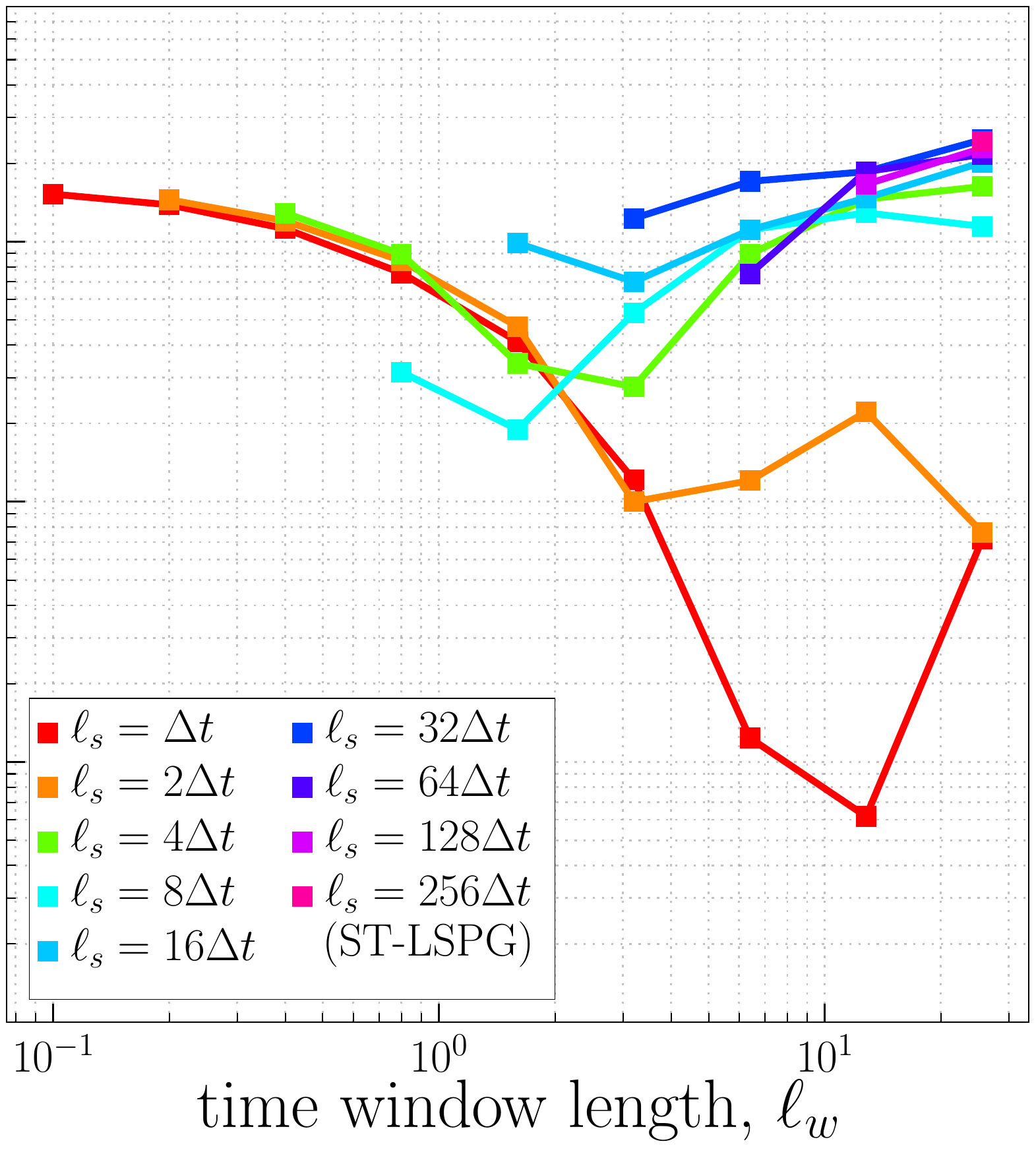}}\vspace{-2mm}
\caption{\emph{Burgers' equation.} WST-LSPG integrated MSE for basis set $e_s = 0.99, e_t = 0.99$, $e_s = 0.99, e_t = 0.999$, $e_s = 0.999, e_t = 0.99$, and $e_s = 0.999, e_t = 0.999$. The results are presented as a function of the window length, $\ell_w$, for various sub-window lengths, $\ell_s$.}
\label{f:windowA}
\end{figure}
\section{\emph{Compressible Navier--Stokes equations}. Residual $\ell^2$ norms}\label{s:nsresidualss}
Next, Figure~\ref{f:NSParetoResidual} shows how the residual $\ell^2$ norm vs. relative wall time behaves with varying sub-window lengths for each time window length. Each different Pareto plots refers to results for different time window lengths $\ell_w$. Collectively, these Pareto plots show that the residual $\ell^2$ norm increases with increasing sub-window length, which is contrary to the behavior of the MSE. With hyper-reduction, the relative wall time for the different sub-window lengths decreased by one to two orders of magnitude. Again for $\ell_w=2\Delta t$ and $\ell_w=\Delta t$, no set of WST-GNAT parameters converged due to memory limitations. Figure~\ref{f:NSParetoResidualAll} combines the results for residual $\ell^2$ norms and shows the Pareto front; this plot shows that with hyper-reduction, WST-GNAT was able to decrease the relative wall time by at most two orders of magnitude. Additionally, Figure~\ref{f:NSParetoResidualAll} shows that the residual $\ell^2$ norm decreases with decreasing sub-window lengths; this is consistent with the results for the one-dimensional Burger's equation. 

Figure~\ref{f:NSParetoResidual_e} shows how the residual $\ell^2$ norm vs. relative wall time behaves with varying window lengths for each sub-window length. These Pareto plots show the same exact data points as shown in Figure~\ref{f:NSParetoResidual}, but highlight how the residual $\ell^2$ norm changes with varying window length for each sub-window length. It can be seen that decreasing the time window length for each sub-window length leads to higher residual $\ell^2$ norms. Figure~\ref{f:NSParetoResidualeAll} combines the results for the residual $\ell^2$ norms and shows the Pareto front, which is the same as the one shown in Figure~\ref{f:NSParetoResidualAll}. Overall, it can be seen that larger time windows lead to lower relative wall times for WST-GNAT, but that setting the time window length such that $\ell_w\ge8\Delta t$ leads to lower residual $\ell^2$ norms, but slightly longer relative wall times. Table~\ref{t:NSresiduals} shows the parameters of the cases that lie on the Pareto front found in both Figure~\ref{f:NSParetoResidualAll} and Figure~\ref{f:NSParetoResidualeAll}. Cases 1--3 and 5 refer to ST-GNAT and exhibit lower relative wall times. Cases 4 and 6--7 refer to WST-GNAT and exhibit lower residual $\ell^2$ norms. The highlighted rows again refer to cases that give optimal results for MSE, residual $\ell^2$ norms, relative drag error, and relative lift error.
\begin{figure}[!t]
\centering
\subfloat[$\numWindows=1,\ell_{w}=128\Delta t$\label{f:NSParetoResidual1}]{\pareto{NS_Pareto_Residual_7.pdf}}\hspace{1mm}\nolinebreak
\subfloat[$\numWindows=2,\ell_{w}=64\Delta t$\label{f:NSParetoResidual2}]{\pareto{NS_Pareto_Residual_6.pdf}}\hspace{1mm}\nolinebreak
\subfloat[$\numWindows=4,\ell_{w}=32\Delta t$\label{f:NSParetoResidual3}]{\pareto{NS_Pareto_Residual_5.pdf}}\hspace{1mm}\nolinebreak
\subfloat[$\numWindows=8,\ell_{w}=16\Delta t$\label{f:NSParetoResidual4}]{\pareto{NS_Pareto_Residual_4.pdf}}\hspace{1mm}\nolinebreak
\subfloat[$\numWindows=16,\ell_{w}=8\Delta t$\label{f:NSParetoResidual5}]{\pareto{NS_Pareto_Residual_3.pdf}}\\
\subfloat[$\numWindows=32,\ell_{w}=4\Delta t$\label{f:NSParetoResidual6}]{\pareto{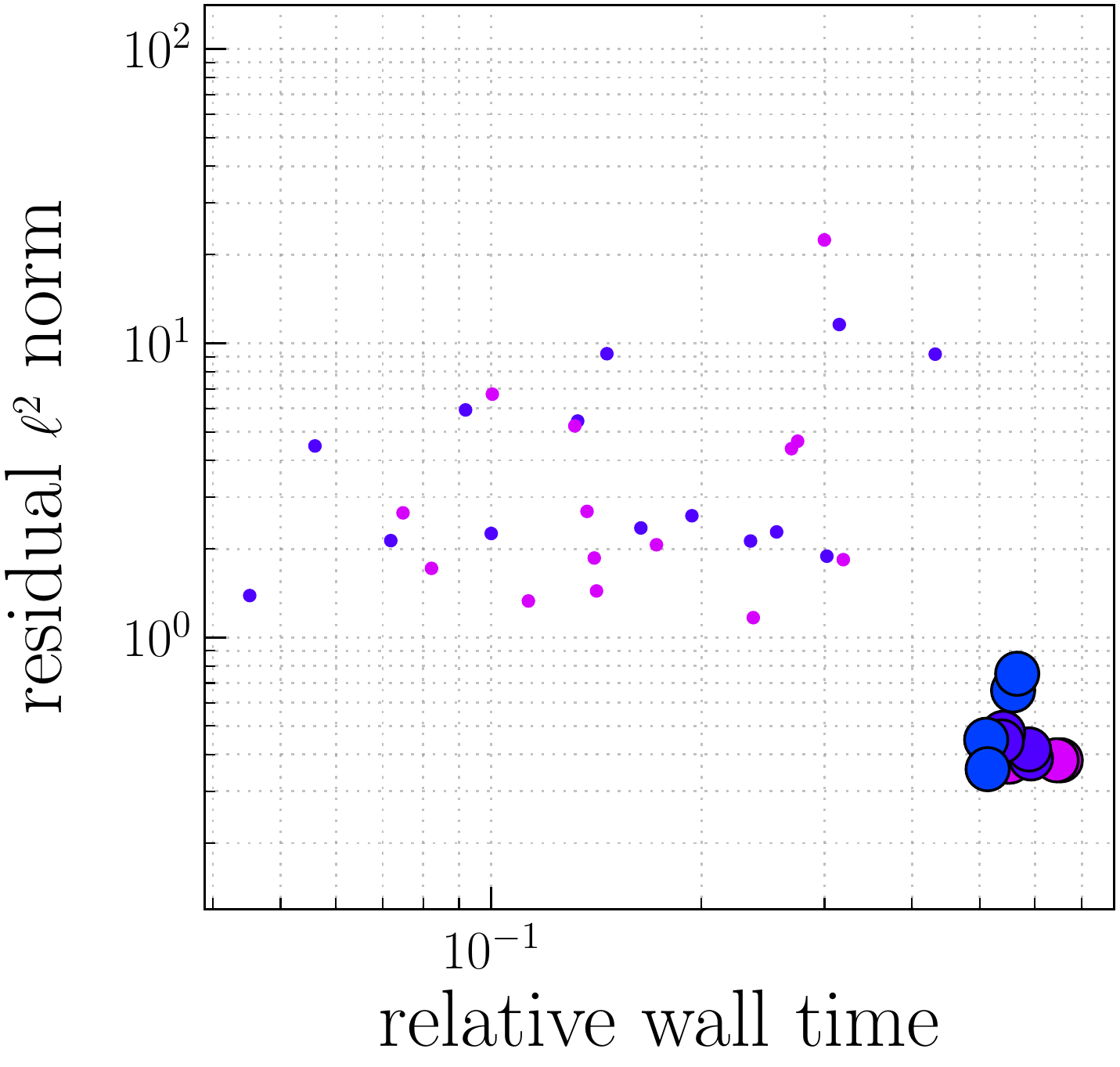}}\hspace{1mm}\nolinebreak
\subfloat[$\numWindows=64,\ell_{w}=2\Delta t$\label{f:NSParetoResidual7}]{\pareto{NS_Pareto_Residual_1.pdf}}\hspace{1mm}\nolinebreak
\subfloat[$\numWindows=128,\ell_{w}=\Delta t$\label{f:NSParetoResidual8}]{\pareto{NS_Pareto_Residual_0.pdf}}\hspace{1mm}\nolinebreak
\legend{ns_legend.pdf}\\
\subfloat[Combined Pareto plot with Pareto front for residual $\ell^2$ norm\label{f:NSParetoResidualAll}]{\includegraphics[width=1\textwidth]{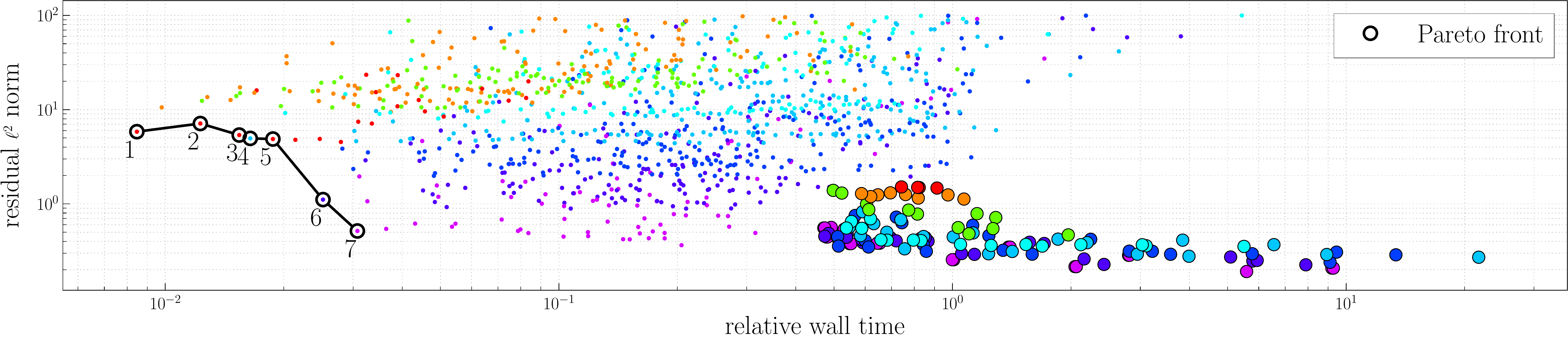}}
\caption{\emph{Compressible Navier--Stokes equations.} Pareto plots for residual $\ell^2$ norm vs. relative wall time for varying sub-window lengths. These Pareto plots highlight how the sub-window length, $\ell_s$, affect the residual $\ell^2$ norm for WST-LSPG and WST-GNAT. Each hyper-reduction case refers to a unique set of spatial sample nodes and temporal sample nodes. The last Pareto plot combines all the results for the residual $\ell^2$ norm and shows the Pareto front.}
\label{f:NSParetoResidual}
\end{figure}
\begin{figure}[!p]
\centering
\subfloat[$\ell_{s}=128\Delta t$\label{f:NSParetoResiduale1}]{\pareto{NS_Pareto_Residual_7_e.pdf}}\hspace{1mm}\nolinebreak
\subfloat[$\ell_{s}=64\Delta t$\label{f:NSParetoResiduale2}]{\pareto{NS_Pareto_Residual_6_e.pdf}}\hspace{1mm}\nolinebreak
\subfloat[$\ell_{s}=32\Delta t$\label{f:NSParetoResiduale3}]{\pareto{NS_Pareto_Residual_5_e.pdf}}\hspace{1mm}\nolinebreak
\subfloat[$\ell_{s}=16\Delta t$\label{f:NSParetoResiduale4}]{\pareto{NS_Pareto_Residual_4_e.pdf}}\hspace{1mm}\nolinebreak
\subfloat[$\ell_{s}=8\Delta t$\label{f:NSParetoResiduale5}]{\pareto{NS_Pareto_Residual_3_e.pdf}}\\
\subfloat[$\ell_{s}=4\Delta t$\label{f:NSParetoResiduale6}]{\pareto{NS_Pareto_Residual_2_e.pdf}}\hspace{1mm}\nolinebreak
\subfloat[$\ell_{s}=2\Delta t$\label{f:NSParetoResiduale7}]{\pareto{NS_Pareto_Residual_1_e.pdf}}\hspace{1mm}\nolinebreak
\subfloat[$\ell_{s}=\Delta t$\label{f:NSParetoResiduale8}]{\pareto{NS_Pareto_Residual_0_e.pdf}}\hspace{1mm}\nolinebreak
\legend{ns_legend2.pdf}\\
\subfloat[Combined Pareto plot with Pareto front for residual $\ell^2$ norm\label{f:NSParetoResidualeAll}]{\includegraphics[width=1\textwidth]{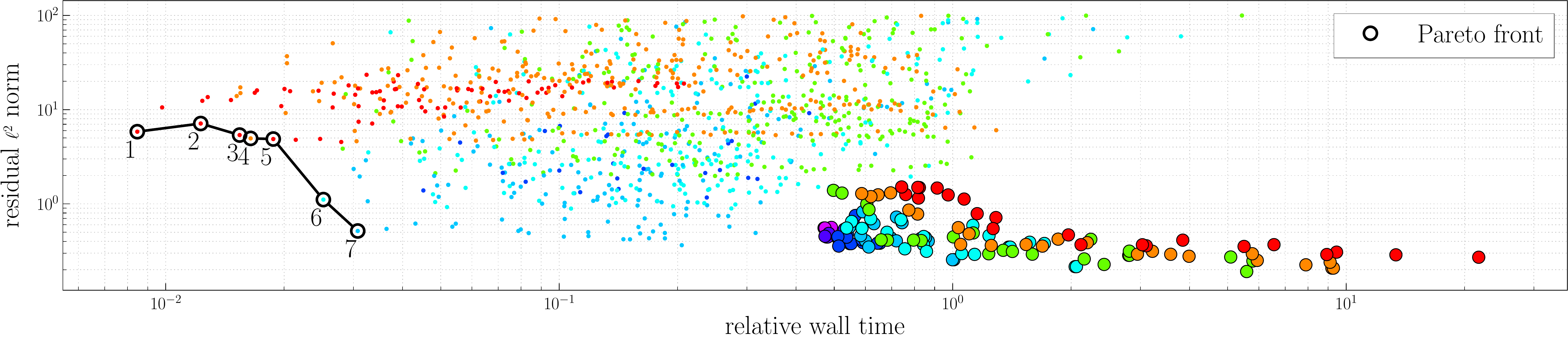}}
\caption{\emph{Compressible Navier--Stokes equations.} Pareto plots for residual $\ell^2$ norm vs. relative wall time for varying window lengths. These Pareto plots highlight how the window length, $\ell_w$, affect the residual $\ell^2$ norm for WST-LSPG and WST-GNAT. The last Pareto plot combines all the results for the residual $\ell^2$ norm and shows the Pareto front.}
\label{f:NSParetoResidual_e}
\begin{center}
\vspace{-3mm}
\captionof{table}{\label{t:NSresiduals}\emph{Compressible Navier--Stokes equations}. Method and parameter values for labeled data points in Figures~\ref{f:NSParetoResidual} and~\ref{f:NSParetoResidual_e} that yield Pareto-optimal performance for for the residual $\ell^2$ norm.}
\resizebox{\columnwidth}{!}{%
\begin{tabular}{ ccccccccccccccc } \toprule
case& method & $\ell_w$ & $\ell_s$ & \thead{total\\$n_{st}$} &\thead{total\\$n^r_{st}$}& $e_s$ & $e_t$ & $e_{r_s}$ & $e_{r_t}$ & $z_t$ & $z_s$ & \thead{residual\\$\ell^2$ norm} & \thead{x-LSPG relative\\ wall time} &\thead{relative wall\\time}\\
\midrule \midrule
\rowcolor{matcha}1 & ST-GNAT	&12.8&	12.8&20 &16 &	0.999&	0.99& 0.999&	0.999&	9	&1080 & 5.832617684 &  0.839231102&0.008468197\\ 
 2 & ST-GNAT &	12.8&	12.8&20 &16&	0.999&	0.99& 0.999&	0.999&	9&	2161& 7.122158120& 0.842223198 &  0.012274356\\
3& ST-GNAT & 12.8&	12.8& 20 & 16&	0.999&	0.99&  0.999&	0.999&	9&	3242 & 5.372841392&  0.844611717& 0.015412119\\
 \rowcolor{matcha}4 & WST-GNAT& 6.4	&0.8& 245 &13& 	0.999&	0.99& 0.999	&0.999	&5	&1080& 4.951202075 &  2.261437990&  0.016439454\\
  5 & ST-GNAT &12.8&	12.8& 20 & 16 &	0.999&	0.99& 0.999&	0.999&	9&	4323&  4.876390151& 0.851598851&  0.018766640\\
 \rowcolor{matcha} 6& WST-GNAT &	1.6&	0.2	&725& 145&0.999&	0.999 &0.999&	0.999&	3&	1080& 1.109987192 & 1.717868245 & 0.025160652\\
  7& WST-GNAT	&0.8&	0.1& 774 & 277&	0.999&	0.999& 0.999&	0.999	&3&	1080 &0.515574079  &1.006124284 &0.030765692\\\bottomrule
\end{tabular}
}
\end{center}
\end{figure}







\clearpage
\bibliographystyle{elsarticle-num-names}
\bibliography{windowedstlspg.bib}
\end{document}

%% file: packages.tex
\usepackage{amsmath}
\usepackage{enumitem}
\usepackage{soul}
\usepackage[export]{adjustbox}
\usepackage[table,x11names]{xcolor}
\usepackage{nomencl}
\usepackage[margin=2.25cm]{geometry}
\usepackage{comment}
\usepackage{graphicx}
\usepackage{amssymb}
\usepackage{amsmath,graphicx}
\usepackage{algorithm}
\usepackage{booktabs}
\usepackage{siunitx}
\usepackage{algpseudocode}
\usepackage{tabstackengine}
\usepackage{caption}
\usepackage{lineno}
\usepackage{natbib}
\setcitestyle{numbers,sort&compress}
\usepackage{bm}
\usepackage{    mathrsfs}
\usepackage[
 font=scriptsize 
 ]{subfig}
\usepackage{mathtools}
\usepackage{tikz}
\usepackage{relsize}
\usepackage[utf8]{inputenc}
\usepackage[english]{babel}
\usepackage{amsthm}
\newtheorem{theorem}{Theorem}[section]

\newtheorem{corollary}{Corollary}[theorem]
\newtheorem{lemma}[theorem]{Lemma}
\usepackage{scalerel,stackengine}

\usepackage[utf8]{inputenc}
\usepackage[english]{babel}

\usepackage{hyperref}
\hypersetup{
    colorlinks=true,
    linkcolor=blue,
    filecolor=magenta,      
    urlcolor=cyan,
}
\usepackage{tabularx}

\usepackage{makecell}

%% file: wstlspg_main_reorg.bbl
\begin{thebibliography}{46}
\expandafter\ifx\csname natexlab\endcsname\relax\def\natexlab#1{#1}\fi
\providecommand{\url}[1]{\texttt{#1}}
\providecommand{\href}[2]{#2}
\providecommand{\path}[1]{#1}
\providecommand{\DOIprefix}{doi:}
\providecommand{\ArXivprefix}{arXiv:}
\providecommand{\URLprefix}{URL: }
\providecommand{\Pubmedprefix}{pmid:}
\providecommand{\doi}[1]{\href{http://dx.doi.org/#1}{\path{#1}}}
\providecommand{\Pubmed}[1]{\href{pmid:#1}{\path{#1}}}
\providecommand{\bibinfo}[2]{#2}
\ifx\xfnm\relax \def\xfnm[#1]{\unskip,\space#1}\fi
\bibitem[{Benner et~al.(2015)Benner, Gugercin, and Willcox}]{wilcox_benner_rev}
\bibinfo{author}{P.~Benner}, \bibinfo{author}{S.~Gugercin},
  \bibinfo{author}{K.~Willcox},
\newblock \bibinfo{title}{{A survey of projection-based model reduction methods
  for parametric dynamical systems}},
\newblock \bibinfo{journal}{SIAM Review} \bibinfo{volume}{57}
  (\bibinfo{year}{2015}) \bibinfo{pages}{483--531}. \URLprefix
  \url{https://doi.org/10.1137/130932715}. \DOIprefix\doi{10.1137/130932715}.
  \href{http://arxiv.org/abs/https://doi.org/10.1137/130932715}{{\tt
  arXiv:https://doi.org/10.1137/130932715}}.
\bibitem[{{Moore}(1981)}]{moore}
\bibinfo{author}{B.~{Moore}},
\newblock \bibinfo{title}{Principal component analysis in linear systems:
  Controllability, observability, and model reduction},
\newblock \bibinfo{journal}{IEEE Transactions on Automatic Control}
  \bibinfo{volume}{26} (\bibinfo{year}{1981}) \bibinfo{pages}{17--32}.
  \DOIprefix\doi{10.1109/TAC.1981.1102568}.
\bibitem[{Mullis and Roberts(1976)}]{roberts}
\bibinfo{author}{C.~T. Mullis}, \bibinfo{author}{R.~A. Roberts},
\newblock \bibinfo{title}{Synthesis of minimum roundoff noise fixed point
  digital filters},
\newblock \bibinfo{journal}{IEEE Transactions on Circuits and Systems}
  \bibinfo{volume}{23} (\bibinfo{year}{1976}) \bibinfo{pages}{551--562}.
\bibitem[{Gugercin et~al.(2008)Gugercin, Antoulas, and Beattie}]{GugercinIRKA}
\bibinfo{author}{S.~Gugercin}, \bibinfo{author}{A.~Antoulas},
  \bibinfo{author}{C.~Beattie},
\newblock \bibinfo{title}{{$\mathcal{H}_2$} model reduction for large-scale
  linear dynamical systems},
\newblock \bibinfo{journal}{SIAM Journal on Matrix Analysis and Applications}
  \bibinfo{volume}{30} (\bibinfo{year}{2008}) \bibinfo{pages}{609--638}.
  \URLprefix \url{https://doi.org/10.1137/060666123}.
  \DOIprefix\doi{10.1137/060666123}.
  \href{http://arxiv.org/abs/https://doi.org/10.1137/060666123}{{\tt
  arXiv:https://doi.org/10.1137/060666123}}.
\bibitem[{Rovas(2003)}]{rovas_thesis}
\bibinfo{author}{D.~V. Rovas}, \bibinfo{title}{{Reduced-basis output bound
  methods for parametrized partial differential equations}}, Ph.D. thesis,
  Massachusetts Institute of Technology, \bibinfo{year}{2003}.
\bibitem[{Qian et~al.(2020)Qian, Kramer, Peherstorfer, and
  Willcox}]{QIAN2020132401}
\bibinfo{author}{E.~Qian}, \bibinfo{author}{B.~Kramer},
  \bibinfo{author}{B.~Peherstorfer}, \bibinfo{author}{K.~Willcox},
\newblock \bibinfo{title}{Lift \& learn: Physics-informed machine learning for
  large-scale nonlinear dynamical systems},
\newblock \bibinfo{journal}{Physica D: Nonlinear Phenomena}
  \bibinfo{volume}{406} (\bibinfo{year}{2020}) \bibinfo{pages}{132401}.
  \URLprefix
  \url{http://www.sciencedirect.com/science/article/pii/S0167278919307651}.
  \DOIprefix\doi{https://doi.org/10.1016/j.physd.2020.132401}.
\bibitem[{Kramer and Willcox(2019)}]{doi:10.2514/1.J057791}
\bibinfo{author}{B.~Kramer}, \bibinfo{author}{K.~E. Willcox},
\newblock \bibinfo{title}{Nonlinear model order reduction via lifting
  transformations and proper orthogonal decomposition},
\newblock \bibinfo{journal}{AIAA Journal} \bibinfo{volume}{57}
  (\bibinfo{year}{2019}) \bibinfo{pages}{2297--2307}. \URLprefix
  \url{https://doi.org/10.2514/1.J057791}. \DOIprefix\doi{10.2514/1.J057791}.
  \href{http://arxiv.org/abs/https://doi.org/10.2514/1.J057791}{{\tt
  arXiv:https://doi.org/10.2514/1.J057791}}.
\bibitem[{Krath et~al.(2020)Krath, Carpenter, Cizmas, and
  Johnston}]{KRATH2020109959}
\bibinfo{author}{E.~H. Krath}, \bibinfo{author}{F.~L. Carpenter},
  \bibinfo{author}{P.~G. Cizmas}, \bibinfo{author}{D.~A. Johnston},
\newblock \bibinfo{title}{An efficient proper orthogonal decomposition based
  reduced-order model for compressible flows},
\newblock \bibinfo{journal}{Journal of Computational Physics}
  (\bibinfo{year}{2020}) \bibinfo{pages}{109959}. \URLprefix
  \url{http://www.sciencedirect.com/science/article/pii/S0021999120307336}.
  \DOIprefix\doi{https://doi.org/10.1016/j.jcp.2020.109959}.
\bibitem[{Everson and Sirovich(1995)}]{everson_sirovich_gappy}
\bibinfo{author}{R.~Everson}, \bibinfo{author}{L.~Sirovich},
\newblock \bibinfo{title}{{{Karhunen-Lo\`eve} procedure for gappy data}},
\newblock \bibinfo{journal}{Journal of the Optical Society of America A}
  \bibinfo{volume}{12} (\bibinfo{year}{1995}) \bibinfo{pages}{1657--1644}.
\bibitem[{Barrault et~al.(2004)Barrault, Maday, Nguyen, and Patera}]{eim}
\bibinfo{author}{M.~Barrault}, \bibinfo{author}{Y.~Maday},
  \bibinfo{author}{N.~C. Nguyen}, \bibinfo{author}{A.~T. Patera},
\newblock \bibinfo{title}{An `empirical interpolation' method: application to
  efficient reduced-basis discretization of partial differential equations},
\newblock \bibinfo{journal}{C. R. Acad. Sci. Paris} \bibinfo{volume}{339}
  (\bibinfo{year}{2004}) \bibinfo{pages}{667--672}.
\bibitem[{Drmac and Gugercin(2016)}]{qdeim_drmac}
\bibinfo{author}{Z.~Drmac}, \bibinfo{author}{S.~Gugercin},
\newblock \bibinfo{title}{A new selection operator for the discrete empirical
  interpolation method---improved a priori error bound and extensions},
\newblock \bibinfo{journal}{J. Sci. Comput.} \bibinfo{volume}{38}
  (\bibinfo{year}{2016}) \bibinfo{pages}{A631--A648}.
\bibitem[{Bui-Thanh(2007)}]{bui2007model}
\bibinfo{author}{T.~Bui-Thanh}, \bibinfo{title}{Model-constrained optimization
  methods for reduction of parameterized large-scale systems}, Ph.D. thesis,
  Massachusetts Institute of Technology, \bibinfo{year}{2007}.
\bibitem[{Bui-Thanh et~al.(2008{\natexlab{a}})Bui-Thanh, Willcox, and
  Ghattas}]{bui2008model}
\bibinfo{author}{T.~Bui-Thanh}, \bibinfo{author}{K.~Willcox},
  \bibinfo{author}{O.~Ghattas},
\newblock \bibinfo{title}{Model reduction for large-scale systems with
  high-dimensional parametric input space},
\newblock \bibinfo{journal}{SIAM Journal on Scientific Computing}
  \bibinfo{volume}{30} (\bibinfo{year}{2008}{\natexlab{a}})
  \bibinfo{pages}{3270--3288}.
\bibitem[{Bui-Thanh et~al.(2008{\natexlab{b}})Bui-Thanh, Willcox, and
  Ghattas}]{bui2008parametric}
\bibinfo{author}{T.~Bui-Thanh}, \bibinfo{author}{K.~Willcox},
  \bibinfo{author}{O.~Ghattas},
\newblock \bibinfo{title}{Parametric reduced-order models for probabilistic
  analysis of unsteady aerodynamic applications},
\newblock \bibinfo{journal}{AIAA journal} \bibinfo{volume}{46}
  (\bibinfo{year}{2008}{\natexlab{b}}) \bibinfo{pages}{2520--2529}.
\bibitem[{Carlberg et~al.(2011)Carlberg, Bou-Mosleh, and
  Farhat}]{carlberg2011efficient}
\bibinfo{author}{K.~Carlberg}, \bibinfo{author}{C.~Bou-Mosleh},
  \bibinfo{author}{C.~Farhat},
\newblock \bibinfo{title}{Efficient non-linear model reduction via a
  least-squares petrov--galerkin projection and compressive tensor
  approximations},
\newblock \bibinfo{journal}{International Journal for Numerical Methods in
  Engineering} \bibinfo{volume}{86} (\bibinfo{year}{2011})
  \bibinfo{pages}{155--181}.
\bibitem[{Carlberg(2011)}]{carlberg2011model}
\bibinfo{author}{K.~Carlberg}, \bibinfo{title}{Model reduction of nonlinear
  mechanical systems via optimal projection and tensor approximation}, Ph.D.
  thesis, PhD thesis, Stanford University, \bibinfo{year}{2011}.
\bibitem[{Carlberg et~al.(2013)Carlberg, Farhat, Cortial, and
  Amsallem}]{carlberg2013gnat}
\bibinfo{author}{K.~Carlberg}, \bibinfo{author}{C.~Farhat},
  \bibinfo{author}{J.~Cortial}, \bibinfo{author}{D.~Amsallem},
\newblock \bibinfo{title}{The gnat method for nonlinear model reduction:
  effective implementation and application to computational fluid dynamics and
  turbulent flows},
\newblock \bibinfo{journal}{Journal of Computational Physics}
  \bibinfo{volume}{242} (\bibinfo{year}{2013}) \bibinfo{pages}{623--647}.
\bibitem[{Carlberg et~al.(2017)Carlberg, Barone, and
  Antil}]{carlberg2017galerkin}
\bibinfo{author}{K.~Carlberg}, \bibinfo{author}{M.~Barone},
  \bibinfo{author}{H.~Antil},
\newblock \bibinfo{title}{Galerkin v. least-squares petrov--galerkin projection
  in nonlinear model reduction},
\newblock \bibinfo{journal}{Journal of Computational Physics}
  \bibinfo{volume}{330} (\bibinfo{year}{2017}) \bibinfo{pages}{693--734}.
\bibitem[{LeGresley and Alonso(2000)}]{legresley2000airfoil}
\bibinfo{author}{P.~LeGresley}, \bibinfo{author}{J.~Alonso},
\newblock \bibinfo{title}{Airfoil design optimization using reduced order
  models based on proper orthogonal decomposition},
\newblock in: \bibinfo{booktitle}{Fluids 2000 conference and exhibit},
  \bibinfo{year}{2000}, p. \bibinfo{pages}{2545}.
\bibitem[{Abgrall and Crisovan(2018)}]{abgrall2018model}
\bibinfo{author}{R.~Abgrall}, \bibinfo{author}{R.~Crisovan},
\newblock \bibinfo{title}{Model reduction using l 1-norm minimization as an
  application to nonlinear hyperbolic problems},
\newblock \bibinfo{journal}{International Journal for Numerical Methods in
  Fluids} \bibinfo{volume}{87} (\bibinfo{year}{2018})
  \bibinfo{pages}{628--651}.
\bibitem[{Parish and Carlberg(2019)}]{parish2019windowed}
\bibinfo{author}{E.~J. Parish}, \bibinfo{author}{K.~T. Carlberg},
\newblock \bibinfo{title}{Windowed least-squares model reduction for dynamical
  systems},
\newblock \bibinfo{journal}{arXiv preprint arXiv:1910.11388}
  (\bibinfo{year}{2019}).
\bibitem[{Collins et~al.(2019)Collins, Fidkowski, and
  Cesnik}]{collins2019output}
\bibinfo{author}{G.~Collins}, \bibinfo{author}{K.~Fidkowski},
  \bibinfo{author}{C.~E. Cesnik},
\newblock \bibinfo{title}{Output error estimation for projection-based reduced
  models},
\newblock in: \bibinfo{booktitle}{AIAA Aviation 2019 Forum},
  \bibinfo{year}{2019}, p. \bibinfo{pages}{3528}.
\bibitem[{Collins et~al.(2020)Collins, Fidkowski, and
  Cesnik}]{collins2020petrov}
\bibinfo{author}{G.~Collins}, \bibinfo{author}{K.~Fidkowski},
  \bibinfo{author}{C.~E. Cesnik},
\newblock \bibinfo{title}{Petrov-galerkin projection-based model reduction with
  an optimized test space},
\newblock in: \bibinfo{booktitle}{AIAA Scitech 2020 Forum},
  \bibinfo{year}{2020}, p. \bibinfo{pages}{1562}.
\bibitem[{Carlberg(2011)}]{carlberg_thesis}
\bibinfo{author}{K.~Carlberg}, \bibinfo{title}{Model reduction of nonlinear
  mechanical systems via optimal projection and tensor approximation}, Ph.D.
  thesis, Stanford University, \bibinfo{year}{2011}.
\bibitem[{Bui-Thanh(2007)}]{bui_thesis}
\bibinfo{author}{T.~Bui-Thanh}, \bibinfo{title}{{Model-constrained optimization
  methods for reduction of parameterized large-scale systems}}, Ph.D. thesis,
  Massachusetts Institute of Technology, \bibinfo{year}{2007}.
\bibitem[{Urban and Patera(2012)}]{urban2012new}
\bibinfo{author}{K.~Urban}, \bibinfo{author}{A.~T. Patera},
\newblock \bibinfo{title}{A new error bound for reduced basis approximation of
  parabolic partial differential equations},
\newblock \bibinfo{journal}{Comptes Rendus Mathematique} \bibinfo{volume}{350}
  (\bibinfo{year}{2012}) \bibinfo{pages}{203--207}.
\bibitem[{Urban and Patera(2014)}]{urban2014improved}
\bibinfo{author}{K.~Urban}, \bibinfo{author}{A.~Patera},
\newblock \bibinfo{title}{An improved error bound for reduced basis
  approximation of linear parabolic problems},
\newblock \bibinfo{journal}{Mathematics of Computation} \bibinfo{volume}{83}
  (\bibinfo{year}{2014}) \bibinfo{pages}{1599--1615}.
\bibitem[{Yano(2014)}]{yano2014space}
\bibinfo{author}{M.~Yano},
\newblock \bibinfo{title}{A space-time petrov--galerkin certified reduced basis
  method: Application to the boussinesq equations},
\newblock \bibinfo{journal}{SIAM Journal on Scientific Computing}
  \bibinfo{volume}{36} (\bibinfo{year}{2014}) \bibinfo{pages}{A232--A266}.
\bibitem[{Yano et~al.(2014)Yano, Patera, and
  Urban}]{doi:10.1142/S0218202514500110}
\bibinfo{author}{M.~Yano}, \bibinfo{author}{A.~T. Patera},
  \bibinfo{author}{K.~Urban},
\newblock \bibinfo{title}{A space-time hp-interpolation-based certified reduced
  basis method for burgers' equation},
\newblock \bibinfo{journal}{Mathematical Models and Methods in Applied
  Sciences} \bibinfo{volume}{24} (\bibinfo{year}{2014})
  \bibinfo{pages}{1903--1935}. \URLprefix
  \url{https://doi.org/10.1142/S0218202514500110}.
  \DOIprefix\doi{10.1142/S0218202514500110}.
  \href{http://arxiv.org/abs/https://doi.org/10.1142/S0218202514500110}{{\tt
  arXiv:https://doi.org/10.1142/S0218202514500110}}.
\bibitem[{Baumann et~al.(2016)Baumann, Benner, and Heiland}]{baumann2016space}
\bibinfo{author}{M.~Baumann}, \bibinfo{author}{P.~Benner},
  \bibinfo{author}{J.~Heiland},
\newblock \bibinfo{title}{Space-time galerkin pod with application in optimal
  control of semi-linear parabolic partial differential equations},
\newblock \bibinfo{journal}{arXiv preprint arXiv:1611.04050}
  (\bibinfo{year}{2016}).
\bibitem[{Volkwein and Weiland(2006)}]{volkwein2006algorithm}
\bibinfo{author}{S.~Volkwein}, \bibinfo{author}{S.~Weiland},
\newblock \bibinfo{title}{An algorithm for galerkin projections in both time
  and spatial coordinates},
\newblock \bibinfo{journal}{Proc. 17th MTNS}  (\bibinfo{year}{2006}).
\bibitem[{Constantine and Wang(2012)}]{constantine2012residual}
\bibinfo{author}{P.~G. Constantine}, \bibinfo{author}{Q.~Wang},
\newblock \bibinfo{title}{Residual minimizing model interpolation for
  parameterized nonlinear dynamical systems},
\newblock \bibinfo{journal}{SIAM Journal on Scientific Computing}
  \bibinfo{volume}{34} (\bibinfo{year}{2012}) \bibinfo{pages}{A2118--A2144}.
\bibitem[{Choi and Carlberg(2019)}]{choi2019space}
\bibinfo{author}{Y.~Choi}, \bibinfo{author}{K.~Carlberg},
\newblock \bibinfo{title}{Space--time least-squares petrov--galerkin projection
  for nonlinear model reduction},
\newblock \bibinfo{journal}{SIAM Journal on Scientific Computing}
  \bibinfo{volume}{41} (\bibinfo{year}{2019}) \bibinfo{pages}{A26--A58}.
\bibitem[{Amsallem et~al.(2012)Amsallem, Zahr, and
  Farhat}]{amsallem2012nonlinear}
\bibinfo{author}{D.~Amsallem}, \bibinfo{author}{M.~J. Zahr},
  \bibinfo{author}{C.~Farhat},
\newblock \bibinfo{title}{Nonlinear model order reduction based on local
  reduced-order bases},
\newblock \bibinfo{journal}{International Journal for Numerical Methods in
  Engineering} \bibinfo{volume}{92} (\bibinfo{year}{2012})
  \bibinfo{pages}{891--916}.
\bibitem[{Amsallem et~al.(2015)Amsallem, Zahr, and Washabaugh}]{AmsZahWas15}
\bibinfo{author}{D.~Amsallem}, \bibinfo{author}{M.~J. Zahr},
  \bibinfo{author}{K.~Washabaugh},
\newblock \bibinfo{title}{Fast local reduced basis updates for the efficient
  reduction of nonlinear systems with hyper-reduction},
\newblock \bibinfo{journal}{Advances in Computational Mathematics}
  \bibinfo{volume}{41} (\bibinfo{year}{2015}) \bibinfo{pages}{1187--1230}.
  \URLprefix \url{https://doi.org/10.1007/s10444-015-9409-0}.
  \DOIprefix\doi{10.1007/s10444-015-9409-0}.
\bibitem[{Maday and R{\o}nquist(2002)}]{maday_rbe}
\bibinfo{author}{Y.~Maday}, \bibinfo{author}{E.~M. R{\o}nquist},
\newblock \bibinfo{title}{A reduced-basis element method},
\newblock \bibinfo{journal}{Journal of Scientific Computing}
  \bibinfo{volume}{17} (\bibinfo{year}{2002}) \bibinfo{pages}{447--459}.
  \URLprefix \url{https://doi.org/10.1023/A:1015197908587}.
  \DOIprefix\doi{10.1023/A:1015197908587}.
\bibitem[{Hoang et~al.(2020)Hoang, Choi, and
  Carlberg}]{hoang2020domaindecomposition}
\bibinfo{author}{C.~Hoang}, \bibinfo{author}{Y.~Choi},
  \bibinfo{author}{K.~Carlberg}, \bibinfo{title}{Domain-decomposition
  least-squares petrov-galerkin (dd-lspg) nonlinear model reduction},
  \bibinfo{year}{2020}. \href{http://arxiv.org/abs/2007.11835}{{\tt
  arXiv:2007.11835}}.
\bibitem[{Iapichino et~al.(2012)Iapichino, Quarteroni, and
  Rozza}]{IAPICHINO201263}
\bibinfo{author}{L.~Iapichino}, \bibinfo{author}{A.~Quarteroni},
  \bibinfo{author}{G.~Rozza},
\newblock \bibinfo{title}{A reduced basis hybrid method for the coupling of
  parametrized domains represented by fluidic networks},
\newblock \bibinfo{journal}{Computer Methods in Applied Mechanics and
  Engineering} \bibinfo{volume}{221-222} (\bibinfo{year}{2012})
  \bibinfo{pages}{63 -- 82}. \URLprefix
  \url{http://www.sciencedirect.com/science/article/pii/S0045782512000473}.
  \DOIprefix\doi{https://doi.org/10.1016/j.cma.2012.02.005}.
\bibitem[{Phuong~Huynh et~al.(2013)Phuong~Huynh, Knezevic, and
  Patera}]{phuong_huynh_knezevic_patera_2013}
\bibinfo{author}{D.~B. Phuong~Huynh}, \bibinfo{author}{D.~J. Knezevic},
  \bibinfo{author}{A.~T. Patera},
\newblock \bibinfo{title}{A static condensation reduced basis element method :
  approximation and a posteriori error estimation},
\newblock \bibinfo{journal}{ESAIM: Mathematical Modelling and Numerical
  Analysis} \bibinfo{volume}{47} (\bibinfo{year}{2013})
  \bibinfo{pages}{213–251}. \DOIprefix\doi{10.1051/m2an/2012022}.
\bibitem[{Berkooz et~al.(1993)Berkooz, Holmes, and
  Lumley}]{berkooz_turbulence_pod}
\bibinfo{author}{G.~Berkooz}, \bibinfo{author}{P.~Holmes},
  \bibinfo{author}{J.~L. Lumley},
\newblock \bibinfo{title}{The proper orthogonal decomposition in the analysis
  of turbulent flows},
\newblock \bibinfo{journal}{Annu. Rev. Fluid Mech.} \bibinfo{volume}{25}
  (\bibinfo{year}{1993}) \bibinfo{pages}{539--575}.
\bibitem[{Rathinam and Petzold(2003)}]{rathinam2003new}
\bibinfo{author}{M.~Rathinam}, \bibinfo{author}{L.~R. Petzold},
\newblock \bibinfo{title}{A new look at proper orthogonal decomposition},
\newblock \bibinfo{journal}{SIAM Journal on Numerical Analysis}
  \bibinfo{volume}{41} (\bibinfo{year}{2003}) \bibinfo{pages}{1893--1925}.
\bibitem[{Bach et~al.(2018)Bach, Song, Erhart, and Duddeck}]{bach2018stability}
\bibinfo{author}{C.~Bach}, \bibinfo{author}{L.~Song},
  \bibinfo{author}{T.~Erhart}, \bibinfo{author}{F.~Duddeck},
  \bibinfo{title}{Stability conditions for the explicit integration of
  projection based nonlinear reduced-order and hyper reduced structural
  mechanics finite element models}, \bibinfo{year}{2018}.
  \href{http://arxiv.org/abs/1806.11404}{{\tt arXiv:1806.11404}}.
\bibitem[{Chaturantabut and Sorensen(2010)}]{deim}
\bibinfo{author}{S.~Chaturantabut}, \bibinfo{author}{D.~C. Sorensen},
\newblock \bibinfo{title}{{Nonlinear Model Reduction via Discrete Empirical
  Interpolation}},
\newblock \bibinfo{journal}{SIAM J. Sci. Comput.} \bibinfo{volume}{32}
  (\bibinfo{year}{2010}) \bibinfo{pages}{2737--2764}.
\bibitem[{Zahr and Farhat(2015)}]{zahr2015progressive}
\bibinfo{author}{M.~J. Zahr}, \bibinfo{author}{C.~Farhat},
\newblock \bibinfo{title}{Progressive construction of a parametric
  reduced-order model for pde-constrained optimization},
\newblock \bibinfo{journal}{International Journal for Numerical Methods in
  Engineering} \bibinfo{volume}{102} (\bibinfo{year}{2015})
  \bibinfo{pages}{1111--1135}.
\bibitem[{Fidkowski(2015)}]{Fidkowski_2015_VKI}
\bibinfo{author}{K.~J. Fidkowski},
\newblock \bibinfo{title}{Output-based error estimation and mesh adaptation for
  steady and unsteady flow problems},
\newblock in: \bibinfo{editor}{H.~Deconinck}, \bibinfo{editor}{T.~Horvath}
  (Eds.), \bibinfo{booktitle}{38$^{\mbox{th}}$ Advanced CFD Lectures Series;
  Von Karman Institute for Fluid Dynamics (September 14--16 2015)},
  \bibinfo{publisher}{von Karman Institute for Fluid Dynamics},
  \bibinfo{year}{2015}.
\bibitem[{Fidkowski(2016)}]{xflow}
\bibinfo{author}{K.~Fidkowski}, \bibinfo{title}{Xflow: A solution-adaptive
  code},
  \bibinfo{howpublished}{\textsc{URL:}~\url{https://how4.cenaero.be/system/files/filedepot/10/Code_XFlow_Fidkowski.pdf}},
  \bibinfo{year}{2016}.

\end{thebibliography}
